\documentclass[a4paper,12pt]{smfbook}

\usepackage{amssymb}
\usepackage{url}
\usepackage[T1]{fontenc}
\RequirePackage{calrsfs}
\DeclareSymbolFont{rsfscript}{OMS}{rsfs}{m}{b}
\DeclareSymbolFontAlphabet{\mathrsfs}{rsfscript}
\usepackage[utopia,expert]{mathdesign}
\usepackage{lscape}

\usepackage[leftbars]{changebar}

\usepackage{palatino}
\usepackage{rotating}
\usepackage{graphicx}

\usepackage{enumerate}

\usepackage{color}
\definecolor{shadecolor}{gray}{0.90}

\def\bfit{\bfseries\itshape}

\def\proofname{D\'emonstration}

\input xypic
\xyoption{all}
\xyoption{arc}

\makeindex
\def\indexnot#1#2{\index{#1@$#2$ |textbf  {\hskip0.5cm} \textbf }}
\def\indexname{Index des notations}

\newtheorem{theo}{Th\'eor\`eme}[section]
\def\thetheo{\thesection.\arabic{theo}}
\newtheorem{prop}[theo]{Proposition}

\newtheorem{lem}[theo]{Lemme}

\newtheorem{coro}[theo]{Corollaire}

\newtheorem{defi}[theo]{D\'efinition}

\newtheorem{defivide}{D\'efinition}

\renewcommand\theequation{\thetheo}
\def\equat{\refstepcounter{theo}\begin{equation}}
\def\endequat{\end{equation}}

\renewcommand\thetable{\thetheo}

\renewcommand\thesection{\thechapter.\arabic{section}}
\renewcommand\thesubsection{\thesection.\Alph{subsection}}

\def\contentsname{Table des mati\`eres}
\def\listfigurename{Liste des figures}
\def\listtablename{Liste des tables}
\def\appendixname{Appendice}
\def\partname{Chapitre}

\def\bibname{Bibliographie}

\def\AG{{\mathfrak A}}  \def\aG{{\mathfrak a}}  \def\AM{{\mathbb{A}}}
    
\def\CG{{\mathfrak C}}    \def\CM{{\mathbb{C}}}

    \def\FM{{\mathbb{F}}}

\def\IG{{\mathfrak I}}

  \def\mG{{\mathfrak m}}  
    \def\NM{{\mathbb{N}}}
    
  \def\pG{{\mathfrak p}}  \def\PM{{\mathbb{P}}}
  \def\qG{{\mathfrak q}}  \def\QM{{\mathbb{Q}}}
  \def\rG{{\mathfrak r}}  \def\RM{{\mathbb{R}}}
\def\SG{{\mathfrak S}}

    \def\XM{{\mathbb{X}}}
    
  \def\zG{{\mathfrak z}}  \def\ZM{{\mathbb{Z}}}

\def\Ab{{\mathbf A}}  \def\ab{{\mathbf a}}  \def\AC{{\mathcal{A}}}
  \def\bb{{\mathbf b}}  \def\BC{{\mathcal{B}}}
\def\Cb{{\mathbf C}}    
  \def\db{{\mathbf d}}  \def\DC{{\mathcal{D}}}
\def\Eb{{\mathbf E}}  \def\eb{{\mathbf e}}  \def\EC{{\mathcal{E}}}
    \def\FC{{\mathcal{F}}}
\def\Gb{{\mathbf G}}    \def\GC{{\mathcal{G}}}
\def\Hb{{\mathbf H}}    \def\HC{{\mathcal{H}}}
    
\def\Jb{{\mathbf J}}    
\def\Kb{{\mathbf K}}  \def\kb{{\mathbf k}}  \def\KC{{\mathcal{K}}}
\def\Lb{{\mathbf L}}    \def\LC{{\mathcal{L}}}
\def\Mb{{\mathbf M}}    \def\MC{{\mathcal{M}}}
    \def\NC{{\mathcal{N}}}
    \def\OC{{\mathcal{O}}}
\def\Pb{{\mathbf P}}    \def\PC{{\mathcal{P}}}
  \def\qb{{\mathbf q}}  
    
\def\Sb{{\mathbf S}}  \def\sb{{\mathbf s}}  
\def\Tb{{\mathbf T}}  \def\tb{{\mathbf t}}  \def\TC{{\mathcal{T}}}
  \def\ub{{\mathbf u}}

  \def\xb{{\mathbf x}}  
    
    \def\ZC{{\mathcal{Z}}}

  \def\arm{{\mathrm{a}}}  
  \def\brm{{\mathrm{b}}}  
\def\Crm{{\mathrm{C}}}  \def\crm{{\mathrm{c}}}  \def\CCB{{\boldsymbol{\mathcal{C}}}}

\def\Grm{{\mathrm{G}}}    
    \def\HCB{{\boldsymbol{\mathcal{H}}}}

\def\Lrm{{\mathrm{L}}}    
\def\Mrm{{\mathrm{M}}}    \def\MCB{{\boldsymbol{\mathcal{M}}}}

    \def\PCB{{\boldsymbol{\mathcal{P}}}}
\def\Qrm{{\mathrm{Q}}}    \def\QCB{{\boldsymbol{\mathcal{Q}}}}
    \def\RCB{{\boldsymbol{\mathcal{R}}}}
    
\def\Trm{{\mathrm{T}}}

    \def\XCB{{\boldsymbol{\mathcal{X}}}}
    \def\YCB{{\boldsymbol{\mathcal{Y}}}}
\def\Zrm{{\mathrm{Z}}}    \def\ZCB{{\boldsymbol{\mathcal{Z}}}}

  \def\gti{{\tilde{g}}}

\def\Qti{{\tilde{Q}}}    
\def\Rti{{\tilde{R}}}    \def\RCt{{\tilde{\mathcal{R}}}}

\def\Zti{{\tilde{Z}}}  \def\zti{{\tilde{z}}}  

  \def\ahat{{\hat{a}}}

  \def\ehat{{\hat{e}}}

\def\Qhat{{\hat{Q}}}

  \def\zhat{{\hat{z}}}

          \def\aba{{\bar{a}}}
          \def\bba{{\bar{b}}}
          
\def\Dba{{\bar{D}}}          
          \def\eba{{\bar{e}}}
\def\Fba{{\bar{F}}}          \def\fba{{\bar{f}}}
\def\Gba{{\bar{G}}}          \def\gba{{\bar{g}}}
          \def\hba{{\bar{h}}}
\def\Iba{{\bar{I}}}          \def\iba{{\bar{i}}}

\def\Pba{{\bar{P}}}          
          
          \def\rba{{\bar{r}}}
          
          \def\tba{{\bar{t}}}
          
          \def\vba{{\bar{v}}}

\def\Zba{{\bar{Z}}}          \def\zba{{\bar{z}}}

\def\Tov{{\overline{T}}}

\def\Hbt{{\widetilde{\Hb}}}

\def\Mbt{{\tilde{\Mb}}}

\def\a{\alpha}
\def\b{\beta}
\def\g{\gamma}
\def\G{\Gamma}
\def\d{\delta}
\def\D{\Delta}
\def\e{\varepsilon}
\def\ph{\varphi}
\def\l{\lambda}
\def\L{\Lambda}

\def\o{\omega}
\def\O{\Omega}
\def\r{\rho}
\def\s{\sigma}
\def\Sig{\Sigma}
\def\th{\theta}
\def\Th{\Theta}
\def\t{\tau}

\def\z{\zeta}

\def\betb{{\boldsymbol{\beta}}}

\def\delb{{\boldsymbol{\delta}}}        
\def\Delb{{\boldsymbol{\Delta}}}        
   
        \def\pht{{\tilde{\varphi}}}

\def\mub{{\boldsymbol{\mu}}}            \def\mut{{\tilde{\mu}}}
            
\def\omeb{{\boldsymbol{\omega}}}        
\def\Omeb{{\boldsymbol{\Omega}}}        
\def\pib{{\boldsymbol{\pi}}}            
            
          \def\rhot{{\tilde{\rho}}}
\def\sigb{{\boldsymbol{\sigma}}}

\def\taub{{\boldsymbol{\tau}}}          \def\taut{{\tilde{\t}}}
\def\xib{{\boldsymbol{\xi}}}

               \def\muh{{\hat{\mu}}}
               
\def\omeba{{\bar{\omega}}}           
           
\def\piba{{\bar{\pi}}}               
               
\def\rhoba{{\bar{\rho}}}

\def\pibt{{\tilde{\boldsymbol{\pi}}}}

\def\mubt{{\tilde{\boldsymbol{\mu}}}}

\def\mubh{{\hat{\boldsymbol{\mu}}}}

\DeclareMathOperator{\Aut}{{\mathrm{Aut}}}

\DeclareMathOperator{\End}{{\mathrm{End}}}

\DeclareMathOperator{\Gal}{{\mathrm{Gal}}}
\DeclareMathOperator{\Hom}{{\mathrm{Hom}}}

\DeclareMathOperator{\Id}{{\mathrm{Id}}}

\DeclareMathOperator{\im}{{\mathrm{Im}}}

\DeclareMathOperator{\Ind}{{\mathrm{Ind}}}

\DeclareMathOperator{\Irr}{{\mathrm{Irr}}}
\DeclareMathOperator{\Ker}{{\mathrm{Ker}}}

\DeclareMathOperator{\Mat}{{\mathrm{Mat}}}

\DeclareMathOperator{\Rad}{{\mathrm{Rad}}}
\DeclareMathOperator{\Res}{{\mathrm{Res}}}

\DeclareMathOperator{\Spec}{{\mathrm{Spec}}}

\DeclareMathOperator{\Stab}{{\mathrm{Stab}}}

\DeclareMathOperator{\Tor}{{\mathrm{Tor}}}

\DeclareMathOperator{\val}{{\mathrm{val}}}
\DeclareMathOperator{\disc}{{\mathrm{disc}}}
\DeclareMathOperator{\classe}{{\mathrm{Cl}}}
\DeclareMathOperator{\longueur}{{\mathrm{Length}}}

\def\to{\rightarrow}
\def\longto{\longrightarrow}
\def\injto{\hookrightarrow}

\def\fonction#1#2#3#4#5{\begin{array}{rccc}
{#1} : & {#2} & \longto & {#3} \\
& {#4} & \longmapsto & {#5} 
\end{array}}

\def\fonctio#1#2#3#4{\begin{array}{ccc}
{#1} & \longto & {#2} \\
{#3} & \longmapsto & {#4} 
\end{array}}

\def\bijectio#1#2#3#4{\begin{array}{ccc}
{#1} & \stackrel{\sim}{\longto} & {#2} \\
{#3} & \longmapsto & {#4} 
\end{array}}

\def\DS{\displaystyle}
\def\SS{\scriptstyle}
\def\SSS{\scriptscriptstyle}

\def\finl{~$\blacksquare$}

\def\lexp#1#2{\kern\scriptspace\vphantom{#2}^{#1}\kern-\scriptspace#2}
\def\le{\hspace{0.1em}\mathop{\leqslant}\nolimits\hspace{0.1em}}
\def\ge{\hspace{0.1em}\mathop{\geqslant}\nolimits\hspace{0.1em}}

\mathchardef\inferieur="321E
\mathchardef\superieur="321F

\def\eqna{\begin{eqnarray*}}
\def\endeqna{\end{eqnarray*}}

\def\itemth#1{\item[${\mathrm{(#1)}}$]}

\catcode`\@=11
\long\def\@car#1#2\@nil{#1}
\long\def\@first#1#2{#1}
\long\def\@second#1#2{#2}
\long\def\ifempty#1{\expandafter\ifx\@car#1@\@nil @\@empty
  \expandafter\@first\else\expandafter\@second\fi}
\catcode`\@=12

\def\Jbov{{\bar{\Jb}}}
\def\Hbov{{\bar{\Hb}}}
\def\Kbov{{\bar{\Kb}}}
\def\Lbov{{\bar{\Lb}}}
\def\Mbov{{\bar{\Mb}}}

\def\pGt{{\tilde{\pG}}}
\def\qGt{{\tilde{\qG}}}

\def\rGt{{\tilde{\rG}}}
\def\GL{{\mathrm{GL}}}

\DeclareMathOperator{\Ref}{R\acute{e}f}

\def\boitegrise#1#2{\begin{centerline}{\fcolorbox{black}{shadecolor}{~
    \begin{minipage}[t]{#2}{\vphantom{~}#1\vphantom{$A_{\DS{A_A}}$}}
            \end{minipage}~}}\end{centerline}\medskip}

\def\ve{{\SSS{\vee}}}
\def\cow{{\mathrm{co}(W)}}
\def\surto{\twoheadrightarrow}
\def\module{{\text{-}}{\mathrm{mod}}}

\theoremstyle{remark}
\newtheorem{rema}[theo]{Remarque}
\newtheorem{remavide}{Remarque}

\newtheorem{exemple}[theo]{Exemple}
\newtheorem{contre}[theo]{Contre-exemple}
\theoremstyle{plain}
\newtheorem{question}[theo]{Question}
\newtheorem{probleme}[theo]{Probl\`eme}

\newtheorem{conjecturem}{Conjecture \FAM}

\newtheorem{conjecturebil}{Conjecture \BIL}

\newtheorem{conjectureleft}{Conjecture \GAUCHE}

\newtheorem{conjecturecar}{Conjecture \CAR}

\def\Frac{{\mathrm{Frac}}}

\def\BIL{LR}
\def\GAUCHE{L}
\def\CAR{CAR}
\def\FAM{FAM}

\def\groth{\KC_0}
\DeclareMathOperator{\dec}{d\acute{e}c}

\def\reg{{\mathrm{r\acute{e}g}}}
\def\blocs{{\mathrm{Idem_{pr}}}}
\def\refw{{\Ref(W)/W}}
\def\grad{{\mathrm{gr}}}
\def\gradauto{{\mathrm{bigr}}}
\def\alg{{\mathrm{alg}}}
\def\euler{{\eb\ub}}
\def\eulerq{{\mathrm{eu}}}
\def\eulertilde{\widetilde{\euler}}
\def\casimir{{\mathrm{cas}}}

\def\pGba{\bar{\pG}}
\def\plus{{\hskip1mm +}}
\def\moins{{\hskip1mm -}}

\def\rGba{{\bar{\rG}}}

\def\qGba{{\bar{\qG}}}
\def\zGba{{\bar{\zG}}}
\def\decba{\overline{\dec}}

\def\calo{{\Crm\Mrm}}
\def\eval{{\mathrm{ev}}}
\def\KER{\KC \!\!\! e\!\! r}

\def\xyinj{\ar@{^{(}->}}
\def\xysur{\ar@{->>}}

\def\gauche{{\mathrm{left}}}
\def\droite{{\mathrm{right}}}

\def\isomorphisme#1{{\boldsymbol{[}}\hskip0.5mm #1\hskip0.5mm {\boldsymbol{]}}}
\def\res{{\mathrm{res}}}
\def\gen{{\mathrm{g\acute{e}n}}}
\def\parti{{\mathrm{par}}}
\def\bigrad{{\mathrm{bigr}}}
\def\partname{Chapitre}
\DeclareMathOperator{\carac}{{\mathrm{Car}}}
\bigskip\def\unb{{\boldsymbol{1}}}

\def\petitespace{\vphantom{$\DS{\frac{\DS{A^A}}{\DS{A_A}}}$}}
\def\trespetitespace{\vphantom{$\DS{\frac{\DS{A}}{\DS{A}}}$}}

\def\kl{{\mathrm{KL}}}

\def\singulier{{\mathrm{sing}}}
\def\ramif{{\mathrm{ram}}}
\def\cmcellules{\lexp{\calo}{\mathrm{Cell}}}

\def\klcellules{\lexp{\kl}{\mathrm{Cell}}}

\def\mult{{\mathrm{mult}}}

\def\cyclo{{\mathrm{cyc}}}
\def\heckegenerique{\HCB_{\!\!\! W}^\gen}
\def\heckecyclotomique{\HCB_{\!\!\! W}^\cyclo}

\makeatletter
\def\hlinewd#1{%
\noalign{\ifnum0=`}\fi\hrule \@height #1 %
\futurelet\reserved@a\@xhline}
\makeatother

\newlength\epaisLigne
\newcommand\clinewd[2]{\noalign{\global\epaisLigne\arrayrulewidth\global\arrayrulewidth #1}%
\cline{#2} \noalign{\global\arrayrulewidth\epaisLigne}}

\usepackage{multirow}

\def\prel{\leqslant_{L}^c}

\def\prer{\leqslant_{R}^c}
\def\prelr{\leqslant_{LR}^c}
\def\rell{\stackrel{L,c}{\longleftarrow}}

\def\relr{\stackrel{R,c}{\longleftarrow}}
\def\siml{\sim_{L}^{\kl,c}}
\def\simr{\sim_{R}^{\kl,c}}
\def\simlr{\sim_{LR}^{\kl,c}}

\def\CCBt{{\widetilde{\CCB}}}
\def\CGt{{\widetilde{\CG}}}

\def\copie{{\mathrm{cop}}}
\def\iso{{\mathrm{iso}}_0}

\newcommand{\longiso}{\stackrel{\sim}{\longrightarrow}}
\newcommand{\longbij}{\stackrel{\sim}{\leftrightarrow}}

\def\MCov{{\bar{\MC}}}
\def\LCov{{\bar{\LC}}}

\def\schur{{\mathrm{sch}}}
\def\carac{{\mathrm{car}}}

\def\attractif{{\mathrm{att}}}
\def\repulsif{{\mathrm{r\acute{e}p}}}
\def\limiteattractive{{{\mathrm{lim}}_{{\mathrm{att}}}}}
\def\limiterepulsive{{{\mathrm{lim}}_{{\mathrm{r\acute{e}p}}}}}
\def\limiteattractiveinverse{{{\mathrm{lim}}_{{\mathrm{att}}}^{-1}}}

\def\limitegauche{{{\mathrm{lim}}_{{\mathrm{left}}}}}

\def\sym{{\mathrm{sym}}}

\begin{document}

\baselineskip=16pt

\title{Cellules de Calogero-Moser}

\author{{\sc C\'edric Bonnaf\'e}}
\address{
Institut de Math\'ematiques et de Mod\'elisation de Montpellier (CNRS: UMR 5149), 
Universit\'e Montpellier 2,
Case Courrier 051,
Place Eug\`ene Bataillon,
34095 MONTPELLIER Cedex,
FRANCE} 

\makeatletter
\email{cedric.bonnafe@univ-montp2.fr}
\makeatother

\author{{\sc Rapha\"el Rouquier}}

\address{UCLA Mathematics Department
Los Angeles, CA 90095-1555, 
USA}
\email{rouquier@math.ucla.edu}



\date{\today}

\thanks{The first author is partly supported by the ANR (Project No ANR-12-JS01-0003-01 ACORT)}


\maketitle

\pagestyle{myheadings}

\markboth{\sc C. Bonnaf\'e \& R. Rouquier}{\sc Cellules de Calogero-Moser}

\tableofcontents

%
%
%
%
%
%
%
%
%
%
%
%
%
%
%
%
%
%

\chapter*{Introduction}

\vskip-2cm

Ce m\'emoire est consacr\'e \`a l'\'etude d'un rev\^etement de la vari\'et\'e
de Calogero-Moser associ\'ee par Etingof et Ginzburg \`a un groupe de r\'eflexion complexe. La
ramification de ce rev\^etement donne lieu \`a des partitions du groupe de r\'eflexion
dont nous conjecturons qu'elles co{\"\i}ncident avec
les cellules de Kazhdan-Lusztig, dans le cas d'un groupe de Coxeter.

\smallskip

\'Etant donn\'e un groupe fini de (pseudo-)r\'eflexion $W$ non
trivial agissant
sur un espace vectoriel complexe de dimension finie $V$, la vari\'et\'e quotient
$(V\times V^*)/\Delta W$ par l'action diagonale de $W$ est une vari\'et\'e
singuli\`ere. C'est un rev\^etement ramifi\'e de la vari\'et\'e lisse
$V/W\times V^*/W$, qui est en fait un espace affine. Etingof et Ginzburg
ont construit une d\'eformation $\Upsilon: \ZCB \longto \PCB$ de ce rev\^etement.
La vari\'et\'e $\ZCB$ est l'espace de Calogero-Moser, construit comme spectre
du centre de la $\CM[\PC]$-alg\`ebre de Cherednik rationelle $\Hb$ en $t=0$
associ\'ee \`a $W$. 
La vari\'et\'e $\PC$ est produit de $V/W\times V^*/W$ par un espace vectoriel 
$\CCB$ de param\`etres de base l'ensemble des classes de conjugaison de
r\'eflexions de $W$.  La sp\'ecialisation  en $0\in\CCB$ de $\Hb$ est
l'alg\`ebre $\CM[V\oplus V^*]\rtimes W$.

Le rev\^etement $\Upsilon$, de degr\'e $|W|$, n'est pas galoisien (sauf si 
$W=(\ZM/2)^n$). L'objet de notre travail est l'\'etude d'une cl\^oture
galoisienne $\RCB$ de ce rev\^etement et de la ramification au-dessus des
sous-vari\'et\'es ferm\'ees
$0\times 0$, $0\times V^*/W$ et $V/W\times 0$ de $V/W\times V^*/W$.

Soit $G$ le groupe de Galois de $\RCB \longto \PCB$.
En $0\in\CCB$, le rev\^etement $(V\times V^*)/\Delta W\to V/W\times V^*/W$
admet $(V\times V^*)/\Delta Z(W)$ comme cl\^oture galoisienne. Nous montrons
comment en d\'eduire un plongement de $G$ comme groupe de
permutations de $W$. 

\smallskip

Ceci se reformule en termes de repr\'esentations de $\Hb$: la $\CM(\PCB)$-alg\`ebre
semi-simple $\CM(\PCB) \otimes_{\CM[\PCB]} \Hb$ n'est pas
d\'eploy\'ee et $\CM(\RCB)$ est un corps de d\'eploiement. Les
$\CM(\RCB) \otimes_{\CM[\PCB]} \Hb$-modules simples sont en bijection avec 
$W$. Notre travail consiste alors en l'\'etude de la r\'epartition en blocs
de ces modules, selon le choix d'un id\'eal premier de $\CM[\RCB]$.

\smallskip

Soit $X$ une sous-vari\'et\'e ferm\'ee
irr\'eductible de $\RCB$. Nous d\'efinissons les $X$-cellules de $W$
comme les orbites du groupe d'inertie de $X$. Etant donn\'e
un choix de param\`etre $c\in\CCB$, nous \'etudions les $c$-cellules
bilat\`eres, d\'efinies pour un $X$ composante irr\'eductible de l'image
inverse de $\bar{X}=c\times 0\times 0$. Nous
\'etudions aussi les $c$-cellules \`a gauche (cas
$\bar{X}=c\times V/W\times 0$) et
\`a droite (cas $\bar{X}=c\times 0\times V^*/W$).

\smallskip

Lorsque $W$ est un groupe de Coxeter, nous conjecturons que ces $c$-cellules
co{\"\i}ncident avec les cellules de Kazhdan-Lusztig, d\'efinies par
Kazhdan-Lusztig~\cite{KL} et Lusztig~\cite{Lu2},~\cite{lusztig}. Ceci d\'epend d'un choix
appropri\'e de $X$ dans une $G$-orbite.

\smallskip

Nous analysons le cas o\`u $W$ est cyclique ($\dim V=1$)~: c'est le
seul cas o\`u nous disposons d'une description compl\`ete de tous les objets \'etudi\'es 
dans ce m\'emoire. Nous \'etudions aussi en d\'etail le cas o\`u
$W$ est un groupe de Weyl de type $B_2$~: 
le groupe de Galois est un groupe de Weyl de type $D_4$ et nous d\'emontrons
que les cellules de Calogero-Moser co{\"\i}ncident avec les cellules de
Kazhdan-Lusztig. Notre approche repose sur une \'etude d\'etaill\'ee de
la vari\'et\'e $\ZCB$ et de la ramification, sans toutefois construire
la vari\'et\'e $\RCB$.

Etingof et Ginzburg ont introduit un \'el\'ement de $\CM[\ZCB]$ qui
d\'eforme l'\'el\'ement d'Euler. Nous d\'emontrons que $G$ est le groupe
de Galois de son polyn\^ome minimal. Cet \'el\'ement joue un r\^ole
important pour l'\'etude de la ramification, mais ne suffit pas \`a
s\'eparer les cellules.

\smallskip

Nous montrons que l'ensemble des $c$-cellules bilat\`eres est en bijection
avec l'ensemble des blocs de l'alg\`ebre de Cherednik r\'eduite
$\Hbov_c$, sp\'ecialisation de $\Hb$ en $(c,0,0) \in \CCB \times V/W\times V^*/W =\PCB$. Cet
ensemble est en bijection avec $\Upsilon^{-1}(c, 0,0)$.
A tout $\CM W$-module simple $E$ est associ\'e une repr\'esentation
ind\'ecomposable de type Verma (aussi appel\'ee ``Baby Verma'') $\MCov_c(E)$ 
de $\Hbov$, d'unique quotient simple $\LCov_c(E)$~\cite{gordon}.
La partition en blocs de ces modules donne une partition de $\Irr(W)$ en 
{\it familles de Calogero-Moser}, qui sont conjecturalement reli\'ees
aux familles de l'alg\`ebre de Hecke de $W$ (voir~\cite{gordon B, gordon martino, bellamy, martino, martino 2}). 
Nous d\'emontrons que, dans une famille de Calogero-Moser donn\'ee, la matrice
des multiplicit\'es $[\MCov_c(E):\LCov_c(F)]$ est de rang $1$, une propri\'et\'e
conjectur\'ee par Ulrich Thiel.

Nous introduisons une notion de module simple cellulaire, associ\'e \`a
une cellule \`a gauche. Nous conjecturons que la multiplicit\'e d'un tel
module simple dans un module de Verma \`a gauche associ\'e \`a une repr\'esentation irr\'eductible 
$E$ de $W$ est la m\^eme que la
multiplicit\'e de $E$ dans la
repr\'esentation cellulaire de Kazhdan-Lusztig, lorsque $W$ est un groupe
de Coxeter. Nous \'etudions le cas d'une cellule bilat\`ere associ\'ee \`a
un point lisse de $\Upsilon^{-1}(c,0,0)$~: dans ce cas, Gordon
a d\'emontr\'e que le bloc correspondant contient un unique b\'eb\'e module de Verma
(voir la section~\ref{subsection:spe-hbar} pour la d\'efinition). 
Nous d\'emontrons que la multiplicit\'e d'un module simple cellulaire 
dans le module de Verma \`a gauche correspondant est \'egale \`a $1$ (pour une cellule 
\`a gauche contenue dans la cellule bilat\`ere donn\'ee)~: ceci 
constitue notre meilleur argument en faveur de nos conjectures. 

\bigskip

\noindent{\sc Commentaire - } Cette introduction a \'et\'e \'ecrite dans un langage 
{\it g\'eom\'etrique}, car il s'adapte mieux \`a un pr\'esentation rapide. 
Le m\'emoire est quant \`a lui \'ecrit plut\^ot dans un langage {\it alg\'ebrique}, 
m\^eme si les r\'ef\'erences \`a la g\'eom\'etrie sous-jacente sont nombreuses.

\bigskip

\noindent{\sc Remerciements - } 
Nous tenons \`a remercier chaleureusement G. Bellamy, I. Gordon, M. Martino et U. Thiel pour 
les nombreuses discussions que nous avons eues depuis trois ans et pour leurs \'eclaicissements 
sur les points les plus d\'elicats de la th\'eorie des repr\'esentations des alg\`ebres de 
Cherednik.  

Nous remercions G. Malle pour nous avoir fait profiter de ses comp\'etences 
en th\'eorie de Galois, et pour une lecture attentive d'une version pr\'eliminaire de ce 
m\'emoire~: ses nombreuses suggestions ont permis d'am\'eliorer grandement le manuscrit. 

Nous remercions aussi M. Chlouveraki pour des pr\'ecisions n\'ecessaires sur 
les familles de Hecke.

\chapter*{English Introduction}

\vskip-2cm

This memoir is devoted to the study of a covering of the Calogero-Moser space
associated with a complex reflection group by Etingof and Ginzburg.
The ramification of this covering gives rise to partitions
of the set of elements of the reflection group. We conjecture that these
partitions coincide with the Kazhdan-Lusztig cells, in the case of a
Coxeter group.

Given a non-trivial finite group acting on a finite-dimensional complex
vector space $V$ and generated by (pseudo-)reflections, the quotient
variety $(V\times V^*)/\Delta W$ by the diagonal action of $W$ is singular.
It is a ramified covering of the smooth variety $V/W\times V^*/W$, which
is actually an affine space. Etingof and Ginzburg have constructed
a deformation $\Upsilon:\ZCB\to\PCB$ of this covering. The variety
$\ZCB$ is the Calogero-Moser space, defined as the spectrum of the center
of the rational Cherednik $\CM[\PCB]$-algebra $\Hb$ associated with
$W$ at $t=0$. The variety $\PCB$ is the product of $V/W\times V^*/W$ by
a vector space $\CCB$ of parameters with basis the set of conjugacy classes
of reflections of $W$. The specialization at $0\in\CCB$ of $\Hb$ is
the algebra $\CM[V\oplus V^*]\rtimes W$.

The covering $\Upsilon$, of degree $|W|$, is not Galois (unless
$W=(\ZM/2)^n$). Our work is a study of a Galois closure $\RCB$ of this
covering and of the ramification above the closed subvarieties
$0\times 0$, $0\times V^*/W$ and $V/W\times 0$ of $V/W\times V^*/W$.

Let $G$ be the Galois group of $\RCB\to\PCB$. At $0\in\CCB$, a
Galois closure of the covering $(V\times V^*)/\Delta W\to V/W\times V^*/W$ 
is given by $(V\times V^*)/\Delta Z(W)$. We show that this provides an
embedding of $G$ as a group of permutations of $W$.

This can be reformulated in terms of representations of $\Hb$: the
semi-simple $\CM(\PCB)$-algebra $\CM(\PCB)\otimes_{\CM[\PCB]}\Hb$
is not split and $\CM(\RCB)$ is a splitting field. The simple
$\CM(\RCB)\otimes_{\CM[\PCB]}\Hb$-modules are in bijection with $W$. Our
work is devoted to the study of the partition into blocks of these
modules, given a prime ideal of $\CM[\RCB]$.

Let $X$ be an irreducible closed subvariety of $\RCB$. We define the
$X$-cells of $W$ as the orbits of the inertia group of $X$. Given a
parameter $c\in\CCB$, we study the two-sided $c$-cells, defined
for $X$ an irreducible component of the inverse image of $\bar{X}=
c\times 0\times 0$. We also study the left $c$-cells (where
$\bar{X}=c\times V/W\times 0$) and the right $c$-cells (where
$\bar{X}=c\times 0\times V/W$).

When $W$ is a Coxeter group, we conjecture that the $c$-cells coincide with
the Kazhdan-Lusztig cells, defined by Kazhdan-Lusztig [KaLu] and
Lusztig [Lus1], [Lus3]. This depends on the choice of an appropriate
$X$ in a $G$-orbit.

We analyze in detail the case $W$ cyclic ($\dim V=1$): this is the only
case where we have a complete description of the objects studied in this
memoir. We also provide a detailed study of the case of a Weyl group $W$ of
type $B_2$: the Galois group is a Weyl group of type $D_4$ and we show
that the Calogero-Moser cells coincide with the Kazhdan-Lusztig cells. Our
approach is based on a detailed study of $\ZCB$ and of the ramification
of the covering, without constructing explicitely the variety $\RCB$.

Etingof and Ginzburg have introduced a deformation of the Euler vector
field. We show that $G$ is the Galois group of its minimal polynomial. This
element plays an important role in the study of ramification, but is not
enough to separate cells.

We construct a bijection between the set of two-sided $c$-cells and the
set of blocks of the restricted rational Cherednik algebra
$\bar{\Hb}_c$ (the specialization of $\Hb$ at 
$(c,0,0)\in\CCB\times V/W\times V^*/W=\PCB$). This latter set is in bijection
with $\Upsilon^{-1}(c,0,0)$. Given a simple $\CM W$-module
$E$, there is an indecomposable representation of Verma type (a ``Baby-Verma''
module) $\bar{\MC}_c^\gauche(E)$ of $\bar{\Hb}$ with a unique simple
quotient $\bar{\LC}_c^\gauche(E)$ [Gor1]. The partition into blocks of
those modules gives a partition of $\Irr(W)$ into {\em Calogero-Moser
families}, which are conjecturally related to the families of
the Hecke algebra of $W$ (cf [Gor2, GoMa, Bel5, Mar1, Mar2]). We show that,
in a given Calogero-Moser family, the matrix of multiplicities
$[\bar{\MC}_c^\gauche(E):\bar{\LC}_c^\gauche(F)]$ has rank $1$, a property
conjectured by Ulrich Thiel.

We introduce a notion of simple cell module associated with a left cell.
We conjecture that the multiplicity of such a simple module in a left Verma
module is the same as the multiplicity of $E$ in
the cell representation of Kazhdan-Lusztig, when $W$ is a Coxeter group.
We study two-sided cells associated with a smooth point of
$\Upsilon^{-1}(c,0,0)$: in that case, Gordon has shown that the corresponding
block contains a unique Verma module. We show
that the multiplicity of any simple cell module in that Verma module is $1$
(for a left cell contained in the given two-sided cell). This is the main
general result we have obtained.

\chapter*{R\'esum\'e}

\noindent{\bf Partie~\ref{part:reflexions}. Groupes de r\'eflexions, alg\`ebres de Hecke.} 
Dans cette partie, nous rappelons les d\'efinitions et notions classiques associ\'ees aux groupes de 
r\'eflexions complexes~: alg\`ebre d'invariants, alg\`ebre des coinvariants, degr\'es fant\^omes, alg\`ebres de Hecke, 
familles de Hecke... Nous fixons entre autres un corps $\kb$ de caract\'eristique z\'ero, 
un $\kb$-espace vectoriel $V$ de dimension $n$ et 
un sous-groupe fini $W$ de $\Grm\Lrm_\kb(V)$ engendr\'e par l'ensemble $\Ref(W)$ de ses r\'eflexions. 
Le chapitre~\ref{chapter:coxeter} s'int\'eresse plus particuli\`erement aux groupes de Coxeter, 
et propose un r\'esum\'e de la th\'eorie des cellules de Kazhdan-Lusztig~: nous rappelons la notion 
de famille $\Irr_\G^\kl(W)$ de caract\`eres irr\'eductibles associ\'es \`a une cellule bilat\`ere $\G$, 
ainsi que la notion de caract\`ere cellulaire $\isomorphisme{C}^\kl$ associ\'e \`a une cellule \`a gauche $C$. 

\bigskip

\noindent{\bf Partie~\ref{part:cherednik}. Alg\`ebres de Cherednik.} 
Nous rappelons dans cette partie les propri\'et\'es essentielles des alg\`ebres de 
Cherednik~: d\'ecomposition PBW, alg\`ebre sph\'erique, \'el\'ement d'Euler. Si le chapitre~\ref{chapter:cherednik-1} 
se place dans le cadre g\'en\'eral, avec un param\`etre $t$ g\'en\'erique, nous nous pla\c{c}ons d\`es le 
chapitre~\ref{chapter:cherednik-0} dans le cas o\`u $t=0$, cas que nous ne quittons plus jusqu'\`a la fin de 
ce m\'emoire. Le parti pris dans ce chapitre, essentiel pour les m\'ethodes que nous d\'eveloppons, 
est de travailler {\it g\'en\'eriquement}. 
Si on note $\CCB$ l'espace vectoriel des fonctions $c : \Ref(W) \longto \kb$ invariantes par conjugaison, 
alors $\Hb$ d\'esignera l'{\it alg\`ebre de Cherednik en $t=0$ g\'en\'erique}~: comme $\kb$-espace 
vectoriel, $\Hb$ est isomorphe \`a $\kb[\CCB] \otimes \kb[V] \otimes \kb W \otimes \kb[V^*]$, avec 
les r\'egles de commutations usuelles faisant intervenir les param\`etres g\'en\'eriques 
de l'alg\`ebre $\kb[\CCB]$ des fonctions polyn\^omiales sur $\CCB$. Si $c \in \CCB$, alors la sp\'ecialisation 
en $c$  de $\Hb$ sera not\'ee $\Hb_c$. 

Le centre $Z$ de cette alg\`ebre $\Hb$ jouit de nombreuses propri\'et\'es 
(on note ici $e=\frac{1}{|W|} \sum_{w \in W} w$)~:
\begin{list}{$\square$}{\leftmargin=1.3cm \itemindent-0.2cm}
\item[(Z1)] $Z$ contient la sous-alg\`ebre $P=\kb[\CCB] \otimes \kb[V]^W \otimes \kb[V^*]^W$.

\item[(Z2)] $Z =  \End_\Hb(\Hb e) \simeq e\Hb e$ et $\Hb = \End_Z(\Hb e)$.

\item[(Z3)] $Z$ est un $P$-module libre de rang $|W|$ et est int\`egre et int\'egralement clos.
\end{list}
Si on note $Z_c$ la sp\'ecialisation en $c$ de $Z$, alors $Z_c$ est \'egal au centre de $\Hb_c$ et v\'erifie 
aussi les analogues de (Z1), (Z2) et (Z3). 

Une autre particularit\'e de cette alg\`ebre $\Hb$ est l'existence de nombreux automorphismes naturels~: 
si on note $\NC$ le normalisateur de $W$ dans $\Grm\Lrm_\kb(V)$, alors le groupe 
$\kb^\times \times \kb^\times \times (W^\wedge \rtimes \NC)$ agit sur $\Hb$. L'action de 
$\kb^\times \times \kb^\times$ \'equivaut \`a la donn\'ee de la bigraduation sur $\Hb$ pour laquelle 
les \'el\'ements de $V$ ont pour degr\'e $(0,1)$, les \'el\'ements de $V^*$ ont pour degr\'e $(1,0)$, 
les \'el\'ements de $W$ ont pour degr\'e $(0,0)$ et les \'el\'ements de $\CCB$ ont pour degr\'e $(1,1)$. 

Nous abordons aussi les sp\'ecificit\'es des groupes de Coxeter, pour lesquels l'isomorphisme de $\kb W$-modules 
$V \simeq V^*$ induit un nouvel automorphisme de $\Hb$. 

\bigskip

\noindent{\bf Partie~\ref{part:extension}. L'extension $Z/P$.} 
Cette partie est le c\oe ur de notre m\'emoire~: nous y construisons les cellules 
de Calogero-Moser \`a travers l'\'etude de la ramification de l'extension d'anneaux $Z/P$. 
Pour nos besoins, nous introduisons une copie $Q$ 
de la $P$-alg\`ebre $Z$ munie d'un isomorphisme de $P$-alg\`ebres $\copie : Z \longiso Q$. 
Notons $\Kb$ (respectivement $\Lb$) le corps des fractions de $P$ (respectivement $Q$) et fixons 
une cl\^oture galoisienne $\Mb$ de l'extension $\Lb/\Kb$~: nous noterons $G=\Gal(\Mb/\Kb)$, $H=\Gal(\Mb/\Lb)$ et 
$R$ la cl\^oture int\'egrale de $P$ dans $\Mb$. Il d\'ecoule de (Z2) que 
$$|G/H|=|W|.\leqno{(\clubsuit)}$$
Par sp\'ecialisation en $c=0$, nous construisons un morphisme de groupes $\iota : W \times W \to G$, 
dont le noyau est $\D(\Zrm(W))$ (o\`u $\D : W \injto W \times W$, $w \mapsto (w,w)$ est l'injection 
diagonale) et v\'erifiant 
$$\iota(W \times W) \cap H = \iota (\D(W))\quad\text{et}\quad \iota(W \times W) \cdot H = H \cdot \iota(W \times W) = G.
\leqno{(\diamondsuit)}$$
Cette construction est faite dans la sous-section~\ref{subsection:specialisation galois 0}~: c'est 
la cl\'e de notre d\'efinition des cellules de Calogero-Moser. Si on identifie $W$ et le 
sous-groupe $\iota(W \times 1)$ de $G$, alors les propri\'et\'es $(\diamondsuit)$ impliquent~:
\begin{itemize}
 \item[($\heartsuit)$\hskip1.6mm] {\it $G/H$ est en bijection avec $W$. Ainsi, $G$ s'identifie avec un 
sous-groupe de $\SG_W$ et, vu comme sous-groupe de $G$, $W=\iota(W \times 1)$ s'identifie avec le groupe 
des permutations de $W$ induites par les translations \`a gauche.}
\end{itemize}
Tout ceci est fait dans le chapitre~\ref{chapter:galois-CM}. Nous y montrons aussi que 
l'action de $\kb^\times \times \kb^\times \times (W^\wedge \rtimes \NC)$ sur $\Hb$ 
induit une action sur $R$ qui normalise $G$.

Un autre aspect essentiel de ce chapitre~\ref{chapter:galois-CM} est de commencer l'\'etude des 
repr\'esentations de l'alg\`ebre $\Hb$, ou plut\^ot de l'alg\`ebre $\Mb\Hb=\Mb \otimes_P \Hb$. 
Il y est d\'emontr\'e que~:
\begin{itemize}
\item[($\spadesuit)$\hskip1.6mm] 
{\it $\Mb\Hb$ est semi-simple, d\'eploy\'ee, et ses modules simples sont en bijection avec $G/H$ c'est-\`a-dire, in fine, 
avec $W$. Cette bijection sera not\'ee $W \longiso \Irr(\Mb\Hb)$, $w \longmapsto \LC_w$.}
\end{itemize}

\medskip  

Le chapitre~\ref{chapter:geometrie-CM} livre une rapide version g\'eom\'etrique des 
pr\'ec\'edentes constructions~: en effet, toutes les alg\`ebres impliqu\'ees ($P$, $Q$, $R$) sont de type fini et 
peuvent donc \^etre vues comme des alg\`ebres de fonctions polynomiales sur des $\kb$-vari\'et\'es. 
Rappelons que $\Spec(Z_c)$ est appel\'e l'{\it espace de Calogero-Moser} associ\'e \`a $(W,c)$. 

\medskip

Nous d\'efinissons dans le chapitre~\ref{chapter:cellules-CM} la notion de {\it cellule de Calogero-Moser}. 
Si $\rG$ est un id\'eal premier de $R$, on note $D_\rG$ (respectivement $I_\rG$) le groupe 
de d\'ecomposition (respectivement d'inertie) de $\rG$ dans $G$ et $k_R(\rG)$ le corps r\'esiduel de $R$ en $\rG$, 
c'est-\`a-dire le corps des fractions de $R/\rG$. Autrement dit, 
$$D_\rG=\{g \in G~|~g(\rG)=\rG\}\quad\text{et}\quad I_\rG=\{g \in G~|~\forall~r \in R,~g(r) \equiv r \mod \rG\}.$$
Alors $I_\rG$ est un sous-groupe distingu\'e de $D_\rG$ et $D_\rG/I_\rG$ s'identifie au groupe de Galois de 
l'extension (galoisienne) $k_R(\rG)/k_P(\rG \cap P)$.

\medskip

\begin{quotation}
\noindent{\bf D\'efinition 1.} 
{\it On appellera {\bfit $\rG$-cellule de Calogero-Moser} toute $I_\rG$-orbite dans $W$.}
\end{quotation}

\medskip

\noindent Une grande partie de ce chapitre~\ref{chapter:cellules-CM} est consacr\'ee \`a relier cette notion avec 
la th\'eorie des repr\'esentations de $\Hb$. Le r\'esultat essentiel est le suivant~:

\begin{quotation}
\noindent{\bf Th\'eor\`eme 1.} 
{\it Les idempotents primitifs centraux de $R_\rG \Hb$ (qui sont en bijection, via la r\'eduction modulo $\rG$, avec 
les idempotents primitifs centraux de $k_R(\rG) \Hb$) sont en bijection avec les $\rG$-cellules de Calogero-Moser~: 
un idempotent primitif central $b$ de $R_\rG \Hb$ est associ\'e \`a une $\rG$-cellule de Calogero-Moser 
$C$ si et seulement si $b\LC_w = \LC_w$ pour un (ou, de mani\`ere \'equivalente, pour tout) 
$w \in C$.}
\end{quotation}

\bigskip

\noindent{\bf Partie~\ref{part:verma}. Cellules et modules de Verma.} 
Dans la d\'efinition pr\'ec\'edente, une grande latitude est laiss\'ee quant au choix de l'id\'eal premier $\rG$ 
de $R$. Nous nous focalisons dans cette partie sur certains id\'eaux premiers, susceptibles de 
donner lieu \`a des g\'en\'eralisations des cellules de Kazhdan-Lusztig, et pour 
lesquels la th\'eorie des repr\'esentations de $k_R(\rG)\Hb$ est plus avanc\'ee~: 
en particulier, pour tous nos choix, il existera une notion de {\it module de Verma} 
associ\'e \`a un caract\`ere irr\'eductible de $W$.

Voyons $P$ comme l'alg\`ebre des fonctions polynomiales sur la $\kb$-vari\'et\'e 
$\CCB \times V/W \times V^*/W$. Fixons maintenant $c \in \CCB$ et notons $\pG_c^\gauche$ 
(respectivement $\pG_c^\droite$, respectivement $\pGba_c$) l'id\'eal premier de $P$ correspondant \`a la sous-vari\'et\'e 
ferm\'ee irr\'eductible $\{c\} \times V/W \times \{0\}$ (respectivement $\{c\} \times \{0\} \times V^*/W$, 
respectivement $\{(c,0,0)\}$). On fixe un id\'eal premier $\rGba_c$ de $R$ au-dessus de $\pGba_c$ ainsi que 
deux id\'eaux premiers $\rG_c^\gauche $ et $\rG_c^\droite$ au-dessus de $\pG_c^\gauche$ et $\pG_c^\droite$ 
respectivement, et contenus dans $\rGba_c$. 

\medskip

\begin{quotation}
\noindent{\bf D\'efinition 2.} 
{\it On appelle {\bfit $c$-cellule de Calogero-Moser bilat\`ere} 
(respectivement {\bfit \`a gauche}, respectivement {\bfit \`a droite})
toute $\rGba_c$-cellule (respectivement $\rG_c^\gauche$-cellule, respectivement $\rG_c^\droite$-cellule) 
de Calogero-Moser.}
\end{quotation}

\medskip

\noindent Bien s\^ur, cette d\'efinition comporte une part d'ambig\"{u}it\'e, car le choix de l'id\'eal 
premier de $R$ au-dessus d'un id\'eal premier de $P$ n'est absolument pas unique. 

\medskip

Dans le chapitre~\ref{chapter:bebe-verma}, nous rappelons la contruction de Gordon des modules simples de l'alg\`ebre 
$k_P(\pGba_c) \Hb$ (souvent appel\'ee {\it alg\`ebre de Cherednik restreinte})~: cette alg\`ebre est 
d\'eploy\'ee 
et ses modules simples sont naturellement param\'etr\'es 
par les caract\`eres irr\'eductibles de $W$. Si $\chi \in \Irr(W)$, nous noterons $\LCov_c(\chi)$ le 
module simple de $k_P(\pGba_c) \Hb$ ou $k_R(\rGba_c) \Hb$ correspondant. Nous poursuivons l'\'etude de ces 
modules simples en int\'egrant notamment l'action du groupe 
$\kb^\times \times \kb^\times \times (W^\wedge \rtimes \NC$). 

\medskip

\begin{quotation}
\noindent{\bf D\'efinition 3.} 
{\it On appelle {\bfit $c$-famille de Calogero-Moser} toute partie de $\Irr(W) \simeq \Irr(k_R(\rGba_c)\Hb)$ 
correspondant \`a un bloc de $k_R(\rGba_c)\Hb$.}
\end{quotation}

\medskip

Dans le chapitre~\ref{chapter:martino}, nous rappelons la conjecture de Martino pr\'edisant 
que les $c$-familles de Calogero-Moser sont des unions de familles de Hecke.

\medskip

Le chapitre~\ref{chapter:bilatere} \'etudie les $c$-cellules de Calogero-Moser bilat\`eres. 
En vertu du th\'eor\`eme~1 et de la d\'efinition~3, on peut associer \`a toute $c$-cellule de Calogero-Moser bilat\`ere 
une $c$-famille de Calogero-Moser $\Irr_\G^\calo(W)$. Entre autres propri\'et\'es \'etablies dans ce chapitre, 
nous montrons que
$$|\G|=\sum_{\chi \in \Irr_\G^\calo(W)} \chi(1)^2.\leqno{(\arm_\calo)}$$

\medskip

Le chapitre~\ref{chapter:gauche} \'etudie les cellules de Calogero-Moser \`a gauche (ou \`a droite~: 
par sym\'etrie, seul le cas ``\`a gauche'' est trait\'e). Par des r\'esultats g\'en\'eraux 
\'etablis dans la partie~\ref{part:extension}, un bloc de l'alg\`ebre $k_R(\rG_c^\gauche)\Hb$ 
(qui correspond donc \`a une $c$-cellule de Calogero-Moser \`a gauche $C$ en vertu du th\'eor\`eme 1) 
ne poss\`ede qu'un seul module simple (que nous noterons $\LC_c^\gauche(C)$), de dimension $|W|$. 
Il est aussi possible d'associer \`a un caract\`ere irr\'eductible $\chi$ de  $W$ un {\it module de Verma 
\`a gauche} $\MC_c^\gauche(\chi)$, de dimension $|W|\chi(1)$. On note 
$\mult_{C,\chi}^\calo$ la multiplicit\'e de $\LC_c^\gauche(C)$ dans une s\'erie de Jordan-H\"older 
de $\MC_c^\gauche(\chi)$. 

\medskip

\begin{quotation}
\noindent{\bf D\'efinition 4.} 
{\it Si $C$ est une $c$-cellule de Calogero-Moser \`a gauche, on notera $\isomorphisme{C}_c^\calo$ le caract\`ere 
de $W$ d\'efini par
$$\isomorphisme{C}_c^\calo=\sum_{\chi \in \Irr(W)} \mult_{C,\chi}^\calo \cdot \chi.$$}
\end{quotation}

\medskip

\noindent Nous d\'emontrons dans ce chapitre les propri\'et\'es suivantes. 
Si $C$ est une $c$-cellule de Calogero-Moser \`a gauche, alors 
$$|C|=\sum_{\chi \in \Irr(W)} \mult_{C,\chi}^\calo \cdot \chi(1).\leqno{(\brm_\calo)}$$
Si $\chi$ est un caract\`ere irr\'eductible de $W$, alors 
$$\chi(1)=\sum_C \mult_{C,\chi}^\calo,\leqno{(\crm_\calo)}$$
o\`u $C$ parcourt l'ensemble des $c$-cellules de Calogero-Moser \`a gauche.

\medskip

Le chapitre~\ref{chapter:bb} est consacr\'e \`a la preuve du r\'esultat le plus important 
que nous ayons d\'emontr\'e sur les cellules de Calogero-Moser~: il s'applique aux cellules 
bilat\`eres associ\'ees \`a un point lisse de l'espace de Calogero-Moser $\Spec(Z_c)$. 
Introduisons pour cela quelques notations~: l'injection $Z \longiso Q \injto R$ 
induit un morphisme de vari\'et\'es $\r : \Spec(R) \to \Spec(Z)$. L'id\'eal premier  
$\rGba_c$ \'etant maximal, il correspond \`a un point de cette vari\'et\'e et, si $w \in W$, 
nous noterons $\r(w(\rGba_c))$ son image dans $\Spec(Z_c)$.

\medskip

\begin{quotation}
\noindent{\bf Th\'eor\`eme 2.} {\it Soit $\G$ une $c$-cellule de Calogero-Moser et supposons que 
le point $\r(w^{-1}(\rGba_c))$ est un point lisse de $\Spec(Z_c)$ pour un $w \in \G$ (ou, de mani\`ere \'equivalente, 
pour tout $w \in \G$). Il d\'ecoule alors d'un r\'esultat de 
Gordon que $\Irr_\G^\calo(W)$ est un singleton (notons $\chi$ son unique \'el\'ement). 
Si $C$ une $c$-cellule de Calogero-Moser \`a gauche contenue dans $\G$, alors
$$\isomorphisme{C}_c^\calo=\chi.$$} 
\end{quotation}

\bigskip

\noindent{\bf Partie~\ref{part:coxeter}. Groupes de Coxeter~: Calogero-Moser vs Kazhdan-Lusztig.} 
Dans le chapitre~\ref{chapter:conjectures}, nous proposons les conjectures suivantes~:

\medskip

\begin{quotation}
\noindent{\bf Conjecture 1.} {\it Si $(W,S)$ est un syst\`eme de Coxeter et si $c(s) \in \RM_{\ge 0}$ pour 
tout $s \in S$, alors~:
\begin{itemize}
\itemth{a} Les $c$-cellules de Kazhdan-Lusztig (\`a gauche, \`a droite, bilat\`eres) 
et les $c$-cellules de Calogero-Moser (\`a gauche, \`a droite, bilat\`eres) co\"{\i}ncident.

\itemth{b} Supposant $(\arm)$ vrai, si $\G$ est une $c$-cellule bilat\`ere, alors $\Irr_\G^\kl(W)=\Irr_\G^\calo(W)$.

\itemth{c} Supposant $(\arm)$ vrai, si $C$ est une $c$-cellule \`a gauche, alors $\isomorphisme{C}_c^\kl=\isomorphisme{C}_c^\calo$.
\end{itemize}}
\end{quotation}

\medskip

Si on reste au niveau des caract\`eres, les \'enonc\'es (b) et (c) de la conjecture 1 impliquent la conjecture suivante~:

\medskip

\begin{quotation}
\noindent{\bf Conjecture 2.} {\it Si $(W,S)$ est un syst\`eme de Coxeter, alors~:
\begin{itemize}
\itemth{a} Les $c$-familles de Kazhdan-Lusztig et les $c$-familles de Calogero-Moser co\"{\i}ncident (Gordon-Martino).

\itemth{b} L'ensemble des KL-caract\`eres $c$-cellulaires et l'ensemble des CM-caract\`eres $c$-cellulaires co\"{\i}ncident.
\end{itemize}}
\end{quotation}

\medskip

Le chapitre~\ref{chapter:arguments} est une succession de r\'esultats corroborant ces conjectures. 
Si les propri\'et\'es num\'eriques $(\arm_\calo)$, $(\brm_\calo)$ et $(\crm_\calo)$ sont aussi 
satisfaites par les cellules de Kazhdan-Lusztig, les familles de Kazhdan-Lusztig et les 
KL-caract\`eres $c$-cellulaires, notre r\'esultat le plus probant est le suivant, qui d\'ecoule 
du th\'eor\`eme 2 et du fait que $\Spec(Z_c)$ est lisse si $(W,S)$ est de type $A$ et $c\neq 0$~:

\medskip

\begin{quotation}
\noindent{\bf Corollaire 3.} 
{\it Supposons que $(W,S)$ soit de type $A$ et que $c \neq 0$. Alors~:
\begin{itemize}
\itemth{a} Les $c$-familles de Calogero-Moser co\"{\i}ncident avec les $c$-familles de Kazhdan-Lusztig (Gordon)~: 
ce sont des singletons.

\itemth{b} Les CM-caract\`eres $c$-cellulaires co\"{\i}ncident avec les KL-caract\`eres $c$-cellulaires~: 
ce sont les caract\`eres irr\'eductibles de $W$.

\itemth{c} Il existe une bijection $\ph : W \to W$ qui envoie les $c$-cellules de Calogero-Moser 
(\`a gauche ou bilat\`eres) sur les $c$-cellules de Kazhdan-Lusztig.
\end{itemize}}
\end{quotation}

\bigskip

\noindent{\bf Partie~\ref{part:exemples}. Exemples.} 
Le chapitre~\ref{chapitre:rang 1} traite du premier exemple de groupe de r\'eflexions complexe 
qui ne soit pas un groupe de Coxeter, \`a savoir le cas o\`u $\dim_\kb V = 1$. Alors $W$ est cyclique 
(notons $d=|W|$). Fixons alors un g\'en\'erateur $s$ de $W$ et notons $\e : W \to \kb^\times$ 
la restriction du d\'eterminant 
(ainsi, $\Irr(W)=\{1,\e,\e^2,\dots,\e^{d-1}\}$). Nous montrons 
dans ce chapitre~\ref{chapitre:rang 1} que
$$G=\SG_W \simeq \SG_d\leqno{(*)}$$
et que, si on note $k_i$ le scalaire via lequel l'\'el\'ement d'Euler agit sur 
$\LCov_c(\e^i)$, alors~:

\medskip

\begin{quotation}
\noindent{\bf Proposition 4.} 
{\it Soient $i$ et $j$ dans $\ZM$. Alors~:
\begin{itemize}
\itemth{a} $\e^i$ et $\e^j$ sont dans la m\^eme $c$-famille de Calogero-Moser si et seulement si 
$k_i=k_j$. 

\itemth{b} $s^i$ et $s^j$ sont dans la m\^eme $c$-cellule de Calogero-Moser (\`a gauche, \`a droite ou bilat\`ere) 
si et seulement si $k_i=k_j$. 

\itemth{c} Si $\G$ est une $c$-cellule de Calogero-Moser bilat\`ere 
(c'est aussi une $c$-cellule de Calogero-Moser 
\`a gauche en vertu de $(\brm)$), alors 
$$\Irr_\G^\calo(W)=\{\e^{-i}~|~s^i \in \G\}\quad\text{et}\quad 
\isomorphisme{\G}_c^\calo=\sum_{\substack{0 \le i \le d-1 \\ s^i \in \G}} \e^{-i}.$$
\end{itemize}}
\end{quotation}

\medskip

Le chapitre~\ref{chapitre:b2} traite du plus petit groupe de Coxeter irr\'eductible dans lequel il y a deux classes 
de conjugaison de r\'eflexions, \`a savoir le cas du type $B_2$. C'est l'occasion d'y v\'erifier les conjectures 
1 et 2. Supposons donc que $(W,S)$ soit un syst\`eme de Coxeter de type $B_2$. On note $w_0=-\Id_V \in W$. 
Notons $\AG_W$ le groupe des permutations paires de $W$. Alors, apr\`es un calcul explicite d'une pr\'esentation 
du centre $Z$ de $\Hb$, nous montrons dans ce chapitre que 
$$G = \AG_W \cap \{\s \in \SG_W~|~\forall~w \in W,~\s(w_0w)=w_0\s(w)\}.\leqno{(**)}$$
Notons que $\{\s \in \SG_W~|~\forall~w \in W,~\s(w_0w)=w_0\s(w)\}$ est un groupe de Weyl de type $B_4$ 
et que
$$\text{\it $G$ est un groupe de Weyl de type $D_4$.}$$
Une fois ce r\'esultat \'etabli, et en utilisant les r\'esultats num\'eriques 
$(\arm_\calo)$, $(\brm_\calo)$ et $(\crm_\calo)$, la v\'erification des conjectures 1 et 2 est
facile pour ce qui concerne les cellules bilat\`eres et les familles, alors que, pour 
les cellules \`a gauche et les caract\`eres cellulaires, cette v\'erification est plus d\'elicate.

\bigskip

\noindent{\bf Appendices.} Nous concluons ce m\'emoire par une s\'erie d'appendices 
constitu\'es de rappels de r\'esultat classiques~: th\'eorie de Galois, graduations et blocs d'alg\`ebres de 
dimension finie...

\part{Groupes de r\'eflexions, alg\`ebres de Hecke}\label{part:reflexions}

Consacr\'ee essentiellement \`a des rappels, cette partie a pour but 
d'introduire diverses notions attach\'ees \`a la donn\'ee d'un groupe de r\'eflexions~: 
arrangement d'hyperplans, th\'eorie des invariants, alg\`ebres de Hecke, familles de Hecke, 
th\'eorie de Kazhdan-Lusztig.

\chapter{Groupes de r\'eflexions}\label{chapter:reflexions}

\boitegrise{{\it Dans tout ce m\'emoire, nous fixons un corps commutatif $\kb$ 
\indexnot{ka}{\kb} de caract\'eristique $0$, un $\kb$-espace vectoriel $V$ \indexnot{V}{V} de dimension 
finie $n$ \indexnot{na}{n} ainsi qu'un sous-groupe {\bfit fini} $W$ \indexnot{W}{W} de $\GL_\kb(V)$. 
Tout au long de ce m\'emoire, la notation $\otimes$ remplacera $\otimes_\kb$. 
L'alg\`ebre de 
groupe de $W$ sur $\kb$ sera not\'ee $\kb W$. Posons
$$\Ref(W)=\{s \in W~|~\dim_\kb \im(s-\Id_V)=1\},\indexnot{R}{\Ref(W)}$$
de sorte que $\Ref(W)$ est l'ensemble des {\bfit r\'eflexions} de $W$. 
{\bfit Nous supposerons que $W$ est engendr\'e par $\Ref(W)$.}}}{0.75\textwidth}

\bigskip

\section{D\'eterminant, racines, coracines}

\medskip

Nous noterons
$$\fonction{\e}{W}{\kb^\times}{w}{\det_V(w).}\indexnot{ez}{\e}$$
La dualit\'e parfaite entre $V$ et son dual $V^*$ sera not\'ee
$$\langle,\rangle : V \times V^* \longto \kb.\index{ZZZ@ ${{\langle}},{{\rangle}}$ |textbf {\hskip0.5cm} \textbf}  $$
Si $s \in \Ref(W)$, nous 
noterons $\a_s$ et $\a_s^\ve$ \indexnot{az}{\a_s,~\a_s^\ve} des \'el\'ements non nuls de 
$V^*$ et $V$ respectivement tels que 
$$\Ker(s-\Id_V)=\Ker \a_s\quad\text{et}\quad \im(s-\Id_V)=\kb \a_s^\ve$$
ou, de mani\`ere \'equivalente, 
$$\Ker(s-\Id_{V^*})=\Ker \a_s^\ve\quad\text{et}\quad \im(s-\Id_{V^*})=\kb \a_s.$$
Notons que, puisque $\kb$ est de caract\'eristique nulle, tous les \'el\'ements de $\Ref(W)$ 
sont diagonalisables et donc
\equat\label{non nul}
\langle \a_s^\ve,\a_s \rangle \neq 0.
\endequat
Ainsi, si $x \in V^*$ et $y \in V$, alors
\equat\label{action s V}
s(y)=y - (1 -\e(s)) \frac{\langle y,\a_s\rangle}{\langle \a_s^\ve,\a_s\rangle} \a_s^\ve
\endequat
et
\equat\label{action s V*}
s(x)=x - (1 -\e(s)^{-1}) \frac{\langle \a_s^\ve,x\rangle}{\langle \a_s^\ve,\a_s\rangle} \a_s.
\endequat

\section{Invariants}

\medskip

Nous noterons $\kb[V]$ (respectivement $\kb[V^*]$) \indexnot{ka}{\kb[V],~\kb[V^*]} l'alg\`ebre sym\'etrique de $V^*$ 
(respectivement $V$), que nous identifierons avec l'alg\`ebre des fonctions 
polynomiales sur $V$ (respectivement $V^*$). Nous noterons
$$\kb[V]_+=\{f \in \kb[V]~|~f(0)=0\}\quad\text{et}\quad
\kb[V^*]_+=\{f \in \kb[V^*]~|~f(0)=0\}.\indexnot{ka}{\kb[V]_+,~\kb[V^*]_+}$$
Le groupe $W$ agit naturellement 
sur les alg\`ebres $\kb[V]$ et $\kb[V^*]$~: les alg\`ebres d'invariants seront 
not\'ees respectivement $\kb[V]^W$ et $\kb[V^*]^W$. \indexnot{ka}{\kb[V]^W,~\kb[V^*]^W} Nous appellerons 
{\it alg\`ebres des coinvariants}, et nous noterons $\kb[V]^\cow$ et $\kb[V^*]^\cow$ 
\indexnot{ka}{\kb[V]^\cow,~\kb[V^*]^\cow} 
les $\kb$-alg\`ebres de dimension finie
$$\kb[V]^\cow=\kb[V]/<\kb[V]_+^W>\quad\text{et}\quad \kb[V^*]^\cow=\kb[V^*]/<\kb[V^*]_+^W>.$$
En vertu du th\'eor\`eme de Shephard-Todd-Chevalley, 
le fait que $W$ soit engendr\'e par ses r\'eflexions a de nombreuses 
cons\'equences sur la structure de $\kb[V]^W$. Nous en donnons ici une version 
agr\'ement\'ee de r\'esultats quantitatifs (nous l'\'enon\c{c}ons 
\`a travers ses cons\'equences sur l'action de $W$ sur la $\kb$-alg\`ebre 
$\kb[V]$, mais les m\^emes \'enonc\'es restent valables en rempla\c{c}ant 
$\kb[V]$ par $\kb[V^*]$)~:

\bigskip

\begin{theo}[Shephard-Todd, Chevalley]\label{chevalley}
Rappelons que $n =\dim_\kb V$. Alors~:
\begin{itemize}
\itemth{a} Il existe des polyn\^omes homog\`enes $f_1$,\dots, $f_n$ dans $\kb[V]^W$, 
alg\'ebriquement ind\'ependants, et tels que $\kb[V]^W=\kb[f_1,\dots,f_n]$. Si on note 
$d_i$ \indexnot{da}{d_i} le degr\'e de $f_i$, alors 
$$|W|=d_1\cdots d_n\quad \text{et}\quad |\Ref(W)|=\sum_{i=1}^n (d_i-1).$$

\itemth{b} Le $k[V]^W[W]$-module $\kb[V]$ est libre de rang $1$. En particulier, 
il existe une $\kb[V]^W$-base $(b_1,b_2,\dots,b_{|W|})$ du $\kb[V]^W$-module $\kb[V]$ 
form\'ee d'\'el\'ements homog\`enes.

\itemth{c} Le $\kb W$-module $\kb[V]^\cow$ est libre de rang $1$~; 
en particulier, $\dim_\kb \kb[V]^\cow=|W|$.

\itemth{d} Si $d_1 \le \cdots \le d_n$, alors la suite $(d_1,\dots,d_n)$ 
ne d\'epend pas du choix des $f_i$~: elle est uniquement d\'etermin\'ee par $W$. 
Nous l'appellerons {\bfit suite des degr\'es} de $W$.
\end{itemize}
\end{theo}

\begin{proof}
Voir par exemple~\cite[th\'eor\`eme~4.1]{broue}.
\end{proof}

\bigskip

\begin{rema}\label{rema:deg-v-v}
Gr\^ace \`a la formule de Molien (voir~\cite[Lemme~3.28]{broue}), 
la suite des degr\'es pour l'action de $W$ sur $V$ 
(telle qu'elle est d\'efinie dans le th\'eor\`eme~\ref{chevalley}) co\"{\i}ncide avec la suite des 
degr\'es de $W$ pour son action sur $V^*$.\finl
\end{rema}

\bigskip

\section{Hyperplans, sous-groupes paraboliques}\label{section:hyperplans}

\medskip

\boitegrise{{\bf Notation.} 
{\it Nous noterons $\EC(\kb)$ \indexnot{E}{\EC(\kb)} l'ensemble des entiers naturels non nuls $e$ tels que 
$\kb$ contienne une racine primitive $e$-i\`eme de l'unit\'e. Nous fixons une famille 
$(\z_e)_{e \in \EC(\kb)}$ o\`u, pour tout $e \in \EC(\kb)$, $\z_e$ \indexnot{zz}{\z_e} est une racine 
primitive $e$-i\`eme de l'unit\'e. Nous choisissons cette famille de fa\c{c}on 
``coh\'erente'', c'est-\`a-dire que, si $e \in \EC(\kb)$ et si $d | e$, alors 
$\z_d=\z_e^{e/d}$.}}{0.75\textwidth}

\medskip

Nous noterons $\AC$ \indexnot{A}{\AC} l'ensemble des hyperplans de r\'eflexion de $W$~:
$$\AC=\{\Ker(s-\Id_V)~|~s \in \Ref(W)\}.$$
Si $X$ est une partie de $V$, nous noterons $W_X$ \indexnot{W}{W_X} le fixateur de $X$, c'est-\`a-dire
$$W_X=\{w \in W~|~\forall~x \in X,~w(x)=x\}.$$

\bigskip

\noindent{\bfit Th\'eor\`eme de Steinberg. ---} 
{\it Si $X \subset V$, alors $W_X$ est engendr\'e par des r\'eflexions.}

\begin{proof}
Voir par exemple~\cite[th\'eor\`eme~4.7]{broue}.
\end{proof}

\bigskip

Si $H \in \AC$, nous noterons $e_H$ \indexnot{ea}{e_H} l'ordre du sous-groupe cyclique $W_H$ de $W$. 
Nous noterons $s_H$ \indexnot{sa}{s_H} le g\'en\'erateur de $W_H$ de d\'eterminant $\z_{e_H}$ 
(c'est une r\'eflexion d'hyperplan $H$). On a alors
$$\Ref(W)=\{s_H^j~|~H \in \AC\text{ et }1 \le j \le e_H-1\}.$$
Le lemme suivant est imm\'ediat~:

\bigskip

\begin{lem}\label{conjugaison ref}
$s_H^j$ et $s_{H'}^{j'}$ sont conjugu\'es dans $W$ 
si et seulement si $H$ et $H'$ sont dans la m\^eme $W$-orbite et 
$j=j'$.
\end{lem}

\bigskip

Si $\O$ est une $W$-orbite d'hyperplans de $\AC$, nous noterons $e_\O$ \indexnot{ea}{e_\O} la valeur 
commune de tous les $e_H$, o\`u $H \in \O$. Ainsi, le lemme~\ref{conjugaison ref} 
nous fournit une bijection entre l'ensemble $\Omeb_W$ \indexnot{oz}{\Omeb_W} des couples $(\O,j)$ o\`u $\O \in \AC/W$ et 
$1 \le j \le e_\O-1$ et l'ensemble $\refw$. 

\bigskip

\section{Caract\`eres irr\'eductibles}

\medskip

Une cons\'equence du fait que $W$ est engendr\'e par $\Ref(W)$, 
dont la preuve repose sur la classification des groupes 
de r\'eflexions (voir~\cite{benard} et~\cite{bessis}), est la suivante 
(notons que la semi-simplicit\'e est \'evidente, puisque $\kb$ est de 
caract\'eristique z\'ero)~:

\bigskip

\begin{theo}[Benard, Bessis]\label{deploiement}
Soit $\kb'$ un sous-corps de $\kb$ contenant les traces des \'el\'ements de $W$. Alors 
l'alg\`ebre $\kb' W$ est semi-simple d\'eploy\'ee. En particulier, $\kb W$ est 
semi-simple d\'eploy\'ee. 
\end{theo}

\bigskip

Nous noterons $\Irr(W)$ \indexnot{I}{\Irr(W),~\Irr(\kb W)} l'ensemble des caract\`eres irr\'eductibles de $W$ 
(sur le corps $\kb$) que nous identifierons avec l'ensemble $\Irr(\kb W)$ 
des classes d'isomorphie de $\kb W$-modules simples. Nous identifierons 
le groupe de Grothendieck $\groth(\kb W)$ \indexnot{K}{\groth(\kb W)} avec le $\ZM$-module libre de base $\Irr(W)$ 
que nous noterons $\ZM\Irr(W)$. Si $M$ est un $\kb W$-module de dimension finie, 
nous noterons $\isomorphisme{M}_{\kb W}$ \indexnot{M}{\isomorphisme{M}_{\kb W}} son image dans $\groth(\kb W)$. 
Nous noterons $\langle,\rangle_W$ 
\index{ZZZ@ ${{\langle}},{{\rangle}}_W$ |textbf {\hskip0.5cm} \textbf}  le produit scalaire sur $\groth(\kb W)$ 
qui fait de $\Irr(W)$ une base orthonormale.
Le groupe $\Hom(W,\kb^\times)$ des caract\`eres lin\'eaires de $W$ \`a valeurs dans $\kb^\times$ sera not\'e 
$W^\wedge$ \indexnot{W}{W^\wedge} (ainsi, $W^\wedge \subset \Irr(W)$, avec \'egalit\'e si et seulement si $W$ est ab\'elien).

\bigskip

\section{S\'eries de Hilbert}

\medskip

Soient $\tb$ et $\ub$ \indexnot{ta}{\tb,~\ub} deux ind\'etermin\'ees et soit $\kb[[\tb,\ub]]$ 
la $\kb$-alg\`ebre des s\'eries formelles en deux variables $\tb$ et $\ub$. 
Si $M$ est une $\kb$-espace vectoriel $(\NM \times \NM)$-gradu\'e 
(de d\'ecomposition associ\'ee $M=\bigoplus_{(i,j) \in \NM \times \NM} M_{i,j}$) 
dont les composantes homog\`enes sont de dimension finie, 
on notera $\dim_\kb^\bigrad(M)$ \indexnot{da}{\dim_\kb^\bigrad} sa {\it s\'erie de Hilbert bi-gradu\'ee} 
d\'efinie par 
$$\dim_\kb^\bigrad(M)=\sum_{(i,j) \in \NM \times \NM} \dim_\kb(M_{i,j})~\tb^i~\ub^j.$$
Bien s\^ur, si $N$ est un autre $\kb$-espace vectoriel $(\NM \times \NM)$-gradu\'e 
dont les composantes homog\`enes sont de dimension finie, alors
\equat\label{+}
\dim_\kb^\bigrad(M \oplus N) = \dim_\kb^\bigrad(M) + \dim_\kb^\bigrad(N)
\endequat
et
\equat\label{x}
\dim_\kb^\bigrad(M \otimes N) = \dim_\kb^\bigrad(M)\cdot\dim_\kb^\bigrad(N).
\endequat

\subsection{Invariants}
L'alg\`ebre $\kb[V \times V^*]=\kb[V] \otimes \kb[V^*]$ admet une 
bi-graduation standard, en attribuant aux \'el\'ements de $V^* \subset \kb[V]$ 
le bi-degr\'e $(0,1)$ et \`a ceux de $V \subset \kb[V^*]$ le bi-degr\'e 
$(1,0)$. On a 
\equat\label{hilbert kv}
\dim_\kb^\bigrad(\kb[V \times V^*]) = \frac{1}{(1-\tb)^n(1-\ub)^n},
\endequat
gr\^ace notamment \`a~\ref{x}.

En reprenant les notations du th\'eor\`eme~\ref{chevalley}(a), on obtient, toujours 
gr\^ace \`a~\ref{x}, 
\equat\label{hilbert biinv}
\dim_\kb^\bigrad(\kb[V \times V^*]^{W \times W})
=\prod_{i=1}^n \frac{1}{(1-\tb^{d_i})(1-\ub^{d_i})}.
\endequat
D'autre part, la s\'erie de Hilbert bi-gradu\'ee de l'alg\`ebre des invariants 
diagonaux $\kb[V \times V^*]^W$ est donn\'ee par une formule {\it \`a la Molien}
\equat\label{hilbert molien}
\dim_\kb^\bigrad(\kb[V \times V^*]^W)=\frac{1}{|W|} 
\sum_{w \in W} \frac{1}{\det(1-w\tb)~\det(1-w^{-1}\ub)},
\endequat
dont la preuve copie presque mot pour mot la preuve de la formule de Molien 
pour la s\'erie de Hilbert gradu\'ee.

\bigskip

\subsection{Degr\'es fant\^omes} 
Si $M= \bigoplus_{i \in \ZM} M_i$ est un $\kb$-espace vectoriel $\ZM$-gradu\'e et si 
$d \in \ZM$, on notera $M\langle d \rangle$ le $\kb$-espace vectoriel gradu\'e 
obtenu en d\'ecalant la graduation de $d$, c'est-\`a-dire $M\langle d \rangle_i=M_{i+d}$. 
Si $M_i$ est nul pour $i \ll 0$, on d\'efinit aussi la {\it s\'erie de Hilbert} de $M$ 
par 
$$\dim_\kb^\grad(M)=\sum_{i \in \ZM} \dim_\kb(M_i)~\tb^i \in \kb((\tb)).$$
Il est imm\'ediat que $\dim_\kb^\grad(M\langle d \rangle)=\tb^{-d}~\dim_\kb^\grad(M)$. 

Notons 
$\groth(\kb W)[\tb,\tb^{-1}]$ l'anneau des polyn\^omes de Laurent \`a 
coefficients dans $\groth(\kb W)$. C'est un $\ZM[\tb,\tb^{-1}]$-module libre de 
base $\Irr(W)$.

Si $M = \bigoplus_{i \in \ZM} M_i$ est un $\kb W$-module $\ZM$-gradu\'e de dimension finie, 
on notera $\isomorphisme{M}_{\kb W}^\grad$ \indexnot{M}{\isomorphisme{M}_{\kb W}^\grad}
l'\'el\'ement de $\groth(\kb W)[\tb,\tb^{-1}]$ d\'efini par 
$$\isomorphisme{M}_{\kb W}^\grad = \sum_{i \in \ZM}~\isomorphisme{M_i}_{\kb W}~\tb^i.$$
Il est clair que $\isomorphisme{M}_{\kb W}$ est l'\'evaluation en $1$ de 
$\isomorphisme{M}_{\kb W}^\grad$ 
et que $\isomorphisme{M\langle n \rangle}_{\kb W}^\grad = 
\tb^{-n} ~\isomorphisme{M}_{\kb W}^\grad$. 
Si $M$ est un $\kb W$-module bi-gradu\'e, on d\'efinit de fa\c{c}on similaire 
$\isomorphisme{M}_{\kb W}^\bigrad$~: c'est un \'el\'ement de 
$\groth(\kb W)[\tb,\ub,\tb^{-1},\ub^{-1}]$.

Soit $(f_\chi(\tb))_{\chi \in \Irr(W)}$ \indexnot{fa}{f_\chi(\tb)}
l'unique famille d'\'el\'ements de $\NM[\tb]$ telle que 
\equat\label{fchi def}
\isomorphisme{k[V^*]^\cow}_{\kb W}^\bigrad=\sum_{\chi \in \Irr(W)} f_\chi(\tb) ~\chi.
\endequat

\bigskip
\begin{defi}\label{defi:degre-fantome}
Le polyn\^ome $f_\chi(\tb)$ est appel\'e le {\bfit degr\'e fant\^ome} de $\chi$. 
Sa $\tb$-valuation sera not\'ee $\bb_\chi$ \indexnot{ba}{\bb_\chi} et sera appel\'ee le {\bfit $\bb$-invariant} de $\chi$.
\end{defi}

\bigskip

Le degr\'e fant\^ome de $\chi$  v\'erifie 
\equat\label{chi 1}
f_\chi(1) = \chi(1).
\endequat
Remarquons aussi que
\equat\label{fchi}
\isomorphisme{\kb[V]^\cow}_{\kb W}^\bigrad=\sum_{\chi \in \Irr(W)} f_\chi(\ub) ~\chi^*,
\endequat
(ici, $\chi^*$ d\'esigne le caract\`ere dual de $\chi$, c'est-\`a-dire 
$\chi^*(w)=\chi(w^{-1})$). Notons aussi que, si $\unb_W$ d\'esigne le caract\`ere 
trivial de $W$, alors 
$$\isomorphisme{\kb[V^*]^\cow}_{\kb W}^\bigrad \equiv \unb_W \mod \tb \groth(\kb W)[\tb]$$
$$\isomorphisme{\kb[V]^\cow}_{\kb W}^\bigrad \equiv \unb_W \mod \ub  
\groth(\kb W)[\ub].\leqno{\text{et}}$$
On en d\'eduit~:

\bigskip

\begin{lem}\label{non zero}
$\isomorphisme{\kb[V]^\cow}_{\kb W}^\grad$ et $\isomorphisme{\kb[V^*]^\cow}_{\kb W}^\grad$ 
ne sont pas des diviseurs de $0$ dans $\groth(\kb W)[\tb,\ub,\tb^{-1},\ub^{-1}]$.
\end{lem}

\bigskip

\begin{rema}
A contrario, 
$$\isomorphisme{\kb[V]^\cow}_{\kb W}=\isomorphisme{\kb[V^*]^\cow}_{\kb W}=
\isomorphisme{\kb W}_{\kb W} = 
\sum_{\chi \in \Irr(W)} \chi(1) \chi$$ 
est un diviseur de $0$ dans $\groth(\kb W)$ (d\`es que $W \neq 1$).\finl
\end{rema}

\bigskip

Nous pouvons alors donner une autre formule pour $\dim_\kb^\bigrad(\kb[V \times V^*]^W)$~:

\bigskip

\begin{prop}\label{dim bigrad Q0 fantome}
$\DS{\dim_\kb^\bigrad(\kb[V \times V^*]^W)=\frac{1}{\prod_{i=1}^n (1-\tb^{d_i})(1-\ub^{d_i})} 
\sum_{\chi \in \Irr(W)} f_\chi(\tb) ~f_\chi(\ub).}$
\end{prop}

\begin{proof}
Soit $\HC$ un suppl\'ementaire $W$-stable et gradu\'e de $<\kb[V]^W>$ dans $\kb[V]$. 
Alors, puisque $\kb[V]$ est un $\kb[V]^W$-module libre, on a des isomorphismes 
de $\kb[W]$-modules gradu\'es
$$\kb[V] \simeq \kb[V]^W \otimes \HC\qquad\text{et}\qquad 
\kb[V]^\cow \simeq \HC.$$
De m\^eme, si $\HC^*$ est un suppl\'ementaire $W$-stable et gradu\'e de 
$<\kb[V^*]^W>$ dans $\kb[V^*]$, alors on a des isomorphismes 
de $\kb[W]$-modules gradu\'es
$$\kb[V^*] \simeq \kb[V^*]^W \otimes \HC^*\qquad\text{et}\qquad 
\kb[V^*]^\cow \simeq \HC^*.$$
En d'autres termes, en tant que $\kb[W]$-modules bi-gradu\'es, on a 
$$\kb[V] \simeq \kb[V]^W \otimes \kb[V]^\cow\qquad\text{et}\qquad 
\kb[V^*] \simeq \kb[V^*]^W \otimes \kb[V^*]^\cow.$$
On en d\'eduit un isomorphisme de $\kb$-espaces vectoriels bi-gradu\'es 
$$(\kb[V] \otimes \kb[V^*])^W \simeq 
(\kb[V]^W \otimes \kb[V^*]^W) \otimes (\kb[V]^\cow \otimes \kb[V^*]^\cow)^W.$$
Mais, d'apr\`es~\ref{fchi def} et~\ref{fchi},
$$\dim_\kb^\bigrad(\kb[V]^\cow \otimes \kb[V^*]^\cow)^W 
= \sum_{\chi, \psi \in \Irr(W)} f_\chi(\tb) f_\psi(\ub) \langle \chi\psi^*,\unb_W\rangle_W.$$
La formule annonc\'ee s'obtient maintenant en remarquant que 
$\langle \chi\psi^*,\unb_W\rangle  = \langle \chi,\psi \rangle_W$. 
\end{proof}

\bigskip

En conclusion, nous pouvons regrouper dans une m\^eme formule la formule de 
Molien~\ref{hilbert molien} et la proposition~\ref{dim bigrad Q0 fantome}:
\eqna
\DS{
\dim_\kb^\bigrad(\kb[V \times V^*]^W)
}
&=&
\DS{
\frac{1}{|W|} \sum_{w \in W} \frac{1}{\det(1-w\tb)~\det(1-w^{-1}\ub)}
}\\
&=&
\DS{
\frac{1}{\prod_{i=1}^n (1-\tb^{d_i})(1-\ub^{d_i})} 
\sum_{\chi \in \Irr(W)} f_\chi(\tb) ~f_\chi(\ub).
}
\endeqna

\chapter{Alg\`ebres de Hecke}\label{chapter:hecke}

\boitegrise{{\bf Notation.} 
{\it Dor\'enavant, et ce jusqu'\`a la fin de ce m\'emoire, 
nous fixons un corps de nombres $F$, \indexnot{F}{F} contenu dans $\kb$, galoisien sur $\QM$, 
contenant toutes les traces des \'el\'ements de $W$ et nous noterons $\OC$ \indexnot{O}{\OC} 
la cl\^oture int\'egrale de $\ZM$ dans $F$. Nous fixons aussi un plongement 
$F \injto \CM$. Par le th\'eor\`eme de Benard-Bessis, la $F$-alg\`ebre $F W$ est 
semi-simple d\'eploy\'ee~: il existe donc un sous-$F$-espace 
vectoriel $W$-stable $V_F$ \indexnot{V}{V_F}  de $V$ tel que $V = \kb \otimes_F V_F$. 
On notera $a \mapsto \aba$ la conjugaison 
complexe (elle stabilise $F$ car $F$ est Galoisien sur $\QM$). 
Pour finir, on notera $\mub_W$ \indexnot{mz}{\mub_W}  le groupe des racines de l'unit\'e du 
corps engendr\'e par les traces des \'el\'ements de $W$.}}{0.75\textwidth}

\bigskip

L'existence d'un tel corps $F$ est facile~: il suffit de prendre le corps engendr\'e par 
les traces des \'el\'ements de $W$ (c'est une extension galoisienne de $\QM$ car 
elle est contenue dans un corps cyclotomique). 
Notons aussi que $F$ contient toutes les racines de l'unit\'e de la forme $\z_{e_H}$, o\`u $H \in \AC$.

\bigskip

\section{D\'efinitions}

\bigskip

\subsection{Groupe de tresses}

Posons $V_\CM= \CM \otimes_F V_F$ 
et, si $H \in \AC$, notons $H_\CM = \CM \otimes_F (H \cap V_F)$. 
On d\'efinit alors
$$V_\CM^\reg = V_\CM \setminus \bigcup_{H \in \AC} H_\CM \indexnot{V}{V_\CM^\reg}$$
et on fixe un point $v_\CM \in V_\CM^\reg$. \indexnot{va}{v_\CM}  Si $v \in V_\CM$, on note $\vba$ son image 
dans la vari\'et\'e $V_\CM/W$. Le {\it groupe de tresses} associ\'e \`a $W$, not\'e $B_W$, 
est alors d\'efini par
$$B_W = \pi_1(V_\CM^\reg/W,\vba_\CM).\indexnot{B}{B_W}$$
Le {\it groupe de tresses pures} associ\'e \`a $W$, not\'e $P_W$, 
est lui d\'efini par
$$P_W = \pi_1(V_\CM^\reg,v_\CM).\indexnot{P}{P_W}$$
Le rev\^etement $V_\CM^\reg \to V_\CM^\reg/W$ \'etant non ramifi\'e 
(en vertu du th\'eor\`eme de Steinberg), on obtient une suite exacte
\equat\label{eq:suite exacte braid}
1 \longto P_W \longto B_W \stackrel{p_W}{\longto} W \longto 1.
\endequat
Si $H \in \AC$, nous noterons $\sigb_H$ \indexnot{sz}{\sigb_H}  un {\it g\'en\'erateur de la monodromie} 
autour de l'hyperplan $H$, tel que d\'efini dans~\cite[\S{2.A}]{BMR}, et tel que 
$p_W(\sigb_H)=s_H$. Rappelons~\cite[th\'eor\`eme~2.17]{BMR} que 
\equat\label{eq:BW engendre}
\text{\it $B_W$ est engendr\'e par $(\sigb_H)_{H \in \AC}$.}
\endequat

Nous noterons $\pib$ \indexnot{pzz}{\pib}   le lacet dans $V_\CM^\reg$ d\'efini par 
$$\fonction{\pib}{[0,1]}{V_\CM^\reg}{t}{e^{2i\pi t} v_\CM.}$$
Alors~\cite[Lemme~2.22]{BMR},
\equat\label{eq:pi central}
\pib\in P_W \cap \Zrm(B_W).
\endequat

\bigskip

\subsection{Alg\`ebre de Hecke g\'en\'erique}\label{sub:hecke} 
Rappelons que $\Omeb_W$ est l'ensemble des couples $(\O,j)$ tels que $\O \in \AC/W$ et $1 \le j \le e_\O -1$ 
(voir \S\ref{section:hyperplans}). 

\bigskip

\boitegrise{{\bf Notation.} 
{\it Nous noterons $\Omeb_W^\circ$ \indexnot{oz}{\Omeb_W^\circ}  l'ensemble des couples $(\O,j)$ tels que 
$\O \in \AC/W$ et $0 \le j \le e_\O -1$. Fixons une famille d'ind\'etermin\'ees 
$\qb_\gen=(\qb_{\O,j})_{(\O,j) \in \Omeb_W^\circ}$ \indexnot{qa}{\qb_\gen,~\qb_{\O,j}}  alg\'ebriquement 
ind\'ependantes sur $\OC$. On notera $\OC[\qb_\gen^{\pm 1}]$ \indexnot{O}{\OC[\qb_\gen^{\pm 1}]} 
l'anneau commutatif int\`egre et int\'egralement clos 
$\OC[(\qb_{\O,j}^{\pm 1})_{(\O,j) \in \Omeb_W^\circ}]$~: son corps des fractions 
sera not\'e $F(\qb_\gen)$. \indexnot{F}{F(\qb_\gen)}  
Si $\O \in \AC/W$, $H \in \O$ et $0 \le j \le e_H-1=e_\O-1$, on posera  
$\qb_{H,j}=\qb_{\O,j}$.\indexnot{qa}{\qb_{H,j}}}}{0.75\textwidth}

\bigskip

L'{\it alg\`ebre de Hecke g\'en\'erique} associ\'ee \`a $W$, not\'ee $\heckegenerique$, \indexnot{H}{\heckegenerique}
est le quotient de l'alg\`ebre de groupe $\OC[\qb_\gen^{\pm 1}] B_W$ par l'id\'eal engendr\'e par les \'el\'ements de la forme 
\equat\label{eq:relations Hecke}
\prod_{j = 0}^{e_H-1} (\sigb_H - \z_{e_H}^j \qb_{H,j}^{|\mub_W|}),
\endequat
o\`u $H$ parcourt $\AC$. Par cons\'equent, en notant $\Tb_H$ \indexnot{T}{\Tb_H}  l'image de 
$\sigb_H$ dans $\heckegenerique$, alors, d'apr\`es~\ref{eq:BW engendre},
\equat\label{eq:HW-engendre}
\text{\it $\heckegenerique$ est  engendr\'ee par $(\Tb_H)_{H \in \AC}$,}
\endequat
et, si $H \in \AC$, alors 
\equat\label{eq:TH-relation}
\prod_{j = 0}^{e_H-1} (\Tb_H - \z_{e_H}^j \qb_{H,j}^{|\mub_W|})=0.
\endequat
Notons que
\equat\label{eq:TH-inversible}
\text{\it $\Tb_H$ est inversible dans $\heckegenerique$.}
\endequat
Le lemme suivant d\'ecoule imm\'ediatement de~\cite[proposition~2.18]{BMR}~:

\bigskip

\begin{lem}\label{lem:hecke specialisation}
\`A travers la sp\'ecialisation $\qb_{\O,j} \mapsto 1$, on a un isomorphisme 
de $\OC$-alg\`ebres 
$\OC \otimes_{\OC[\qb_\gen^{\pm 1}]} \heckegenerique \longiso \OC W$.
\end{lem}

\bigskip

Notons par $a \mapsto \aba$ l'unique automorphisme de la $\ZM$-alg\`ebre 
$\OC[\qb_\gen^{\pm 1}]$ \'etendant la conjugaison complexe sur $\OC$ 
et tel que $\overline{\qb}_{\O,j}=\qb_{\O,j}^{-1}$. 
Lorsque nous parlerons d'alg\`ebres de Hecke, nous travaillerons le plus souvent sous 
l'hypoth\`ese suivante~\cite[\S{2.A}]{BMM}~:

\bigskip

\def\libertesymetrie{{(Lib-Sym)}}
\def\HHC{{\!\SSS{\HC}}}
\def\HHb{{\!\SSS{\Hb}}}

\begin{quotation}
{\bf Hypoth\`ese \libertesymetrie.} 
{\it 
\begin{itemize}
\itemth{a} $\heckegenerique$ est un $\OC[\qb_\gen^{\pm 1}]$-module libre de rang $|W|$.

\itemth{b} Il existe une unique forme sym\'etrisante $\taub_\HHC : \heckegenerique \to A$ 
telle que~:
\begin{itemize}
\itemth{1} Apr\`es la sp\'ecialisation du lemme~\ref{lem:hecke specialisation} (i.e. 
$\qb_{\O,j} \mapsto 1$), $\taub_\HHC$ se sp\'ecialise en la forme sym\'etrisante 
canonique de $\OC W$ (i.e. $w \mapsto \d_{w,1}$, o\`u $\d_?$ est le symbole de Kronecker). 

\itemth{2} Si $b \in B_W$, alors 
$$\taub_\HHC(\pib) \overline{\taub_\HHC(b^{-1})} = \taub_\HHC(b\pib).\indexnot{tx}{\taub_\HHC}$$
Ici, $\taub_\HHC(b)$ d\'esigne l'image par $\taub_\HHC$ de l'image de $b$ dans $\heckegenerique$. 
\end{itemize}
\end{itemize}
}
\end{quotation}

\bigskip

\noindent{\sc Commentaire - } 
Il est conjectur\'e~\cite[\S{2.A}]{BMM} que l'hypoth\`ese \libertesymetrie~est toujours 
v\'erifi\'ee. Elle est connue 
dans de nombreux cas, incluant les sous-familles infinies suivantes~:
\begin{itemize}
\item[$\bullet$]  $W$ groupe de Coxeter.

\item[$\bullet$] $W$ de type $G(de,e,r)$ dans la classification de Shephard-Todd.
\end{itemize}
D'autre part, notons que, si l'hypoth\`ese \libertesymetrie~ est v\'erifi\'ee, alors 
$\taub_\HHC(\pib) \neq 0$ car, en vertu du point (1) de l'assertion (b) et de~\ref{eq:pi central}, 
$\taub_\HHC(\pib)$ se sp\'ecialise en $1$ via $\qb_{H,j} \mapsto 1$.\finl

\bigskip

\subsection{Alg\`ebre de Hecke cyclotomique} 
Nous n'utiliserons pas ici les d\'efinitions usuelles d'alg\`ebre de Hecke 
cyclotomique~\cite[\S{6.A}]{BMM},~\cite[D\'efinition~4.3.1]{chlouveraki LNM}, 
car nous aurons besoin de travailler sur un anneau assez gros permettant 
de faire varier les param\`etres autant que possible. 

\bigskip

\boitegrise{{\bf Notation.} 
{\it Suivant~\cite{bonnafe two},~\cite{bonnafe continu} et~\cite{bonnafe faux}, nous 
utiliserons une notation exponentielle pour l'alg\`ebre de groupe $\OC[\RM]$, que nous noterons 
$\OC[\qb^\RM]$~: $\OC[\qb^\RM]= \bigoplus_{r \in \RM} ~\OC~\!\qb^r$,  \indexnot{O}{\OC[\qb^\RM]}  avec 
$\qb^r \qb^{r'}=\qb^{r+r'}$. Puisque $\OC$ est int\`egre et 
$\RM$ est sans torsion, $\OC[\qb^\RM]$ est int\`egre 
et nous noterons $F(\qb^\RM)$ \indexnot{F}{F(\qb^\RM)}  son corps des fractions. 
Si $a = \sum_{r \in \RM} a_r\qb^r$, nous noterons 
$\deg(a)$ \indexnot{da}{\deg}  (respectivement $\val(a)$) \indexnot{va}{\val}  son degr\'e (respectivement sa valuation), 
c'est-\`a-dire l'\'el\'ement de $\RM \cup \{-\infty\}$ 
(respectivement $\RM \cup \{+\infty\}$) d\'efini par \\
\centerline{$\deg(a) = \max\{r \in \RM~|~a_r \neq 0\}$}\\
\centerline{(respectivement \quad $\val(a)=\min\{r \in \RM~|~a_r \neq 0\}$).}}}{0.75\textwidth}

\bigskip

On prend $\deg(a) = -\infty$ (respectivement $\val(a)=+\infty$) si et seulement 
si $a=0$. Les propri\'et\'es usuelles des degr\'es et valuations (vis-\`a-vis de la somme et du produit) 
sont bien s\^ur v\'erifi\'ees. Commen\c{c}ons par une remarque facile~:

\medskip

\begin{lem}\label{lem:A-normal}
L'anneau $\OC[\qb^\RM]$ est int\'egralement clos.
\end{lem}

\begin{proof}
Cela r\'esulte du fait que $\OC[\qb^\RM]=\bigcup_{\L \subset \RM} \OC[\qb^\L]$, 
o\`u $\L$ parcourt les sous-groupes de type fini de $\RM$, et que, si $\L$ est de rang $e$, alors 
$\OC[\qb^\L] \simeq \OC[\tb_1^{\pm 1},\dots,\tb_e^{\pm 1}]$ est int\'egralement clos.
%
\end{proof}

\bigskip

Fixons une famille $k=(k_{\O,j})_{(\O,j) \in \Omeb_W^\circ}$ de nombres r\'eels 
(comme d'habitude, si $H \in \O$ et $0 \le i \le e_H -1$, alors 
on posera $k_{H,j}=k_{\O,j}$). On 
appellera {\it alg\`ebre de Hecke cyclotomique} ({\it de param\`etre $k$}), et on notera 
$\heckecyclotomique(k)$ \indexnot{H}{\heckecyclotomique(k)}  la $\OC[\qb^\RM]$-alg\`ebre 
$\OC[\qb^\RM] \stackrel{k}{\otimes}_{\OC[\qb_\gen^{\pm 1}]} \heckegenerique$, o\`u 
$\stackrel{k}{\otimes}_{\OC[\qb_\gen^{\pm 1}]}$ d\'esigne le produit tensoriel effectu\'e 
en voyant $\OC[\qb^\RM]$ comme une $\OC[\qb_\gen^{\pm 1}]$-alg\`ebre \`a travers le morphisme 
$$\fonction{\Th_k^\cyclo}{\OC[\qb_\gen^{\pm 1}]}{\OC[\qb^\RM]}{\qb_{\O,j}}{\qb^{k_{\O,j}}.}\indexnot{ty}{\Th_k^\cyclo}$$
Si le contexte est suffisamment clair, nous noterons $\otimes_{\OC[\qb_\gen^{\pm 1}]}$ le produit 
tensoriel $\stackrel{k}{\otimes}_{\OC[\qb_\gen^{\pm 1}]}$. 
Notons $T_H$ (ou $T_H^{(k)}$ \indexnot{T}{T_H,~T_H^{(k)}}  s'il peut y avoir une ambigu\"{\i}t\'e) l'image de 
$\Tb_H$ dans $\heckecyclotomique(k)$~; alors 
\equat\label{eq:hecke-engendre-cyclotomique}
\text{\it $\heckecyclotomique(k)$ est engendr\'ee par $(T_H)_{H \in \AC}$}
\endequat
et, si $H \in \AC$, alors 
\equat\label{eq:TW-relation-cyclotomique}
\prod_{j = 0}^{e_H-1} (T_H - \z_{e_H}^j \qb^{|\mub_W| k_{H,j}})=0.
\endequat

\bigskip

\begin{rema}
Il d\'ecoule du lemme~\ref{lem:hecke specialisation} qu'apr\`es la sp\'ecialisation 
$\OC[\qb^\RM] \to \OC$, $\qb^r \mapsto 1$ (c'est le morphisme d'augmentation pour le groupe 
$\RM$), on obtient $\OC \otimes_{\OC[\qb^\RM]} \heckecyclotomique(k) \simeq \OC W$. 

De m\^eme, $\heckecyclotomique(0)\simeq \OC[\qb^\RM] W$.\finl 
\end{rema}

\bigskip

\boitegrise{{\bf Notation.} 
{\it On notera $\CCB_\RM$ \indexnot{C}{\CCB_\RM}  le $\RM$-espace vectoriel des familles 
$k=(k_{\O,j})_{(\O,j) \in \Omeb_W^\circ}$ de nombres r\'eels telles que 
$\sum_{j=0}^{e_\O-1} k_{\O,j}=0$ pour tout $\O \in \AC/W$.}}{0.75\textwidth}

\bigskip

\begin{rema}\label{rem:translation}
Soit $(\l_\O)_{\O \in \AC/W}$ une famille de nombres r\'eels et, si $H \in \O$, 
posons $\l_H=\l_\O$. Soit $k_{\O,j}'=k_{\O,j} + \l_\O$ et notons $k'=(k_{\O,j}')_{(\O,j) \in \Omeb_W^\circ}$. 
Alors l'application $T_H^{(k)} \mapsto \qb^{-\l_H} T_H^{(k')}$ s'\'etend en un isomorphisme 
de $\OC[\qb^\RM]$-alg\`ebres $\heckecyclotomique(k) \simeq \heckecyclotomique(k')$. 

Ainsi, si on prend $\l_\O=-(k_{\O,0}+k_{\O,1}+\cdots+k_{\O,e_\O-1})/e_\O$, 
alors $\heckecyclotomique(k)\simeq \heckecyclotomique(k')$, avec 
$k' \in \CCB_\RM$.\finl 
\end{rema}

\section{Repr\'esentations}\label{section:representations-hecke}

\medskip

\boitegrise{{\bf Hypoth\`ese.} 
{\it Dor\'enavant, et ce jusqu'\`a la fin de ce chapitre, 
nous supposerons que l'hypoth\`ese \libertesymetrie~ est satisfaite.}}{0.75\textwidth}

\bigskip

\subsection{Cas g\'en\'erique}\label{subsection:cas-generique} 
Le r\'esultat suivant est d\^u \`a Malle~\cite[th\'eor\`eme~5.2]{malle} (la difficult\'e r\'eside 
dans le d\'eploiement)~:

\bigskip

\begin{theo}[Malle]\label{theo:hecke-deployee}
La $F(\qb_\gen)$-alg\`ebre $F(\qb_\gen)\heckegenerique$ est semi-simple et d\'eploy\'ee.
\end{theo}

\bigskip

Puisque l'alg\`ebre $FW$ est elle aussi semi-simple et d\'eploy\'ee (th\'eor\`eme de 
Benard-Bessis~\ref{deploiement}), il d\'ecoule du th\'eor\`eme de d\'eformation de 
Tits~\cite[th\'eor\`eme~7.4.6]{geck} que l'on a une bijection 
$$
\begin{array}{ccc}
\Irr(W) & \longiso & \Irr(F(\qb_\gen)\heckegenerique) \\
\chi & \longmapsto & \chi^\gen
\end{array}
$$
d\'efinie par la propri\'et\'e suivante~: $\chi$ est la sp\'ecialisation de $\chi^\gen$ 
par $\qb_{\O,j} \mapsto 1$. D'apr\`es~\cite[th\'eor\`eme~7.2.6]{geck}, il existe une unique famille 
$(\sb_\chi^\gen)_{\chi \in \Irr(W)}$ \indexnot{sa}{\sb_\chi^\gen}  d'\'el\'ements de $\OC[\qb_\gen^{\pm 1}]$ telle que 
$$\taub_\HHC=\sum_{\chi \in \Irr(W)} \frac{1}{\sb_\chi^\gen}~\chi^\gen.$$
L'\'el\'ement $\sb_\chi^\gen$ est appel\'e 
l'{\it \'el\'ement de Schur g\'en\'erique associ\'e \`a $\chi$}. 
Les \'el\'ements de Schur ont \'et\'e calcul\'es, au cas par cas, par de nombreux auteurs 
(Alvis, Benson, Geck, Iancu, Lusztig, Malle, Mathas, Surowski,\dots) et cette
description a \'et\'e achev\'ee par M. Chlouveraki qui en a d\'eduit une
propri\'et\'e uniforme 
des \'el\'ements de Schur (voir~\cite[th\'eor\`eme~3.2.5]{maria} 
ou~\cite[th\'eor\`eme~4.2.5]{chlouveraki LNM}). 
Ce r\'esultat a \'et\'e d\'emontr\'e ind\'ependamment par le second auteur 
par un argument g\'en\'eral utilisant les alg\`ebres de Cherednik~\cite{rouquier schur}.

\bigskip

\begin{theo}\label{theo:schur}
Soit $\chi \in \Irr(W)$. Alors il existe un entier naturel non nul $m$, un \'el\'ement 
$\xib_\chi^\gen \in \OC \setminus\{0\}$, des mon\^omes $M_{\chi,0}$, $M_{\chi,1}$,\dots, $M_{\chi,m}$ 
dans $\OC[\qb_\gen^{\pm 1}]$ et des polyn\^omes $F$-cyclotomiques $\Psi_{\chi,1}$,\dots, 
$\Psi_{\chi,m}$ tels que 
$$\sb_\chi^\gen  = \xib_\chi^\gen M_{\chi,0} 
\Psi_{\chi,1}(M_{\chi,1}) \cdots \Psi_{\chi,m}(M_{\chi,m}).$$
De plus, $\Psi_{\chi,i}(\tb) \neq \tb-1$. 
\end{theo}

 \bigskip

\begin{rema}\label{rem-poly-cyclo}
Un polyn\^ome {\it $F$-cyclotomique} est un polyn\^ome minimal, sur $F$, d'une 
racine de l'unit\'e. Par exemple, sur $\QM(\sqrt{2})$, $\tb^2 - \sqrt{2} \tb + 1$ est un polyn\^ome 
$F$-cyclotomique (c'est le polyn\^ome minimal d'une des racines primitives huiti\`emes 
de l'unit\'e).

La derni\`ere assertion d\'ecoule du th\'eor\`eme de d\'eformation de Tits~\cite[th\'eor\`eme~7.4.6]{geck}~:  
en effet, d'apr\`es le lemme~\ref{lem:hecke specialisation}, la sp\'ecialisation $\qb_{\O,j} \mapsto 1$ 
nous donne l'alg\`ebre de groupe $\kb W$ qui est semi-simple d\'eploy\'ee.\finl
\end{rema}

\bigskip

Notons $\omeb_\chi^\gen : \Zrm(F(\qb_\gen)\heckegenerique) \longto F(\qb_\gen)$ \indexnot{ozz}{\omeb_\chi^\gen}
le caract\`ere central associ\'e au caract\`ere $\chi$~: si $z \in \Zrm(F(\qb_\gen)\heckegenerique)$, 
alors $\omeb_\chi^\gen(z)$ est l'\'el\'ement de $F(\qb_\gen)$ par lequel $z$ agit 
sur un module simple de caract\`ere $\chi^\gen$. C'est un morphisme 
de $F(\qb_\gen)$-alg\`ebres. 
Puisque $\OC[\qb_\gen^{\pm 1}]$ est int\'egralement clos, $\omeb_\chi^\gen$ 
se restreint en un morphisme de $\OC[\qb_\gen^{\pm 1}]$-alg\`ebres 
$\omeb_\chi^\gen : \Zrm(\heckegenerique) \longto \OC[\qb_\gen^{\pm 1}]$. 

Puisque $\pib \in \Zrm(B_W)$, son image dans $\heckegenerique$ appartient 
au centre de cette alg\`ebre. On peut donc \'evaluer $\omeb_\chi^\gen$ en $\pib$. 
Avant de donner la formule qui d\'ecrit cette \'evaluation, 
nous aurons besoin de la notation suivante~: si $(\O,j) \in \Omeb_W^\circ$ 
et $H \in \O$, posons 
$$m_{\O,j}^\chi = \langle \Res_{W_H}^W \chi, {\det}^j \rangle_{W_H}.\indexnot{ma}{m_{\O,j}^\chi}$$
Alors, on a~\cite[4.17]{broue-michel}
\equat\label{eq:m-entier}
\frac{m_{\O,j}^\chi|\O|e_\O}{\chi(1)} \in \NM
\endequat
et~\cite[proposition~4.16]{broue-michel}
\equat\label{eq:action-pi}
\omeb_\chi^\gen(\pib) = \prod_{(\O,j) \in \Omeb_W^\circ} 
\qb_{\O,j}^{|\mub_W| \cdot \frac{\SS{m_{\O,j}^\chi|\O|e_\O}}{\SS{\chi(1)}}}.
\endequat

\bigskip

\subsection{Cas cyclotomique}\label{sub:cas-cyclo}
Le corollaire suivant d\'ecoule facilement de la forme des \'el\'ements de Schur 
donn\'ee par le th\'eor\`eme~\ref{theo:schur}~:

\bigskip

\begin{coro}\label{coro:schur}
Si $\chi \in \Irr(W)$, alors $\Th_k^\cyclo(\sb_\chi^\gen) \neq 0$. 
\end{coro}

\bigskip

Le r\'esultat suivant d\'ecoule du th\'eor\`eme de Malle~\ref{theo:hecke-deployee} et du corollaire~\ref{coro:schur} 
(voir~\cite[th\'eor\`eme~2.4.12 et proposition~1.4.1]{chlouveraki LNM} et~\cite[proposition~3.2.1]{chlouveraki LNM}).

\bigskip

\begin{coro}\label{coro:hecke-deployee}
La $F(\qb^\RM)$-alg\`ebre $F(\qb^\RM)\heckecyclotomique(k)$ 
est semi-simple d\'eploy\'ee. 
\end{coro}

\bigskip

Par le th\'eor\`eme de d\'eformation de Tits~\cite[th\'eor\`eme~7.4.6]{geck}, on obtient encore une suite  
de bijections 
$$\begin{array}{ccccc}
\Irr(W) &\longiso & \Irr(F(\qb^\RM)\heckecyclotomique(k)) & \longiso & \Irr(F(\qb_\gen)\heckegenerique) \\
\chi & \longmapsto & \chi_k^\cyclo & \longmapsto & \chi^\gen
\end{array}\indexnot{kz}{\chi^\gen,~\chi_k^\cyclo} 
$$
telles que $\chi_k^\cyclo = \Th_k^\cyclo \circ \chi^\gen$. On notera 
$\tau_k^\cyclo$ \indexnot{tx}{\t_k^\cyclo}  la sp\'ecialisation de la forme sym\'etrisante $\taub_\HHC$ 
en une forme sym\'etrisante sur $\heckecyclotomique(k)$ et, si $\chi \in \Irr(W)$, 
on notera $s_\chi^\cyclo(k)$ \indexnot{sa}{s_\chi^\cyclo(k)}  l'\'el\'ement de Schur associ\'e \`a $\chi_k^\cyclo$~: 
on a alors 
\equat\label{eq:schur-specialise}
s_\chi^\cyclo(k) = \Th_k^\cyclo(\sb_\chi^\gen)
\endequat
et
\equat\label{eq:tau-specialise}
\t_k^\cyclo = \sum_{\chi \in \Irr(W)} \frac{1}{s_\chi^\cyclo(k)} ~\chi_k^\cyclo.
\endequat
Compte tenu du th\'eor\`eme~\ref{theo:schur}, il d\'ecoule de~(\ref{eq:schur-specialise}) que~:

\bigskip

\begin{coro}\label{coro:schur-specialise}
Soit $\chi \in \Irr(W)$. Alors il existe un entier naturel non nul $m$, un \'el\'ement 
$\xi_{\chi,k}^\cyclo \in \OC\setminus\{0\}$, un nombre r\'eel $r_0$, 
des nombres r\'eels {\bfit non nuls} $r_1$,\dots, $r_m$, 
et des polyn\^omes $F$-cyclotomiques $\Psi_{\chi,1}$,\dots, 
$\Psi_{\chi,m}$ tels que 
$$s_{\chi,k}^\cyclo  = \xi_{\chi,k}^\cyclo \qb^{r_0} 
\Psi_{\chi,1}(\qb^{r_1}) \cdots \Psi_{\chi,m}(\qb^{r_m}).$$
\end{coro}

\bigskip

\begin{rema}\label{rem:comparaison-schur}
L'entier $m$ du corollaire~\ref{coro:schur-specialise} et l'entier $m$ du th\'eor\`eme~\ref{theo:schur} 
ne sont pas forc\'ement les m\^emes, de m\^eme que les polyn\^omes $F$-cyclotomiques $\Psi_{\chi,j}$~: 
en effet, il se peut que 
certains des mon\^omes $M_{\chi,j}$ du th\'eor\`eme~\ref{theo:schur} 
se sp\'ecialisent en $1=\qb^0$ via le morphisme $\Th_k^\cyclo$. 
Dans ce cas, $\Psi_{\chi,j}(1)$ devient un \'el\'ement de $\OC$ et est int\'egr\'e 
dans la constante $\xi_{\chi,k}^\cyclo$~: cela montre aussi que $\xib_\chi^\gen$ et 
$\xi_{\chi,k}^\cyclo$ peuvent \^etre diff\'erents. On peut simplement dire que 
$\xib_\chi^\gen$ divise $\xi_{\chi,k}^\cyclo$ dans $\OC$.\finl
\end{rema}

Pour finir, notons $\o_{\chi,k}^\cyclo : \Zrm(\heckecyclotomique(k)) \longto \OC[\qb^\RM]$ 
le caract\`ere central associ\'e \`a $\chi_k^\cyclo$. \indexnot{ozz}{\o_{\chi,k}^\cyclo}  Posons
$$c_\chi(k)=\sum_{(\O,j) \in \Omeb_W^\circ} 
k_{\O,j} \cdot \frac{m_{\O,j}^\chi|\O|e_\O}{\chi(1)}.\indexnot{ca}{c_\chi(k)}$$
Alors il d\'ecoule de~(\ref{eq:action-pi}) que 
\equat\label{eq:action-pi-cyclo}
\o_{\chi,k}^\cyclo(\pib) = \qb^{|\mub_W| c_\chi(k)}.
\endequat

\bigskip

\subsection{Familles de Hecke}\label{sub:rouquier}
On appelle {\it anneau anti-cyclotomique}, et on note $\OC^\cyclo[\qb^\RM]$, \indexnot{O}{\OC^\cyclo[\qb^\RM]}  l'anneau 
$$\OC^\cyclo[\qb^\RM] = \OC[\qb^\RM][\bigl((1-\qb^r)^{-1}\bigr)_{r \in \RM\setminus\{0\}}].$$
Si $b$ est un idempotent central (pas n\'ecessairement primitif) 
de $\OC^\cyclo[\qb^\RM]\heckecyclotomique(k)$, nous noterons 
$\Irr_\HC(W,b)$ \indexnot{I}{\Irr_\HC(W,b)}  l'ensemble des caract\`eres irr\'eductibles $\chi$ de $W$ tels que 
$\chi_k^\cyclo \in \Irr(F(\qb^\RM)\heckecyclotomique(k)b)$. 

\bigskip

\begin{defi}\label{defi:famille-rouquier}
On appellera {\bfit $k$-famille de Hecke} toute partie de $\Irr(W)$ de la forme 
$\Irr_\HC(W,b)$, o\`u $b$ est un idempotent primitif central de $\OC^\cyclo[\qb^\RM]\heckecyclotomique(k)$.
\end{defi}

\bigskip

Les $k$-familles de Hecke forment donc une partition de $\Irr(W)$.

\bigskip

\begin{lem}[Brou\'e-Kim]\label{lem:c-contant}
Si $\chi$ et $\chi'$ sont dans la m\^eme $k$-famille de Hecke, alors 
$c_\chi(k)=c_{\chi'}(k)$.
\end{lem}

\begin{proof}
On pourrait appliquer l'argument contenu dans~\cite[proposition~2.9(2)]{broue kim}. 
Cependant, notre cadre est l\'eg\`erement diff\'erent et nous proposons une preuve diff\'erente, 
reposant sur la forme particuli\`ere de $\o_{\chi,k}^\cyclo(\pib)$ 
(voir~\ref{eq:action-pi-cyclo}). 

Notons $\EC=\{r_1,r_2,\dots, r_m\}$, avec $r_i \neq r_j$ si $i \neq j$, 
l'image de l'application $\Irr(W) \to \RM$, $\chi \mapsto |\mub_W| c_\chi(k)$. 
Si $1 \le j \le m$, nous noterons 
$$\FC_j=\{\chi \in \Irr(W)~\bigl|~|\mub_W| c_\chi(k)=r_j\}.$$
Si $\chi \in \Irr(W)$, nous noterons $e_{\chi,k}$ l'idempotent 
primitif central de $F(\qb^\RM)\heckecyclotomique(k)$ associ\'e. 
On pose
$$b_j = \sum_{\chi \in \FC_j} e_{\chi,k}.$$
Pour montrer le lemme, il suffit de montrer que $b_j \in \OC^\cyclo[\qb^\RM]\heckecyclotomique(k)$. 
Si on note $\pibt$ l'image de $\pib$ dans $\heckecyclotomique(k)$, alors 
$$\pibt=\qb^{r_1} b_1 + \qb^{r_2} b_2 + \cdots + \qb^{r_m} b_m.$$
Ainsi, 
$$
\left\{\begin{array}{ccccccccc}
b_1 &+& b_2 &+& \cdots &+& b_m &=& 1\\
\qb^{r_1} b_1 &+& \qb^{r_2} b_2 &+& \cdots &+& \qb^{r_m} b_m &=&  \pibt \\
 && \cdots & \\
\qb^{(m-1)r_1} b_1 &+&  \qb^{(m-1)r_2} b_2 &+& \cdots &+& \qb^{(m-1)r_m} b_m &=&  \pibt^{m-1}.
\end{array}\right.
$$
Or, le d\'eterminant de ce syst\`eme est un d\'eterminant de Vandermonde \'egal \`a 
$$\prod_{1 \le i < j \le m} (\qb^{r_i}-\qb^{r_j}),$$
qui est inversible dans l'anneau de Hecke $\OC^\cyclo[\qb^\RM]$ par construction. 
De plus, $1$, $\pibt$,\dots, $\pibt^{m-1} \in \heckecyclotomique(k)$, 
d'o\`u le r\'esultat.
\end{proof}

\bigskip

\section{Cas des r\'eflexions d'ordre $2$}\label{sec:ordre-2}

\medskip

\boitegrise{{\bf Hypoth\`ese.} 
{\it Dans cette section, et dans cette section seulement, nous supposerons 
que toutes les r\'eflexions de $W$ sont d'ordre $2$.}}{0.75\textwidth}

\bigskip

Notons $\OC[\qb_\gen^{\pm 1}] \to \OC[\qb_\gen^{\pm 1}]$, $f \mapsto f^\dagger$ \indexnot{ZZZ}{\dagger}  
l'unique automorphisme involutif 
de $\OC$-alg\`ebre \'echangeant $\qb_{\O,0}$ et $\qb_{\O,1}$ pour tout $\O \in \AC/W$. Notons encore 
$\OC[\qb_\gen^{\pm 1}] B_W \to \OC[\qb_\gen^{\pm 1}] B_W$, $a \mapsto a^\dagger$ l'unique 
automorphisme semi-lin\'eaire (pour l'involution $f \mapsto f^\dagger$ de $\OC[\qb_\gen^{\pm 1}]$) 
tel que $\b^\dagger=\e(p_W(\b))\b$ pour tout $\b \in B_W$. 
Les relations~(\ref{eq:relations Hecke}) sont stables par cet automorphisme. 
Il induit donc un automorphisme semi-lin\'eaire $\heckegenerique \to \heckegenerique$, $h \mapsto h^\dagger$ de l'alg\`ebre de Hecke 
g\'en\'erique. 

Cet automorphisme, apr\`es la sp\'ecialisation $\qb_{\O,j} \mapsto 1$, devient l'unique 
automorphisme $\OC$-lin\'eaire de $\OC W$ qui envoie $w \in W$ sur $\e(w)w$. En d'autres termes, 
c'est l'automorphisme induit par le caract\`ere lin\'eaire $\e$. 
 
De m\^eme, puisque $k_{\O,0}+k_{\O,1}=0$, si on note encore $\OC[\qb^\RM] \to \OC[\qb^\RM]$, $f \mapsto f^\dagger$ 
l'unique automorphisme de $\OC$-alg\`ebre tel que $(\qb^r)^\dagger=\qb^{-r}$, alors la sp\'ecialisation 
$\qb_{\O,j} \mapsto \qb^{k_{\O,j}}$ induit un automorphisme $\OC[\qb^\RM]$-semi-lin\'eaire de l'alg\`ebre $\heckecyclotomique$ 
toujours not\'e $h \mapsto h^\dagger$. 
Si $\chi \in \Irr(W)$, notons $(\chi^\gen)^\dagger$ (respectivement $(\chi^\cyclo_k)^\dagger$) 
la composition de $\chi^\gen$ (respectivement $\chi_k^\cyclo$) 
avec l'automorphisme $\dagger$~: 
c'est un nouveau caract\`ere irr\'eductible de $F(\qb_\gen)\heckegenerique$ 
(respectivement $F(\qb^\RM)\heckecyclotomique(k)$). 
Puisqu'il est d\'etermin\'e par sa sp\'ecialisation via $\qb_{\O,j} \mapsto 1$ 
(respectivement $\qb^r \mapsto 1$), 
on a 
\equat\label{eq:chi-dagger}
(\chi^\gen)^\dagger = (\chi\e)^\gen\qquad\text{et}\qquad (\chi^\cyclo_k)^\dagger = (\chi\e)^\cyclo_k.
\endequat
L'automorphisme $f \mapsto f^\dagger$ de $\OC[\qb^\RM]$ s'\'etendant \`a l'anneau 
$\OC^\cyclo[\qb^\RM]$, le lemme suivant d\'ecoule imm\'ediatement de ces observations~:

\bigskip

\begin{lem}\label{lem:rouquier-ordre-2}
Supposons que toutes les r\'eflexions de $W$ sont d'ordre $2$. Si 
$\FC$ est une $k$-famille de Hecke, alors $\FC\e$ est une $k$-famille de Hecke. 
\end{lem}

\bigskip

\section{Commentaire sur le choix de l'anneau}

\medskip
 
Il peut para\^{\i}tre \'etrange de travailler avec un anneau de coefficients aussi 
\'enorme (tr\`es loin d'\^etre noeth\'erien notamment). Tout d'abord, comme cela 
a \'et\'e remarqu\'e dans la preuve du lemme~\ref{lem:A-normal}, 
cet anneau est une r\'eunion croissante d'anneaux noeth\'eriens raisonnables, 
ce qui permet souvent de raisonner comme s'il \'etait noeth\'erien. 

D'autre part, ce choix nous permet de travailler avec un anneau fixe, quel que soit notre choix de param\`etre $k$~: 
comme nous aurons besoin de faire varier $k$ dans un espace de param\`etres 
{\it r\'eels}, ce choix s'est impos\'e. Comme cela a \'et\'e vu dans le 
corollaire~\ref{coro:hecke-deployee}, le fait qu'il soit possible d'extraire 
des racines quelconques de toutes les ``puissances'' de $\qb$ entra\^{\i}ne le d\'eploiement 
de toutes les alg\`ebres de Hecke cyclotomiques sur ce m\^eme anneau fixe. 

Pour finir, cet anneau est de la forme $\OC[\G]$, 
o\`u $\G$ est un groupe ab\'elien totalement ordonn\'e, ce qui permet, 
gr\^ace aux notions de degr\'es et de valuation, de d\'efinir les invariants 
$\ab$ et $\Ab$ associ\'es aux caract\`eres irr\'eductibles de $W$. 
Comme nous le verrons aussi dans le chapitre~\ref{chapter:coxeter}, 
c'est le cadre g\'en\'eral de la th\'eorie de Kazhdan-Lusztig, dont nous 
souhaitons proposer une possible g\'en\'eralisation aux groupes de r\'eflexions 
complexes.

\bigskip

\section{Autres choix d'anneaux de base}

\medskip

Soit $B$ une $\OC[\qb_\gen^{\pm 1}]$-alg\`ebre commutative int\`egre, de corps
des fractions $F_B$. On pose
$\HCB_B=B\otimes_{\OC[\qb_\gen^{\pm 1}]} \heckegenerique$.

Soit $L_B$ le sous-groupe de $B^\times$ engendr\'e par les $1_Bq_{\Omega,j}$. C'est
un quotient de $\ZM^{\Omeb_W^\circ}$.
On suppose que $L_B$ n'a pas de torsion, que $\OC[L_B]$ s'injecte dans $B$,
et que $F_{\OC[L_B]}\cap B=\OC[L_B]$.

Comme dans le corollaire~\ref{coro:hecke-deployee}, la $F_B$-alg\`ebre
$F_B\HCB_B$ est semi-simple d\'eploy\'ee et on obtient une
bijection $\Irr(W)\iso\Irr(F_B\HCB_B)$.

\bigskip

\begin{exemple}
\label{groupeabelien}
Soit $A$ un groupe ab\'elien sans torsion.
On note, comme dans le lemme~\ref{lem:A-normal}, que $\OC[A]$ est int\'egralement clos.
Consid\'erons une application $q:\Omeb_W^\circ\to A$. Elle s'\'etend en un morphisme
de groupes $\ZM^{\Omeb_W^\circ}\to A$ et en un morphisme entre alg\`ebres de groupes
 $\OC[\qb_\gen^{\pm 1}]\to \OC[A]$. L'alg\`ebre $B=\OC[A]$ satisfait les hypoth\`eses
pr\'ec\'edentes.\finl
\end{exemple}

\bigskip

Soit $B^\cyclo=B[(1-v)_{v\in L_B-\{0\}}^{-1}]$ (nous notons additivement
les groupes ab\'eliens dans cette section). On d\'efinit comme dans
\S\ref{sub:rouquier} la notion de $B$-{\em famille de Hecke}.

\bigskip

\begin{prop}
\label{anneaubase}
Les $B$-familles de Hecke co{\"\i}ncident avec les $\OC[L_B]$-familles de Hecke.
\end{prop}

\begin{proof}
Soit $A=L_B$.
On a $B^\cyclo\cap F_{\OC[A]}=\OC[A]^\cyclo$. On en d\'eduit que, 
si $b$ est un idempotent primitif central de $F_{\OC[A]}\HCB_{\OC[A]}$
tel que $b\in B^\cyclo\HCB_B$, alors $b\in \OC[A]^\cyclo\HCB_{\OC[A]}$.
\end{proof}

\bigskip

La proposition pr\'ec\'edente ram\`ene l'\'etude des familles de Hecke
au cas o\`u $B=\OC[A]$ et $A$ est un groupe ab\'elien sans torsion
quotient de $\ZM^{\Omeb_W^\circ}$.

\medskip
Soit $\MCB$ l'ensemble des mon\^omes $M_{\chi,i}$ donn\'es par le 
th\'eor\`eme~\ref{theo:schur}, pour $\chi\in\Irr(W)$ et $i\not=0$. C'est un sous-ensemble
fini de $\ZM^{\Omeb_W^\circ}$.

Consid\'erons maintenant un groupe ab\'elien sans torsion $A$ et une
application $q:\Omeb_W^\circ\to A$ comme dans l'exemple~\ref{groupeabelien}.

Soit $A'$ un groupe ab\'elien sans torsion et
$f:A\to A'$ un morphisme surjectif de groupes.

\begin{prop}
\label{quotientabelien}
Si $q(\MCB)\cap \ker f=\{0\}$, alors les $\OC[A]$-familles de Hecke
co{\"\i}ncident avec les $\OC[A']$-familles de Hecke.
\end{prop}

\begin{proof}
Le morphisme $f$ induit un morphisme surjectif entre alg\`ebres de groupes
$\OC[A]\to\OC[A']$ qui s'\'etend en un morphisme d'alg\`ebres surjectif entre
localisations
$f:\OC[A][(1-v)^{-1}_{v\in A-\ker f}]\to \OC[A']^\cyclo$.
Le th\'eor\`eme~\ref{theo:schur} fournit des polyn\^omes $F$-cyclotomiques
$\Psi_{\chi,i}$.
Soit 
$$h\in F[A][\{\Psi_{\chi,i}(q(M_{\chi,i}))^{-1}\}_{M_{\chi,i}\not\in\ker q}].$$
Si $h\in \OC[A]^\cyclo$, alors 
$h\in f^{-1}(\OC[A']^\cyclo)$.

Il r\'esulte du th\'eor\`eme~\ref{theo:schur} que les idempotents de
$Z(F(A)\HCB_{\OC[A]})$ sont dans l'alg\`ebre 
$F[A][\{\Psi_{\chi,i}(q(M_{\chi,i}))^{-1}\}_{M_{\chi,i}\not\in\ker q}]
\HCB_{\OC[A]}$. Par cons\'equent, tout
idempotent de $Z(\OC[A]^\cyclo\HCB_{\OC[A]})$ est contenu dans
$\OC[A][(1-v)^{-1}_{v\in A-\ker f}]\HCB_{\OC[A]}$. 
La proposition~\ref{muller} montre que les idempotents centraux de 
$\OC[A][(1-v)^{-1}_{v\in A-\ker f}]\HCB_{\OC[A]}$ sont en bijection avec ceux
de $\OC[A']^\cyclo\HCB_{\OC[A']}$ et la proposition en d\'ecoule.
\end{proof}

\'Etant donn\'es $A$ et $q$ comme ci-dessus, il existe un morphisme 
de groupes $f:A\to\ZM$ tel que $\ker f\cap q(\MCB)=\{0\}$. La proposition 
\ref{quotientabelien} ram\`ene donc l'\'etude des $\OC[A]$-familles de Hecke 
(et donc celle des $B$-familles de Hecke d'apr\`es ce qui pr\'ec\`ede) 
au cas de $\OC[t^{\pm 1}]$-familles de Hecke, pour un choix 
d'entiers $m_{\Omega,j}\in\ZM$ d\'efinissant un morphisme de groupes 
$\ZM^{\Omeb_W^\circ}\to t^\ZM,\ q_{\Omega,j}\mapsto t^{m_{\Omega,j}}$.

\chapter{Sp\'ecificit\'es des groupes de Coxeter}\label{chapter:coxeter}

\section{Groupes}\label{sec:coxeter-general}

\medskip

\subsection{Groupe de r\'eflexions} 
Rappelons les \'equivalences classiques suivantes~:

\bigskip

\begin{prop}\label{prop:coxeter}
Les assertions suivantes sont \'equivalentes~:
\begin{itemize}
\itemth{1} Il existe une partie $S$ de $\Ref(W)$ telle que $(W,S)$ soit un syst\`eme de Coxeter.

\itemth{2} $V \simeq V^*$ comme $\kb W$-modules.

\itemth{3} Il existe une forme bilin\'eaire sym\'etrique non d\'eg\'en\'er\'ee 
et $W$-invariante $V \times V \to \kb$.

\itemth{4} Il existe un sous-corps $\kb_\RM$ de $\kb$ et un sous-$\kb_\RM$-espace vectoriel 
$W$-stable $V_{\kb_\RM}$ de $V$ tels que $V = \kb \otimes_{\kb_\RM} V_{\kb_\RM}$ et 
$\kb_\RM$ est isomorphe \`a un sous-corps de $\RM$.
\end{itemize}
\end{prop}

\bigskip

Lorsque l'une des (ou toutes les) assertions de la proposition~\ref{prop:coxeter} 
sont v\'erifi\'ees, nous dirons que {\it $W$ est un groupe de Coxeter}. 
Lorsque ce sera le cas, le texte sera accompagn\'e d'un liseret gris sur la gauche, 
comme ci-dessous.

\bigskip

\cbstart

\boitegrise{{\bf Hypoth\`ese, choix.} {\it Dor\'enavant, et ce jusqu'\`a la fin de ce chapitre, 
nous supposerons que $W$ est un groupe de Coxeter. 
Nous fixons alors un sous-corps $F_\RM$  \indexnot{F}{F_\RM}  de $F \cap \RM$ contenant les traces des \'el\'ements 
de $W$, un sous-$F_\RM$-espace vectoriel $W$-stable $V_{F_\RM}$  \indexnot{V}{V_{F_\RM}}  de $V$ tel que 
$V=\kb \otimes_{F_\RM} V_{F_\RM}$, un \'el\'ement $v_\RM$ \indexnot{va}{v_\RM}  de $V_{F_\RM}$ tel que 
$s(v_\RM) \neq v_\RM$ pour tout 
$s \in \Ref(W)$ (et nous supposerons que $v_\CM=v_\RM$) 
et nous choisirons $\a_s$ de sorte que 
$\langle v_\RM,\a_s \rangle > 0$ pour tout $s \in \Ref(W)$. 
Nous noterons alors $S$ l'ensemble des $s \in \Ref(W)$ tels que 
$\Ker_{V_\RM}(\a_s)$ soit le seul hyperplan s\'eparant $v_\RM$ de $s(v_\RM)$. 
Ainsi $(W,S)$ est un syst\`eme de Coxeter. Ces notations 
seront en vigueur dans tout le reste de ce m\'emoire, d\`es que $W$ est un 
groupe de Coxeter.}}{0.75\textwidth}

\bigskip

Il d\'ecoule du th\'eor\`eme~\ref{deploiement} que~:

\bigskip

\begin{lem}\label{lem:coxeter}
La $F_\RM$-alg\`ebre $F_\RM W$ est d\'eploy\'ee. En particulier, 
les caract\`eres de $W$ sont \`a valeurs r\'eelles, c'est-\`a-dire que $\chi=\chi^*$ 
pour tout caract\`ere $\chi$ de $W$.
\end{lem}

\bigskip

Rappelons aussi que, par la force des choses,

\bigskip

\begin{lem}\label{lem:coxeter-2}
Si $s \in \Ref(W)$, alors $s$ est d'ordre $2$ et $\e(s)=-1$.
\end{lem}

\bigskip

En particulier, tous les r\'esultats de \S\ref{sec:ordre-2} s'appliquent. 

\bigskip

\begin{coro}\label{coro:coxeter}
L'application $\Ref(W) \to \AC$, $s \mapsto \Ker(s-\Id_V)$ est bijective et $W$-\'equivariante. 
En particulier, $|\AC|=|\Ref(W)|=\sum_{i=1}^n (d_i-1)$ et $|\AC/W|=|\Ref(W)/\!\sim|$.
\end{coro}

\bigskip

\boitegrise{{\bf Notations.} 
{\it Notons $\ell : W \to \NM$ la fonction longueur par rapport \`a $S$. 
Si $w=s_1s_2\cdots s_l$ avec $s_i \in S$ et $l=\ell(w)$, on dit alors que 
$w=s_1s_2\cdots s_l$ est une {\bfit d\'ecomposition r\'eduite} de $w$. 
Nous noterons $w_0$ \indexnot{wa}{w_0}  l'\'el\'ement le plus long de $W$~: on a $\ell(w_0)=|\Ref(W)|=|\AC|$. 
}}{0.75\textwidth}

\bigskip

\begin{rema}\label{rem:w0-central}
Si $-\Id_V \in W$, alors $w_0=-\Id_V$. R\'eciproquement, si 
$w_0$ est central et $V^W=0$, alors $w_0=-\Id_V$.\finl
\end{rema}

%
%
%
%
%
%
%

\subsection{Groupe de tresses} 
Pour $s$, $t \in S$, notons $m_{st}$ \indexnot{ma}{m_{st}}  l'ordre de $st$ dans $W$. Pour $s \in S$ 
et $H=\Ker(s-\Id_V)$, notons $\sigb_s = \sigb_H$ \indexnot{sz}{\sigb_s}  le lacet image dans $V_\CM^\reg/W$ 
du chemin 
$$\fonctio{[0,1]}{V_\CM^\reg}{t}{e^{i\pi t} \Bigl(\DS{\frac{v_\RM-s(v_\RM)}{2} + \frac{v_\RM+s(v_\RM)}{2}}\Bigr)}$$
de $v_\RM$ vers $s(v_\RM)$. 
Avec ces notations, $B_W$ admet la pr\'esentation suivante~\cite{brieskorn}~:
\equat\label{eq:coxeter-tresse}
B_W\quad : \quad 
\begin{cases}
\text{G\'en\'erateurs~:} & (\sigb_s)_{s \in S}, \\
\text{Relations~:} & \forall~s,t \in S,~\underbrace{\sigb_s\sigb_t\sigb_s\cdots}_{\text{$m_{st}$ fois}} 
= \underbrace{\sigb_t\sigb_s\sigb_t\cdots}_{\text{$m_{st}$ fois}}.
\end{cases}
\endequat
Si $w=s_1s_2\cdots s_l$ est une  d\'ecomposition r\'eduite de $w$, on pose 
$\sigb_w=\sigb_{s_1}\sigb_{s_2}\cdots \sigb_{s_l}$~: \indexnot{sz}{\sigb_w}  il est classique que 
$\sigb_w$ ne d\'epend pas du choix de la d\'ecomposition r\'eduite. D'autre part, 
\equat\label{eq:pi-coxeter}
\pib = \sigb_{\!\!\!w_0}^2.
\endequat

\section{Alg\`ebres de Hecke}

\medskip

\subsection{Cas g\'en\'erique} 
Pour $s \in S$, posons $\qb_{s,j}=\qb_{\Ker(s-\Id_V),j}$. \indexnot{qa}{\qb_{s,j}}
Il d\'ecoule de~\ref{eq:coxeter-tresse} que l'alg\`ebre de Hecke g\'en\'erique $\heckegenerique$ 
admet la pr\'esentation suivante, o\`u $\Tb_s$ \indexnot{T}{\Tb_s}  d\'esigne l'image de $\sigb_s$ dans $\heckegenerique$~:
\equat\label{eq:coxeter-hecke}
\heckegenerique \quad : \quad 
\begin{cases}
\text{G\'en\'erateurs~:} & (\Tb_s)_{s \in S}, \\
\text{Relations~:} & \forall~s \in S,~(\Tb_s-\qb_{s,0}^2)(\Tb_s+\qb_{s,1}^2)=0,\\
& \forall~s,t \in S,~\underbrace{\Tb_s\Tb_t\Tb_s\cdots}_{\text{$m_{st}$ fois}} 
= \underbrace{\Tb_t\Tb_s\Tb_t\cdots}_{\text{$m_{st}$ fois}}.
\end{cases}
\endequat
Si $w=s_1s_2\cdots s_l$ est une {\it d\'ecomposition r\'eduite} de $w$, posons 
$\Tb_w=\Tb_{s_1}\Tb_{s_2}\cdots\Tb_{s_l}$~: \indexnot{T}{\Tb_w}  
c'est l'image de $\sigb_w$ dans $\heckegenerique$ et donc 
$\Tb_w$ ne d\'epend pas du choix de la d\'ecomposition r\'eduite. De plus, 
\equat\label{eq:hecke-base}
\heckegenerique=
\mathop{\bigoplus}_{w \in W} \OC[\qb_\gen^{\pm 1}]~\! \Tb_w.
\endequat
Notons que $\Tb_w\Tb_{w'}=\Tb_{ww'}$ si $\ell(ww')=\ell(w)+\ell(w')$. Remarquons aussi que 
la base $(\Tb_w)_{w \in W}$ de $\heckegenerique$ d\'epend fortement du choix de $S$, c'est-\`a-dire 
de $v_\RM$. 
% ...

\subsection{Cas cyclotomique}\label{sub:coxeter-cyclotomique}
Fixons dans cette section $k=(k_{\O,j})_{\O \in \AC/W, j \in \{0,1\}} \in \CCB_\RM$~: 
la remarque~\ref{rem:translation} montre que supposer $k_{\O,0}+k_{\O,1}=0$ 
ne restreint pas la classe d'alg\`ebres \`a laquelle on s'int\'eresse.  
Pour $H \in \AC$, on posera $c_{s_H} = k_{H,0} - k_{H,1}=2k_{H,0}=-2k_{H,1}$~: \indexnot{ca}{c_{s_H}}
cette notation sera justifi\'ee dans le chapitre~\ref{chapter:cherednik-1}. 
Ainsi, l'alg\`ebre de Hecke cyclotomique $\heckecyclotomique(k)$ est la $\OC[\qb^\RM]$-alg\`ebre 
admettant la pr\'esentation suivante~:
\equat\label{eq:coxeter-hecke-cyclotomique}
\heckecyclotomique(k) \quad : \quad 
\begin{cases}
\text{G\'en\'erateurs~:} & (T_s)_{s \in S}, \\
\text{Relations~:} & \forall~s \in S,~(T_s-\qb^{c_s})(T_s+\qb^{-c_s})=0,\\
& \forall~s,t \in S,~\underbrace{T_s T_t T_s\cdots}_{\text{$m_{st}$ fois}} 
= \underbrace{T_t T_s T_t\cdots}_{\text{$m_{st}$ fois}}.
\end{cases}
\endequat

\section{Cellules de Kazhdan-Lusztig}\label{section:cellules-kl}

\medskip

\boitegrise{{\bf Hypoth\`ese.} 
{\it Dor\'enavant, et ce jusqu'\`a la fin de ce chapitre, nous 
fixons une famille $k=(k_{\O,j})_{\O \in \AC/W, j \in \{0,1\}} \in \CCB_\RM$. 
Nous noterons $c : \Ref(W) \to \RM$, \indexnot{ca}{c}  $s \mapsto c_s$.}}{0.75\textwidth}

\bigskip

La donn\'ee de la fonction $c : \Ref(W) \to \RM$ constante sur les classes de conjugaison 
est \'equivalente \`a la donn\'ee de $k$. 

\bigskip

\subsection{Base de Kazhdan-Lusztig} 
L'involution $a \mapsto \aba$ de $\OC[\qb^\RM]$ s'\'etend en une involution $\OC[\qb^\RM]$-semilin\'eaire 
de l'alg\`ebre $\heckecyclotomique(k)$ en posant
$$\overline{T}_w = T_{w^{-1}}^{-1}.\indexnot{T}{\Tov_w}$$
Si $\XM$ est une partie de $\RM$, nous noterons $\OC[\qb^\XM]=\mathop{\bigoplus}_{r \in \XM} \OC~\qb^r$. 
\indexnot{O}{\OC[\qb^\XM]}  
Nous poserons 
$$\heckecyclotomique(k)_{> 0}=\mathop{\bigoplus}_{w \in W} \OC[\qb^{\RM_{>0}}]~T_w.\indexnot{H}{\heckecyclotomique(k)_{>0}}$$
Le th\'eor\`eme suivant est d\'emontr\'e dans~\cite{KL}.

\bigskip

\begin{theo}[Kazhdan-Lusztig]\label{theo:base-kl}
Pour $w \in W$, il existe un unique \'el\'ement $C_w \in \heckecyclotomique(k)$ \indexnot{C}{C_w}  tel que 
$$
\begin{cases}
\overline{C}_w  =  C_w, \\
C_w \equiv T_w \mod \heckecyclotomique(k)_{> 0}.
\end{cases}
$$
La famille $(C_w)_{w \in W}$ est une $\OC[\qb^\RM]$-base de $\heckecyclotomique(k)$. 
\end{theo}

\bigskip

Notons que $C_w$ d\'epend de $k$ (i.e., de $c$). Par exemple, si $s \in S$, alors 
$$C_s=
\begin{cases}
T_s - q^{c_s} & \text{si $c_s > 0$,}\\
T_s & \text{si $c_s=0$,}\\
T_s + q^{-c_s} & \text{si $c_s < 0$.}
\end{cases}
$$
De m\^eme, tout comme $T_w$, $C_w$ d\'epend du choix de $S$. La base $(C_w)_{w \in W}$ sera appel\'ee 
la base de Kazhdan-Lusztig de $\heckecyclotomique(k)$~: si n\'ecessaire, $C_w$ sera not\'e 
$C_w^{(k)}$. \indexnot{C}{C_w^{(k)}}  

\bigskip

\subsection{Cellules de Kazhdan-Lusztig} 
Pour $x$, $y \in W$, nous \'ecrirons $x \rell y$ \indexnot{ZZZ}{\rell,~\relr}  s'il existe $h \in \heckecyclotomique(k)$ 
tel que $C_x$ appara\^{\i}t avec un coefficient non nul dans la d\'ecomposition de 
de $hC_y$ dans la base de Kazhdan-Lusztig. Nous noterons $\prel$ \indexnot{ZZZ}{\prel,~\prer,~\prelr}  
la cl\^oture transitive de cette relation~; c'est un pr\'e-ordre et nous noterons $\siml$ 
\indexnot{ZZZ}{\siml,~\simr,~\simlr}  la 
relation d'\'equivalence associ\'ee. 

On d\'efinit de m\^eme $\relr$ en multipliant \`a droite par $h$ ainsi que $\prer$ et $\simr$.  
Nous noterons $\prelr$ la relation r\'eflexive et sym\'etrique engendr\'ee par 
$\prel$ et $\prer$, et $\simlr$ d\'esignera la relation d'\'equivalence associ\'ee. 

\bigskip

\begin{defi}\label{defi:cellules}
On appellera $c$-{\bfit cellule de Kazhdan-Lusztig \`a gauche} 
(respectivement {\bfit \`a droite}, respectivement {\bfit bilat\`ere}) 
de $W$ une classe d'\'equivalence pour la relation $\siml$ (respectivement $\simr$, 
respectivement $\simlr$). Si $? \in \{L,R,LR\}$, 
on notera $\klcellules_?^c(W)$ l'ensemble correspondant de $c$-cellules de Kazhdan-Lusztig de $W$. 
\end{defi}

\bigskip

Si $? \in \{L,R,LR\}$ et si $\G$ est une classe d'\'equivalence pour la relation 
$\sim_?$ (c'est-\`a-dire une $c$-cellule de Kazhdan-Lusztig du bon type), posons
$$\heckecyclotomique(k)_{\leqslant_?^{\kl,c} \G} = 
\mathop{\bigoplus}_{w \leqslant_?^{\kl,c} \G} \OC[\qb^\RM]~C_w
\qquad\text{et}\qquad
\heckecyclotomique(k)_{<_?^{\kl,c} \G} = 
\mathop{\bigoplus}_{w <_?^{\kl,c} \G} \OC[\qb^\RM]~C_w,
\indexnot{H}{\heckecyclotomique(k)_{\leqslant_?^{\kl,c} \G},~\heckecyclotomique(k)_{<_?^{\kl,c} \G}}
$$
ainsi que 
$$\MC_\G^? = \heckecyclotomique(k)_{\leqslant_?^{\kl,c} \G}/\heckecyclotomique(k)_{<_?^{\kl,c} \G}.$$
Par construction, $\heckecyclotomique(k)_{\leqslant_?^{\kl,c} \G}$ et $\heckecyclotomique(k)_{<_?^{\kl,c} \G}$ 
sont des id\'eaux (\`a gauche si $?=L$, \`a droite si $?=R$ ou bilat\`eres si $?=LR$) 
et $\MC_\G^?$ est un $\heckecyclotomique(k)$-module \`a gauche si $?=L$, 
\`a droite si $?=R$ ou un $(\heckecyclotomique(k),\heckecyclotomique(k))$-bimodule 
si $?=LR$. Notons que
\equat\label{eq:liberte-cellulaire}
\text{\it $\MC_\G^?$ est un $\OC[\qb^\RM]$-module libre de base l'image de $(C_w)_{w \in \G}$.}
\endequat

\bigskip

\begin{defi}\label{defi:cellulaires-familles}
Si $C$ est une $c$-cellule de Kazhdan-Lusztig \`a gauche de $W$, on notera 
$\isomorphisme{C}_c^\kl$ \indexnot{C}{\isomorphisme{C}_c^\kl}  la classe 
de $\kb \otimes_{\OC[\qb^\RM]} \MC_C^L$ dans le groupe de Grothendieck $\groth(\kb W)=\ZM\Irr(W)$ 
(ici, le produit tensoriel $\kb \otimes_{\OC[\qb^\RM]} -$ est vu \`a travers la sp\'ecialisation 
$\qb^r \mapsto 1$). On appellera {\bfit $\kl$-caract\`ere $c$-cellulaire} de $W$ 
tout caract\`ere de la forme $\isomorphisme{C}_c^\kl$, o\`u $C$ est une $c$-cellule 
de Kazhdan-Lusztig \`a gauche. 

Si $\G$ est une $c$-cellule de Kazhdan-Lusztig bilat\`ere de $W$, 
on notera $\Irr_\G^\kl(W)$ \indexnot{I}{\Irr_\G^\kl(W)}  l'ensemble des caract\`eres 
irr\'eductibles de $W$ apparaissant dans $\kb \otimes_{\OC[\qb^\RM]} \MC_\G^{LR}$, 
vu comme $\kb W$-module \`a gauche. On appellera {\bfit $c$-famille de Kazhdan-Lusztig} 
toute partie de $\Irr(W)$ de la forme $\Irr_\G^\kl(W)$ o\`u $\G$ est une 
$c$-cellule de Kazhdan-Lusztig bilat\`ere. Nous dirons que $\Irr_\G^\kl(W)$ est la $c$-famille 
de Kazhdan-Lusztig {\bfit asoci\'ee} \`a $\G$, ou que $\G$ est la $c$-cellule de Kazhdan-Lusztig bilat\`ere 
{\bfit recouvrant} $\Irr_\G^\kl(W)$. 
\end{defi}

\bigskip

Puisque $\kb W$ est semi-simple et que 
$\kb \otimes_{\OC[\qb^\RM]} \MC_\G^{LR}$ est un quotient d'id\'eaux bilat\`eres de $\kb W$, 
les $k$-familles de Kazhdan-Lusztig forment une partition de $\Irr(W)$ 
\equat\label{eq:partition-familles}
\Irr(W) = \mathop{\coprod}_{\G \in \klcellules_{LR}^c(W)} \Irr_\G^\kl(W)
\endequat
et, puisque $\kb W$ est d\'eploy\'ee,
\equat\label{eq:cardinal-cellule-kl}
|\G|=\sum_{\chi \in \Irr_\G^\kl(W)} \chi(1)^2.
\endequat
D'autre part, si $C$ est une $c$-cellule de Kazhdan-Lusztig {\it \`a gauche} de $W$, posons 
$$\isomorphisme{C}_c^\kl = \sum_{\chi \in \Irr(W)} \mult_{C,\chi}^\kl~\chi,\indexnot{ma}{\mult_{C,\chi}^\kl}$$
o\`u $\mult_{C,\chi}^\kl \in \NM$. 
Alors~:

\bigskip

\begin{lem}\label{lem:mult-kl}
Avec les notations pr\'ec\'edentes, on a~:
\begin{itemize}
\itemth{a} Si $C\in \klcellules_L^c(W)$, alors 
$\DS{\sum_{\chi \in \Irr(W)} \mult_{C,\chi}^\kl~\chi(1)=|C|}$.

\smallskip

\itemth{b} Si $\chi \in \Irr(W)$, alors 
$\DS{\sum_{C \in \klcellules_L^c(W)} \mult_{C,\chi}^\kl = \chi(1)}$. 
\end{itemize}
\end{lem}

\begin{proof}
L'\'egalit\'e (a) exprime simplement que la dimension de $\isomorphisme{C}_c^\kl$ 
est \'egale \`a $|C|$ d'apr\`es~\ref{eq:liberte-cellulaire}. L'\'egalit\'e (b) traduit 
le fait que, puisque $W$ est la r\'eunion disjointe des $c$-cellules de Kazhdan-Lusztig \`a gauche, on a 
$\isomorphisme{\kb W}_{\kb W} = \sum_{C \in \klcellules_L^c(W)} \isomorphisme{C}_c^\kl$. 
\end{proof}

\subsection{Autres propri\'et\'es des cellules}\label{sub:proprietes-cellules} 
L'alg\`ebre $\heckecyclotomique(k)$ est muni d'un anti-automorphisme $\OC[\qb^\RM]$-lin\'eaire 
qui envoie $T_w$ sur $T_{w^{-1}}$~: il sera not\'e $h \mapsto h^*$. \indexnot{ZZZ}{*}  Il est imm\'ediat que 
\equat\label{eq:cw-etoile}
C_w^*=C_{w^{-1}},
\endequat
ce qui implique que, si $x$ et $y$ sont deux \'el\'ements de $W$, alors 
\equat\label{eq:rell-relr}
\text{\it $x \prel y$ si et seulement si $x^{-1} \prer y^{-1}$}
\endequat
et donc 
\equat\label{eq:siml-simr}
\text{\it $x \siml y$ si et seulement si $x^{-1} \simr y^{-1}$.}
\endequat
En d'autres termes, l'application $\klcellules_L^c(W) \to \klcellules_R^c(W)$, 
$\G \mapsto \G^{-1}$, est bien d\'efinie et bijective.

Moins \'evidente est la propri\'et\'e suivante~\cite[corollaire~11.7]{lusztig}
\equat\label{eq:rel-w0}
\text{\it $x \leqslant_?^c y$ si et seulement si $w_0 y \leqslant_?^c w_0 x$ 
si et seulement si $y w_0 \leqslant_?^c x w_0$.}
\endequat
Il en d\'ecoule que 
\equat\label{eq:sim-w0}
\text{\it $x \sim_?^c y$ si et seulement si $w_0x \sim_?^c w_0y$ si et seulement si $xw_0 \sim_?^c yw_0$.}
\endequat
Par ailleurs, si $C \in \klcellules_L^c(W)$, alors~\cite[proposition~21.5]{lusztig}
\equat\label{eq:caractere-cellulaire-w0}
\isomorphisme{w_0 C}_c^\kl = \isomorphisme{C w_0}_c^\kl =\isomorphisme{C}_c^\kl \cdot \e.
\endequat
De m\^eme, si $\G \in \klcellules_{LR}^c(W)$, alors~\cite[proposition~21.5]{lusztig}
\equat\label{eq:familles-cellulaire-w0}
\Irr_{w_0 \G}^\kl(W)=\Irr_{\G w_0}^\kl(W) = \Irr_\G^\kl(W) \cdot \e.
\endequat
Cela montre en particulier que 
\equat\label{eq:w0gw0}
w_0\G w_0 = \G.
\endequat
La tensorisation par $\e$ induit donc une permutation des $c$-familles de 
Kazhdan-Lusztig et des $c$-caract\`eres cellulaires.

Si $\g : W \to \kb^\times$ est un caract\`ere lin\'eaire (notons 
que $\g$ est \`a valeurs dans $\{1,-1\}$), on note $\g\cdot c : \Ref(W) \to \RM$, \indexnot{gz}{\g \cdot c}  
$s \mapsto \g(s)c_s$. Le lemme suivant est d\'emontr\'e 
dans~\cite[corollaire~2.5~et~2.6]{bonnafe continu}~:

\bigskip

\begin{lem}\label{lem:gamma-c-kl}
Soit $\g \in W^\wedge$ et soit $? \in \{L,R,LR\}$. Alors~:
\begin{itemize}
\itemth{a} Les relations $\leqslant_?^c$ et $\leqslant_?^{\g \cdot c}$ co\"{\i}ncident.

\itemth{b} Les relations $\sim_?^{\kl,c}$ et $\sim_?^{\kl,\g \cdot c}$ co\"{\i}ncident.

\itemth{c} Si $C \in \klcellules_L^c(W)=\klcellules_L^{\g \cdot c}(W)$, alors 
$\isomorphisme{C}_c^\kl = \g \cdot \isomorphisme{C}_{\g \cdot c}^\kl$. 
\end{itemize}
\end{lem}

\bigskip

Le r\'esultat suivant est facile~\cite[lemme~8.6]{lusztig}~:

\bigskip

\begin{lem}\label{lem:1-cellule}
Supposons $c_s \neq 0$ pour tout $s \in \Ref(W)$. Alors~:
\begin{itemize}
 \itemth{a} $\{1\}$ et $\{w_0\}$ sont des $c$-cellules de Kazhdan-Lusztig 
(\`a gauche, \`a droite et bilat\`eres).

\itemth{b} Notons $\g : W \to \kb^\times$ l'unique caract\`ere lin\'eaire 
tel que $\g(s)=1$ si $c_s > 0$ et $\g(s)=-1$ si $c_s < 0$. Alors 
$\isomorphisme{1}_{\kb W} = \g$ et $\isomorphisme{w_0}_{\kb W}=\g\e$.
\end{itemize}
\end{lem}

\bigskip

\begin{rema}\label{rem:lusztig-positif}
En fait,~\cite[lemme~8.6]{lusztig} est d\'emontr\'e lorsque $c_s > 0$ pour tout $s$. 
Pour passer \`a l'\'enonc\'e g\'en\'eral du lemme~\ref{lem:1-cellule}, il suffit 
alors d'appliquer le lemme~\ref{lem:gamma-c-kl}.\finl
\end{rema}

\bigskip

Le dernier lemme exprime une forme de compatibilit\'e entre les notions de cellules 
de Kazhdan-Lusztig \`a gauche (ou \`a droite) et les sous-groupes 
paraboliques standard~: il faut noter qu'il n'y a pas de r\'esultat analogue 
pour les cellules de Kazhdan-Lusztig bilat\`eres. Nous aurons besoin de la notation 
suivante~: pour $I \subset S$, notons $W_I$ \indexnot{W}{W_I}  le sous-groupe de $W$ engendr\'e 
par $I$ (c'est un {\it sous-groupe parabolique standard} de $W$) et notons 
$W^I$ \indexnot{W}{W^I}  l'ensemble des \'el\'ements $x \in W$ qui sont de longueur minimale dans $xW_I$ 
(rappelons que l'application $W^I \to W/W_I$, $x \mapsto x W_I$ est bijective). Notons 
$c_I$ \indexnot{ca}{c_I}  la restriction de $c$ \`a $\Ref(W_I)=\Ref(W) \cap W_I$.

\bigskip

\begin{lem}\label{lem:kl-cells-parabolique}
Soit $I \subset S$. Alors~:
\begin{itemize}
\itemth{a} Si $\G$ est une $c_I$-cellule de Kazhdan-Lusztig \`a gauche de $W_I$, alors 
$W^I\cdot \G$ est une union de $c$-cellules de Kazhdan-Lusztig \`a gauche de $W$.

\itemth{b} Si $w$, $w' \in W_I$ et $x \in W^I$ sont tels que $w \leqslant_L^{c_I} w'$ 
(respectivement $w \sim_L^{\kl,c_I} w'$), alors $wx^{-1} \leqslant_L^{c_I} w'x^{-1}$ 
(respectivement $wx^{-1} \sim_L^{\kl,c_I} w'x^{-1}$).
\end{itemize}
\end{lem}

\bigskip

\begin{proof}
(a) est d\^u \`a Geck~\cite[th\'eor\`eme~1]{geck induction} 
tandis que (b) est d\^u \`a Lusztig~\cite[proposition~9.11]{lusztig}. 
\end{proof}

\bigskip

Pour finir, rappelons que Lusztig a propos\'e de nombreuses conjectures concernant les cellules 
et sa fonction $\ab$~\cite[\S{14.2},~conjectures~P1-P15]{lusztig}. Nous ne les rappellerons pas 
toutes ici mais mentionnons simplement les deux suivantes~:

\bigskip

\begin{quotation}\refstepcounter{theo}
{\bf Conjectures \thetheo~(Lusztig).} 
{\it {\rm (a)} Toute $c$-cellule de Kazhdan-Lusztig \`a gauche contient une involution. 
Ainsi, si $w \in W$, alors $w \simlr w^{-1}$.

{\rm (b)} Si $x \simlr y$ et $x \prel y$, alors $x \siml y$.

{\rm (c)} $\simlr$ est la relation d'\'equivalence engendr\'ee par $\siml$ et $\simr$.}
\end{quotation}

\bigskip

\begin{exemple}[Param\`etres nuls]\label{exemple:c=0}
Si $c=0$ (i.e. si $c_s=0$ pour tout $s$), alors $C_w=T_w$, $\heckecyclotomique(0)=\OC[\qb^\RM][W]$ 
et il n'y a qu'une seule $0$-cellule de Kazhdan-Lusztig (\`a gauche, \`a droite, ou bilat\`ere), 
c'est $W$. On a alors $\Irr_W^{\kl,0}(W)=\Irr(W)$ et 
$\isomorphisme{W}_{\kb W}^{\kl,0} = \sum_{\chi \in \Irr(W)} \chi(1)\chi$.\finl
\end{exemple}

\bigskip

\cbend

\part{Alg\`ebre de Cherednik}\label{part:cherednik}

\chapter{Alg\`ebre de Cherednik g\'en\'erique}\label{chapter:cherednik-1}

Notons $\CCB$ \indexnot{C}{\CCB}  le $\kb$-espace vectoriel des fonctions $c : \Ref(W) \to \kb$, $s \mapsto c_s$ 
\indexnot{ca}{c_s}  
qui sont constantes sur les classes de conjugaison. Nous l'appellerons {\it espace des 
param\`etres} et nous l'identifions avec l'espace des fonctions $\refw \to \kb$. 

Si $s \in \Ref(W)$ (ou $s \in \refw$), nous noterons $C_s$ \indexnot{C}{C_s}  la forme lin\'eaire 
sur $\CCB$ correspondant \`a l'\'evaluation en $s$. L'alg\`ebre $\kb[\CCB]$ des 
fonctions polynomiales sur $\CCB$ est alors l'alg\`ebre de polyn\^omes en 
les ind\'etermin\'ees $(C_s)_{s \in \refw}$~:
$$\kb[\CCB]=\kb[(C_s)_{s \in \refw}].$$
Notons $\CCBt$ \indexnot{C}{\CCBt}  le $\kb$-espace vectoriel $\kb \times \CCB$ et 
notons $T : \CCBt \to \kb$, \indexnot{T}{T}  $(t,c) \mapsto t$. Ainsi, $T \in \CCBt^*$ et 
$$\kb[\CCBt]=\kb[T,(C_s)_{s \in \refw}].$$

\section{D\'efinition}

\medskip

Nous appellerons {\it alg\`ebre de Cherednik g\'en\'erique rationnelle} (ou plus 
simplement {\it alg\`ebre de Cherednik g\'en\'erique}) la $\kb[\CCBt]$-alg\`ebre $\Hbt$ quotient 
de $\kb[\CCBt] \otimes \bigl(\Trm_\kb(V \oplus V^*) \rtimes W\bigr)$ 
par l'id\'eal engendr\'e par les relations suivantes (ici, $\Trm_\kb(V \oplus V^*)$ 
d\'esigne l'alg\`ebre tensorielle de $V \oplus V^*$)~:
\equat\label{relations-1}\begin{cases}
[x,x']=[y,y']=0, \\
\\
[y,x] = T \langle y,x\rangle + \DS{\sum_{s \in \Ref(W)} (\e(s)-1)\hskip1mm C_s 
\hskip1mm\frac{\langle y,\a_s \rangle \cdot \langle \a_s^\ve,x\rangle}{\langle \a_s^\ve,\a_s\rangle}
\hskip1mm s,} 
\end{cases}\endequat
pour tous $x$, $x' \in V^*$ et $y$, $y' \in V$. 

\bigskip

\begin{rema}
Gr\^ace \`a~(\ref{action s V}), la deuxi\`eme relation peut aussi s'\'ecrire 
\equat\label{eq:relations-1-sans-alpha}
[y,x] = T \langle y,x\rangle + \sum_{s \in \Ref(W)} C_s 
\langle s(y)-y,x\rangle 
\hskip1mm s,
\endequat
ce qui \'evite d'utiliser les $\a_s$ et les $\a_s^\ve$.\finl
\end{rema}

\bigskip

\subsection{D\'ecomposition PBW}\label{subsection:PBW-1} 
Compte tenu des relations~\ref{relations-1}, 
les faits suivants sont imm\'ediats~:
\begin{itemize}
\item[$\bullet$] Il existe un unique morphisme de $\kb$-alg\`ebres 
$\kb[V] \to \Hbt$ qui envoie $y \in V^*$ sur la classe de 
$y \in \Trm_\kb(V \oplus V^*) \rtimes W$ dans $\Hbt$.

\item[$\bullet$] Il existe un unique morphisme de $\kb$-alg\`ebres 
$\kb[V^*] \to \Hbt$ qui envoie $x \in V$ sur la classe de 
$x \in \Trm_\kb(V \oplus V^*) \rtimes W$ dans $\Hbt$.

\item[$\bullet$] Il existe un unique morphisme de $\kb$-alg\`ebres 
$\kb W \to \Hbt$ qui envoie $w \in W$ sur la classe de 
$w \in \Trm_\kb(V \oplus V^*) \rtimes W$ dans $\Hbt$.

\item[$\bullet$] L'application $\kb$-lin\'eaire $\kb[\CCBt] \otimes \kb[V] \otimes \kb W 
\otimes \kb[V^*] \longto \Hbt$ induite par les trois morphismes pr\'ec\'edents 
et par la multiplication est surjective. Elle est de plus $\kb[\CCBt]$-lin\'eaire.
\end{itemize}
Cette derni\`ere propri\'et\'e est en fait grandement am\'elior\'ee par 
le r\'esultat fondamental suivant~:

\bigskip

\begin{theo}[Etingof-Ginzburg]\label{PBW-1}
La multiplication  $\kb[\CCBt] \otimes \kb[V] \otimes \kb W 
\otimes \kb[V^*] \longto \Hbt$ est un isomorphisme de 
$\kb[\CCBt]$-modules.
\end{theo}

\begin{proof}
Dans~\cite[th\'eor\`eme 1.3]{EG}, ce th\'eor\`eme n'est pas \'enonc\'e dans cette 
version ``g\'en\'erique''. N\'eanmoins, il s'en d\'eduit ais\'ement de la fa\c{c}on suivante. 
Soit $(\b_i)_{i \in I}$ une $\kb$-base de 
$\kb[V] \otimes \kb W \otimes \kb[V^*]$ et soit $f$ un \'el\'ement du noyau de 
l'application lin\'eaire de l'\'enonc\'e. Il existe alors une unique 
famille \`a support fini $(f_i)_{i \in I}$ d'\'el\'ements de $\kb[\CCBt]$ tels que 
$f=\sum_{i \in I} f_i \otimes \b_i$. Mais, d'apr\`es~\cite[th\'eor\`eme 1.3]{EG}, 
la sp\'ecialisation en tout $(t,c) \in \CCBt$ du th\'eor\`eme est vraie. Par 
cons\'equent, $f_i(t,c)=0$ pour tout $i \in I$ et tout $(t,c) \in \CCBt$, ce qui 
montre que $f_i=0$ (car $\kb$ est infini) et donc que $f=0$.
\end{proof}

\bigskip

\subsection{Hyperplans, changement de variables, conventions}\label{section:hyperplans-variables}
Rappelons que $(C_s)_{s \in \refw}$ est 
une $\kb$-base de $\CCB^*$. Nous allons construire ici une nouvelle base de 
$\CCB^*$. Pour cela, posons $C_1=0 \in \CCB^*$

Nous notons $(K_{\O,j})_{(\O,j) \in \Omeb_W^\circ}$ \indexnot{K}{K_{\O,j},~K_{H,j}}  l'unique 
famille d'\'el\'ements de $\CCB^*$ telle que, pour tout $H \in \AC$ et pour tout 
$i \in \{0,1,\dots,e_H-1\}$, on ait 
$$C_{s_H^i}= \sum_{j=0}^{e_H-1} \z_{e_H}^{i(j-1)} K_{H,j}.$$
Ici, $K_{H,j}$ d\'esigne la variable $K_{\O,j}$, o\`u $\O$ est la $W$-orbite de $H$. 
L'existence et l'unicit\'e de la famille $(K_{\O,j})_{(\O,j) \in \Omeb_W}$ 
d\'ecoule de l'invertibilit\'e du d\'eterminant de Vandermonde, 
et on obtient \'egalement, en se restreignant aux \'el\'ements de $\Omeb_W \subset \Omeb_W^\circ$, 
\equat\label{nouvelle base}
\text{\it $(K_{\O,j})_{(\O,j) \in \Omeb_W}$ est une $\kb$-base de $\CCB^*$.} 
\endequat
Notons que $K_{\O,0}$ est d\'etermin\'e par l'\'equation 
$K_{\O,0} + K_{\O,1} + \cdots + K_{\O,e_\O-1}=C_1=0$. 
Pour finir, notons l'\'egalit\'e suivante~:
\equat\label{escs}
\sum_{s \in \Ref(W)} \e(s)~C_s = \sum_{H \in \AC} e_H K_{H,0} = 
-\sum_{H \in \AC} \sum_{i=1}^{e_H-1} e_H K_{H,i}.
\endequat

%
%
%

\bigskip

Si $H \in \AC$, on note $\a_H \in V^*$ une forme lin\'eaire telle que $H=\ker(\a_H)$ et 
$\a_H^\ve \in V$ un vecteur tel que $V=H \oplus \kb \a_H^\ve$ et tel que la droite 
$\kb\a_H^\ve$ soit stable par $W_H$. Avec les conventions ci-dessus, 
la deuxi\`eme relation de~(\ref{relations-1}) s'\'ecrit 
\equat\label{eq:rel-1-H}
[y,x]=T\langle y,x\rangle +\sum_{H\in\mathcal{A}} \sum_{i=0}^{e_H-1}
e_H(K_{H,i}-K_{H,i+1}) \frac{\langle y,\a_H \rangle \cdot \langle \a_H^\ve,x\rangle}{\langle \a_H^\ve,\a_H\rangle} 
\e_{H,i}
\endequat
pour $x\in V^*$ et $y\in V$, o\`u $\e_{H,i}=e_H^{-1}\sum_{w\in W_H}\e(w)^i w$. 

\bigskip

\noindent{\sc Commentaire - } 
La convention utilis\'ee ici dans la d\'efinition de l'alg\`ebre de Cherednik 
n'est pas celle de~\cite[\S{3.1}]{ggor}~: en effet, nous avons rajout\'e un coefficient $\e(s)-1$ 
devant le terme $C_s$. En revanche, c'est la m\^eme que celle de~\cite[\S{1.15}]{EG}, 
avec $c_s=c_{\a_s}$ (lorsque $W$ est un groupe de Coxeter). D'autre part, les $k_{H,i}$ de~\cite{ggor} sont reli\'es 
aux $K_{H,i}$ ci-dessus par $k_{H,i}=K_{H,0}-K_{H,i}$.\finl

%

\bigskip

\subsection{Sp\'ecialisation}\label{subsection:specialisation-1}
Si $(t,c) \in \CCBt$, nous noterons $\CGt_{t,c}$ \indexnot{C}{\CGt_{t,c}}  l'id\'eal maximal de $\kb[\CCBt]$ d\'efini 
par $\CGt_{t,c}=\{f \in \kb[\CCBt]~|~f(t,c)=0\}$~: c'est l'id\'eal engendr\'e par $T-t$ et 
$(C_s-c_s)_{s \in \refw}$. Posons alors 
$$\Hbt_{t,c} = \kb[\CCBt]/\CGt_{t,c} \otimes_{\kb[\CCBt]} \Hbt = \Hbt/\CGt_{t,c} \Hbt.\indexnot{H}{\Hbt_{t,c}}$$
La $\kb$-alg\`ebre 
$\Hbt_{t,c}$ est le quotient de 
$\Trm_\kb(V \oplus V^*) \rtimes W$ par l'id\'eal engendr\'e par les relations suivantes~: 
\equat\label{relations specialisees}\begin{cases}
[x,x']=[y,y']=0, \\
\\
[y,x] = t \langle y,x \rangle + \DS{\sum_{s \in \Ref(W)} (\e(s)-1)\hskip1mm c_s 
\hskip1mm\frac{\langle y,\a_s \rangle \cdot \langle \a_s^\ve,x\rangle}{\langle \a_s^\ve,\a_s\rangle}
\hskip1mm s,} 
\end{cases}\endequat
pour $x$, $x' \in V^*$ et $y$, $y' \in V$. 

\bigskip

\begin{exemple}\label{exemple zero-1} 
Il est imm\'ediat que $\Hbt_{0,0}=\kb[V \oplus V^*] \rtimes W$.\finl
\end{exemple}

\bigskip

\section{Graduations}\label{section:graduation-1}

\medskip

L'alg\`ebre $\Hbt$ poss\`ede une $\NM\times\NM$-graduation naturelle, 
gr\^ace \`a laquelle on peut associer, \`a chaque morphisme de mono\"{\i}des 
$\NM \times \NM \to \ZM$ (ou $\NM \times \NM \to \NM$), une $\ZM$-graduation 
(ou une $\NM$-graduation).

\medskip

On bi-gradue (sur $\NM \times \NM$) 
l'alg\`ebre tensorielle \'etendue $\kb[\CCBt] \otimes \bigl(\Trm_\kb(V \oplus V^*) \rtimes W\bigr)$ 
en d\'ecr\'etant que les \'el\'ements de $V$ ont pour bi-degr\'e $(1,0)$, 
les \'el\'ements de $V^*$ ont pour bi-degr\'e $(0,1)$, les \'el\'ements de $\CCBt^*$ 
ont pour degr\'e $(1,1)$ et ceux de $W$ 
ont pour degr\'e $(0,0)$, alors les relations~\ref{relations-1} deviennent homog\`enes. 
Ainsi, $\Hbt$ h\'erite d'une $(\NM \times \NM)$-graduation dont nous noterons 
$\Hbt^{\NM \times \NM}[i,j]$ \indexnot{H}{\Hbt^{\NM \times \NM}[i,j],~\Hbt^\ph[i],~\Hbt^\NM[i],\Hbt^\ZM[i]}  
la composante homog\`ene de bi-degr\'e $(i,j)$. Alors
$$\Hbt=\mathop{\bigoplus}_{(i,j) \in \NM \times \NM} 
\Hbt^{\NM \times \NM}[i,j]\quad\text{et}\quad \Hbt^{\NM \times \NM}[0,0]=\kb W.$$
D'autre part, chaque composante homog\`ene est de dimension finie sur $\kb$. 

\medskip

Si $\ph : \NM \times \NM \to \ZM$ est un morphisme de mono\"{\i}des, 
alors $\Hbt$ h\'erite d'une $\ZM$-graduation dont nous noterons $\Hbt^\ph[i]$ 
la composante homog\`ene de degr\'e $i$~:
$$\Hbt^\ph[i]=\mathop{\bigoplus}_{\ph(a,b) = i} \Hbt^{\NM \times \NM}[a,b].\indexnot{H}{\Hbt_\ph[i]}$$
Dans cette graduation, les \'el\'ements de $V$ ont pour degr\'e $\ph(1,0)$, 
les \'el\'ements de $V^*$ ont pour degr\'e $\ph(0,1)$, les \'el\'ements de $\CCBt^*$ 
ont pour degr\'e $\ph(1,1)$ et ceux de $W$ ont pour degr\'e $0$.

\bigskip

\begin{exemple}[$\ZM$-graduation]\label{Z graduation-1}
Le morphisme de mono\"{\i}des $\NM \times \NM \to \ZM$, $(i,j) \mapsto j-i$ induit 
une $\ZM$-graduation sur $\Hbt$ pour laquelle les \'el\'ements de $V$ ont pour degr\'e $-1$, 
les \'el\'ements de $V^*$ ont pour degr\'e $1$ et les \'el\'ements de $\CCBt^*$ et $W$ 
ont pour degr\'e $0$. Nous noterons $\Hbt^\ZM[i]$ la composante 
homog\`ene de degr\'e $i$. Alors
$$\Hbt=\mathop{\bigoplus}_{i \in \ZM} \Hbt^\ZM[i].$$
Par sp\'ecialisation en tout $(t,c) \in \CCBt$, l'alg\`ebre $\Hbt_{t,c}$ h\'erite d'une $\ZM$-graduation 
dont la composante homog\`ene de degr\'e $i$ sera not\'ee $\Hbt_{t,c}^\ZM[i]$.\finl
\end{exemple}

\bigskip

\begin{exemple}[$\NM$-graduation]\label{N graduation-1}
Le morphisme de mono\"{\i}des $\NM \times \NM \to \NM$, $(i,j) \mapsto i+j$ induit 
une $\NM$-graduation sur $\Hbt$ pour laquelle les \'el\'ements de $V$ ou de $V^*$ 
ont pour degr\'e $1$, les \'el\'ements de $\CCBt^*$ ont pour degr\'e $2$ et les 
\'el\'ements de $W$ ont pour degr\'e $0$. 
Nous noterons $\Hbt^\NM[i]$ la composante homog\`ene de degr\'e $i$. Alors
$$\Hbt=\mathop{\bigoplus}_{i \in \NM} \Hbt^\NM[i]\quad\text{et}\quad \Hbt^\NM[0]=\kb W.$$
Notons que $\dim_\kb \Hbt^\NM[i] < \infty$ pour tout $i$. 
Cette graduation ne s'h\'erite pas par sp\'ecialisation en tout $(t,c) \in \CCBt$~: 
elle s'h\'erite seulement par la sp\'ecialisation en $(0,0)$~: on retrouve 
la $\NM$-graduation usuelle sur $\Hbt_{0,0} = \kb[V \times V^*] \rtimes W$ 
(voir l'exemple~\ref{exemple zero-1}).\finl
\end{exemple}

\bigskip

\section{Filtration}\label{section:filtration-1}

\medskip

Munissons l'alg\`ebre tensorielle \'etendue $\kb[\CCBt] \otimes 
\bigl(\Trm_\kb(V \oplus V^*) \rtimes W\bigr)$ 
de la graduation pour laquelle les \'el\'ements de $V$ et $V^*$ ont pour degr\'e $1$, 
tandis que les \'el\'ements de $\CCBt^*$ et $W$ ont pour degr\'e $0$. Si $n \in \NM$, 
notons $\FC_n \Hbt$ \indexnot{F}{\FC_n\Hbt}  l'image dans $\Hbt$ de l'espace form\'e des \'el\'ements de degr\'e 
$\le n$. Alors 
$$\FC_0 \Hbt = \kb[\CCBt] \otimes \kb W \subset \FC_1 \Hbt \subset \FC_2\Hbt \subset \cdots 
\subset \Hbt \qquad\text{et}\qquad \bigcup_{m \ge 0} \FC_m \Hbt = \Hbt.$$
On a 
$$\FC_l \Hbt \cdot \FC_m \Hbt \subset \FC_{l+m} \Hbt,$$
ce qui d\'efinit sur $\Hbt$ une filtration $\FC$. Nous noterons $\grad_\FC(\Hbt)$ 
l'alg\`ebre gradu\'ee associ\'ee~: le th\'eor\`eme~\ref{PBW-1} 
fournit un isomorphisme de $\kb[\CCBt]$-alg\`ebres 
\equat\label{gradue-1}
\kb[\CCBt] \otimes (\kb[V \times V^*] \rtimes W) \longiso \grad_\FC(\Hbt). 
\endequat

Par sp\'ecialisation en $(t,c) \in \CCBt$, la filtration $\FC$ de $\Hbt$ induit 
une filtration (que nous noterons encore $\FC$) de $\Hbt_{t,c}$, et 
\equat\label{gradue c-1}
\grad_\FC(\Hbt_{t,c})\simeq \kb[V \times V^*] \rtimes W.
\endequat

\bigskip

\section{\'El\'ement d'Euler}\label{section:eulertilde}

\medskip

Soit $(x_1,\dots,x_n)$ une $\kb$-base de $V^*$ et soit $(y_1,\dots,y_n)$ 
sa base duale. Posons 
$$\eulertilde=~-nT + \sum_{i=1}^n y_i x_i + \sum_{s \in \Ref(W)} C_s s ~= 
\sum_{i=1}^n x_i y_i + \sum_{s \in \Ref(W)} \e(s) C_s s \in \Hbt.\indexnot{ea}{\eulertilde}$$
Alors $\eulertilde$ est appel\'e l'{\it \'el\'ement d'Euler g\'en\'erique} de $\Hbt$. 
Il est facile de v\'erifier que $\eulertilde$ ne d\'epend pas du choix 
de la base $(x_1,\dots,x_n)$ de $V^*$. Notons que
\equat\label{eq:eulertilde-1-1}
\eulertilde \in \Hbt^{\NM \times \NM}[1,1]
\endequat
et rappelons le lemme classique suivant~\cite[\S 3.1(4)]{ggor}~:

\bigskip

\begin{lem}\label{lem:eulertilde}
Si $x \in V^*$, $y \in V$ et $w \in W$, alors 
$$[\eulertilde,x]=Tx,\qquad[\eulertilde,y]=-Ty\qquad\text{et}\qquad [\eulertilde,w]=0.$$
\end{lem}

\bigskip

Dans~\cite{ggor}, l'\'el\'ement d'Euler joue un r\^ole primordial dans l'\'etude de 
la cat\'egorie $\OC$ associ\'ee \`a $\Hbt_{1,c}$. Nous verrons dans ce m\'emoire  
le r\^ole que joue cet \'el\'ement dans la th\'eorie des cellules de 
Calogero-Moser.

\bigskip

\begin{prop}\label{prop:eulertilde}
Si $h \in \Hbt^\ZM[i]$, alors $[\eulertilde,h]=iTh$.
\end{prop}

\bigskip

\section{Alg\`ebre sph\'erique}\label{section:spherique-1}

\medskip

\boitegrise{{\bf Notation.} 
{\it Tout au long de ce m\'emoire, nous noterons $e$ \indexnot{ea}{e}  l'idempotent primitif central 
de $\kb W$ d\'efini par
$$e=\frac{1}{|W|}\sum_{w \in W} w.$$
La $\kb[\CCBt]$-alg\`ebre $e \Hbt e$ sera appel\'ee l'{\bfit alg\`ebre sph\'erique 
g\'en\'erique}.}}{0.75\textwidth}

\medskip

Par sp\'ecialisation en $(t,c)$, et puisque $e\Hbt_{t,c} e$ est un facteur direct du $\kb[\CCBt]$-module 
$\Hbt$, on obtient que
\equat\label{spherique-1}
e \Hbt_{t,c} e = (\kb[\CCBt]/\CGt_{t,c}) \otimes_{\kb[\CCBt]} e\Hbt e.
\endequat
Puisque $e$ est de degr\'e $0$, la filtration $\FC$ induit une filtration (que nous 
noterons encore $\FC$) de l'alg\`ebre sph\'erique g\'en\'erique
$$e(\FC_0 \Hbt) e = \kb[\CCBt] \subset e(\FC_1 \Hbt) e \subset e(\FC_1 \Hbt) e 
\subset \cdots \quad\text{et}\quad 
\bigcup_{m \ge 0} e(\FC_m \Hbt)e = e\Hbt e.$$
On v\'erifie ais\'ement que
$$e (\FC_m \Hbt) e = \FC_m\Hbt \cap e\Hbt e$$
et que
\equat\label{spherique graduee-1}
\grad_\FC(e\Hbt e) = e \grad_\FC(\Hbt) e \simeq \kb[\CCBt] \otimes \kb[V \times V^*]^W.
\endequat
Il en d\'ecoule le th\'eor\`eme suivant, dont la preuve apr\`es sp\'ecialisation 
se trouve dans~\cite[th\'eor\`eme 1.5]{EG} mais dont tous les arguments s'adaptent ici 
mot pour mot au cas g\'en\'erique~:

\bigskip

\begin{theo}[Etingof-Ginzburg]\label{EG spherique-1}
Soit $\CGt$ un id\'eal premier de $\kb[\CCBt]$ et notons $\Hbt_\CGt=\Hbt/\CGt\Hbt$. Alors~:
\begin{itemize}
\itemth{a} L'alg\`ebre $e \Hbt_\CGt e$ est une $\kb$-alg\`ebre de type fini 
et sans diviseur de $0$.

\itemth{b} $\Hbt_\CGt e$ est un $e\Hbt_\CGt e$-module \`a droite de type fini.

\itemth{c} L'action naturelle \`a gauche de $\Hbt_\CGt$ sur le module projectif $\Hbt_\CGt e$ induit un isomorphisme 
$\Hbt_\CGt \stackrel{\sim}{\longto} \End_{e\Hbt_\CGt e}(\Hbt_\CGt e)$.

\itemth{d} Si $\kb[\CCBt]/\CGt$ est Gorenstein (respectivement Cohen-Macaulay), alors il en est de m\^eme de l'alg\`ebre 
$e\Hbt_\CGt e$ ainsi que du $e\Hbt_\CGt e$-module \`a droite $\Hbt_\CGt e$. 
\end{itemize}
\end{theo}

\bigskip
%

%
%
%

\section{Quelques automorphismes de $\Hbt$}\label{section:automorphismes-1}

\medskip

Notons $\Aut_{\kb\text{-}\alg}(\Hbt)$ le groupe des automorphismes de la $\kb$-alg\`ebre $\Hbt$.

\bigskip

\subsection{Bi-graduation}\label{subsection:bi-graduation-1}
De mani\`ere \'equivalente, la bi-graduation sur $\Hbt$ peut se voir comme 
une action du groupe $\kb^\times \times \kb^\times$ sur $\Hbt$. En effet, 
si $(\xi,\xi') \in \kb^\times \times \kb^\times$, on d\'efinit 
l'automorphisme $\gradauto_{\xi,\xi'}$ \indexnot{ba}{\gradauto_{\xi,\xi'}} de $\Hbt$ par la formule suivante~:
$$\forall~(i,j) \in \NM \times \NM,~\forall~h \in \Hbt^{\NM \times \NM}[i,j],~
\gradauto_{\xi,\xi'}(h)=\xi^i\xi^{\prime j} h.$$
Alors
\equat\label{graduation automorphisme-1}
\grad : \kb^\times \times \kb^\times \longto \Aut_{\kb\text{-}\alg}(\Hbt)
\endequat
est un morphisme de groupes. Concr\`etement,
$$
\begin{cases}
\forall~y \in V,~\gradauto_{\xi,\xi'}(y)=\xi y,\\
\forall~x \in V^*,~\gradauto_{\xi,\xi'}(x)=\xi' x,\\
\forall~C \in \CCBt^*,~\gradauto_{\xi,\xi'}(C)=\xi\xi' C,\\
\forall~w \in W,~\gradauto_{\xi,\xi'}(w)=w.\\
\end{cases}
$$
Par sp\'ecialisation, pour tout $\xi \in \kb^\times$ et 
pour tout $(t,c) \in \CCBt$, l'action de $(\xi,\xi)$ fournit un isomorphisme de $\kb$-alg\`ebres
\equat\label{xi c}
\Hbt_{t,c} \longiso \Hbt_{\xi t, \xi c}.
\endequat

\bigskip

\subsection{Caract\`eres lin\'eaires}\label{subsection:lineaires-1} 
Soit $\g : W \longto \kb^\times$ un caract\`ere lin\'eaire. Alors $\g$ induit 
un automorphisme de $\CCB$ par multiplication~: 
si $c \in \CCB$, on d\'efinit $\g \cdot c$ comme la fonction $\Ref(W) \to \kb$, $s \mapsto \g(s)c_s$. 
Cela induit un automorphisme $\g_\CCB : \kb[\CCB] \to \kb[\CCB]$, \indexnot{gz}{\g_\CCB}
$f \mapsto (c \mapsto f(\g^{-1} \cdot c))$, qui envoie en particulier 
$C_s$ sur $\g(s)^{-1} C_s$. On l'\'etend en un automorphisme $\g_\CCBt$ de $\kb[\CCBt]$ en posant 
$\g_\CCBt(T)=T$. 

D'autre part, $\g$ induit aussi un automorphisme de l'alg\`ebre de groupe $\kb W$, 
envoyant $w \in W$ sur $\g(w) w$. Ainsi, $\g$ induit un automorphisme 
de la $\kb[\CCBt]$-alg\`ebre $\kb[\CCBt] \otimes \bigl(\Trm_\kb(V \oplus V^*) \rtimes W\bigr)$ 
agissant trivialement sur $V$ et $V^*$~: nous le noterons $\g_\Trm$. Bien s\^ur,
$$(\g\g')_\Trm=\g_\Trm \g_\Trm'.$$
Les relations~\ref{relations-1} \'etant clairement stables par l'action de $\g_\Trm$, 
ce dernier induit par passage au quotient un automorphisme $\g_*$ de 
la $\kb$-alg\`ebre $\Hbt$. L'application
\equat\label{action caracteres lineaires-1}
\fonctio{W^\wedge}{\Aut_{\kb\text{-}\alg}(\Hbt)}{\g}{\g_*}
\endequat
est un morphisme injectif de groupes. Par sp\'ecialisation en $(t,c) \in \CCBt$ 
et si $\g \in W^\wedge$, alors $\g_*$ induit un isomorphisme de $\kb$-alg\`ebres
\equat\label{gamma c-1}
\Hbt_{t,c} \longiso \Hbt_{t,\g \cdot c}.
\endequat


\bigskip

\subsection{Normalisateur}\label{subsection:normalisateur-1}
Notons $\NC$ \indexnot{N}{\NC}  le normalisateur, dans $\GL_\kb(V)$, de $W$. Alors~:
\begin{itemize}
\item[$\bullet$] $\NC$ agit naturellement sur $V$ et $V^*$~;

\item[$\bullet$] $\NC$ agit sur $W$ par conjugaison~;

\item[$\bullet$] l'action de $\NC$ sur $W$ stabilise $\Ref(W)$ et donc $\NC$ agit 
sur $\CCB$~: si $g \in \NC$ et $c \in \CCB$, alors 
$\lexp{g}{c} : \Ref(W) \to \kb$, $s \mapsto c_{g^{-1}sg}$.

\item[$\bullet$] l'action de $\NC$ sur $\CCB$ induit une action de $\NC$ sur $\CCB^*$ 
(et donc sur $\kb[\CCB]$) telle que, si $g \in \NC$ et $s \in \Ref(W)$, alors 
$\lexp{g}{C_s}=C_{gsg^{-1}}$.

\item[$\bullet$] $\NC$ agit trivialement sur $T$. 
\end{itemize}
En cons\'equence, $\NC$ agit sur la $\kb[\CCBt]$-alg\`ebre 
$\kb[\CCBt] \otimes \bigl(\Trm_\kb(V \oplus V^*) \rtimes W\bigr)$ et 
il est facile de v\'erifier, gr\^ace aux relations~\ref{relations-1}, que 
cette action induit une action sur $\Hbt$~: si $g \in \NC$ et $h \in \Hbt$, 
nous noterons $\lexp{g}{h}$ l'image de $h$ par l'action de $g$.
Par sp\'ecialisation en $(t,c) \in \CCBt$, un \'el\'ement $g \in \NC$ induit un isomorphisme de $\kb$-alg\`ebres
\equat\label{N c}
\Hbt_{t,c} \longiso \Hbt_{t,\lexp{g}{c}}.
\endequat

\bigskip
%
%
%

\begin{exemple}\label{Z graduation et normalisateur-1}
Si $\xi \in \kb^\times$, alors on peut voir $\xi$ comme un automorphisme de $V$ 
(par homoth\'etie) normalisant (et m\^eme centralisant) $W$. On retrouve 
alors l'automorphisme de $\Hbt$ induisant la $\ZM$-graduation (au signe pr\`es)~: 
si $h \in \Hbt$, alors $\lexp{\xi}{h}=\gradauto_{\xi,\xi^{-1}}(h)$.\finl
\end{exemple}

\bigskip

\subsection{Compilation}\label{subsection:compilation-1} 
Les automorphismes induits par $\kb^\times \times \kb^\times$ et 
ceux induits par $W^\wedge$ commutent. D'autre part, le 
groupe $\NC$ agit sur le groupe $W^\wedge$ et sur la $\kb$-alg\`ebre $\Hbt$. 
On v\'erifie facilement que cela induit une action de $W^\wedge \rtimes \NC$ 
sur $\Hbt$. De plus, cette action respecte la bi-graduation, c'est-\`a-dire commute avec 
l'action de $\kb^\times \times \kb^\times$. 
Si $\g \in W^\wedge$ et $g \in \NC$, nous noterons $\g \rtimes g$ l'\'el\'ement de 
$W^\wedge \rtimes \NC$ correspondant. 
On a donc un morphisme de groupes
$$\fonctio{\kb^\times \times \kb^\times \times (W^\wedge \rtimes \NC)}{\Aut_{\kb\text{-}\alg}(\Hbt)}{
(\xi,\xi',\g \rtimes g)}{(h \mapsto \gradauto_{\xi,\xi'} \circ \g_*(\lexp{g}{h})).}$$
Si $\t=(\xi,\xi',\g \rtimes g) \in \kb^\times \times \kb^\times \times (W^\wedge \rtimes \NC)$ 
et si $h \in \Hbt$, nous poserons
$$\lexp{\t}{h}=\g_*(\gradauto_{\xi,\xi'}(\lexp{g}{h})).$$\indexnot{ha}{{{^\t{h}}}}
Le lemme suivant est imm\'ediat~:

\bigskip

\begin{lem}\label{lem:automorphismes-1}
Soit $\t=(\xi,\xi',\g \rtimes g) \in \kb^\times \times \kb^\times \times (W^\wedge \rtimes \NC)$. 
Alors~:
\begin{itemize}
\itemth{a} $\t$ stabilise les sous-alg\`ebres $\kb[\CCBt]$, $\kb[V]$, $\kb[V^*]$ et $\kb W$.

\itemth{b} $\t$ respecte la bi-graduation.

\itemth{c} $\lexp{\t}{\eulertilde}=\xi\xi' \eulertilde$.

\itemth{d} $\lexp{\t}{e}=e$ si et seulement si $\g=1$.
\end{itemize}
\end{lem}

%
%

\section{Sp\'ecificit\'e des groupes de Coxeter}\label{section:coxeter-htilde}

\medskip

\cbstart

\boitegrise{{\bf Hypoth\`ese.} 
{\it Dans cette section, et seulement dans cette section, nous supposons 
que $W$ est un groupe de Coxeter, et nous reprenons les notations du 
chapitre~\ref{chapter:coxeter}.}}{0.75\textwidth}

\bigskip

D'apr\`es la proposition~\ref{prop:coxeter}, il existe une forme bilin\'eaire 
sym\'etrique non d\'eg\'en\'er\'ee $W$-invariante $\betb : V \times V \to \kb$. \indexnot{bz}{\betb}  
Nous noterons $\s : V \stackrel{\sim}{\longto} V^*$ \indexnot{sz}{\s}  l'isomorphisme 
induit par $\betb$~: si $y$, $y' \in V$, alors 
$$\langle y,\s(y') \rangle = \betb(y,y').$$
La $W$-invariance de $\betb$ implique que $\s$ est un isomorphisme de $\kb W$-modules 
et la sym\'etrie de $\betb$ implique que 
\equat\label{eq:sigma-sigma-inverse}
\langle y , x \rangle = \langle \s^{-1}(x),\s(y) \rangle
\endequat
pour tous $x \in V^*$ et $y \in V$. Par suite, notons 
$\s_\Trm : \Trm_\kb(V \oplus V^*) \to \Trm_\kb(V \oplus V^*)$ l'isomorphisme d'alg\`ebres induit par l'automorphisme de 
l'espace vectoriel $V \oplus V^*$ d\'efini par $(y , x) \mapsto (-\s^{-1}(x), \s(y))$. Il est $W$-invariant 
donc il s'\'etend en un automorphisme de $\Trm_\kb(V \oplus V^*) \rtimes W$, 
en agissant trivialement sur $W$. Ce dernier s'\'etend ensuite par extension des 
scalaires en un automorphisme, toujours not\'e $\s_\Trm$, de 
$\kb[\CCBt] \otimes (\Trm_\kb(V \oplus V^*) \rtimes W)$. 
Il est facile de v\'erifier que $\s_\Trm$ induit un automorphisme $\s_\Hbt$ \indexnot{sz}{\s_\Hbt}  de $\Hbt$. 
On a donc montr\'e la proposition suivante~:

\bigskip

\begin{prop}\label{prop:auto-coxeter-1}
Il existe un unique automorphisme de $\kb$-alg\`ebres $\s_\Hbt$ de $\Hbt$ tel que 
$$
\begin{cases}
\s_\Hbt(y)=\s(y) & \text{si $y \in V$,}\\
\s_\Hbt(x)=-\s^{-1}(x) & \text{si $x \in V^*$,}\\
\s_\Hbt(w)=w & \text{si $w \in W$,}\\
\s_\Hbt(C)=C & \text{si $C \in \CCBt^*$.}\\
\end{cases}
$$
\end{prop}

\bigskip

\begin{prop}\label{prop-bis:auto-coxeter-1}
On a les propri\'et\'es suivantes~:
\begin{itemize}
\itemth{a} $\s_\Hbt$ stabilise les sous-alg\`ebres $\kb[\CCBt]$ et $\kb W$ et \'echange 
les sous-alg\`ebres $\kb[V]$ et $\kb[V^*]$.

\itemth{b} Si $h \in \Hbt^{\NM \times \NM}[i,j]$, 
alors $\s_\Hbt(h) \in \Hbt^{\NM \times \NM}[j,i]$.

\itemth{c} Si $h \in \Hbt^\NM[i]$ (respectivement $h \in \Hbt^\ZM[i]$), 
alors $\s_\Hbt(h) \in \Hbt^\NM[i]$ (respectivement $\s_\Hbt(h) \in \Hbt^\ZM[-i]$).

\itemth{d} $\s_\Hbt$ commute \`a l'action de $W^\wedge$ sur $\Hbt$.

\itemth{e} Si $(t,c) \in \CCBt$, alors $\s_\Hbt$ induit un automorphisme 
de $\Hbt_{t,c}$, toujours not\'e $\s_\Hbt$ (ou $\s_{\Hbt_{t,c}}$) si n\'ecessaire).

\itemth{f} $\s_\Hbt(\eulertilde)=nT - \eulertilde$.
\end{itemize}
\end{prop}

\begin{rema}[Actions de $\Gb\Lb_2(\kb)$ et $\Sb\Lb_2(\kb)$]\label{rem:sl2-1}
Soit $\r = \begin{pmatrix} a & b \\ c & d \end{pmatrix} \in \Gb\Lb_2(\kb)$. 
Alors l'application $\kb$-lin\'eaire 
$$\fonctio{V \oplus V^*}{V \oplus V^*}{y \oplus x}{ay+b\s^{-1}(x) \oplus c \s(y) + dx}$$
est un automorphisme du $\kb W$-module $V \oplus V^*$. Elle s'\'etend donc en un automorphisme 
de la $\kb$-alg\`ebre $\Trm_\kb(V \oplus V^*) \rtimes W$ et en un 
automorphisme $\r_\Trm$ de $\kb[\CCBt] \otimes (\Trm_\kb(V \oplus V^*) \rtimes W)$ 
par $\r_\Trm(C)=\det(\r) C$ si $C \in \CCBt^*$. 

Il est alors facile de v\'erifier que $\r_\Trm$ induit un automorphisme $\r_\Hbt$ \indexnot{rz}{\r_\Hbt}  
de $\Hbt$. De plus, $(\r \r')_\Hbt=\r_\Hbt \circ \r_\Hbt'$ pour tous $\r$, $\r' \in \Gb\Lb_2(\kb)$, 
ce qui d\'efinit une action de $\Gb\Lb_2(\kb)$ sur $\Hbt$. Cette action respecte la $\NM$-graduation 
$\Hbt^\NM$. 

Pour finir, notons que, pour $\r=\begin{pmatrix} 0 & -1 \\ 1 & 0 \end{pmatrix}$, on a 
$\r_\Hbt=\s_\Hbt$ et, si $\r=\begin{pmatrix} \xi & 0 \\ 0 & \xi' \end{pmatrix}$, alors 
$\r_{\Hbt}=\gradauto_{\xi,\xi'}$. On a ainsi \'etendu l'action de 
$\kb^\times \times \kb^\times \times (W^\wedge \rtimes \NC)$ en une action de 
$\Gb\Lb_2(\kb) \times (W^\wedge \rtimes \NC)$.\finl
\end{rema}

\cbend

\chapter{Alg\`ebre de Cherednik en $t=0$}\label{chapter:cherednik-0}

\boitegrise{{\bf Notation.} {\it Nous posons $\Hb=\Hbt/T\Hbt$. \indexnot{H}{\Hb}  La $\kb$-alg\`ebre 
$\Hb$ est appel\'ee l'{\bfit alg\`ebre de Cherednik g\'en\'erique en $t=0$}.}}{0.75\textwidth}

\bigskip

\section{G\'en\'eralit\'es}

\medskip

Nous rassemblons ici toutes les propri\'et\'es qui se d\'eduisent imm\'ediatement 
de ce que nous avons \'etabli dans le chapitre pr\'ec\'edent~\ref{chapter:cherednik-1}. 
Nous en profitons pour fixer quelques notations.

Tout d'abord, les relations~\ref{relations-1} s'\'ecrivent maintenant ainsi. 
L'alg\`ebre $\Hb$ est la $\kb[\CCB]$-alg\`ebre quotient 
de $\kb[\CCB] \otimes \bigl(\Trm_\kb(V \oplus V^*) \rtimes W\bigr)$ 
par l'id\'eal engendr\'e par les relations suivantes~:
\equat\label{relations-0}\begin{cases}
[x,x']=[y,y']=0, \\
\\
[y,x] = \DS{\sum_{s \in \Ref(W)} (\e(s)-1)\hskip1mm C_s 
\hskip1mm\frac{\langle y,\a_s \rangle \cdot \langle \a_s^\ve,x\rangle}{\langle \a_s^\ve,\a_s\rangle}
\hskip1mm s,} 
\end{cases}\endequat
pour $x$, $x' \in V^*$ et $y$, $y' \in V$. 

La d\'ecomposition PBW (th\'eor\`eme~\ref{PBW-1}) se r\'e\'ecrit ainsi~:

\bigskip

\begin{theo}[Etingof-Ginzburg]\label{PBW-0}
L'application $\kb[\CCB]$-lin\'eaire $\kb[\CCB] \otimes \kb[V] \otimes \kb W 
\otimes \kb[V^*] \longto \Hb$ induite par la multiplication est un isomorphisme de 
$\kb$-espaces vectoriels.
\end{theo}

\bigskip

Si $c \in \CCB$, nous noterons $\CG_c$ \indexnot{C}{\CG_c} l'id\'eal maximal de $\kb[\CCB]$ d\'efini 
par $\CG_c=\{f \in \kb[\CCB]~|~f(c)=0\}$~: c'est l'id\'eal engendr\'e par 
$(C_s-c_s)_{s \in \refw}$. Posons alors 
$$\Hb_c = (\kb[\CCB]/\CG_c) \otimes_{\kb[\CCB]} \Hb = \Hb/\CG_c \Hb = \Hbt_{0,c}.\indexnot{H}{\Hb_c}$$
La $\kb$-alg\`ebre $\Hb_c$ est le quotient de la $\kb$-alg\`ebre 
$\Trm_\kb(V \oplus V^*) \rtimes W$ par l'id\'eal engendr\'e par les relations suivantes~: 
\equat\label{relations specialisees-0}\begin{cases}
[x,x']=[y,y']=0, \\
\\
[y,x] = \DS{\sum_{s \in \Ref(W)} (\e(s)-1)\hskip1mm c_s 
\hskip1mm\frac{\langle y,\a_s \rangle \cdot \langle \a_s^\ve,x\rangle}{\langle \a_s^\ve,\a_s\rangle}
\hskip1mm s,} 
\end{cases}\endequat
pour $x$, $x' \in V^*$ et $y$, $y' \in V$. 

Puisque $T$ est bi-homog\`ene, la $\kb$-alg\`ebre $\Hb$ 
h\'erite de toutes les graduations, filtrations de l'alg\`ebre $\Hbt$~: 
nous utiliserons les notations \'evidentes $\Hb^{\NM \times \NM}[i,j]$, 
$\Hb^\NM[i]$ et $\Hb^\ZM[i]$ sans les d\'efinir. 
Nous noterons $\euler$ \indexnot{ea}{\euler}  l'image de $\eulertilde$ dans $\Hb$. 
Alors $\euler$ est appel\'e l'{\it \'el\'ement d'Euler g\'en\'erique} de $\Hb$. 
Notons que
\equat\label{eq:euler-1-1}
\euler \in \Hb^{\NM \times \NM}[1,1]
\endequat
L'id\'eal engendr\'e par $T$ est aussi stable par l'action de 
$\kb^\times \times \kb^\times \times (W^\wedge \rtimes \NC)$, donc 
$\Hb$ h\'erite d'une action de 
$\kb^\times \times \kb^\times \times (W^\wedge \rtimes \NC)$ qui sera toujours not\'ee 
$\lexp{\t}{h}$, si $h \in \Hb$ et 
$\t \in \kb^\times \times \kb^\times \times (W^\wedge \rtimes \NC)$. 
Le lemme suivant se d\'eduit imm\'ediatement du lemme~\ref{lem:automorphismes-1}~:

\bigskip

\begin{lem}\label{lem:automorphismes-0}
Soit $\t=(\xi,\xi',\g \rtimes g) \in \kb^\times \times \kb^\times \times (W^\wedge \rtimes \NC)$. 
Alors~:
\begin{itemize}
\itemth{a} $\t$ stabilise les sous-alg\`ebres $\kb[\CCB]$, $\kb[V]$, $\kb[V^*]$ et $\kb W$.

\itemth{b} $\t$ respecte la bi-graduation.

\itemth{c} $\lexp{\t}{\euler}=\xi\xi'~ \euler$.
\end{itemize}
\end{lem}

Terminons par l'alg\`ebre sph\'erique. Le th\'eor\`eme~\ref{EG spherique-1} implique~:

\bigskip

\begin{theo}[Etingof-Ginzburg]\label{EG spherique-0}
Soit $\CG$ un id\'eal premier de $\kb[\CCB]$ et notons $\Hb_\CG=\Hb/\CG\Hb$. \indexnot{H}{\Hb_\CG}  Alors~:
\begin{itemize}
\itemth{a} L'alg\`ebre $e \Hb_\CG e$ est une $\kb$-alg\`ebre de type fini, Gorenstein 
et sans diviseur de $0$.

\itemth{b} Si $\kb[\CCB]/\CG$ est Gorenstein (respectivement Cohen-Macaulay), alors 
le $e\Hb_\CG e$-module \`a droite de type fini $\Hb_\CG e$ l'est aussi.

\itemth{c} L'action naturelle \`a gauche de $\Hb_\CG$ sur le module projectif $\Hb_\CG e$ induit un isomorphisme 
$\Hb_\CG \stackrel{\sim}{\longto} \End_{e\Hb_\CG e}(\Hb_\CG e)$.
\end{itemize}
\end{theo}

\bigskip


\section{Centre}

\medskip

\boitegrise{{\bf Notation.} {\it Tout au long de ce m\'emoire, 
nous noterons $Z=\Zrm(\Hb)$ \indexnot{Z}{Z,~Z_c}  le centre de $\Hb$. Si $c \in \CCB$, nous noterons 
$Z_c=Z/\CG_c Z$. Nous noterons $P$ \indexnot{P}{P}  la $\kb[\CCB]$-alg\`ebre 
obtenue par produit tensoriel d'alg\`ebres $P=\kb[\CCB] \otimes \kb[V]^W \otimes \kb[V^*]^W$.}}{0.75\textwidth}

\bigskip

\subsection{Une sous-alg\`ebre de $Z$} 
Le premier r\'esultat, fondamental, concernant le centre $\Zrm(\Hb)$ de $\Hb$, 
est le suivant~\cite[Proposition 4.15]{EG} 
(voir aussi~\cite[Proposition 3.6]{gordon} pour une preuve plus directe)~:
\equat\label{inclusion centre}
\text{\it $P \subset Z$.}
\endequat
Le lemme suivant est imm\'ediat. Bien s\^ur, le centre $Z$ est stable par l'action 
de $\kb^\times \times \kb^\times \times (W^\wedge \rtimes \NC)$ et h\'erite 
en particulier des graduations de $\Hb$. Il est facile de v\'erifier qu'il en est de m\^eme 
de $P$~:

\bigskip

\begin{lem}\label{P stable}
$P$ est une sous-alg\`ebre du centre stable par l'action de 
$\kb^\times \times \kb^\times \times (W^\wedge \rtimes \NC)$. En particulier, 
elle est $\NM \times \NM$-gradu\'ee.
\end{lem}

\bigskip

\begin{coro}\label{coro:P-libre}
La d\'ecomposition PBW est un isomorphisme de $P$-modules. En particulier, on a des isomorphismes 
de $P$-modules~:
\begin{itemize}
\itemth{a} $\Hb \simeq \kb[\CCB] \otimes \kb[V] \otimes \kb W \otimes \kb[V^*]$.

\itemth{b} $\Hb e \simeq \kb[\CCB] \otimes \kb[V] \otimes \kb[V^*]$.

\itemth{c} $e\Hb e \simeq \kb[\CCB] \otimes \kb[V \times V^*]^W$.
\end{itemize}
En particulier, $\Hb$ (respectivement $\Hb e$, respectivement $e\Hb e$) est un $P$-module libre 
de rang $|W|^3$ (respectivement $|W|^2$, respectivement $|W|$). 
\end{coro}

\bigskip

Le th\`eme principal de ce m\'emoire est l'\'etude de l'alg\`ebre $\Hb$, vue comme 
$P$-alg\`ebre~: ainsi, si $\pG$ est un id\'eal premier de $P$, 
nous nous int\'eresserons aux repr\'esentations de la $\kb_P(\pG)$-alg\`ebre 
de dimension finie $|W|^3$ \'egale \`a $\kb_P(\pG) \otimes_P \Hb$ (d\'eploiement, 
modules simples, blocs, \'eventuellement modules standard, matrice de d\'ecomposition...).

\bigskip

\begin{rema}\label{rem:P-libre}
Soit $(b_i)_{1 \le i \le |W|}$ une $\kb[V]^W$-base de $\kb[V]$ et soit 
$(b_i^*)_{1 \le i \le |W|}$ une $\kb[V^*]^W$-base de $\kb[V^*]$. Le corollaire~\ref{coro:P-libre} 
montre que $(b_i w b_j^*)_{\substack{1 \le i,j \le |W| \\ w \in W}}$ est une $P$-base de 
$\Hb$ et que $(b_i b_j^* e)_{1 \le i,j \le |W|}$ est une $P$-base de $\Hb e$.\finl
\end{rema}

\bigskip

Posons 
$$P_\bullet = \kb[V]^W \otimes \kb[V^*]^W.\indexnot{P}{P_\bullet}$$
Si $c \in \CCB$, alors 
$$P_\bullet \simeq \kb[\CCB]/\CG_c \otimes_{\kb[\CCB]} P = P/\CG_c P .$$
On d\'eduit du corollaire~\ref{coro:P-libre} le corollaire suivant~:

\bigskip

\begin{coro}\label{coro:P-libre-c}
On a des isomorphismes de $P_\bullet$-modules~:
\begin{itemize}
\itemth{a} $\Hb_c \simeq \kb[V] \otimes \kb W \otimes \kb[V^*]$.

\itemth{b} $\Hb_c e \simeq \kb[V] \otimes \kb[V^*]$.

\itemth{c} $e\Hb_c e \simeq \kb[V \times V^*]^W$.
\end{itemize}
En particulier, $\Hb_c$ (respectivement $\Hb_c e$, respectivement $e\Hb_c e$) est un 
$P_\bullet$-module libre de rang $|W|^3$ (respectivement $|W|^2$, respectivement $|W|$). 
\end{coro}

\bigskip

\subsection{Isomorphisme de Satake} 
Il d\'ecoule du lemme~\ref{lem:eulertilde} que
\equat\label{eq:euler-central}
\euler \in Z.
\endequat
Si $c \in \CCB$, nous noterons $\euler_c$ \indexnot{ea}{\euler_c}  l'image de $\euler$ dans $\Hb_c$. 

\bigskip

\begin{exemple}\label{exemple zero-0} 
Puisque $\Hb_0=\kb[V \oplus V^*] \rtimes W$, on a 
$\Zrm(\Hb_0) =\kb[V \times V^*]^W$ et $P_\bullet = \kb[V/W \times V^*/W] \subset Z_0$. 
De plus, il d\'ecoule du th\'eor\`eme~\ref{chevalley} 
que $Z_0$ est un $P_\bullet$-module libre de rang $|W|$. D'autre part, 
$\euler_0=\sum_{i=1}^n x_i y_i$ (en reprenant les notations de la section~\ref{section:eulertilde}). 
On peut alors v\'erifier facilement (voir par exemple~\cite[proposition 1.1]{BK}) que 
le polyn\^ome minimal de $\euler_0$ sur $P_\bullet$ est de degr\'e $|W|$ et donc que 
le corps des fractions $\kb(V \times V^*)^W$ est engendr\'e, sur 
$\kb(V/W \times V^*/W)$, par $\euler_0$.\finl
\end{exemple}

\bigskip

Le th\'eor\`eme de structure suivant est au c\oe ur de tous les r\'esultats concernant la th\'eorie des 
repr\'esentations de $\Hb$. Nous proposons ici une preuve diff\'erente de celles contenues dans la 
litt\'erature, reposant uniquement sur l'exemple~\ref{exemple zero-0} 
ci-dessus et le fait que $P[\euler] \subset Z$. 

\bigskip

\begin{theo}[Etingof-Ginzburg]\label{theo:satake}
Le morphisme d'alg\`ebres $Z \longto e\Hb e$, $z \mapsto ze$ est un isomorphisme d'alg\`ebres 
$(\NM \times \NM)$-gradu\'ees. En particulier, $e\Hb e$ est commutative.
\end{theo}

\bigskip

\def\opp{{\mathrm{opp}}}

\begin{proof}
Commen\c{c}ons par un lemme pr\'eparatoire classique~:

\bigskip

\begin{quotation}
{\small
\begin{lem}\label{lem:ZA-ZB}
Soient $A$ et $B$ deux anneaux et soit $M$ un $(A,B)$-bimodule tel que 
$\End_A(M)=B^{\mathrm{opp}}$ et $A=\End_B(M)$. Alors on a un isomorphisme 
$\Zrm(A) \longiso \Zrm(B)$.
\end{lem}

\begin{proof}
La multiplication \`a gauche sur $M$ induit un morphisme d'anneaux 
$\a : \Zrm(A) \to \Zrm(B)$ tel que $zm=m\a(z)$ pour tous $z \in \Zrm(A)$ et $m \in M$. 
De m\^eme, la multiplication \`a droite fournit un morphisme d'anneaux $\b : \Zrm(B) \to \Zrm(A)$ 
tel que $mz=\b(z)m$ pour tout $z \in \Zrm(B)$ et $m \in M$. Ainsi, si $z \in \Zrm(A)$ et 
$m \in M$, alors $zm=\b(\a(z))m$, et donc $\b \circ \a = \Id_{\Zrm(A)}$ 
car l'action de $A$ sur $M$ est fid\`ele par hypoth\`ese. De m\^eme $\a \circ \b = \Id_{\Zrm(B)}$.
\end{proof}
}
\end{quotation}

\bigskip

Tout d'abord, puisque $\End_\Hb(\Hb e)=(e\Hb e)^\opp$ et $\End_{(e\Hb e)^{\opp}}(\Hb e) = \Hb$, 
l'application $\pi_e : \Zrm(\Hb) \to \Zrm(e\Hb e)$, $z \mapsto ze$ est un isomorphisme d'alg\`ebres 
d'apr\`es le lemme~\ref{lem:ZA-ZB}. Il suffit donc de montrer que $e\Hb e$ est commutative. 

Notons aussi que $\pi_e$ est un morphisme de $P$-alg\`ebres. Puisque $e\Hb e$ est 
un $P$-module libre de rang $|W|$ (voir le corollaire~\ref{coro:P-libre}(c)), 
le polyn\^ome minimal de $\euler$ sur $P$, 
qui est le m\^eme que le polyn\^ome minimal de $\pi_e(\euler)$ sur $\pi_e(P)\simeq P$, est de degr\'e 
inf\'erieur ou \'egal \`a $|W|$. Mais la sp\'ecialisation $\euler_0$ de $\euler$ en $0 \in \CCB^*$ 
a un polyn\^ome minimal sur $P_0$ de degr\'e $|W|$ 
(voir l'exemple~\ref{exemple zero-0}) On en d\'eduit donc que 
\equat\label{eq:minimal-euler}
\text{\it Le polyn\^ome minimal de $\euler$ sur $P$ est de degr\'e $|W|$.}
\endequat
On a alors
$$\pi_e(P[\euler]) \subset \pi_e(Z) =\Zrm(e\Hb e) \subset e\Hb e.$$
Il d\'ecoule de~\ref{eq:minimal-euler} que $\pi_e(P[\euler])$ est un sous-$P$-module de rang $|W|$ du 
$P$-module $e\Hb e$, qui est lui aussi de rang $|W|$. Si on note $\Kb$ le corps des fractions de $P$ 
(rappelons que $P$ est int\`egre), alors 
$\Kb \otimes_P \pi_e(P[\euler]) = \Kb \otimes_P (e\Hb e)$. Ainsi, $\Kb \otimes_P e\Hb e$ est 
commutative, ce qui montre que $e\Hb e$ l'est aussi.
\end{proof}

\bigskip

\begin{coro}\label{coro:endo-bi}
On a~:
\begin{itemize}
\itemth{a} $\End_\Hb(\Hb e) = Z$ et $\End_Z(\Hb e)=\Hb$.

\itemth{b} $\Hb=Z \oplus e\Hb (1-e) \oplus (1-e)\Hb e \oplus (1-e)\Hb(1-e)$.
En particulier, $Z$ est un facteur direct du $Z$-module $\Hb$.

\itemth{c} $Z$ est un $P$-module libre de rang $|W|$. 

\itemth{d} $Z$ est int\`egre, int\'egralement clos et Cohen-Macaulay.

\itemth{e} $\Hb e$ est un $Z$-module de Cohen-Macaulay.
\end{itemize}
\end{coro}

\bigskip

\begin{proof}
Notons que $\Hb=e\Hb e \oplus e\Hb (1-e) \oplus (1-e)\Hb e \oplus (1-e)\Hb(1-e)$ et 
que l'isomorphisme de Satake nous dit que l'application $Z \to e\Hb e$, $z \mapsto eze=ze$, est 
un isomorphisme, do\`u (a). 
\`A part le fait que $Z$ est int\'egralement clos, tout le reste d\'ecoule des 
propri\'et\'es correspondantes de $e\Hb e$. Le fait que $Z \simeq e\Hb e$ est int\'egralement clos 
d\'ecoule de ce que $\grad_\FC(e \Hb e) \simeq \kb[\CCB] \otimes \kb[V \times V^*]^W$ l'est 
(voir~\cite[Lemme~3.5]{EG}).
\end{proof}

\bigskip

Puisque $Z$ est un facteur direct de $\Hb$, $Z_c$ est un facteur direct de $\Hb_c$. 
De plus, $Z_c \subset \Zrm(\Hb_c)$. Mieux, le corollaire~\ref{coro:endo-bi} 
reste valide apr\`es sp\'ecialisation~:

\bigskip

\begin{coro}\label{coro:satake-c}
$Z_c=\Zrm(\Hb_c)$ et l'application $Z_c \to e\Hb_c e$, $z \mapsto ze$ est 
un isomorphisme.
\end{coro}

\bigskip

\begin{proof}
D'apr\`es le th\'eor\`eme~\ref{EG spherique-0} et le lemme~\ref{lem:ZA-ZB}, 
on a $\Zrm(\Hb_c) \simeq \Zrm(e\Hb_c e)$. Mais d'apr\`es l'isomorphisme de Satake, 
$e\Hb_c e$ est commutative. Donc l'application $\Zrm(\Hb_c) \to e\Hb_c e$, $z \mapsto ze$ 
est un isomorphisme. On en d\'eduit que 
$$\Hb_c=\Zrm(\Hb_c) \oplus e\Hb_c (1-e) \oplus (1-e)\Hb_c e \oplus (1-e)\Hb_c(1-e).$$
De plus, d'apr\`es le corollaire~\ref{coro:endo-bi}(b), 
$$\Hb_c=Z_c \oplus e\Hb_c (1-e) \oplus (1-e)\Hb_c e \oplus (1-e)\Hb_c(1-e).$$
Puisque $Z_c \subset \Zrm(\Hb_c)$, l'\'egalit\'e $Z_c=\Zrm(\Hb_c)$ 
devient \'evidente. 
\end{proof}

\bigskip

\begin{coro}\label{coro:endo-bi-c}
On a~:
\begin{itemize}
\itemth{a} $\End_{\Hb_c}(\Hb_c e) = Z_c$ et $\End_{Z_c}(\Hb_c e)=\Hb_c$.

\itemth{b} $\Hb_c=Z_c \oplus e\Hb_c (1-e) \oplus (1-e)\Hb_c e \oplus (1-e)\Hb_c(1-e)$.
En particulier, $Z_c$ est un facteur direct du $Z_c$-module $\Hb_c$.

\itemth{c} $Z_c$ est un $P_\bullet$-module libre de rang $|W|$. 

\itemth{d} $Z_c$ est int\`egre, int\'egralement clos et Cohen-Macaulay.

\itemth{e} $\Hb_c e$ est un $Z_c$-module de Cohen-Macaulay.
\end{itemize}
\end{coro}

\bigskip

%


\bigskip

\begin{coro}\label{coro:minimal-euler-c}
Si $c \in \CCB$, alors le polyn\^ome minimal de $\euler_c$ sur $P_\bullet$ 
est de degr\'e $|W|$~: c'est la sp\'ecialisation en $c$ du polyn\^ome 
minimal de $\euler$ sur $P$.
\end{coro}

\bigskip

\begin{proof}
Notons $F_\euler(\tb)$ le polyn\^ome minimal de $\euler$ sur $P$, et notons 
$F_{\euler,c}(\tb) \in P_\bullet[\tb]$ sa sp\'ecialisation en $c$. 
Alors $P_\bullet[\euler_c] \subset Z_c$ est int\`egre et 
$P_\bullet[\euler_c] \simeq P[\tb]/F_{\euler,c}(\tb)$. Donc 
$F_{\euler,c}(\tb)$ est irr\'eductible et est bien le poyn\^ome minimal 
de $\euler_c$ sur $P_\bullet$.
\end{proof}

\bigskip

\section{\'Equivalence de Morita}\label{sec:morita}

\medskip

D'apr\`es le corollaire~\ref{coro:endo-bi}, on a $\Hb = \End_Z(\Hb e)$ et 
$Z=\End_\Hb(\Hb e)$. De plus, $\Hb e$ est un $\Hb$-module projectif. 
Pour que $\Hb e$ induise une \'equivalence de Morita entre $Z$ et $\Hb$, il faut et il suffit que 
$\Hb e$ soit un $Z$-module projectif~: or, ce n'est pas le cas en g\'en\'eral, 
mais $\Hb e$ est un $Z$-module de Cohen-Macaulay. Puisqu'un module de Cohen-Macaulay 
sur un anneau r\'egulier est projectif~\cite[chapitre~4,~corollaire~2]{serre}, on obtient le r\'esultat suivant 
(si $U$ est une partie multiplicative de $Z$, nous notons $Z[U^{-1}]$ \indexnot{Z}{Z[U^{-1}]}  le localis\'e de 
$Z$ en $U$ et nous posons $\Hb[U^{-1}]=Z[U^{-1}] \otimes_Z \Hb$)~:

\bigskip

\begin{prop}\label{prop:morita}
Si $U$ est une partie multiplicative de $Z$ telle que $Z[U^{-1}]$ soit un anneau r\'egulier, 
alors les alg\`ebres $\Hb[U^{-1}]$ et $Z[U^{-1}]$ sont Morita \'equivalentes 
(l'\'equivalence de Morita est induite par le bimodule $\Hb[U^{-1}]e$).
\end{prop}

\bigskip

\subsection{Localisation en $V^\reg$} 
Rappelons que 
$$V^\reg=V \setminus \bigcup_{H \in \AC} H=\{v \in V~|~\Stab_G(v)=1\}.$$
On note $P^\reg=\kb[\CCB] \otimes \kb[V^\reg]^W \times \kb[V^*]^W$, \indexnot{P}{P^\reg} 
$Z^\reg=P^\reg \otimes_P Z$ \indexnot{Z}{Z^\reg}  
et $\Hb^\reg=P^\reg \otimes_P \Hb=Z^\reg \otimes_Z \Hb$. \indexnot{H}{\Hb^\reg} 
Si $s \in \Ref(W)$, notons $\a_s^W=\prod_{w \in W} w(\a_s) \in P$, \indexnot{az}{\a_s^W}  de sorte 
que $P^\reg$ (resp. $Z^\reg$) est le localis\'e de $P$ (resp. $Z$) 
en la partie multiplicative engendr\'ee par $(\a_s^W)_{s \in \Ref(W)}$. 
Il s'en suit que 
\equat\label{inversible alpha}
\text{\it $\a_s$ est inversible dans $\Hb^\reg$.}
\endequat
En effet, $\prod_{w \in W} w(\a_s) \in (P^\reg)^\times$.

\bigskip

\begin{prop}[Etingof-Ginzburg]
Il existe un unique isomorphisme de $\kb[\CCB]$-alg\`ebres 
$$\Th : \Hb^\reg \longiso \kb[\CCB] \otimes (\kb[V^\reg \times V^*] \rtimes W)\indexnot{ty}{\Th}$$
tel que 
$$
\begin{cases}
\Th(w) = w & \text{si $w \in W$},\\
\Th(y) = y - \DS{\sum_{s \in \Ref(W)} 
\e(s)C_s \hskip1mm\frac{\langle y,\a_s\rangle}{\a_s}\hskip1mm s} & \text{si $y \in V$},\\
\Th(x) = x & \text{si $x \in V^*$.}\\
\end{cases}
$$
\end{prop}

\begin{proof}
Voir~\cite[proposition 4.11]{EG}.
\end{proof}

\bigskip

\begin{rema}\label{rem:theta-H}
En termes des variables $K_{H,i}$, on a
$$\Theta(y)=y-\sum_{H\in\mathcal{A}}\sum_{i=0}^{e_H-1}\frac{\langle y,\alpha_H\rangle}{\alpha_H}
e_H K_{H,i}\e_{H,i}$$
pour $y\in V$.\finl
\end{rema}

\bigskip

\begin{coro}\label{centre reg}
$\Th$ se restreint en un isomorphisme de $\kb[\CCB]$-alg\`ebres 
$Z^\reg \longiso \kb[\CCB] \otimes \kb[V^\reg \times V^*]^W$. En particulier, 
$Z^\reg$ est r\'egulier. 
\end{coro}

\begin{proof}
La premi\`ere assertion d\'ecoule simplement de la comparaison des centres 
de $\Hb^\reg$ et $\kb[V^\reg \times V^*] \rtimes W$. La deuxi\`eme r\'esulte du fait 
que $W$ agit librement sur $V^\reg \times V^*$.
\end{proof}

\bigskip

Si $c \in \CCB$, notons $Z^\reg_c$ \indexnot{Z}{Z_c^\reg}  
la localisation de $Z_c$ en $P^\reg_\bullet =\kb[V^\reg]^W \otimes \kb[V^*]^W$. 
Le corollaire~\ref{centre reg} montre que
\equat\label{lisse reg}
\text{\it $Z^\reg_c \simeq \kb[V^\reg \times V^*]^W$ est un anneau r\'egulier.}
\endequat


\subsection{Corps des fractions} 
Notons $\Kb$ le corps des fractions de $P$ et posons $\Kb Z = \Kb \otimes_P Z$: 
alors, puisque $Z$ est int\`egre et entier sur $P$, 
\equat\label{eq:KZ-fraction}
\text{\it $\Kb Z$ est le corps des fractions de $Z$.}
\endequat
En particulier, $\Kb Z$ est un anneau r\'egulier.

\bigskip

\begin{theo}\label{theo:KH-mat}
Les $\Kb$-alg\`ebres $\Kb\Hb$ et $\Kb Z$ sont Morita \'equivalentes, 
l'\'equivalence de Morita \'etant induite par $\Kb \Hb e$. Plus pr\'ecis\'ement, 
$$\Kb\Hb \simeq \Mat_{|W|}(\Kb Z).$$
\end{theo}

\begin{proof}
Compte tenu de la proposition~\ref{prop:morita}, 
il ne reste qu'\`a montrer la derni\`ere assertion: il suffit donc de montrer 
que $\Kb\Hb e$ est un $\Kb Z$-espace vectoriel de dimension $|W|$. Or, $\Hb e$ est un $P$-module 
libre de rang $|W|^2$ tandis que $Z$ est un $P$-module libre de rang $|W|$, d'o\`u 
le r\'esultat.
\end{proof}

\bigskip

\section{Compl\'ements}

\bigskip

\subsection{Structure de Poisson}

La d\'ecomposition PBW fournit un isomorphisme de {\it $\kb$-espaces vectoriels} 
$\kb[T] \otimes \Hb \longiso \Hbt$. Si $h \in \Hb$, notons $\tilde{h}$ l'image de 
$1 \otimes h \in \kb[T] \otimes \Hb$ dans $\Hbt$ via cet isomorphisme. 
Si $z$, $z' \in Z$, alors $[z,z']=0$, ce qui signifie que 
$[\zti,\zti'] \in T \Hbt$. On note $\{z,z'\}$ l'image de $[\zti,\zti']/T \in \Hbt$ dans 
$\Hb = \Hbt/T\Hbt$. Il est facile de v\'erifier que $\{z,z'\} \in Z$ et que
\equat\label{eq:poisson}
\{-,-\} : Z \times Z \longto Z\indexnot{ZZZ}{{{\{,\}}}}
\endequat
est un {\it crochet de Poisson} $\kb[\CCB]$-lin\'eaire. Si $c \in \CCB$, alors 
il induit un crochet de Poisson
\equat\label{eq:poisson-c}
\{-,-\} : Z_c \times Z_c \longto Z_c.
\endequat

\bigskip

\subsection{Forme sym\'etrisante}\label{sub:symetrisante}

Fixons 
une $\kb[V]^W$-base $(b_i)_{1 \le i \le |W|}$ 
de $\kb[V]$ 
form\'ee d'\'el\'ements homog\`enes ainsi 
qu'une $\kb[V^*]^W$-base $(b_i^*)_{1 \le i \le |W|}$ de $\kb[V^*]$ 
elle aussi form\'ee d'\'el\'ements homog\`enes. 
Il est bien connu~\cite[proposition~4.16]{broue} que, parmi la base homog\`ene $(b_i)_{1 \le i \le |W|}$, 
un seul \'el\'ement est de degr\'e $|\Ref(W)|$, tous les autres \'etant de 
degr\'e strictement inf\'erieur. Nous supposerons 
que $\deg b_1 = \deg b_1^* = |\Ref(W)|$. Notons
$$\taub_\HHb : \Hb \longto P\indexnot{tx}{\taub_\HHb}$$
l'unique application $P$-lin\'eaire telle que 
$$\taub_\HHb(b_i w b_j^*) = 
\begin{cases}
1 & \text{si $i=j=1$ et $w=1$,}\\
0 & \text{sinon.}
\end{cases}$$
Alors~\cite[th\'eor\`eme 4.4]{BGS}
\equat\label{symetrique}
\text{\it $\taub_\HHb$ est une forme sym\'etrisante sur la $P$-alg\`ebre $\Hb$.}
\endequat

\bigskip

Soit $(h_i)_{1 \le i \le |W|^3}$ une $P$-base de $\Hb$ et soit $(h_i^\ve)_{1 \le i \le |W|^3}$ 
la $P$-base duale (pour la forme sym\'etrisante $\taub_\HHb$). Posons
$$\casimir_\HHb=\sum_{i=1}^{|W|^3} h_i h_i^\ve.$$
L'\'el\'ement $\casimir_\HHb$ de $\Hb$ est appel\'e l'{\it \'el\'ement de Casimir central} de $\Hb$ 
(ou $(\Hb,\taub_\HHb)$)~: 
il est ind\'ependant du choix de la $P$-base $(h_i)_{1 \le i \le |W|^3}$ (il ne d\'epend que du 
couple $(\Hb,\t)$) et il v\'erifie 
$$\casimir_\HHb \in Z\indexnot{ca}{\casimir_\HHb}$$
(voir par exemple~\cite[d\'efinition~2.2.4~et~lemme~2.2.6]{chlouveraki}).

\bigskip

\section{Comment calculer $Z$~?}\label{section:calcul Q}

\bigskip

Nous allons ici donner quelques ingr\'edients permettant d'aider 
au calcul du centre $Z$ de l'alg\`ebre $\Hb$~: nous ne donnerons pas d'algorithme, 
mais juste quelques principes permettant de r\'esoudre 
les petits cas. Nous les mettrons en \oe uvre dans les chapitres 
\ref{chapitre:rang 1} (le rang $1$) et~\ref{chapitre:b2} (le type $B_2$). 
Un des ingr\'edients essentiels est la bi-graduation de $Z$. 

\bigskip

\subsection{S\'eries de Hilbert}
Nous allons ici calculer les s\'eries de Hilbert bi-gradu\'ees 
de $\Hb$, $P$, $Z$ et $e\Hb e$. Tout d'abord, remarquons que 
$$\dim_\kb^\bigrad(\kb[\CCB]) = \frac{1}{(1-\tb\ub)^{|\refw|}},$$
ce qui permet de d\'eduire facilement le r\'esultat pour $\Hb$, 
en utilisant~(\ref{x}) et la d\'ecomposition PBW~:
\equat\label{hilbert H}
\dim_\kb^\bigrad(\Hb) = \frac{|W|}{(1-\tb)^n~(1-\ub)^n
~ (1-\tb\ub)^{|\refw|}}.
\endequat
D'autre part, en reprenant les notations du th\'eor\`eme~\ref{chevalley}, 
on obtient, gr\^ace \`a~(\ref{hilbert biinv}), 
\equat\label{hilbert P}
\dim_\kb^\bigrad(P)=
\frac{1}{(1-\tb\ub)^{|\refw|} ~\DS{\prod_{i=1}^n (1-\tb^{d_i})(1-\ub^{d_i})}}.
\endequat
Pour finir, notons que la d\'ecomposition PBW est un isomorphisme 
$W$-\'equivariant de $\kb[\CCB]$-modules bi-gradu\'es, d'o\`u l'on d\'eduit 
que $\Hb e \simeq \kb[\CCB] \otimes \kb[V] \otimes \kb[V^*]$ comme 
$\kb W$-modules bi-gradu\'es et donc que 
\equat\label{iso bigrad}
\text{\it les $\kb$-espaces vectoriels bi-gradu\'es $Z$ et 
$\kb[\CCB] \otimes \kb[V \times V^*]^W$ sont isomorphes.}
\endequat
Ainsi, $\dim_\kb^\bigrad(Z)=\dim_\kb^\bigrad(\kb[\CCB]) \cdot 
\dim_\kb^\bigrad(\kb[V \times V^*]^W)$ ce qui, 
en vertu de~(\ref{hilbert molien}) et de la proposition~\ref{dim bigrad Q0 fantome},
nous donne
\equat\label{hilbert Q molien}
\dim_\kb^\bigrad(Z) = \frac{1}{|W|~(1-\tb\ub)^{|\refw|}} 
\sum_{w \in W} \frac{1}{\det(1-w\tb)~\det(1-w^{-1}\ub)}
\endequat
et
\equat\label{hilbert Q fantome}
\dim_\kb^\bigrad(Z) = 
\frac{\DS{\sum_{\chi \in \Irr(W)} 
f_\chi(\tb)~f_\chi(\ub)}}{(1-\tb\ub)^{|\refw|} ~\DS{\prod_{i=1}^n (1-\tb^{d_i})(1-\ub^{d_i})}}.
\endequat

\subsection{Comparaison de $Z$ et $Z_0$}\label{subsection:Q Q0} 
Si $z \in Z$, nous noterons $z_0$ son image dans $Z_0$. Rappelons que 
$Z_0=Z/\CG_0 Z$, o\`u $\CG_0$ est l'unique id\'eal homog\`ene 
maximal de $\kb[\CCB]$ ($\CG_0$ est m\^eme bi-homog\`ene, par rapport 
\`a la $(\NM \times \NM)$-graduation). Le lemme de Nakayama gradu\'e 
permet imm\'ediatement d'obtenir le r\'esultat suivant~:

\bigskip

\begin{lem}\label{QQ}
Soient $z^{(1)}$,\dots, $z^{(r)}$ des \'el\'ements $\NM$-homog\`enes 
(par exemple, des \'el\'ements $(\NM \times \NM)$-homog\`enes) de $Z$. Alors 
$Z=\kb[\CCB][z^{(1)},\dots,z^{(r)}]$ si et seulement si $Z_0=\kb[z^{(1)}_0,\dots,z^{(r)}_0]$. 
\end{lem}

\bigskip

Si le lemme~\ref{QQ} peut aider \`a trouver des g\'en\'erateurs de $Z$, 
le probl\`eme du calcul de $Z$ n'est pas encore totalement r\'esolu \`a 
ce stade, car il faut d\'eterminer les relations entre ces g\'en\'erateurs. 
Fixons donc des \'el\'ements $z^{(1)}$,\dots, $z^{(r)}$ de $Z$ qui soient 
$\NM$-homog\`enes (notons $e_1$,\dots, $e_r$ leurs degr\'es) et tels que 
$Z=\kb[\CCB][z^{(1)},\dots,z^{(r)}]$. Graduons 
l'alg\`ebre $\kb[\CCB][\xb_1,\dots,\xb_r]$ en attribuant \`a $\xb_i$ le degr\'e 
$e_i$. Alors le morphisme naturel de $\kb[\CCB]$-alg\`ebres 
$$\ph : \kb[\CCB][\xb_1,\dots,\xb_r] \longto Z$$
qui envoie $\xb_i$ sur $z^{(i)}$ est surjectif et gradu\'e. On notera 
$$\ph_0 : \kb[\xb_1,\dots,\xb_r] \longto Z_0$$
le morphisme surjectif gradu\'e 
de $\kb$-alg\`ebres d\'eduit de $\ph$ par r\'eduction modulo $\CG_0$ 
(c'est-\`a-dire par sp\'ecialisation en $c=0$). 
Supposons trouv\'es des \'el\'ements homog\`enes 
$f^{(1)}$,\dots, $f^{(m)}$ dans $\Ker \ph$ et notons $I$ 
(respectivement $I_0$) l'id\'eal de $\kb[\CCB][\xb_1,\dots,\xb_r]$ 
(respectivement $\kb[\xb_1,\dots,\xb_r]$) engendr\'e par $f^{(1)}$,\dots, $f^{(m)}$ 
(respectivement $f^{(1)}_0$,\dots, $f^{(m)}_0$). Ici, si $f \in \kb[\CCB][\xb_1,\dots,\xb_r]$, 
alors $f_0$ d\'esigne son image dans $\kb[\xb_1,\dots,\xb_r]$ par sp\'ecialisation en $c=0$. 

\bigskip

\begin{prop}\label{presentation Q}
Supposons que $\kb[\CCB][\xb_1,\dots,\xb_r]/I$ soit un $\kb[\CCB]$-module libre. 
Alors~:
\begin{itemize}
\itemth{a}  $I_0=I/\CG_0 I$

\itemth{b} $I=\Ker \ph$ si et seulement si $I_0=\Ker \ph_0$.
\end{itemize}
\end{prop}

\begin{proof}
Notons $A$ la $\kb[\CCB]$-alg\`ebre $\kb[\CCB][\xb_1,\dots,\xb_r]$ et posons 
$A_0=A/\CG_0 A =\kb[\xb_1,\dots,\xb_r]$. 

\medskip

(a) La suite exacte de $\kb[\CCB]$-modules 
$$0 \longto I \longto A \longto A/I \longto 0$$
induit une suite exacte
$${\mathrm{Tor}}^{\kb[\CCB]}_1(\kb[\CCB]/\CG_0,A/I) \longto I/\CG_0 I \longto A_0 
\longto (A/I)/\CG_0(A/I) \longto 0.$$
Mais, par hypoth\`ese, $A/I$ est libre donc 
$\Tor^{\kb[\CCB]}_1(\kb[\CCB]/\CG_0,A/I)=0$ 
ce qui montre que $I/\CG_0 I$ s'identifie \`a l'image de $I$ dans $A_0$, 
c'est-\`a-dire $I_0$. D'o\`u (a).

\medskip

(b) Puisque $A/\Ker \ph \simeq Z$ est un $\kb[\CCB]$-module libre, l'argument 
pr\'ec\'edent montre que $(\Ker \ph)/\CG_0 (\Ker \ph)$ s'identifie \`a 
son image dans $A_0$, c'est-\`a-dire \`a $\Ker \ph_0$. 
Le r\'esultat d\'ecoule alors encore du lemme de Nakayama gradu\'e.
\end{proof}

\bigskip

\begin{exemple}\label{qpe}
Supposons ici que $n=\dim_\kb(V) = 1$ et posons $d=|W|$. 
Soit $y \in V \setminus\{0\}$ et soit $x \in V^*$ 
tel que $\langle y,x\rangle = 1$. Alors $P_\bullet=\kb[x^d,y^d]$, 
$\euler_0=xy$ et il est facile de v\'erifier que $Z_0=\kb[x^d,y^d,xy]$, 
c'est-\`a-dire $Z_0=P_\bullet[\euler_0]$. Il d\'ecoule donc du lemme 
\ref{QQ} que $Z=P[\euler]$. 

Cela peut se montrer d'une autre mani\`ere. 
En effet, $\Irr(W)=\{\e^i~|~0 \le i \le d-1\}$ et $f_{\e^i}(\tb)=\tb^i$ pour $0 \le i \le d-1$. 
Par cons\'equent, la formule~\ref{hilbert Q fantome} 
montre que 
$$\dim_\kb^\bigrad(Z) = \frac{1+(\tb\ub)+\cdots + (\tb\ub)^{d-1}}{(1-\tb\ub)^{d-1}~
(1-\tb^d)~(1-\ub^d)}$$
tandis que, puisque $P[\euler]=P \oplus P \euler \oplus \cdots \oplus P \euler^{d-1}$ 
(d'apr\`es~(\ref{eq:minimal-euler})), 
on a 
$$\dim_\kb^\bigrad(P[\euler])=\frac{1+(\tb\ub)+\cdots + (\tb\ub)^{d-1}}{(1-\tb\ub)^{d-1}~
(1-\tb^d)~(1-\ub^d)}.$$
Ainsi, $\dim_\kb^\bigrad(P[\euler])=\dim_\kb^\bigrad(Z)$, ce qui force l'\'egalit\'e 
$Z=P[\euler]$.\finl
\end{exemple}

\bigskip

En fait, il n'arrive pratiquement jamais que $Z=P[\euler]$, comme le montre la proposition 
suivante (dont nous ne nous servirons pas, mais dont nous int\'egrons 
la preuve pour montrer l'efficacit\'e des calculs de s\'eries de Hilbert)~:

\bigskip

\begin{prop}\label{Q=Pe}
On a $Z=P[\euler]$ si et seulement si $W$ est engendr\'e par 
une seule r\'eflexion.
\end{prop}

\begin{proof}
Si $W$ est engendr\'e par une seule r\'eflexion, alors un argument imm\'ediat 
permet de se ramener au cas o\`u $\dim_\kb V=1$. Dans ce cas, l'exemple 
\ref{qpe} montre que $Z=P[\euler]$. 

\medskip

R\'eciproquement, si $Z=P[\euler]$, alors
$$Z=\mathop{\bigoplus}_{i=0}^{|W|-1} P \euler^i$$
puisque le polyn\^ome minimal de $\euler$ sur $P$ 
est de degr\'e $|W|$ (d'apr\`es~(\ref{eq:minimal-euler})). 
On en d\'eduit, \`a l'aide de~(\ref{hilbert P}) que 
$$\dim_\kb^\bigrad(Z)=\frac{\DS{\sum_{j=0}^{|W|-1} 
(\tb\ub)^j}}{(1-\tb\ub)^{|\refw|} ~\DS{\prod_{i=1}^n (1-\tb^{d_i})(1-\ub^{d_i})}}.$$
Il d\'ecoule donc de~(\ref{hilbert Q molien}) que
$$\frac{1}{|W|}\sum_{w \in W} \frac{(1-\tb)^n}{\det(1-w\tb)~\det(1-w^{-1}\ub)} = 
\frac{\DS{\sum_{j=0}^{|W|-1}(\tb\ub)^j}}{\DS{\prod_{i=1}^n
(1+\tb+\cdots+\tb^{d_i-1})(1-\ub^{d_i})}}.$$
En sp\'ecialisant $\tb \mapsto 1$ dans cette \'egalit\'e, 
le terme de gauche ne contribue que lorsque $w=1$ et on obtient 
(car $|W|=d_1\cdots d_n$ d'apr\`es le th\'eor\`eme~\ref{chevalley}(a)) 
$$\frac{1}{(1-\ub)^n} = \frac{\DS{\sum_{j=0}^{|W|-1}\ub^j}}{\DS{\prod_{i=1}^n
(1-\ub^{d_i})}},$$
ou encore
$$\prod_{i=1}^n (1+\ub + \cdots + \ub^{d_i-1}) = \sum_{j=0}^{|W|-1}\ub^j.$$
Par comparaison des degr\'es, on obtient
$$|W|-1 = \sum_{i=1}^n (d_i-1).$$
Mais, toujours d'apr\`es le th\'eor\`eme~\ref{chevalley}(a), on 
a $|\Ref(W)|=\sum_{i=1}^n (d_i-1)$, ce qui montre que 
$$\Ref(W)=W\setminus\{1\}.$$
Par suite, si $w$, $w' \in W$, alors $ww'w^{-1}w^{\prime -1}$ est de d\'eterminant $1$, 
donc ne peut \^etre une r\'eflexion. Donc $ww'=w'w$, c'est-\`a-dire que $W$ est ab\'elien, 
donc diagonalisable. La proposition en d\'ecoule.
\end{proof}

\bigskip

\section{Sp\'ecificit\'e des groupes de Coxeter}\label{section:coxeter-h}

\cbstart

\boitegrise{{\bf Hypoth\`ese.} 
{\it Dans cette section, et seulement dans cette section, nous supposons 
que $W$ est un groupe de Coxeter, et nous reprenons les notations du 
chapitre~\ref{chapter:coxeter}.}}{0.75\textwidth}

\bigskip

En lien avec les probl\`emes \'evoqu\'es dans ce chapitre, une des particularit\'es de 
cette situation est que l'alg\`ebre $\Hb$ admet un autre automorphisme $\s_\Hb$, \indexnot{sz}{\s_\Hb}  
induit par l'isomorphisme de $W$-modules $\s : V \longiso V^*$. C'est la r\'eduction 
modulo $T$ de l'automorphisme $\s_\Hbt$ de $\Hbt$ d\'efini dans la 
section~\ref{section:coxeter-htilde}. Les propositions~\ref{prop:auto-coxeter-1} 
et~\ref{prop-bis:auto-coxeter-1} deviennent alors~:

\bigskip

\begin{prop}\label{prop:auto-coxeter-0}
Il existe un unique automorphisme $\s_\Hb$ de $\Hb$ tel que 
$$
\begin{cases}
\s_\Hb(y)=\s(y) & \text{si $y \in V$,}\\
\s_\Hb(x)=-\s^{-1}(x) & \text{si $x \in V^*$,}\\
\s_\Hb(w)=w & \text{si $w \in W$,}\\
\s_\Hb(C)=C & \text{si $C \in \CCB^*$.}\\
\end{cases}
$$
\end{prop}

\bigskip

\begin{prop}\label{prop-bis:auto-coxeter-0}
De plus,
\begin{itemize}
\itemth{a} $\s_\Hb$ stabilise les sous-alg\`ebres $\kb[\CCB]$ et $\kb W$ et \'echange 
les sous-alg\`ebres $\kb[V]$ et $\kb[V^*]$.

\itemth{b} Si $h \in \Hb^{\NM \times \NM}[i,j]$, 
alors $\s_\Hb(h) \in \Hb^{\NM \times \NM}[j,i]$.

\itemth{c} Si $h \in \Hb^\NM[i]$ (respectivement $h \in \Hb^\ZM[i]$), 
alors $\s_\Hb(h) \in \Hb^\NM[i]$ (respectivement $h \in \Hb^\ZM[-i]$).

\itemth{d} $\s_\Hb$ commute \`a l'action de $W^\wedge$ sur $\Hb$.

\itemth{e} Si $c \in \CCB$, alors $\s_\Hb$ induit un automorphisme 
de $\Hb_{c}$, toujours not\'e $\s_\Hb$ (ou $\s_{\Hb_{c}}$) si n\'ecessaire).

\itemth{f} $\s_\Hb(\euler)=- \euler$.
\end{itemize}
\end{prop}

De m\^eme, il existe une action de $\Sb\Lb_2(\kb)$ sur $\Hb$, 
qui s'obtient par r\'eduction modulo $T$ de l'action sur $\Hbt$ d\'efinie 
dans la remarque~\ref{rem:sl2-1}.

\cbend

\bigskip

\part{L'extension $Z/P$}\label{part:extension}

\boitegrise{{\bf Notation importante.} 
{\it 
Tout au long de ce m\'emoire, nous fixons une copie $Q$ \indexnot{Q}{Q} de la $P$-alg\`ebre $Z$, 
ainsi qu'un isomorphisme de $P$-alg\`ebres $\copie : Z \longiso Q$. \indexnot{ca}{\copie}  Cela signifie que $P$ 
sera vu comme sous-$\kb$-alg\`ebre de $Z$ et $Q$ simultan\'ement, mais que 
$Q$ et $Z$ seront consid\'er\'ees comme diff\'erentes.\\ 
\hphantom{A} Nous notons alors $\Kb=\Frac(P)$ \indexnot{K}{\Kb}  et $\Lb=\Frac(Q)$  \indexnot{L}{\Lb} 
et nous fixons une {\bfit cl\^oture galoisienne} $\Mb$ \indexnot{M}{\Mb}  de l'extension $\Lb/\Kb$. 
Posons $G=\Gal(\Mb/\Kb)$ \indexnot{G}{G}  et $H=\Gal(\Mb/\Lb)$.  \indexnot{H}{H} 
Nous notons $R$  \indexnot{R}{R}  la cl\^oture int\'egrale de $P$ dans $\Mb$. On a ainsi 
$P \subset Q \subset R$ et, d'apr\`es le corollaire~\ref{coro:endo-bi}, 
$Q=R^H$ et $P=R^G$. On peut alors utiliser les r\'esultats de l'appendice~\ref{chapter:galois-rappels}.\\ 
\hphantom{A} Rappelons que $\Kb Z=\Kb \otimes_P Z$ est le corps des fractions de l'alg\`ebre $Z$ (voir~(\ref{eq:KZ-fraction})). 
Nous noterons encore $\copie : \Frac(Z) \longiso \Lb$ l'extension de $\copie$ aux corps 
des fractions.\\
\hphantom{A} Notons $Z^{\NM \times \NM}$, $Z^\NM$ et $Z^\ZM$ respectivement la $(\NM \times \NM)$-graduation, 
la $\NM$-graduation, la $\ZM$-graduation induite par celle correspondante de $\Hbt$ (voir~\S\ref{section:graduation-1}, 
et les exemples~\ref{N graduation-1} et~\ref{Z graduation-1}). Par transport \`a travers 
$\copie$, on obtient des graduations 
$Q^{\NM \times \NM}$, $Q^\NM$ et $Q^\ZM$ sur $Q$.
}}{0.75\textwidth}

\bigskip
\medskip

L'objet essentiel de ce m\'emoire est l'\'etude de cette extension 
galoisienne, notamment des groupes d'inertie des id\'eaux premiers de $R$ 
en lien avec l'\'etude des repr\'esentations de $\Hb$. Nous commen\c{c}ons cette 
\'etude dans ce chapitre.

\chapter{Th\'eorie de Galois}\label{chapter:galois-CM}

%
%
%

\section{Action de $G$ sur l'ensemble $W$}

\medskip

Puisque $Q$ est un $P$-module libre de rang $|W|$, l'extension de corps $\Lb/\Kb$ est 
de degr\'e $|W|$~:
\equat\label{degre lk}
[\Lb : \Kb] = |W|.
\endequat
Rappelons que le fait que $\Mb$ soit la cl\^oture galoisienne de $\Lb/\Kb$ 
implique que
\equat\label{intersection}
\bigcap_{g \in G} \lexp{g}{H}=1.
\endequat
Il r\'esulte de~(\ref{degre lk}) que
\equat\label{ghw}
|G/H|=|W|.
\endequat
Cette \'egalit\'e \'etablit un premier lien entre le 
couple $(G,H)$ et le groupe $W$. Nous allons ici construire une bijection de nature 
galoisienne (d\'ependant de choix) entre $G/H$ et $W$. 

\medskip

\subsection{Sp\'ecialisation}\label{subsection:specialisation galois}
Fixons $c \in \CCB$. Rappelons que $\CG_c$ est l'id\'eal maximal de $\kb[\CCB]$ 
form\'e des fonctions qui s'annulent en $c$. On pose
$$\pG_c=\CG_c P\qquad\text{et}\qquad \qG_c = \CG_c Q=\pG_c Q.\indexnot{pa}{\pG_c}\indexnot{qa}{\qG_c}
\indexnot{pa}{\pG_c}\indexnot{qa}{\qG_c}$$
Puisque $P_\bullet=P/\pG_c \simeq \kb[V]^W \otimes \kb[V^*]^W$ \indexnot{P}{P_\bullet}  est int\`egre 
et $Q_c=Q/\qG_c$ \indexnot{Q}{Q_c}  l'est aussi (voir le corollaire~\ref{coro:endo-bi-c}(d)), 
on en d\'eduit que $\pG_c$ et $\qG_c$ sont des id\'eaux premiers de $P$ et $Q$ respectivement. 
Fixons un id\'eal premier $\rG_c$ \indexnot{ra}{\rG_c}  de $R$ au-dessus de $\pG_c$ et notons 
$R_c=R/\rG_c$. \indexnot{R}{R_c}  Notons alors $D_c$ \indexnot{D}{D_c}  
(respectivement $I_c$) \indexnot{I}{I_c}  le groupe de d\'ecomposition 
(respectivement d'inertie) $G_{\rG_c}^D$ (respectivement $G_{\rG_c}^I$). 
Posons
$$\Kb_c=\Frac(P_\bullet),\qquad \Lb_c=\Frac(Q_c)\qquad\text{et}\qquad \Mb_c=\Frac(R_c).
\indexnot{K}{\Kb_c}\indexnot{L}{\Lb_c}\indexnot{M}{\Mb_c}$$
En d'autres termes, $\Kb_c=k_P(\pG_c)$, $\Lb_c=k_Q(\qG_c)$ et 
$\Mb_c=k_R(\rG_c)$.

\bigskip

\begin{rema}\label{premier probleme} 
Ici, le choix de l'id\'eal premier $\rG_c$ n'est pas anodin. Ce genre de probl\`eme 
nous poursuivra tout au long de ce m\'emoire.\finl
\end{rema}

\bigskip

Puisque $\qG_c=\pG_c Q$ est premier, on obtient que 
$\Upsilon^{-1}(\pG_c)$ est r\'eduit \`a un \'el\'ement (ici, $\qG_c$) ce qui 
permet de d\'eduire de la proposition~\ref{reduction} que 
\equat\label{G=HD}
G=H \cdot D_c = D_c \cdot H.
\endequat
On obtient aussi que $Q$ est nette sur $P$ en 
$\qG_c$ (par d\'efinition). 
Le th\'eor\`eme~\ref{raynaud} nous dit alors que $I_c \subset H$. 
Puisque $I_c$ est distingu\'e dans $D_c$, on d\'eduit de~(\ref{intersection}) et~(\ref{G=HD}) que 
$I_c \subset \bigcap_{d \in D_c} \lexp{d}{H} = \bigcap_{g \in G} \lexp{g}{H}=1$ et donc 
\equat\label{eq:net}
I_c=1.
\endequat
Il d\'ecoule alors de la proposition~\ref{cloture galoisienne} 
que 
\equat\label{cloture Lc}
\text{\it $\Mb_c$ est la cl\^oture galoisienne de l'extension $\Lb_c/\Kb_c$.}
\endequat
Pour finir, d'apr\`es~(\ref{eq:net}) et le th\'eor\`eme~\ref{bourbaki}, on a 
\equat\label{gal Dc}
\Gal(\Mb_c/\Kb_c) = D_c\qquad\text{et}\qquad \Gal(\Mb_c/\Lb_c)= D_c \cap H.
\endequat
Nous noterons $\copie_c : Z_c \to Q_c$ \indexnot{ca}{\copie_c}  la sp\'ecialisation de $\copie$ en $c$ 
et nous noterons encore $\copie_c : \Frac(Z_c) \to \Lb_c$ l'extension de $\copie_c$ 
aux corps de fractions. 

\bigskip

\begin{rema}\label{rema:impossible}
Nous allons dans la section suivante \'etudier le cas o\`u $c=0$, et obtenir alors une description 
explicite de $D_0$. En revanche, donner une description explicite de $D_c$ en g\'en\'eral se r\'ev\`ele 
\^etre une t\^ache insurmontable, comme le montreront les quelques exemples trait\'es dans le 
chapitre~\ref{chapitre:rang 1} (qui \'etudie le cas o\`u $\dim_\kb(V)=1$)~: 
voir~\S\ref{rema:dc-cyclique}.\finl
\end{rema}

\bigskip

\subsection{Sp\'ecialisation en $0$}\label{subsection:specialisation galois 0} 
Rappelons que $P_0=P_\bullet = \kb[V]^W \otimes \kb[V^*]^W$ et 
$Q_0 \simeq Z_0=\Zrm(\Hb_0) \simeq \kb[V \times V^*]^{\D W}$ (voir l'exemple~\ref{exemple zero-0}), 
o\`u $\D : W \to W \times W$, $w \mapsto (w,w)$ est le morphisme diagonal.  \indexnot{dz}{\D}  
Ainsi, 
$$\Kb_0=\kb(V \times V^*)^{W \times W}\qquad\text{et}\qquad
\Lb_0 = \kb(V \times V^*)^{\D W},$$
D'autre part, l'extension $\kb(V \times V^*)/\Kb_0$ est galoisienne de groupe 
$W \times W$ tandis que l'extension $\kb(V \times V^*)/\Lb_0$ est galoisienne 
de groupe $\D W$. Puisque $\D \Zrm(W)$ est le plus grand sous-groupe distingu\'e 
de $W \times W$ contenu dans $\D W$, il d\'ecoule de~(\ref{cloture Lc}) que 
$$\Mb_0 \simeq \kb(V \times V^*)^{\D \Zrm(W)}.$$

\medskip

\boitegrise{
{\bf Choix fondamental.} 
 {\it 
Nous fixons une fois pour toutes un id\'eal premier $\rG_0$ de $R$ 
au-dessus de $\qG_0=\CG_0 Q$ ainsi qu'un isomorphisme de corps 
$$\iso : \kb(V \times V^*)^{\D \Zrm(W)} \stackrel{\sim}{\longto} \Mb_0 \indexnot{ia}{\iso}$$
dont la restriction \`a $\kb(V \times V^*)^{\D W}$ est l'isomorphisme canonique 
$\kb(V \times V^*)^{\D W} \longiso \Frac(Z_0) \longiso \Lb_0$. Ici, 
l'isomorphisme $\Frac(Z_0) \longiso \Lb_0$ est $\copie_0$.\\
~\\
{\bf Convention.} 
L'action du groupe $W \times W$ sur le corps $\kb(V \times V^*)$ s'effectue ainsi~: 
$V \times V^*$ sera vu comme un sous-espace vectoriel de $\kb(V \times V^*)$ et engendrant ce 
corps, et l'action de $(w_1,w_2)$ envoie $(y,x) \in V \times V^*$ sur $(w_1(y),w_2(x))$.
}}{0.75\textwidth}

\bigskip

\begin{rema}\label{rema:action-dif}
L'action de $W \times W$ sur $\kb(V \times V^*)$ d\'ecrite ci-dessus 
n'est pas celle obtenue en faisant agir $W \times W$ sur 
la vari\'et\'e $V \times V^*$ et induisant une action par 
pr\'ecomposition sur le corps des fonctions 
$\kb(V \times V^*)$~: on passe de l'une \`a l'autre \`a travers 
l'isomorphisme $W \times W \longiso W \times W$, $(w_1,w_2) \mapsto (w_2,w_1)$. 
Cette distinction aura son importance (voir la remarque~\ref{rema:w-w}).\finl
\end{rema}

\medskip

Ces choix \'etant effectu\'es, on obtient un isomorphisme  canonique 
$\Gal(\Mb_0/\Kb_0) \longiso (W \times W)/\D \Zrm(W)$ induisant 
un isomorphisme $\Gal(\Mb_0/\Lb_0) \longiso \D W /\D \Zrm(W)$. 
Puisque $D_0 = \Gal(\Mb_0/\Kb_0)$ d'apr\`es~(\ref{gal Dc}),  
on obtient un morphisme de groupes 
$$\iota : W \times W \longto G\indexnot{iz}{\iota}$$
v\'erifiant les propri\'et\'es suivantes~:

\bigskip

\begin{prop}\label{WW}
\begin{itemize}
\itemth{a} $\Ker \iota = \D \Zrm(W)$.

\itemth{b} $\im \iota = D_0$.

\itemth{c} $\iota^{-1}(H) = \D W$.
\end{itemize}
\end{prop}

\bigskip

En vertu de~(\ref{G=HD}), cette proposition fournit une bijection 
\equat\label{bij W}
(W \times W)/\D W \stackrel{\sim}{\longleftrightarrow} G/H.
\endequat
On peut bien s\^ur mettre $(W \times W)/\D W$ en bijection avec $W$ par injection \`a gauche ou \`a droite. 
Nous en fixons une~: 

\medskip

\boitegrise{{\bf Identification.} {\it Le morphisme $W \to W \times W$, $w \mapsto (w,1)$ 
compos\'e avec le morphisme $\iota : W \times W \to G$ est injectif, et nous identifierons 
le groupe $W$ avec son image dans $G$.}}{0.75\textwidth}

\medskip

Ainsi, $w \in W \subset G$ est l'unique 
automorphisme de $P$-alg\`ebre de $R$ tel que, pour tout $r \in R$, 
\equat\label{definition W}
\bigl(w(r) \mod \rG_0\bigr) = (w,1)
\bigl(r \mod \rG_0\bigr)\quad\text{dans $\kb(V \times V^*)^{\D \Zrm(W)}$}.
\endequat
On a alors, d'apr\`es~\ref{bij W},
\equat\label{G=HW}
G=H \cdot W=W \cdot H\qquad\text{et}\qquad H \cap W = 1.
\endequat

\bigskip

\begin{coro}\label{Dc W}
Pour tout $c \in \CCB$, l'application naturelle $D_c \to G/H \stackrel{\sim}{\to} W$ 
induit une bijection $D_c/(D_c \cap H) \stackrel{\sim}{\longto} W$.
\end{coro}

\begin{proof}
Cela d\'ecoule de~(\ref{G=HD}) de~(\ref{G=HW}).
\end{proof}

\bigskip

\subsection{Action de $G$ sur $W$}\label{subsection:action G} 
Notons maintenant $\SG_W$ le groupe des bijections de l'ensemble $W$. 
Nous identifierons le groupe $\SG_{W\setminus\{1\}}$ des bijections de l'ensemble 
$W \setminus \{1\}$ avec le stabilisateur de $1$ dans $\SG_W$. 
L'identification $G/H \stackrel{\sim}{\longleftrightarrow} W$ et 
l'action de $G$ par translation \`a gauche sur $G/H$ permet d'identifier 
$G$ \`a un sous-groupe du groupe $\SG_W$ des bijections de l'ensemble fini $W$. 
Ainsi, 
\equat\label{GHW}
G \subseteq \SG_W\qquad\text{et}\qquad H = G \cap \SG_{W \setminus \{1\}}.
\endequat
Si $g \in G$ et $w \in W$, nous noterons $g(w)$ l'unique \'el\'ement de $W$ 
tel que $g\iota(w,1) H = \iota(1,g(w))H$. \`A travers cette identification de $G$ 
comme sous-groupe de $\SG_W$, l'application $\iota : W \times W \to G$ se d\'ecrit ainsi~: 
si $(w_1,w_2) \in W \times W$ et $w \in W$, alors 
\equat\label{iota concret}
\iota(w_1,w_2)(w)=w_1 w w_2^{-1}.
\endequat
C'est l'action naturelle de $W \times W$ sur l'ensemble $W$ par translation 
\`a gauche et \`a droite. Puisque $\D W$ est le stabilisateur de $1 \in W$ 
pour cette action, on obtient
\equat\label{DW S}
\iota(\D W) = \iota(W \times W) \cap \SG_{W \setminus \{1\}},
\endequat
ce qui est bien s\^ur compatible avec la proposition~\ref{WW}(c) 
et~(\ref{GHW}).

Pour finir, le choix de l'injection de $W$ dans $G$ \`a travers $w \mapsto \iota(w,1)$ 
revient \`a identifier $W$ comme sous-groupe de $\SG_W$ \`a travers l'action 
sur lui-m\^eme par translation \`a gauche.

\bigskip

\subsection{\'El\'ement d'Euler et groupe de Galois}\label{section:galois euler} 

Posons $\eulerq=\copie(\euler)$. \indexnot{ea}{\eulerq}  
Puisque le polyn\^ome minimal de $\eulerq$ sur $P$ (donc sur $K$) est de degr\'e $|W|$ 
(voir~(\ref{eq:minimal-euler})), on d\'eduit de~(\ref{degre lk}) que 
\equat\label{euler engendre}
\Lb=\Kb[\eulerq].
\endequat
%
%
%
Le calcul du groupe de Galois $G=\Gal(\Mb/\Kb)$ est donc \'equivalent 
au calcul du groupe de Galois du polyn\^ome minimal de $\euler$ (ou $\eulerq$). 
Les m\'ethodes classiques (r\'eduction modulo un id\'eal premier notamment, 
voir par exemple la section~\ref{sec:calcul}) 
nous serons utiles dans de petits exemples.

Revenons au calcul de l'injection $W \injto G \subseteq \SG_W$. 
Si $w \in W$, posons $\eulerq_w=w(\eulerq) \in \Mb$. \indexnot{ea}{\eulerq_w}  
Rappelons que $G$ est vu comme un sous-groupe de $\SG_W$ 
(voir~(\ref{GHW}))~: si $g \in G$ et $w \in W$, alors $g(w)$ est d\'efini par l'\'egalit\'e $g(w)H=gwH$. 
Puisque $H$ agit trivialement sur $\eulerq$, on en d\'eduit que 
\equat\label{action G W euler}
g(\eulerq_w) = \eulerq_{g(w)}
\endequat
et donc, si $(w_1,w_2) \in W \times W$, 
\equat\label{action WW euler}
\iota(w_1,w_2)(\eulerq_w)=\eulerq_{w_1ww_2^{-1}}.
\endequat
Cela \'etend l'\'egalit\'e
\equat\label{action W euler}
w_1(\eulerq_w)=\eulerq_{w_1w}
\endequat
qui est une cons\'equence imm\'ediate de la d\'efinition de $\eulerq_w$. 
En particulier, d'apr\`es~(\ref{G=HW}), 
\equat\label{euler minimal}
\text{\it Le polyn\^ome minimal de $\eulerq$ sur $P$ est } 
\prod_{w \in W} (\tb-\eulerq_w).
\endequat
Notons aussi que, d'apr\`es~(\ref{definition W}) et la convention choisie pour l'action de 
$W \times W$ sur $\kb(V \times V^*)$, on a 
$$\iso^{-1}(\eulerq_w \mod \rG_0) = \sum_{i=1}^n w(y_i) x_i \in \kb[V \times V^*]^{\D \Zrm(W)}.$$
%

\bigskip

\section{D\'eploiement de $\Kb\Hb$}\label{section:deploiement}

\medskip

Rappelons que, d'apr\`es le th\'eor\`eme~\ref{theo:KH-mat}, l'\'equivalence de Morita 
induite par $\Kb \Hb e$ montre l'existence d'un isomorphisme 
$$\Kb\Hb \simeq \Mat_{|W|}(\Kb Z).$$ 
Rappelons aussi que $\Kb Z$ est le corps des fractions de $Z$ (voir~(\ref{eq:KZ-fraction})) 
et que $\copie : \Kb Z \longiso \Lb$ d\'esigne encore l'extension de 
$\copie : Z \longiso Q$. 
En particulier, $\Kb\Hb$ est semi-simple, mais n'est pas $\Kb$-d\'eploy\'ee en g\'en\'eral.

Si $g \in G$, le morphisme 
$\Kb Z \to \Mb$, $z \mapsto g(\copie(z))$ obtenu par restriction de $g$ \`a $\Lb$ 
(\`a travers l'isomorphisme $\copie$) 
est $\Kb$-lin\'eaire et il s'\'etend en un unique morphisme de $\Mb$-alg\`ebres 
$$\fonction{g_Z}{\Mb \otimes_\Kb \Kb Z}{\Mb}{m \otimes_\Kb z}{mg(\copie(z)).}$$
Bien s\^ur, $g_Z=(gh)_Z$ pour tout $h \in H$ et c'est un fait classique 
(voir la proposition~\ref{iso galois}) que 
$$(g_Z)_{gH \in G/H} : \Mb \otimes_\Kb \Kb Z \longto \prod_{gH \in G/H} \Mb$$
est un isomorphisme de $\Mb$-alg\`ebres. Compte tenu de~(\ref{G=HW}), ceci peut se 
r\'e\'ecrire ainsi~: on a un isomorphisme de $\Mb$-alg\`ebres
\equat\label{W M}
\bijectio{\Mb \otimes_\Kb \Kb Z}{\prod_{w \in W} \Mb}{x}{(w_Z(x))_{w \in W}}.
\endequat

Fixons une $\Kb Z$-base ordonn\'ee $\BC$ de $\Kb\Hb e$~: rappelons que $|\BC|=|W|$. Ce choix 
fournit un isomorphisme d'alg\`ebres 
$$\r^\BC : \Kb\Hb \stackrel{\sim}{\longto} \Mat_{|W|}(\Kb Z).\indexnot{rz}{\r^\BC}$$
Maintenant, si $w \in W$, notons $\r_w^\BC$ \indexnot{rz}{\r_w^\BC}   le morphisme de $\Mb$-alg\`ebres 
$\Mb\Hb \longto \Mat_{|W|}(\Mb)$ d\'efini par
$$\r_w^\BC(m \otimes_P h) = m \cdot  w(\copie(\r^\BC(h)))$$
pour tous $m \in \Mb$ et $h \in \Hb$. Alors $\r_w^\BC$ est une repr\'esentation 
irr\'eductible de $\Mb\Hb$~: via l'\'equivalence de Morita de la proposition~\ref{prop:morita}, 
elle correspond au $\Mb \otimes_\Kb \Kb Z$-module simple d\'efini par $w$. 
On notera $\LC_w$ \indexnot{L}{\LC_w}  le $\Mb\Hb$-module simple dont le $\Mb$-espace sous-jacent est 
$\Mb^{|W|}$ et sur lequel l'action de $\Mb\Hb$ s'effectue via $\r_w^\BC$. 

Si on note $\Irr(\Mb\Hb)$ l'ensemble des classes d'isomorphie de 
$\Mb\Hb$-modules simples, alors on a une bijection 
\equat\label{bij irr}
\bijectio{W}{\Irr \Mb\Hb}{w}{\LC_w}
\endequat
est bijective et un isomorphisme de $\Mb$-alg\`ebres 
\equat\label{iso MH}
\prod_{w \in W} \r_w^\BC : \Mb\Hb \stackrel{\sim}{\longto} 
\prod_{w \in W} \Mat_{|W|}(\Mb).
\endequat
Par cons\'equent,  
\equat\label{deployee}
\text{\it la $\Mb$-alg\`ebre $\Mb\Hb$ est semi-simple d\'eploy\'ee.}
\endequat
De plus, la bijection~(\ref{bij irr}) nous permet d'identifier 
son groupe de Grothendieck $\groth(\Mb\Hb)$ avec le $\ZM$-module 
$\ZM W$~:
\equat\label{identification}
\groth(\Mb\Hb) \longiso \ZM W.
\endequat
Puisque l'alg\`ebre $\Mb\Hb$ est semi-simple d\'eploy\'ee, il d\'ecoule 
de~\cite[th\'eor\`eme~7.2.6]{geck} qu'il existe une unique famille d'\'el\'ements 
$(\schur_w)_{w \in W}$ \indexnot{sa}{\schur_w}  de $R$ telle que
$$\taub_{\!\SSS{\Mb\Hb}} = \sum_{w \in W} \frac{\carac_w}{\schur_w},\indexnot{tx}{\taub_{\!\SSS{\Mb\Hb}}}$$
o\`u $\carac_w : \Mb\Hb \to \Mb$ \indexnot{ca}{\carac_w}  d\'esigne le caract\`ere du $\Mb\Hb$-module 
simple $\LC_w$ et $\taub_{\!{\SSS{\Mb\Hb}}} : \Mb\Hb \to \Mb$ d\'esigne l'extension de la forme 
sym\'etrisante $\taub_\HHb : \Hb \to P$. L'\'el\'ement $\schur_w$ de $R$ est 
appel\'e l'{\it \'el\'ement de Schur} associ\'e au module simple $\LC_w$. 
D'apr\`es~\cite[th\'eor\`eme~7.2.6]{geck}, $\schur_w$ est 
\'egal au scalaire par lequel l'\'el\'ement de Casimir $\casimir_\Hb \in Z$ d\'efini dans 
la section~\ref{sub:symetrisante} agit sur le module simple $\LC_w$. Ainsi,
\equat\label{eq:formule-schur}
\schur_w = w(\copie(\casimir_\Hb)).
\endequat

\bigskip

\begin{rema}
Dans la th\'eorie g\'en\'erale des alg\`ebres sym\'etriques, 
l'\'el\'ement de Schur $\schur_w$ est un invariant important, qui peut \^etre tr\`es utile 
pour d\'eterminer les blocs d'une r\'eduction de $R\Hb$ modulo un id\'eal premier de $R$. 
Ici, la formule~(\ref{eq:formule-schur}) montre que ce calcul est \'equivalent \`a la r\'esolution des 
deux probl\`emes suivants~:
\begin{itemize}
\itemth{1} Calculer l'\'el\'ement de Casimir $\casimir_\Hb$.

\itemth{2} Comprendre l'action de $W$ (ou de $G$) sur l'image de $\casimir_\Hb$ 
dans $Q \subset R$. 
\end{itemize}
Si le probl\`eme (1) semble attaquable (et sa solution serait int\'eressante car 
elle fournirait, apr\`es l'\'el\'ement d'Euler, un nouvel \'el\'ement du centre $Z$ de $\Hb$), 
il est en revanche peu probable que l'on puisse obtenir des informations pr\'ecises 
concernant le probl\`eme (2), tant les calculs de l'anneau $R$ et du groupe de Galois 
$G$ semblent hors de port\'ee. Tout au plus peut-on esp\'erer pour le moment comprendre 
leur r\'eduction modulo $\rG_0$ par exemple, ou d'autres id\'eaux premiers 
bien choisis.
\end{rema}

\bigskip

\section{Graduations sur $R$}\label{section:graduation R}

\medskip

La $(\NM \times \NM)$-graduation sur $Z \simeq Q$ induit deux $\NM$-graduations, 
associ\'ees aux morphismes de mono\"{\i}des $\ph_1$ et $\ph_2$ 
d\'efinis par $\ph_m : \NM \times \NM \longto \NM$, $(i_1,i_2) \mapsto i_m$. 
Nous noterons $Q=\bigoplus_{i \ge 0} Q^{\ph_m}[i]$ la graduation associ\'ee 
\`a $\ph_m$. Notons $\r_1 : \kb^\times \to \Aut_{\kb\text{-}\alg}(Q)$, $\xi \mapsto \gradauto_{\xi,1}$ 
et $\r_2 : \kb^\times \to \Aut_{\kb\text{-}\alg}(Q)$, $\xi \mapsto \gradauto_{1,\xi}$ 
les morphismes associ\'es respectivement aux graduations $Q^{\ph_1}$ et $Q^{\ph_2}$. 

D'apr\`es la proposition~\ref{R gradue}, la graduation $Q^{\ph_m}$ s'\'etend 
en une graduation $R=\bigoplus_{i \ge 0} R^{(m)}[i]$ \`a laquelle correspond 
un morphisme de groupes $\rhot_m : \kb^\times \to \Aut_{\kb\text{-}\alg}(R)$ tel 
que $\rhot_m(\xi)(q)=\r_m(\xi)(q)$ pour tous $\xi \in \kb^\times$ et 
$q \in Q$. Maintenant, si $\xi \in \kb^\times$, alors 
$\r_1(\xi)$ et $\r_2(\xi')$ commutent pour tout $\xi' \in \kb^\times$, 
donc $\rhot_1(\xi)$ est un automorphisme de $R$ qui stabilise $Q$ et respecte 
la graduation $Q^{\ph_2}$ de $Q$. D'apr\`es le corollaire~\ref{graduation et automorphisme}, 
$\rhot_1(\xi)$ stabilise la graduation $R^{(2)}$ de $R$, ce qui signifie 
que $\rhot_1(\xi)$ et $\rhot_2(\xi')$ commutent. En conclusion~:

\bigskip

\begin{prop}\label{bigraduation sur R}
Il existe une unique $(\NM \times \NM)$-graduation sur $R$ (nous la noterons 
$R=\bigoplus_{(i,j) \in \NM \times \NM} R^{\NM \times \NM}[i,j]$) \'etendant 
celle de $Q$. 

Elle induit un morphisme de groupes 
$\gradauto^R : \kb^\times \times \kb^\times \to \Aut_{\kb\text{-}\alg}(R)$, \indexnot{ba}{\gradauto^R}
$(\xi,\xi') \mapsto \rhot_1(\xi)\rhot_2(\xi')=\rhot_2(\xi')\rhot_1(\xi)$ 
tel que, pour tout $(\xi,\xi') \in \kb^\times \times \kb^\times$, 
$\gradauto_{\xi,\xi'}^R$ stabilise $Q$ et co\"{\i}ncide avec $\gradauto_{\xi,\xi'}$ 
sur $Q$.

Le groupe de Galois $G$ stabilise cette $(\NM \times \NM)$-graduation.
\end{prop}

\begin{proof}
Seules l'unicit\'e et la stabilit\'e sous l'action de $G$ n'ont pas \'et\'e 
montr\'ees. Elles d\'ecoulent respectivement de la proposition~\ref{unicite graduation} 
et du corollaire~\ref{graduation et automorphisme}.
\end{proof}

\bigskip

Nous noterons $R=\bigoplus_{(i,j) \in \NM \times \NM} R^{\NM \times \NM}[i,j]$ la $(\NM \times \NM)$-graduation 
\'etendant celle de $Q$. 
De m\^eme, $R=\bigoplus_{i \in \NM} R^\NM[i]$ (respectivement $R=\bigoplus_{i \in \ZM} R^\ZM[i]$) 
d\'esignera la $\NM$-graduation (respectivement $\ZM$-graduation) 
\'etendant celle de $Q$~: en d'autres termes, 
$$R^\NM[i]=\mathop{\bigoplus}_{i_1+i_2=i} R^{\NM \times \NM}[i_1,i_2] \quad\text{et}\quad 
R^\ZM[i]=\mathop{\bigoplus}_{i_1+i_2=i} 
R^{\NM \times \NM}[i_1,i_2].$$

\bigskip

\begin{coro}\label{ideaux homogenes}
L'id\'eal premier $\rG_0$ de $R$ choisi dans la sous-section~\ref{subsection:specialisation galois 0} 
est bi-homog\`ene (en particulier, il est homog\`ene).
\end{coro}

\begin{proof}
Si $\ph : \NM \times \NM \to \ZM$ est un morphisme de mono\"{\i}des, 
nous noterons $R^\ph$ la $\ZM$-graduation associ\'ee. 
Puisque $\rG_0 \cap Q=\qG_0$ est bi-homog\`ene, 
il d\'ecoule du corollaire~\ref{premier homogene} que $\rG_0$ est $R^\ph$-homog\`ene 
pour tout morphisme de mono\"{\i}de $\ph : \NM \times \NM \to \ZM$. 
Cela force $\rG_0$ \`a \^etre bi-homog\`ene.
\end{proof}

\bigskip

La sous-alg\`ebre $R^{\NM \times \NM}[0,0]$ de $R$ \'etant int\`egre et enti\`ere et de type fini 
sur $P^{\NM \times \NM}[0,0]=\kb$, c'est une extension finie de $\kb$. 

\bigskip

\begin{coro}\label{decomposition naturelle}
Si $R_+$ \indexnot{R}{R_+}  d\'esigne l'unique id\'eal bi-homog\`ene maximal de $R$, 
alors $G$ stabilise $R_+$ (ce qui signifie que le groupe de d\'ecomposition  
de $R_+$ dans $G$ est $G$ lui-m\^eme).
\end{coro}

\begin{proof}
Cela d\'ecoule de la proposition~\ref{bigraduation sur R}.
\end{proof}

\bigskip

\begin{coro}\label{r0}
On a $R^{\NM \times \NM}[0,0] = \kb$, .
\end{coro}

\begin{proof}
D'apr\`es~\ref{ideaux homogenes}, on a $\rG_0 \subset R_+$. Par cons\'equent, 
$R^{\NM \times \NM}[0,0]$ est isomorphe \`a la composante homog\`ene de bi-degr\'e $(0,0)$ 
de $R/\rG_0$. Mais $k_R(\rG_0) \simeq \kb(V \times V^*)^{\D\Zrm(W)}$ 
et $R/\rG_0$ est entier sur $Q_0=\kb[V \times V^*]^{\D W}$, donc 
$R/\rG_0 \subset \kb[V \times V^*]^{\D\Zrm(W)}$, et cette inclusion 
respecte la bi-graduation, par unicit\'e de 
la bi-graduation sur $R/\rG_0$ \'etendant celle de $Q_0=\kb[V \times V^*]^{\D W}$ 
(voir la proposition~\ref{unicite graduation}). D'o\`u le r\'esultat.
\end{proof}

\bigskip

\begin{coro}\label{r0 DI}
Soit $D_+$  \indexnot{D}{D_+}  (respectivement $I_+$)\indexnot{I}{I_+}   le groupe de d\'ecomposition 
(respectivement inertie) de $R_+$ dans $G$. Alors $D_+=I_+=G$.
\end{coro}

\begin{proof}
Notons $\pG_+=R_+ \cap P$. Alors $k_R(R_+)/k_P(\pG_+)$ est une extension 
galoisienne de groupe de Galois $D_+/I_+$ (voir le th\'eor\`eme~\ref{bourbaki}) et, 
d'apr\`es le corollaire~\ref{r0}, $k_R(R_+)=\kb=k_P(\pG_+)$, donc 
$D_+/I_+=1$. Puisque $D_+=G$ d'apr\`es le corollaire~\ref{decomposition naturelle}, 
le corollaire~\ref{r0 DI} est d\'emontr\'e.
\end{proof}

\bigskip

\section{Action sur $R$ des automorphismes naturels de $\Hb$}
\label{section:auto galois}

\medskip

La section pr\'ec\'edente~\ref{section:graduation R} \'etudiait l'extension \`a $R$ 
des automorphismes de $Q$ induits par $\kb^\times \times \kb^\times$. 
Dans la section~\ref{section:automorphismes-1}, nous avons aussi introduit une action de 
$W^\wedge \rtimes \NC$ sur $\Hb$ qui stabilisait $Z$ (forc\'ement), $P$ 
mais aussi $\pG_0$ et $\pG_0 Z$~: cette action se transporte \`a $Q \simeq Z$ et stabilise 
encore $\qG_0=\pG_0 Q$. Nous allons montrer comment \'etendre cette action 
\`a $R$, et en tirer les cons\'equences sur le groupe de Galois $G$. Pour cela, nous 
nous placerons dans un cadre relativement g\'en\'eral~:

\bigskip

\boitegrise{{\bf Hypoth\`ese.} {\it Dans cette section, et seulement dans cette section, 
nous fixons un groupe $\GC$ agissant \`a la fois sur $Z$ et sur $\kb[V \times V^*]$ et 
v\'erifiant les propri\'et\'es suivantes~:
\begin{itemize}
\itemth{1} $\GC$ stabilise $P$ et $\pG_0$.
\itemth{2} L'action de $\GC$ sur $\kb[V \times V^*]$ normalise l'action de $W \times W$ et celle 
de $\D W$.
\itemth{3} L'isomorphisme canonique de $\kb$-alg\`ebres $Z_0 \longiso \kb[V \times V^*]^{\D W}$ est $\GC$-\'equivariant.
\end{itemize}}\vskip-0.5cm
}{0.75\textwidth}

\bigskip

On transporte, \`a travers l'isomorphisme 
$\copie$, l'action de $\GC$ sur $Z$ en une action de $\GC$ sur $Q$. 
Si $\t \in \GC$, on note $\t^\circ$ l'automorphisme de $\kb[V \times V^*]$ 
induit par $\t$~: d'apr\`es (2), $\t^\circ$ stabilise $\kb[V \times V^*]^{\D\Zrm(W)}$, 
$\kb[V \times V^*]^{\D W}$ et $\kb[V \times V^*]^{W \times W}$. 

%
%
%
%
%
%
%
 \bigskip

\begin{prop}\label{extension tau}
Soit $\t \in \GC$. Alors il existe une unique extension $\taut$ de $\t$ \`a $R$ v\'erifiant les deux 
propri\'et\'es suivantes~:
\begin{itemize}
\itemth{1} $\taut(\rG_0)=\rG_0$~;

\itemth{2} L'automorphisme de $R/\rG_0$ induit par $\taut$ est \'egal \`a $\t^\circ$, 
via l'identification $\iso : \kb(V \times V^*)^{\D \Zrm(W)} \longiso \Mb_0$ de 
\S\ref{subsection:specialisation galois 0}. 
\end{itemize}
\end{prop}

\begin{proof}
Commen\c{c}ons par montrer l'existence. 
Tout d'abord, $\Mb$ \'etant une cl\^oture galoisienne de l'extension $\Lb/\Kb$, 
il existe une extension $\tau_\Mb$ de $\tau$ \`a $\Mb$. Puisque $R$ est la cl\^oture 
int\'egrale de $Q$ dans $\Mb$, $\t_\Mb$ stabilise $R$. De plus, puisque 
$\t(\qG_0)=\qG_0$, il existe $h \in H$ tel que $\t_\Mb(\rG_0)=h(\rG_0)$. 
Posons donc $\taut_\Mb=h^{-1} \circ \t_\Mb$. Ainsi
$$\taut_\Mb(\rG_0)=\rG_0\qquad\text{et}\qquad (\taut_\Mb)|_{\SSS{\Lb}} =\t.$$
Notons $\taut_{\Mb,0}$ l'automorphisme de $R/\rG_0$ induit par $\taut_\Mb$. 

Par construction, la restriction de $\taut_{\Mb,0}$ \`a $Q/\qG_0$ est \'egale 
\`a la restriction de $\iso \circ \t^\circ \circ \iso^{-1}$. 
Par suite, il existe $d \in D_0 \cap H$ tel que 
$\taut_{\Mb,0} = d \circ (\iso \circ \t_0 \circ \iso^{-1})$. 
On pose alors $\taut=d^{-1} \circ \taut_\Mb$~: il est clair que 
$\taut$ v\'erifie (1) et (2).

\medskip

Montrons maintenant l'unicit\'e. Si $\taut_1$ est une autre extension de $\t$ 
\`a $R$ satisfaisant (1) et (2), posons $\s=\taut^{-1}\taut_1(\rG_0)$. 
Alors $\s \in G$ et, d'apr\`es (1), $\s$ stabilise $\rG_0$. 
Donc $\s \in D_0$. De plus, d'apr\`es (2), $\s$ induit sur $R/\rG_0$ 
l'automorphisme identit\'e. Donc $\s\in I_0=1$. Donc $\taut=\taut_1$.
\end{proof}

\bigskip

L'existence et l'unicit\'e dans l'\'enonc\'e de la proposition pr\'ec\'edente~\ref{extension tau} 
impliquent le corollaire suivant~:

\bigskip

\begin{coro}\label{coro:extension-action}
L'action de $\GC$ sur $Q$ s'\'etend (de mani\`ere unique) en une action de $\GC$ sur $R$ 
qui stabilise $\rG_0$ et est compatible avec l'isomorphisme $\iso$.
\end{coro}

\bigskip

Par la suite, on notera encore $\t$ l'extension $\taut$ d\'efinie dans la proposition~\ref{extension tau}. 
Puisque $\GC$ stabilise $P$, $Q$, $\pG_0$, $\qG_0=\pG_0 Q$ et $\rG_0$, on obtient~:

\bigskip

\begin{coro}\label{coro:action-stabilise}
L'action de $\GC$ sur $R$ normalise $G$, $H$, $D_0=\iota(W \times W)$ et $D_0 \cap H=\iota(\D W)=W/\Zrm(W)$.
\end{coro}

\bigskip

Du corollaire~\ref{coro:action-stabilise}, on d\'eduit que $\GC$ agit sur l'ensemble $G/H \simeq W$ 
et que 
\equat\label{eq:action-normalise}
\text{\it l'image de $\GC$ dans $\SG_W$ normalise $G$.}
\endequat

\begin{exemple}\label{ex:action-k-k-h-n}
Le groupe $\GC=\kb^\times \times \kb^\times \times (W^\wedge \rtimes \NC)$ 
agit sur $\Hb$ et stabilise $P$ et $\pG_0$~; par les m\^emes formules, il agit sur 
$\kb[V \times V^*]$ en normalisant $W \times W$ et $\D W$ (en fait, 
$\kb^\times \times \kb^\times \times \Hom(W ,\kb^\times)$ commute avec $W \times W$ et 
seul $\NC$ agit non trivialement sur $W \times W$). 

Il d\'ecoule des r\'esultats pr\'ec\'edents que l'action de 
$\kb^\times \times \kb^\times \times  (W^\wedge \rtimes \NC)$ sur $Q$ 
s'\'etend de mani\`ere unique en une action sur $R$ qui stabilise $\rG_0$ et 
qui est compatible avec l'isomorphisme $\iso$. Par l'unicit\'e, l'extension de l'action de 
$\kb^\times \times \kb^\times \times W^\wedge$ \`a $R$ commute avec l'action de $G$ 
tandis que celle de $\NC$ rend le morphisme $G \injto \SG_W$ \'equivariant pour l'action de 
$\NC$. 

Pour finir, toujours par unicit\'e, l'extension de l'action du sous-groupe $\kb^\times \times \kb^\times$ 
correspond \`a l'extension \`a $R$ de la $(\NM \times \NM)$-graduation d\'ecrite 
dans la proposition~\ref{bigraduation sur R}.\finl
\end{exemple}

\section{Une situation particuli\`ere~: r\'eflexions d'ordre 2}\label{sec:w0}

\medskip

\boitegrise{\noindent{\bf Hypoth\`ese et notation.} 
{\it Dans cette section, et dans cette section seulement, nous supposons 
que toutes les r\'eflexions de $W$ sont d'ordre $2$ et que $-\Id_V \in W$. 
Nous noterons $w_0=-\Id_V$ et \indexnot{tx}{\tau_0}  $\t_0=(-1,1,\e) \in \kb^\times \times \kb^\times 
\times W^\wedge$.}}{0.75\textwidth}

\bigskip

Par construction, la restriction de $\t_0$ \`a $\kb[\CCB]$ est \'egale \`a 
l'identit\'e. Puisque $-\Id_V \in W$, la restriction de $\t_0$ \`a $\kb[V]^W$ 
est \'egale \`a l'identit\'e. De plus, la restriction de $\t_0$ \`a $\kb[V^*]^W$ 
est aussi \'egale \`a l'identit\'e. En conclusion,
\equat\label{tau 0}
\forall~p \in P,~\t_0(p)=p.
\endequat
Rappelons que $\t_0$ d\'esigne aussi l'automorphisme de $R$ 
d\'efini par la proposition~\ref{extension tau}. Par d\'efinition du groupe de Galois, 
on a $\t_0 \in G$. Plus pr\'ecis\'ement~:

\bigskip

\begin{prop}\label{tau 0 in G}
Supposons que toutes les r\'eflexions de $W$ sont d'ordre $2$ et que $w_0=-\Id_V \in W$. 
Alors $\t_0$ est un \'el\'ement central de $G$. Son action sur $W$ est donn\'ee par 
$\t_0(w)=w_0w$ (ce qui signifie que $\t_0=w_0=\iota(w_0,1)$, via l'injection canonique 
$W \injto G$) et, \`a travers l'injection $G \injto \SG_W$, on a
$$G \subset \{\s \in \SG_W~|~\forall~w \in W,~\s(w_0w)=w_0\s(w)\}.$$
De plus, si $w \in W$, alors 
$$\t_0(\eulerq_w)=-\eulerq_w=\eulerq_{w_0w}.$$
\end{prop}

\begin{proof}
D'apr\`es le lemme~\ref{lem:automorphismes-1}(c), on a $\t_0(\eulerq)=-\eulerq$. 
De plus, d'apr\`es l'exemple~\ref{ex:action-k-k-h-n}, l'action de $\t_0$ sur $R$ commute 
\`a l'action de $G$. Ainsi, si $w \in W$, $\t_0(\eulerq_w)=-\eulerq_w$. 

Mais d'autre part, il existe $w_1 \in W$ tel que $\t_0(\eulerq)=\eulerq_{w_1}$. 
Comme $-\eulerq_0 = w_0(\eulerq_0)$, il r\'esulte de la caract\'erisation de l'action de 
$W$ sur $\Lb$ que $\t_0(\eulerq)=\eulerq_{w_0}=-\eulerq$. Puisque $w_0$ est central dans 
$W$, on a $w_0(\eulerq_w)=\eulerq_{w_0w}=\eulerq_{ww_0}=w(\eulerq_{w_0})=-\eulerq_w$. 
Donc $\t_0=w_0$ car $\Mb=\Kb[(\eulerq_w)_{w \in W}]$. 

Maintenant, le fait que $G \subset \{\s \in \SG_W~|~\forall~w \in W,~\s(w_0w)=w_0\s(w)\}$ 
d\'ecoule du fait que $\t_0=w_0$ commute avec l'action de $G$.
\end{proof}

\bigskip

Notons que $w_0w=-w$ et donc l'inclusion de la proposition~\ref{tau 0 in G} peut 
se r\'e\'ecrire
\equat\label{inclusion w0}
G \subset \{\s \in \SG_W~|~\forall~w \in W,~\s(-w)=-\s(w)\}.
\endequat
Vu ainsi, cela montre que, sous les hypoth\`eses de cette section, 
$G$ est contenu dans un groupe de Weyl de type $B_{|W|/2}$.

\bigskip

\section{Sp\'ecificit\'e des groupes de Coxeter}

\cbstart

\boitegrise{{\bf Hypoth\`ese.} 
{\it Dans cette section, et seulement dans cette section, nous supposons 
que $W$ est un groupe de Coxeter, et nous reprenons les notations du 
chapitre~\ref{chapter:coxeter}.}}{0.75\textwidth}

\bigskip

En lien avec les probl\`emes \'evoqu\'es dans ce chapitre, une des particularit\'es de 
cette situation est que l'alg\`ebre $\Hb$ admet un autre automorphisme $\s_\Hb$, 
induit par l'isomorphisme de $W$-modules $\s : V \longiso V^*$. 
Cet automorphisme stabilise $P$, et induit un automorphisme de $\kb[V \times V^*]$ 
qui normalise $W \times W$ et $\D W$. Plus pr\'ecis\'ement, notons 
$\s_2 : V \oplus V^* \longiso V \oplus V^*$, $(y,x) \longmapsto (-\s^{-1}(x),\s(y))$ 
l'automorphisme du $\kb W$-module $V \oplus V^*$. Alors
\equat\label{eq:sigma-ww}
\s_2 (w,w') \s_2^{-1} = (w',w)
\endequat
pour tout $(w,w') \in W \times W$. En vertu de la proposition~\ref{extension tau}, 
$\s_\Hb$ s'\'etend de mani\`ere unique en un automorphisme de $R$, toujours not\'e 
$\s_\Hb$, qui stabilise $\rG_0$ et qui est compatible avec $\iso$. 
Or, $\s_\Hb$ normalise $G$ et son sous-groupe $\iota(W \times W)$ 
(voir~(\ref{eq:action-normalise})) et, compte tenu de~(\ref{eq:sigma-ww}), 
son action sur les \'el\'ements de $W \subset G$ v\'erifie 
\equat\label{eq:sigma-w}
\lexp{\s_\Hb}{w} H = w^{-1} H\qquad\text{et}\qquad H\lexp{\s_\Hb}{w} = H w^{-1} 
\endequat
pour tout $w \in W$. Il d\'ecoule alors de~(\ref{eq:action-normalise}) que~:

\bigskip

\begin{coro}\label{coro:sigma-inverse}
Si $g \in G \subset \SG_W$ et $w \in W$, alors $(\lexp{\s_\Hb}{g})(w)=g(w^{-1})^{-1}$. 
\end{coro}

\cbend

\section{Probl\`emes, questions}\label{section:problemes 2}

\medskip

Voici quelques probl\`emes soulev\'es par ce chapitre.

\bigskip

\begin{probleme}\label{calcul G}
Calculer les groupes de Galois $G=\Gal(\Mb/\Kb)$ et $H=\Gal(\Mb/\Lb)$ 
comme sous-groupes de $\SG_W$.
\end{probleme}
 
\bigskip

Une piste \`a explorer pour r\'esoudre le probl\`eme~\ref{calcul G} est la suivante~: 
soit $W'$ un sous-groupe parabolique de $W$ et notons $G'$ le groupe de Galois 
associ\'e \`a $W'$ de la m\^eme fa\c{c}on que $G$ est associ\'e \`a $W$. 
Notons aussi $H'$ le sous-groupe de $G'$ correspondant \`a $H$.

\bigskip

\begin{probleme}
Relier $G$ et $G'$, $H$ et $H'$.
\end{probleme}

\bigskip

Il est \`a peu pr\`es exclus d'esp\'erer que $R$ soit lisse (c'est 
faux par exemple d\`es le rang $1$ comme on le montrera dans le chapitre~\ref{chapitre:rang 1}). 
Il semble plus raisonnable de se poser les questions suivantes~:

\bigskip

\begin{question}
Est-ce que $R$ est d'intersection compl\`ete, de Gorenstein, de Cohen-Macaulay~? 
Est-ce $R$ est un $P$-module plat~? un $Q$-module plat~?
\end{question}

\bigskip

Nous n'avons aucune id\'ee de la r\'eponse. En revanche, si on note $E^*$ un supp\'ementaire 
$(\NM \times \NM)$-gradu\'e et $G$-stable de $(R_+)^2$ dans $R_+$ et si on 
note $E$ son dual (comme dans la section~\ref{section:GR}), alors 
la question suivante devient raisonnable, sachant que $R^G=P$ est lisse 
(voir la proposition~\ref{intersection complete R})~:

\bigskip

\begin{question}
Est-ce que $G$ est un sous-groupe de $\GL_\kb(E)$ engendr\'e par des r\'eflexions~?
\end{question}

\bigskip

\noindent{\bf Remarque. --- } 
Notons $I$ le noyau du morphisme canonique (surjectif) $\kb[E] \to R$. 
Si $R$ est d'intersection compl\`ete et $G$ agit trivialement sur $I/\kb[E]_+ I$, 
alors $G$ est un sous-groupe de $\GL_\kb(E)$ 
engendr\'e par des r\'eflexions (voir la proposition~\ref{intersection complete R}).\finl

\bigskip

Plus d\'elicat est le probl\`eme qui suit~:

\bigskip

\begin{probleme}\label{calcul R}
Calculer une pr\'esentation de $R$ ou, du moins, donner un algorithme 
permettant de calculer une telle pr\'esentation.
\end{probleme}

\bigskip

Nous avons vu dans la preuve du corollaire~\ref{r0} que 
$R/\rG_0 \subset \kb[V \times V^*]^{\D\Zrm(W)}$ tandis que, par construction, 
$k_R(\rG_0) = \kb(V \times V^*)^{\D\Zrm(W)}$ et $\kb[V \times V^*]^{\D\Zrm(W)}$ 
est int\'egralement clos et entier sur $R/\rG_0$. 
Il est naturel de se poser la question suivante~:

\bigskip

\begin{question}\label{question:r0}
Est-ce que $R/\rG_0 = \kb[V \times V^*]^{\D \Zrm(W)}$~? De mani\`ere 
\'equivalente, est-ce que $R/\rG_0$ est int\'egralement clos~?
\end{question}


\chapter{G\'eom\'etrie}\label{chapter:geometrie-CM}

\section{G\'eom\'etrie de l'extension $Z/P$}\label{section:geometrie-zp}

\medskip

Les alg\`ebres $P$, $Z$, $P_\bullet$ et $Z_c$ \'etant de type fini, 
on peut leur associer des $\kb$-vari\'et\'es alg\'ebriques que nous noterons 
$\PCB$, $\ZCB$, $\PCB_{\!\!\!\bullet}$ et $\ZCB_c$. 
\indexnot{P}{\PCB,~\PCB_{\!\!\!\bullet}}\indexnot{Z}{\ZCB,~\ZCB_c}  
Notons que
$$\PCB=\CCB \times V/W \times V^*/W\quad\text{et}\quad \PCB_{\!\!\!\bullet}=V/W \times V^*/W$$
$$\ZCB_0=(V \times V^*)/W.\leqno{\text{et que}}$$
Il d\'ecoule des corollaires~\ref{coro:endo-bi}(d) et~\ref{coro:endo-bi-c}(d) que 
\equat\label{irr normale}
\text{\it les vari\'et\'es $\ZCB$ et $\ZCB_c$ sont irr\'eductibles et normales.}
\endequat
Puisque tous les \'enonc\'es alg\'ebriques des chapitres pr\'ec\'edents ne d\'ependent pas 
du corps de base, l'\'enonc\'e~(\ref{irr normale}) peut \^etre entendu au sens 
``g\'eom\'etrique''.
Les inclusions $P \subset Z$ et $P_\bullet \subset Z_c$ 
induisent des morphismes de vari\'et\'es 
$$\Upsilon : \ZCB \longto \PCB=\CCB\times V/W \times V^*/W$$ \indexnot{uz}{\Upsilon,~\Upsilon_c}
$$\Upsilon_c : \ZCB_c \longto \PCB_{\!\!\!\bullet}=V/W \times V^*/W.\leqno{\text{et}}$$
Les surjections $P \to P/\CG_cP \simeq P_\bullet$  et $Z \to Z_c$ induisent 
des immersions ferm\'ees $j_c : \ZCB_c \injto \ZCB$ et $i_c : \PCB_{\!\!\!\bullet} \injto \PCB$, \indexnot{ia}{i_c,~j_c}  
$p \mapsto (c,p)$. Par ailleurs, 
l'injection canonique $\kb[\CCB] \injto P$ induit la projection canonique 
$\pi : \PCB \to \CCB$ \indexnot{pz}{\pi}  et, dans le diagramme 
\equat\label{diagramme geometrie}
\diagram
&\ZCB_c \ar@{^{(}->}[rr]^{\DS{j_c}} \ddto_{\DS{\Upsilon_c}} && \ZCB \ddto^{\DS{\Upsilon}}& \\
&&&&\\
V/W \times V^*/W \ar@{=}[r] & \PCB_{\!\!\!\bullet} 
\ar@{^{(}->}[rr]^{\DS{i_c}} \ddto && \PCB \ddto^{\DS{\pi}} \ar@{=}[r]& 
\CCB \times V/W \times V^*/W \\
&&\\
&\{c\} \ar@{^{(}->}[rr] && \CCB =\Ab^{\refw}&,
\enddiagram
\endequat
tous les carr\'es sont cart\'esiens. Notons que, 
d'apr\`es les corollaires~\ref{coro:endo-bi} et~\ref{coro:endo-bi-c}, 
\equat\label{platitude}
\text{\it les morphismes $\Upsilon$ et $\Upsilon_c$ sont finis et plats.}
\endequat
De plus,
\equat\label{eq:pi-plat}
\text{\it $\pi$ est lisse,}
\endequat
car $V/W \times V^*/W$ est lisse.

\bigskip

\begin{exemple}\label{upsilon c=0}
D'apr\`es l'exemple~\ref{exemple zero-0}, on a $\ZCB_0 = (V \times V^*)/W$ et 
$\Upsilon_0 : (V \times V^*)/W = \ZCB_0 \to \PCB_\bullet = V/W \times V^*/W$ 
est le morphisme canonique.\finl 
\end{exemple}

\bigskip

Notons $\ZCB^\reg$ \indexnot{Z}{\ZCB^\reg}  l'ouvert $\Spec(Z^\reg)$ de $\ZCB$.  
Le corollaire~\ref{centre reg} montre que 
\equat\label{eq:zreg}
\text{\it $\ZCB^\reg \simeq (V^\reg \times V^*)/W \times \CCB$ est lisse.}
\endequat

\bigskip

\section{G\'eom\'etrie de l'extension $R/P$}\label{section:geometrie galois}

\medskip

Puisque $R$ et $Q \simeq Z$ sont aussi des $\kb$-alg\`ebres de type fini, on peut leur associer 
des $\kb$-vari\'et\'es $\RCB$ \indexnot{R}{\RCB}  
et \indexnot{Q}{\QCB}  $\QCB \simeq \ZCB$~: l'isomorphisme $\copie^* : \QCB \longiso \ZCB$ est 
induit par $\copie : Z \longiso Q$. 
Alors l'inclusion $P \injto R$ (respectivement $Q \injto R$) 
d\'efinit un morphisme de vari\'et\'es $\r_G : \RCB \to \PCB$ \indexnot{rz}{\r_G}  
(respectivement $\r_H : \RCB \to \QCB$)\indexnot{rz}{\r_H}   et les \'egalit\'es $P=R^G$ et $Q=R^H$ 
montrent que $\r_G$ et $\r_H$ induisent des isomorphismes 
\equat\label{R/G}
\RCB/G \longiso \PCB \qquad\text{et}\qquad \RCB/H \longiso \QCB.
\endequat
Dans cette optique, le choix d'un id\'eal premier $\rG_c$ au-dessus de $\qG_c$ 
\'equivaut au choix d'une composante irr\'eductible $\RCB_c$ \indexnot{R}{\RCB_c}  de $\r_H^{-1}(\QCB_c)$ 
(dont $\rG_c$ est l'id\'eal de d\'efinition). De m\^eme, 
l'argument conduisant \`a la proposition~\ref{WW} implique par exemple que 
le nombre de composantes irr\'eductibles de $\r_G^{-1}(\QCB_0)$ est \'egal 
\`a $|G|\cdot|\D\Zrm(W)|/|W|^2$. Il montre aussi que $\iota(W \times W)$ est 
le stabilisateur de $\RCB_0$ dans $G$ et $\RCB_0/\iota(W \times W) \simeq \PCB_0$, 
que $\iota(\D W)$ est le stabilisateur de $\RCB_0$ dans $H$ et que 
$\RCB_0/\iota(\D W) \simeq \QCB_0$. On a donc un diagramme commutatif
\equat\label{diagramme geometrie bis}
\diagram
&\RCB_c \ar@{^{(}->}[rr] \ddto && \RCB \ddto^{\DS{\r_H}}\ar@/^4pc/[dddd]^{\DS{\r_G}}& \\
&&&&\\
&\QCB_c \ar@{^{(}->}[rr]^{\DS{j_c}} \ddto_{\DS{\Upsilon_c}} && \QCB \ddto^{\DS{\Upsilon}}& \\
&&&&\\
V/W \times V^*/W \ar@{=}[r] & \PCB_{\!\!\!\bullet} 
\ar@{^{(}->}[rr]^{\DS{i_c}} \ddto && \PCB \ddto^{\DS{\pi}} \ar@{=}[r]& 
\CCB \times V/W \times V^*/W \\
&&\\ 
&\{c\} \ar@{^{(}->}[rr] && \CCB&
\enddiagram
\endequat
compl\'etant le diagramme~\ref{diagramme geometrie} (en identifiant $\QCB$ et $\ZCB$ via $\copie^*$). 
Seuls les deux carr\'es inf\'erieurs du diagramme~\ref{diagramme geometrie bis} sont cart\'esiens.

\bigskip

\subsection{Automorphismes}\label{subsection:automorphismes} 
Le groupe $\kb^\times \times \kb^\times \times \bigl(W^\wedge \rtimes \NC\bigr)$ 
(qui agit sur $\Hb$ par automorphismes de $\kb$-alg\`ebres)  
stabilise les $\kb$-sous-alg\`ebres $\kb[\CCB]$, $P$ et $Q$ de $\Hb$. 
Ainsi, $\t$ induit des automorphismes des $\kb$-vari\'et\'es 
$\CCB$, $\PCB$ et $\QCB$ et les morphismes $\Upsilon$ et $\pi$ 
du diagramme~\ref{diagramme geometrie} sont \'equivariants pour 
cette action. 

De m\^eme, cette action s'\'etend en une action sur $\RCB$ (voir le corollaire~\ref{coro:extension-action}) 
qui rend $\r_H$ et $\r_G$ \'equivariants. 

\bigskip

\subsection{Composantes irr\'eductibles de $\RCB \times_\PCB \ZCB$} 
Si $w \in W$, on pose
$$\RCB_w = \{(r,\copie^*(\r_H(w(r))))~|~r \in R\} \subseteq \RCB \times_\PCB \ZCB.\indexnot{R}{\RCB_w}$$

\bigskip

\begin{lem}\label{lem:compo-irr-r}
 Si $w \in W$, alors $\RCB_w$ est une composante irr\'eductible de 
$\RCB \times_\PCB \ZCB$, isomorphe \`a $\RCB$. De plus,
$$\RCB \times_\PCB \ZCB = \bigcup_{w \in W} \RCB_w$$
et $\RCB_w=\RCB_{w'}$ si et seulement si $w=w'$.
\end{lem}

\begin{proof}
C'est seulement la traduction g\'eom\'etrique du fait que le morphisme 
$$\fonctio{R \otimes_P Z}{\prod_{w \in W} R}{x}{(w_Z(x))_{w \in W}}$$
d\'efini par restriction du morphisme~\ref{W M} est fini et 
devient un isomorphisme par extension des scalaires \`a $\Kb$.
\end{proof}

\section{Probl\`emes, questions}\label{section:probleme 1}

\medskip

\begin{probleme}
D\'eterminer le lieu singulier de $\ZCB$ (et de $\ZCB_c$), ou au moins sa codimension.
\end{probleme}

\bigskip

\begin{question}
Soit $c \in \CCB$ et $z \in \ZCB_c$. Est-ce que $z$ est un point lisse de $\ZCB$ si et seulement 
si c'est un point lisse de $\ZCB_c$~? Plus pr\'ecis\'ement, est-ce que les singularit\'es 
de $\ZCB$ et $\ZCB_c$ en $z$ sont \'equivalentes~?
\end{question}

\bigskip

\begin{probleme}
D\'eterminer le lieu de ramification de $\Upsilon$ et de $\Upsilon_c$. 
\end{probleme}

\chapter{Cellules de Calogero-Moser}\label{chapter:cellules-CM}

\boitegrise{{\bf Notation.} 
{\it Dor\'enavant, et ce jusqu'\`a la fin de ce chapitre, 
nous fixons un id\'eal premier $\rG$ de $R$ et posons 
$\qG=\rG \cap Q$ et $\pG = \rG \cap P$. \indexnot{pa}{\pG,~\qG,~\rG}  Nous noterons 
$D_\rG$ (respectivement $I_\rG$) le groupe de d\'ecomposition 
(respectivement d'inertie) de $\rG$ dans $G$.}}{0.75\textwidth}

\bigskip

\section{D\'efinition, premi\`eres propri\'et\'es}\label{section:definition cellules}

\medskip

Rappelons que, maintenant que l'on a choisi une fois pour toutes 
un id\'eal premier $\rG_0$ ainsi qu'un isomorphisme 
$k_R(\rG_0) \longiso \kb(V \times V^*)^{\D \Zrm(W)}$, on peut identifier 
les ensembles $G/H$ et $W$ (voir~\S\ref{subsection:specialisation galois 0}). 
Ainsi, le groupe $G$ agit sur l'{\it ensemble} $W$. 

\bigskip

\begin{defi}\label{defi:CM}
On appelle {\bfit $\rG$-cellule de Calogero-Moser} 
toute orbite du groupe d'inertie $I_\rG$ dans l'{\bfit ensemble} $W$. Nous noterons 
$\sim^\calo_\rG$ \indexnot{ZZZ}{\sim_\rG^\calo}  la relation d'\'equivalence correspondant \`a la partition de $W$ 
en $\rG$-cellules de Calogero-Moser. 

L'ensemble des $\rG$-cellules de Calogero-Moser sera not\'e $\cmcellules_\rG(W)$. 
\indexnot{C}{{{{^\calo{\mathrm{Cell}}_\rG(W)}}}}
\end{defi}

\bigskip

Rappelons que $W$ s'identifie \`a l'ensemble $\Hom_{P-\alg}(Q,R) = \Hom_{\Kb-\alg}(\Lb,\Mb)$. 
En vertu de la proposition~\ref{reduction}, 
si $w$ et $w'$ sont deux \'el\'ements de $W$, alors 
\equat\label{w sim}
\text{\it $w \sim^\calo_\rG w'$ si et seulement si $w(q) \equiv w'(q) \mod \rG$ 
pour tout $q \in Q$.}
\endequat

\bigskip

\begin{rema}\label{semicontinuite}
Si $\rG$ et $\rG'$ sont deux id\'eaux premiers de $R$ tels que $\rG \subset \rG'$, 
alors $I_\rG \subset I_{\rG'}$ et donc les $\rG'$-cellules de Calogero-Moser 
sont des r\'eunions de $\rG$-cellules de Calogero-Moser.\finl
\end{rema}

\begin{exemple}[R\'eflexions d'ordre 2]\label{cellules w0}
Si toutes les r\'eflexions de $W$ sont d'ordre $2$ et si $w_0=-\Id_V \in W$, 
alors il d\'ecoule de la proposition~\ref{tau 0 in G} que 
$G \subset \{\s \in \SG_W~|~\forall~w \in W,~\s(w_0w)=w_0\s(w)\}$. 
Par cons\'equent, si $\G$ est une $\rG$-cellule de Calogero-Moser, 
alors $w_0\G=\G w_0$ est une $\rG$-cellule de Calogero-Moser.\finl 
\end{exemple}

\bigskip

L'action de $G$ \'etant compatible avec la bi-graduation de $R$, le r\'esultat suivant n'est pas \'etonnant~:

\bigskip

\begin{prop}\label{prop:cellules-homogeneise}
Soit $R =\bigoplus_{i \in \ZM} R_i$ une graduation $G$-stable de $R$. Notons $\rGt=\bigoplus_{i \in \ZM} \rG \cap R_i$. 
Alors $I_\rG=I_\rGt$ et donc les $\rGt$-cellules de Calogero-Moser co\"{\i}ncident avec les $\rG$-cellules de 
Calogero-Moser.
\end{prop}

\begin{proof}
Cela d\'ecoule du corollaire~\ref{coro:homogeneise-inertie}.
\end{proof}

\bigskip

\begin{exemple}\label{exemple:graduation-stables}
Le groupe de Galois $G$ stabilisant la bi-graduation naturelle de $R$ (voir la proposition~\ref{bigraduation sur R}), 
il stabilise toutes les graduations induites par des morphismes de mono\"{\i}des $\ph : \NM \times \NM \to \ZM$ 
comme dans la section~\ref{section:graduation-1}.\finl
\end{exemple}

\bigskip

\section{Cellules et blocs}\label{subsection:cellules et blocs}

\medskip

Si $w \in W$, on notera $e_w \in \Mb\Hb$ \indexnot{ea}{e_w}  l'idempotent primitif central de 
$\Mb\Hb$ (qui est semi-simple d\'eploy\'ee d'apr\`es~(\ref{deployee})) 
associ\'e au module simple $\LC_w$~: c'est 
l'unique idempotent primitif central de $\Mb\Hb$ qui agit comme l'identit\'e 
sur le $\Mb$-module simple $\LC_w$. Si 
$b \in \blocs(R_\rG Q)$, nous noterons $\calo_\rG(b)$ \indexnot{C}{\calo_\rG(b)} l'unique partie de $W$ 
telle que
\equat\label{calo idempotent}
b=\sum_{w \in \calo_\rG(b)} e_w.
\endequat
En d'autres termes, la bijection $W \stackrel{\sim}{\longleftrightarrow} \Irr \Mb\Hb$ 
se restreint en une bijection 
$\calo_\rG(b) \stackrel{\sim}{\longleftrightarrow} \Irr \Mb\Hb b$. 
Il est \'evident que $(\calo_\rG(b))_{b \in \blocs(R_\rG \Hb)}$ est une partition 
de $W$. En fait, cette partition co\"{\i}ncide avec la partition en cellules de 
Calogero-Moser~:

\bigskip

\begin{theo}\label{theo:calogero}
Soient $w$, $w' \in W$ et notons $b$ et $b'$ les idempotents primitifs centraux de 
$R_\rG\Hb$ tels que $w \in \calo_\rG(b)$ et $w' \in \calo_\rG(b')$. 
Alors $w \sim_\calo^\rG w'$ si et seulement si $b=b'$.
\end{theo}

\begin{proof}
Notons $\o_w : \Zrm(R\Hb)=R \otimes_P Z \longto R$ \indexnot{ozz}{\o_w}  le caract\`ere central 
associ\'e au $\Mb\Hb$-module simple $\LC_w$ (voir la sous-section~\ref{section centrale}). 
Au vu de la d\'efinition de $\LC_w$, on a 
$$\o_w(r \otimes_P z)= r w(\copie(z))$$
pour tous $r \in R$ et $q \in Q$. Par cons\'equent, d'apr\`es~(\ref{w sim}), on a 
$w \sim_\calo^\rG w'$ si et seulement si $\o_w \equiv \o_{w'} \mod \rG$. 
Le r\'esultat d'ecoule alors du corollaire~\ref{coro:r-blocs}.
\end{proof}

\bigskip

Autrement dit, on a construit des bijections 
\equat\label{bij calogero}
\cmcellules_\rG(W) \stackrel{\sim}{\longleftrightarrow} \blocs(R_\rG Z)
\stackrel{\sim}{\longleftrightarrow} \blocs(k_R(\rG) Z).
\endequat
Puisque $\Mb Z$ est le centre de $\Mb\Hb$, la semi-simplicit\'e 
et le d\'eploiement de $\Mb\Hb$ impliquent imm\'ediatement que
\equat\label{dim centre calogero}
\dim_\Mb(\Mb Z b)=|\calo_\rG(b)|.
\endequat
Rappelons que, puisque $Z$ est un facteur direct de $\Hb$, alors $k_R(\rG)Z$ s'identifie 
\`a son image dans $k_R(\rG) \Hb$~: cependant, cette image peut ne pas co\"{\i}ncider 
avec le centre de $k_R(\rG)\Hb$. 

\bigskip

\begin{coro}\label{dim centre r}
Notons $\bba$ l'image de $b$ dans $k_R(\rG) \Hb = R_\rG \Hb / \rG R_\rG \Hb$. 
Alors 
$$\dim k_R(\rG) Z \bba=|\calo_\rG(b)|.$$
\end{coro}

\begin{proof}
Le $R_\rG$-module $R_\rG Z$ est libre (de rang $|W|$) donc le 
$R_\rG$-module $R_\rG Z b$ est projectif, donc libre car $R_\rG$ est local. 
En vertu de~(\ref{dim centre calogero}), le $R_\rG$-rang de $R_\rG Z b$ 
est $|\calo_\rG(b)|$. D'o\`u le r\'esultat.
\end{proof}

\bigskip

\begin{exemple}[Sp\'ecialisation]\label{specialisation} 
Soit $c \in \CCB$. Choisissons un id\'eal premier $\rG_c$ de $R$ au-dessus de 
$\pG_c$ et notons, comme dans~\S\ref{subsection:specialisation galois}, 
$D_c=G_{\rG_c}^D$ et $I_c=G_{\rG_c}^I$. 
Alors $I_c=1$ d'apr\`es~\ref{eq:net} et donc
$$\text{\it Les $\rG_c$-cellules de Calogero-Moser sont des singletons.~$\square$}$$
\end{exemple}

\bigskip

\section{Lieu de ramification}

\medskip

Nous noterons $\rG_\ramif$ \indexnot{ra}{\rG_\ramif}  l'id\'eal de d\'efinition du lieu de ramification 
du morphisme fini $\Spec(R) \to \Spec(P)$~: en d'autres termes, $R$ est \'etale 
sur $P$ en $\rG$ si et seulement si $\rG_\ramif \not\subset \rG$. 
Rappelons~\cite[expos\'e I, corollaire 9.11 et expos\'e V, corollaires 2.3 et 2.4]{sga} 
aussi que $R$ est \'etale sur $P$ en $\rG$ si et seulement si $I_\rG \neq 1$ et 
que $R$ est \'etale sur $P$ en $\rG$ si et seulement si $R$ est net sur $P$ 
en $\rG$. Comme $G$ agit fid\`element sur $W \stackrel{\sim}{\longleftrightarrow} G/H$, 
on a donc montr\'e le r\'esultat suivant (compte tenu du th\'eor\`eme~\ref{theo:calogero})~:

\bigskip

\begin{prop}\label{ramification inertie}
Les conditions suivantes sont \'equivalentes~:
\begin{itemize}
\itemth{1} $I_\rG \neq 1$.

\itemth{2} $R$ n'est pas \'etale sur $P$ en $\rG$.

\itemth{3} $R$ n'est pas net sur $P$ en $\rG$.

\itemth{4} $\rG_\ramif \subset \rG$.

\itemth{5} $|\blocs(R_\rG Q)| < |W|$.
\end{itemize}
\end{prop}

\bigskip

Il faut cependant noter que $\rG_\ramif$ n'est pas forc\'ement un id\'eal premier 
de $R$. En revanche, 
le th\'eor\`eme de puret\'e~\cite[expos\'e~X,~th\'eor\`eme 3.1]{sga} nous dit que 
$\Spec(R/\rG_\ramif)$ est vide ou purement de codimension $1$ dans $\Spec(R)$ 
(car $R$ est int\'egralement clos et $P$ est r\'egulier). 
D'apr\`es le corollaire~\ref{r0 DI} et la proposition 
pr\'ec\'edente~\ref{ramification inertie}, 
le morphisme $\Spec(R) \to \Spec(P)$ n'est pas \'etale (sauf si $W=1$). Ainsi,
\equat\label{codim 1 ramification}
\text{\it $\Spec(R/\rG_\ramif)$ est purement de codimension $1$ dans $\Spec(R)$.}
\endequat
Bien s\^ur, l'id\'eal $\rG_\ramif$ est tellement naturel que l'on en d\'eduit que
\equat\label{ramification stable}
\text{\it $\rG_\ramif$ est stable sous l'action de 
$\kb^\times \times\kb^\times \times \left((W^\wedge \times G) \rtimes \NC\right)$.}
\endequat
S'il est difficile de calculer l'id\'eal $\rG_\ramif$ (on ne sait m\^eme 
pas calculer l'anneau $R$), il est nettement plus facile de d\'eterminer 
l'id\'eal $\pG_\ramif=\rG_\ramif \cap P$. \indexnot{pa}{\pG_\ramif}  Le r\'esultat suivant est classique~:

\bigskip

\begin{lem}\label{discriminant ramification}
Notons $\disc(Q/P)$ \indexnot{da}{\disc(D/P)}  l'id\'eal discriminant de $Q$ dans $P$. Alors 
$\pG_\ramif=\sqrt{\disc(Q/P)}$. 
\end{lem}
%

\bigskip

\begin{rema}\label{autre preuve ramification}
Nous avons fait d\'ecouler le fait que $\Spec(R/\rG_\ramif)$ est 
purement de codimension $1$ dans $\Spec(R)$ du th\'eor\`eme de puret\'e. 
En vertu de l'\'equivalence entre (4) et (5) dans le lemme~\ref{ramification inertie}, 
nous aurions aussi pu invoquer la proposition~\ref{codimension un}.\finl
\end{rema}

\bigskip

%
%
%
%
%
%
%
%

\section{Cellules et lissit\'e}\label{section:cellules et lissite}

\medskip

Notons $\ZCB_\singulier$  \indexnot{Z}{\ZCB_\singulier}  le lieu singulier de $\ZCB=\Spec(Z)$ et $\zG_\singulier$ 
\indexnot{Z}{\zG_\singulier}  
son id\'eal de d\'efinition. Puisque $Z$ est int\'egralement clos, 
\equat\label{codimension deux Q}
\text{\it $\ZCB_\singulier$ est de codimension $\ge 2$ dans $\ZCB$.}
\endequat
Bien s\^ur, il est possible que $\zG_\singulier$ ne soit pas un id\'eal premier.
Puisque $\Upsilon : \ZCB \to \PCB$ est un morphisme fini et plat, on en d\'eduit que
\equat\label{codimension deux P}
\text{\it $\Upsilon(\ZCB_\singulier)$ est ferm\'e et 
de codimension $\ge 2$ dans $\PCB$.}
\endequat
L'id\'eal de d\'efinition de $\Upsilon(\ZCB_\singulier)$ est $\sqrt{\zG_\singulier \cap P}$. 

\medskip

\boitegrise{\noindent{\bf Hypoth\`ese.} 
{\it Dans cette section, et uniquement dans cette section, 
nous supposerons que $\Spec(P/\pG)$ (qui est une sous-vari\'et\'e 
ferm\'ee irr\'eductible de $\PCB$) n'est pas contenue dans 
$\Upsilon(\ZCB_\singulier)$. En d'autres termes, on suppose que 
$\zG_\singulier \cap P \not\subset \pG$.}}{0.75\textwidth}

\medskip

Par cons\'equent, il existe $p \in \zG_\singulier \cap P$ tel que $p \not\in \pG$. 
Ainsi, $Z[1/p]=P[1/p] \otimes_P Z \subset Z_\pG= P_\pG \otimes_P Z$ et $\Spec(Z[1/p])$ est 
r\'egulier. Il d\'ecoule alors de la proposition~\ref{prop:morita} que 
$P[1/p] \otimes_P \Hb$ et $Q[1/p]$ sont Morita \'equivalentes 
(gr\^ace au bimodule $P[1/p] \otimes_P \Hb e$) et, par extension des 
scalaires, on obtient~:

\bigskip

\begin{prop}\label{morita lissite}
Le $(\Hb_\pG,Z_\pG)$-bimodule $\Hb_\pG e$ est projectif \`a droite et \`a gauche 
et induit une \'equivalence de Morita entre $\Hb_\pG$ et $Z_\pG$. 
\end{prop}

\bigskip

Par r\'eduction modulo $\pG$, on obtient~:

\bigskip

\begin{coro}\label{morita lissite p}
Le $(k_P(\pG)\Hb,k_P(\pG)Z)$-bimodule $k_P(\pG)\Hb e$ est projectif \`a droite et \`a gauche 
et induit une \'equivalence de Morita entre $k_P(\pG)\Hb$ et $k_P(\pG)Z$. 
\end{coro}

\bigskip

Par extension des scalaires, on en d\'eduit~:

\bigskip

\begin{coro}\label{morita lissite r}
Le $(k_R(\rG)\Hb,k_R(\rG)Z)$-bimodule $k_R(\rG)\Hb e$ est projectif \`a droite et \`a gauche 
et induit une \'equivalence de Morita entre $k_R(\rG)\Hb$ et $k_R(\rG)Z$. 
\end{coro}

\bigskip

\begin{theo}\label{lissite et simples}
La $k_R(\rG)$-alg\`ebre $k_R(\rG)\Hb$ est 
d\'eploy\'ee. Chaque bloc de $k_R(\rG)\Hb$ admet un unique module simple, qui est de dimension $|W|$.
En particulier, les $k_R(\rG)\Hb$-modules simples sont param\'etr\'es naturellement 
par les $\rG$-cellules de Calogero-Moser, c'est-\`a-dire par les $I_\rG$-orbites dans $W$.
\end{theo}

\begin{proof}
Montrons tout d'abord que $k_R(\rG)Z = k_R(\rG) \otimes_P Z = k_R(\rG) \otimes_{P_\pG} Z_\pG$ 
est une $k_R(\rG)$-alg\`ebre d\'eploy\'ee. Soient $\zG_1$,\dots, $\zG_l$ 
les id\'eaux premiers de $Z$ au-dessus de $\pG$~: en d'autres termes, 
$k_P(\pG)\zG_1$,\dots, $k_P(\pG)\zG_l$ sont les id\'eaux premiers (et donc maximaux)  
de $k_P(\pG)Z = Z_\pG/\pG Z_\pG$. Alors $k_P(\pG)(\zG_1 \cap \dots \cap \zG_l)$ 
est le radical de $k_P(\pG)Z$. De plus, 
$$k_P(\pG) Z \simeq k_Z(\zG_1) \times \cdots \times k_Z(\zG_l).$$
Ainsi,
$$k_R(\rG) Z \simeq (k_R(\rG) \otimes_{k_P(\pG)} k_Z(\zG_1)) \times \cdots \times 
(k_R(\rG) \otimes_{k_P(\pG)} k_Z(\zG_l)).$$
Mais $k_R(\rG)$ est une extension galoisienne de $k_P(\pG)$ (de groupe $D_\rG/I_\rG$) 
contenant $k_Z(\zG_i)$ (pour tout $i$), ou du moins son image via $\copie^{-1}$, 
et donc $k_R(\rG) \otimes_{k_P(\pG)} k_Z(\zG_i)$ 
est une $k_R(\rG)$-alg\`ebre d\'eploy\'ee (voir la proposition~\ref{iso galois}). 
On a donc montr\'e que $k_R(\rG)Z$ est d\'eploy\'ee~: gr\^ace \`a l'\'equivalence de Morita 
avec $k_R(\rG)\Hb$, cela montre que $k_R(\rG)\Hb$ est elle aussi d\'eploy\'ee.

D'autre part, puisque $k_R(\rG)Z$ est commutative, chaque bloc de $k_R(\rG)Z$ 
admet un unique module simple. Par \'equivalence de Morita, il en est de m\^eme 
de $k_R(\rG)\Hb$. Pour finir, puisque le $Z_\pG$-module projectif $\Hb_\pG e$ est de rang 
$|W|$, il en est  de m\^eme du $k_R(\rG)Z$-module projectif $k_R(\rG)\Hb e$, et donc les 
$k_R(\rG)\Hb$-modules simples sont de dimension $|W|$.

La derni\`ere assertion du th\'eor\`eme devient alors \'evidente.
\end{proof}

\bigskip

\begin{exemple}\label{exemple lisse}
Compte tenu du corollaire~\ref{centre reg}, la condition $\zG_\singulier \cap P \not\subset \pG$ est automatiquement 
satisfaite si $\Spec(P/\pG)$ rencontre l'ouvert $\PCB^\reg = \CCB \times V^\reg/W \times V^*/W$ ou bien, par sym\'etrie, 
l'ouvert $\CCB \times V/W \times V^{* \reg}/W$.\finl
\end{exemple}

\bigskip

\section{Cellules et g\'eom\'etrie}

\medskip

D'apr\`es le lemme~\ref{lem:compo-irr-r}, les composantes irr\'eductibles de 
$\RCB \times_\PCB \ZCB$ sont de la forme 
$$\RCB_w=\{(r,\copie^*(\r_H(w(r))))~|~r \in \RCB\}.$$
Ainsi, le morphisme $\Upsilon_\RCB : \RCB \times_\PCB \ZCB \to \RCB$ obtenu 
\`a partir de $\Upsilon : \ZCB \to \PCB$ par changement de base 
induit un isomorphisme entre la composante irr\'eductible $\RCB_w$ et $\RCB$. 

Par cons\'equent, l'image inverse par $\Upsilon_\RCB$ de la sous-vari\'et\'e ferm\'ee 
irr\'eductible $\RCB(\rG)=\Spec(R/\rG)$ \indexnot{R}{\RCB(\rG)}  est une r\'eunion de sous-vari\'et\'e irr\'eductibles 
$$\Upsilon^{-1}_\RCB(\RCB(\rG)) = \bigcup_{w \in W} \RCB_w(\rG),\indexnot{R}{\RCB_w(\rG)}$$
de sorte que $\RCB_w(\rG) \simeq \RCB(\rG)$ soit l'image inverse de $\RCB(\rG)$ dans 
$\RCB_w$. 

\bigskip

\begin{lem}\label{lem:cellules-geometrie}
Soient $w$ et $w' \in W$. Alors $\RCB_w(\rG)=\RCB_{w'}(\rG)$ si et seulement si 
$w \sim_\rG^\calo w'$.
\end{lem}

\begin{proof}
En effet, $\RCB_w(\rG)=\RCB_{w'}(\rG)$ si et seulement si, pour tout $r \in \RCB(\rG)$, 
$\r_H(w(r))=\r_H(w'(r))$. Traduit au niveau des anneaux $Q$ et $R$, cela est \'equivalent 
\`a dire que, pour tout $q \in Q$, on a $w(q) \equiv w'(q) \mod \rG$. 
\end{proof}

\bigskip

En d'autres termes, le lemme~\ref{lem:cellules-geometrie} montre que les $\rG$-cellules de 
Calogero-Moser param\`etrent les composantes irr\'eductible de l'image inverse de $\RCB(\rG)$ 
dans le produit fibr\'e $\RCB \times_\PCB \ZCB$. 

\bigskip

\part{Cellules et modules de Verma}\label{part:verma}

\boitegrise{{\bf Notations.} {\it Fixons dans cette partie un id\'eal premier $\CG$ de $\kb[\CCB]$. 
Nous notons $\CCB(\CG)$ le \indexnot{C}{\CCB(\CG)}  
sous-sch\'ema irr\'eductible ferm\'e de $\CCB$ d\'efini par $\CG$~: en d'autres 
termes, $\CCB(\CG)=\Spec \kb[\CCB]/\CG$ et nous notons $\pGba_\CG$ 
(resp. $\pG_\CG^\gauche$, resp. $\pG_\CG^\droite$)  \indexnot{pa}{\pGba_\CG, \pG_\CG^\gauche, \pG_\CG^\droite}  
l'id\'eal premier de $P$ correspondant au sous-sch\'ema ferm\'e irr\'eductible 
$\CCB(\CG) \times \{0 \} \times \{0\}$ (resp. $\CCB(\CG) \times V/W \times \{0\}$, 
resp. $\CCB(\CG) \times \{0\} \times V^*/W$). 
On pose $\Pba_\CG=P/\pGba_\CG \simeq \kb[\CCB]/\CG$, 
$P^\gauche_\CG = P/\pG_\CG^\gauche \simeq \kb[\CCB]/\CG \otimes \kb[V]^W$ et 
$P^\droite_\CG=P/\pG^\droite \simeq \kb[\CCB]/\CG \otimes \kb[V^*]^W$,  \indexnot{P}{\Pba_\CG, P_\CG^\gauche, P_\CG^\droite} 
ce qui permet de d\'efinir la $\Pba_\CG$-alg\`ebre $\Hbov_\CG=\Hb/\pGba_\CG \Hb$,  
la $P_\CG^\gauche$-alg\`ebre $\Hb^\gauche_\CG=\Hb/\pG_\CG^\gauche\Hb$ et la 
$P_\CG^\droite$-alg\`ebre $\Hb_\CG^\droite=\Hb/\pG_\CG^\droite\Hb$.  \indexnot{H}{\Hbov_\CG, \Hb_\CG^\gauche, \Hb_\CG^\droite} 
On d\'efinit aussi $\Kbov_\CG=k_P(\pGba_\CG)$, 
$\Kb_\CG^\gauche=k_P(\pG_\CG^\gauche)$ et $\Kb_\CG^\droite = k_P(\pG_\CG^\droite)$. \indexnot{K}{\Kbov_\CG, \Kb_\CG^\gauche, \Kb_\CG^\droite}  \\
\hphantom{A} Pour simplifier les notations, si $\CG$ est l'id\'eal nul, alors l'indice $\CG$ sera omis 
dans toutes les notations pr\'ec\'edentes ($\Pba$, $\pGba$, $\pG^\gauche$, $\Hb^\droite$, $\Kb^\gauche$,\dots) 
et, si $c \in \CCB$ et $\CG=\CG_c$, alors l'indice $\CG_c$ sera remplac\'e par $c$   
\indexnot{pa}{\pGba, \pG^\gauche, \pG^\droite, \pGba_c, \pG_c^\gauche, \pG_c^\droite}  
\indexnot{P}{\Pba, P^\gauche, P_\CG^\droite, \Pba, P_c^\gauche, P_c^\droite}  
\indexnot{H}{\Hbov, \Hb^\gauche, \Hb^\droite, \Hbov_c, \Hb_c^\gauche, \Hb_c^\droite}  
\indexnot{K}{\Kbov, \Kb^\gauche, \Kb^\droite, \Kbov_c, \Kb_c^\gauche, \Kb_c^\droite} 
($\pG_c^\droite$, $\Hb_c^\gauche$, $\Kb_c^\droite$, $\Kbov_c$,\dots). 
Notons par exemple que $\pGba_\CG=\pGba + \CG~P$ (et de m\^eme pour $\pG_\CG^\gauche$ et 
$\pG_\CG^\droite$) et que $\Kbov_c \simeq \kb$.}}{0.75\textwidth}

\bigskip

\begin{defivide}
Fixons un id\'eal premier $\rGba_\CG$ (resp. $\rG_\CG^\gauche$, resp. $\rG_\CG^\droite$) 
de $R$ au-dessus de $\pGba_\CG$ (resp. $\pG_\CG^\gauche$, resp. $\pG_\CG^\droite$).  
\indexnot{ra}{\rGba_\CG, \rG_\CG^\gauche, \rG_\CG^\droite}   
\indexnot{ra}{\rGba, \rG^\gauche, \rG^\droite, \rGba_c, \rG_c^\gauche, \rG_c^\droite}  
On appelera {\bfit $\CG$-cellule de Calogero-Moser bilat\`ere} 
(resp. {\bfit \`a gauche}, resp. {\bfit \`a droite}) toute 
$\rGba_\CG$-cellule (resp. $\rG_\CG^\gauche$-cellule, resp. $\rG_\CG^\droite$-cellule) 
de Calogero-Moser.  

Si $\CG=0$, elles seront aussi appel\'ees {\bfit cellules de Calogero-Moser 
(bilat\`eres, \`a gauche ou \`a droite) g\'en\'eriques}. Si $c \in \CCB$ et $\CG=\CG_c$, 
elles seront appel\'ees {\bfit $c$-cellules de Calogero-Moser 
(bilat\`eres, \`a gauche ou \`a droite)}.
\end{defivide}

\bigskip

\begin{remavide}\label{rem:dependance-r}
Bien s\^ur, la notion de $\CG$-cellule de Calogero-Moser (bilat\`ere, \`a gauche ou \`a droite) 
d\'epend fortement du choix de l'id\'eal $\rGba_\CG$, $\rG_\CG^\gauche$ 
ou $\rG_\CG^\droite$~; cependant, comme tous les id\'eaux premiers de $R$ au-dessus d'un id\'eal 
premier de $P$ sont $G$-conjugu\'es, changer l'id\'eal revient \`a transformer 
les cellules par l'action d'un \'el\'ement de $G$.\finl
\end{remavide}

\bigskip

\begin{remavide}[Semi-continuit\'e]\label{rem:semicontinuite}
D'autre part, il est bien s\^ur possible de choisir les id\'eaux $\rGba_\CG$, $\rG_\CG^\gauche$ 
ou $\rG_\CG^\droite$ de sorte que $\rGba_\CG$ contienne $\rG_\CG^\gauche$ et $\rG_\CG^\droite$~: 
dans ce cas, en vertu de la remarque~\ref{semicontinuite}, 
les $\CG$-cellules de Calogero-Moser bilat\`eres sont des r\'eunions de 
$\CG$-cellules de Calogero-Moser \`a gauche (resp. \`a droite).

De m\^eme, si $\CG'$ est un autre id\'eal premier de $\kb[\CCB]$ tel que $\CG \subset \CG'$, alors 
on peut choisir les id\'eaux $\rGba_{\CG'}$, $\rG_{\CG'}^\gauche$ ou $\rG_{\CG'}^\droite$ 
de sorte qu'ils contiennent respectivement $\rGba_\CG$, $\rG_\CG^\gauche$ 
ou $\rG_\CG^\droite$. Ainsi, les $\CG'$-cellules de Calogero-Moser bilat\`eres (resp. \`a gauche, 
resp. \`a droite) sont des r\'eunions de $\CG$-cellules de Calogero-Moser bilat\`eres 
(resp. \`a gauche, resp. \`a droite).\finl
\end{remavide}

\bigskip

Avec la d\'efinition des cellules de Calogero-Moser bilat\`eres, \`a gauche ou \`a droite donn\'ee ci-dessus, 
le premier objectif de ce m\'emoire est atteint. L'objet de cette partie est l'\'etude de 
ces cellules particuli\`eres, en lien avec la th\'eorie des repr\'esentations de $\Hb$~: 
dans chaque cas, une famille de {\it modules de Verma} 
viendra renforcer l'arsenal mis \`a notre disposition. Ainsi~:
\begin{itemize}
\item Dans le chapitre~\ref{chapter:bilatere}, nous associerons une {\it famille de Calogero-Moser} 
\`a chaque cellule de Calogero-Moser bilat\`ere~: les familles de Calogero-Moser forment une partition 
de l'ensemble $\Irr(W)$.

\item Dans le chapitre~\ref{chapter:gauche}, nous associerons un $\calo$-caract\`ere cellulaire 
\`a chaque cellule \`a gauche.
\end{itemize}
Nous conjecturons que, lorsque $W$ est un groupe de Coxeter, toutes ces notions co\"{\i}ncident 
avec les notions analogues d\'efinies par Kazhdan-Lusztig dans le cadre des groupes de Coxeter. 
Ces conjectures seront \'enonc\'ees pr\'ecis\'ement dans la partie suivante~\ref{part:coxeter} 
(voir le chapitre~\ref{chapter:conjectures}) et des arguments 
th\'eoriques en faveur de ces conjectures y seront d\'evelopp\'es (voir le chapitre~\ref{chapter:arguments}).

\chapter{B\'eb\'es modules de Verma}\label{chapter:bebe-verma}

\bigskip

%

\boitegrise{\noindent{\bf Notations.} 
{\it Rappelons que $\Pba=\kb[\CCB]$ et $\Kbov=\kb(\CCB)$. On fixe dans ce chapitre 
une extension $K$ du corps $\Kbov_\CG=\kb_P(\pGba_\CG)$. Le morphisme naturel obtenu par la composition 
$\kb[\CCB] \to \kb[\CCB]/\pGba_\CG \injto \Kbov_\CG \injto K$ sera not\'e $\th_K : \kb[\CCB] \to K$.  \indexnot{tz}{\th_K}
Son noyau est $\CG$.}}{0.75\textwidth}

\bigskip

Le but de ce chapitre est de rappeler les r\'esultats de I. Gordon~\cite{gordon} 
sur les repr\'esentations de la $K$-alg\`ebre $K\Hbov=K\Hbov_\CG$. Il n'y aura 
qu'une l\'eg\`ere diff\'erence~: I. Gordon travaillait avec un corps alg\'ebriquement 
clos, tandis qu'ici $K$ n'est pas suppos\'e alg\'ebriquement clos. Cela ne posera 
pas de difficult\'e (gr\^ace \`a un r\'esultat facile de d\'eploiement), 
mais aura quelques cons\'equences utiles concernant 
la partition en cellules de Calogero-Moser bilat\`eres 
(voir le chapitre~\ref{chapter:bilatere}).

\bigskip

\section{Alg\`ebre de Cherednik restreinte}\label{section:def hbar}

\medskip


\subsection{D\'ecomposition PBW}\label{subsection:PBW-restreinte}
Nous appellerons {\it alg\`ebre de Cherednik restreinte 
g\'en\'erique} la $\kb[\CCB]$-alg\`ebre $\Hbov$ d\'efinie par
$$\Hbov=\Hb/\pGba\Hb = \kb[\CCB] \otimes_P \Hb.$$
On d\'eduit imm\'ediatement du th\'eor\`eme~\ref{PBW-0} que~:

\bigskip

\begin{prop}\label{PBW restreint}
L'application $\kb[\CCB] \otimes \kb[V]^\cow \otimes \kb W \otimes \kb[V^*]^\cow \to \Hbov$ 
induite par le produit est un isomorphisme de $\kb[\CCB]$-modules. En particulier, 
$\Hbov$ est un $\kb[\CCB]$-module libre de rang $|W|^3$.
\end{prop}

\bigskip

Nous noterons $\Ab^-$ (respectivement $\Ab^0$, respectivement $\Ab^+$)  \indexnot{A}{\Ab^-, \Ab^0, \Ab^+}  
la sous-$\kb[\CCB]$-alg\`ebre $\kb[\CCB] \otimes \kb[V^*]^\cow$ 
(respectivement $\kb[\CCB] \otimes \kb W$, respectivement $\kb[\CCB]\otimes\kb[V]^\cow$) 
de $\Hbov$. Alors 
\equat\label{dim A}
\text{\it $\Ab^-$, $\Ab^0$ et $\Ab^+$ sont des $\kb[\CCB]$-modules libres de 
rang $|W|$}
\endequat
et l'application
\equat\label{iso KC}
\Ab^+ \otimes_{\kb[\CCB]} \Ab^0 \otimes_{\kb[\CCB]} \Ab^- \longto \Hbov
\endequat
induite par le produit est un isomorphisme de $\kb[\CCB]$-modules 
(d'apr\`es la proposition~\ref{PBW restreint}). 

Posons maintenant $\Hbov^\moins=\Ab^- \Ab^0=\Ab^0\Ab^-$ et $\Hbov^\plus=\Ab^+\Ab^0=\Ab^0 \Ab^+$. 
Ce sont des sous-$\kb[\CCB]$-alg\`ebres de $\Hbov$. 
On obtient alors des isomorphismes de $\kb[\CCB]$-alg\`ebres 
\equat\label{iso H plus}
\Hbov^\plus \simeq \kb[\CCB] \otimes \bigl( \kb[V]^\cow \rtimes W\bigr)
\qquad\text{et}\qquad 
\Hbov^\moins \simeq \kb[\CCB] \otimes \bigl( \kb[V^*]^\cow \rtimes W\bigr).
\endequat
En particulier, 
\equat\label{dim H}
\text{\it $\Hbov^\moins$ et $\Hbov^\plus$ sont des $\kb[\CCB]$-modules libre de 
rang $|W|^2$.}
\endequat
Pour finir, on notera $\Zba$ l'image de $Z$ dans $\Hbov$ (c'est un facteur direct de $\Hbov$). 

\bigskip

\subsection{Automorphismes}\label{subsection:auto-restreinte}
Rappelons que le groupe $\kb^\times \times \kb^\times \times (W^\wedge \rtimes \NC)$ 
agit sur $\Hb$ par automorphismes de $\kb$-alg\`ebres. 
D'apr\`es le lemme~\ref{P stable}, $P$ est stable sous cette action et 
il est facile de v\'erifier que l'id\'eal premier $\pGba$ l'est aussi. 
Ainsi, $\Hbov$ h\'erite d'une action de 
$\kb^\times \times \kb^\times \times (W^\wedge \rtimes \NC)$~: si 
$\t \in \kb^\times \times \kb^\times \times (W^\wedge \rtimes \NC)$ et 
$h \in \Hbov$, on notera encore $\lexp{\t}{h}$ l'image de $h$ 
par l'action de $\t$. Rappelons toutefois que $\t$ 
n'induit pas un automorphisme de la $\kb[\CCB]$-alg\`ebre $\Hbov$
(en effet, $\lexp{\t}{C_s}=\xi\xi'\g(s)^{-1} C_{nsn^{-1}}$ si $\t=(\xi,\xi',\g \rtimes n)$). 

\bigskip

\begin{lem}\label{tout stable}
Les sous-alg\`ebres $\kb W$, 
$\Pba=\kb[\CCB]$, $\Zba$, $\Hbov^\plus$, $\Hbov^\moins$, $\Ab^+$, $\Ab^0$ et $\Ab^-$ 
sont stables sous l'action de $\kb^\times \times \kb^\times \times (W^\wedge \rtimes \NC)$. 
\end{lem}

\bigskip

En particulier, l'alg\`ebre $\Hbov$ h\'erite de la $\NM \times \NM$-graduation. Cette bi-graduation 
sera not\'ee
$$\Hbov=\mathop{\bigoplus}_{(i,j) \in \NM \times \NM} \Hbov^{\NM \times \NM}[i,j].$$
Nous d\'efinissons alors, comme pour $\Hb$, une $\ZM$-graduation 
$$\Hbov=\mathop{\bigoplus}_{i \in \ZM} \Hbov^\ZM[i]$$
ainsi qu'une $\NM$-graduation
$$\Hbov=\mathop{\bigoplus}_{i \in \NM} \Hbov^\NM[i].$$
La $\ZM$-graduation poss\`ede les propri\'et\'es suivantes~:

\bigskip

\begin{lem}\label{PBW hbar}
L'anneau de coefficients $\kb[\CCB]$ est en $\ZM$-degr\'e $0$ et 
les sous-$\kb[\CCB]$-alg\`ebres $\Ab^-$, $\Ab^+$ et $\Ab^0$ sont 
des sous-$\kb[\CCB]$-alg\`ebres gradu\'ees. De plus~:
\begin{itemize}
\itemth{a} $\Ab^- \subset \kb[\CCB] \oplus \bigl( \mathop{\bigoplus}_{i < 0} \Hbov^\ZM[i] \bigr)$.  

\itemth{b} $\Ab^+ \subset \kb[\CCB] \oplus \bigl( \mathop{\bigoplus}_{i > 0} \Hbov^\ZM[i] \bigr)$. 

\itemth{c} $\Ab^0 \subset \Hbov^\ZM[0]$. 
\end{itemize}
\end{lem}

\bigskip

\section{B\'eb\'es modules de Verma, modules simples}\label{subsection:spe-hbar}

\medskip

Il r\'esulte de~(\ref{iso KC}) et du lemme~\ref{PBW hbar} que~:

\begin{itemize}\itemindent5mm
\item[$\bullet$] $K\Ab^-$, $K\Ab^0$ et $K\Ab^+$ sont des sous-$K$-alg\`ebres gradu\'ees 
de $K\Hbov$ et le produit induit un isomorphisme de $K$-espaces vectoriels 
$K\Ab^- \otimes_K K\Ab^0 \otimes_K K\Ab^+ \stackrel{\sim}{\longto} K\Hbov$.

\item[$\bullet$] $K\Ab^- \subset K \oplus \bigl(\bigoplus_{i < 0} K\Hbov^\ZM[i]\bigr)$ et 
$K\Ab^+ \subset K \oplus \bigl(\bigoplus_{i > 0} K\Hbov^\ZM[i]\bigr)$.

\item[$\bullet$] $(K\Ab^-)(K\Ab^0)= (K\Ab^0)(K\Ab^-)$, $(K\Ab^+)(K\Ab^0) = 
(K\Ab^0) (K\Ab^+)$ et $K\Ab^0 \subset K\Hbov^\ZM[0]$.
\end{itemize}

Comme l'a remarqu\'e Gordon~\cite{gordon}, ces hypoth\`eses permettent d'appliquer 
les r\'esultats de Holmes et Nakano~\cite{HN}. 
Rappelons que la $\kb$-alg\`ebre $\kb W$ est d\'eploy\'ee 
(voir le th\'eor\`eme~\ref{deploiement})~: ceci 
implique que la $K$-alg\`ebre $K\Ab^0=K \otimes_{\kb[\CCB]} (\kb[\CCB] \otimes \kb W) = KW$ 
est elle aussi d\'eploy\'ee. 

Tout d'abord, notons $\Jbov^\plus$ (respectivement $\Jbov^\moins$) 
l'id\'eal bilat\`ere nilpotent de $\Hbov^\plus$ (respectivement $\Hbov^\moins$) 
engendr\'e par les \'el\'ements homog\`enes de degr\'e strictement positif 
(respectivement n\'egatif). Si $M$ est un $\kb W$-module, on notera $M^{(+)}$ (respectivement $M^{(-)}$) 
le $\Hbov^\plus$-module (respectivement $\Hbov^\moins$-module) d'espace vectoriel sous-jacent 
$\kb[\CCB]\otimes M$ sur lequel $\Ab^0=\kb[\CCB] \otimes \kb W$ agit naturellement et 
$\Jbov^\plus$ (respectivement $\Jbov^\moins$) agit par $0$. 
On pose alors
$$\fonction{\MCov}{\kb W\module}{\Hbov\module}{M}{\Hbov \otimes_{\Hbov^\moins} M^{(-)}.}$$ \indexnot{M}{\MCov}  
Si $P'$ est une $\Pba$-alg\`ebre (i.e. une $\kb[\CCB]$-alg\`ebre), nous noterons 
$P'\MCov(M)$ le $P' \otimes_\Pba \Hbov$-module $P' \otimes_\Pba \MCov(M)$. 
Si $\chi$ est un caract\`ere irr\'eductible de $W$, on notera $V_\chi$   \indexnot{V}{V_\chi}
un $\kb W$-module irr\'eductible admettant le caract\`ere $\chi$ et on posera pour 
simplifier $\MCov(\chi)=\MCov(V_\chi)$. 
%
On v\'erifie facilement que 
\equat\label{H moins}
\MCov(M) \simeq \Hbov^\plus \otimes_{\kb W} M\qquad\text{comme $\Hbov^\plus$-module}
\endequat
et
\equat\label{AA moins}
\MCov(M) \simeq \Ab^+ \otimes_\kb M\qquad\text{comme $\Ab^+$-module.}
\endequat
Par cons\'equent, la d\'ecomposition $\Hbov^\plus = \Ab^0 \oplus \Jbov^\plus$ induit une 
d\'ecomposition
\equat\label{dec AA0}
\MCov(M) = (\kb[\CCB] \otimes M )\oplus \Jbov^\plus \otimes_{\Ab^0} (\kb[\CCB] \otimes M )\qquad\text{comme $\Ab^0$-module,}
\endequat
ainsi qu'un isomorphisme
\equat\label{MH rad}
\MCov(M)/\Jbov^\plus \MCov(M) \simeq M^{(+)}\qquad\text{comme $\Hbov^\plus$-module.}
\endequat
Nous noterons $i_M : (\kb[\CCB] \otimes M ) \injto \MCov(M)$ et $\pi_M : \MCov(M) \to M^{(+)}$ les morphismes 
d\'eduits de la d\'ecomposition~\ref{dec AA0}. Alors $i_M$ est un 
morphisme de $\Ab^0$-modules tandis que $\pi_M$ est un morphisme 
de $\Hbov^\plus$-modules.

%

\bigskip

\begin{prop}\label{verma}
Si $\chi \in \Irr(W)$, alors $K\MCov(\chi)$ est un $K\Hbov$-module 
ind\'ecomposable, admettant un unique quotient simple, que nous notons $\LCov_K(\chi)$. 
L'application $\Irr(W) \longto \Irr(K\Hbov)$, $\chi \mapsto \LCov_K(\chi)$ 
est bijective et l'alg\`ebre $K\Hbov$ est d\'eploy\'ee. 
\end{prop}

\bigskip

\begin{proof}
Cette proposition est d\'emontr\'ee 
dans~\cite[proposition~4.3]{gordon}, en se basant sur~\cite{HN}, dans le cas o\`u 
$K$ est alg\'ebriquement clos. La preuve se copie ici mot pour mot 
pour obtenir toutes les assertions de la proposition, sauf l'assertion sur le d\'eploiement.

Soit donc $V$ un $KW$-module simple~: le foncteur $K\MCov$ induit 
un morphisme de $K$-alg\`ebres $m : \End_{KW}(V) \to \End_{K\Hbov}(\LCov_K(V))$. 
Il suffit de montrer que $m$ est un isomorphisme. 
Puisque $K\Jbov^+ \subset \Rad(K\Hbov^+)$, 
$K\MCov(V)$, vu comme $K\Hbov^+$-module, admet un unique quotient simple (ici, $V^{(+)}$). 
Ainsi, $\Rad K\MCov(V) \subset K\Jbov^+ K\MCov(V)$. 
Notons $i_V : V \injto K\MCov(V)$ et $\pi_V : K\MCov(V) \surto V$ les morphismes 
induits par la d\'ecomposition~\ref{dec AA0}~: ce sont des morphismes de $KW$-modules 
et $\pi_V \circ i_V = \Id_V$. D'autre part, puique $\Rad K\MCov(V) \subset K\Jbov^+ K\MCov(V)$, 
$\pi_V$ induit un morphisme surjectif de $KW$-modules $\piba_V : \LCov_K(V) \surto V$. 
Notons $\iba_V : V \injto \LCov_K(V)$ le morphisme de $KW$-modules 
compos\'e du morphisme canonique $K\MCov(V) \to \LCov_K(V)$ avec $i_V$~: $\iba_V$ est 
injectif car 
$$\piba_V \circ \iba_V = \Id_V.$$
Soit maintenant $\ph \in \End_{K\Hbov}\LCov_K(V)$. Alors 
$\piba_V \circ \ph \circ \iba_V : V \to V$ est un endomorphisme du $KW$-module $V$. 
Notons-le $\ph_V$. Alors $\psi=\ph - m(\ph_V)$ est un endomorphisme du $K\Hbov$-module $\LCov_K(V)$, 
et il est facile de v\'erifier que 
$\piba_V \circ \psi \circ \iba_V = 0$. Cela signifie que le noyau de $\psi$ est 
non nul. Puisque $\LCov_K(V)$ est simple, il r\'esulte du lemme de Schur que 
$\psi=0$, c'est-\`a-dire $\ph=m(\ph_V)$. Donc 
l'application naturelle $m : \End_{KW}(V) \longto \End_{K\Hbov}(\LCov_K(V))$ est 
un morphisme surjectif de $K$-alg\`ebres. Mais, $V$ \'etant simple, $\End_{KW}(V)=K$ 
car $KW$ est d\'eploy\'ee. 
\end{proof}

\bigskip

\`A travers la proposition~\ref{verma}, nous identifierons le groupe de Grothendieck 
de $K\Hbov$ avec le $\ZM$-module $\ZM \Irr W$~:
\equat\label{groth bar}
\groth(K\Hbov) \simeq \ZM \Irr W.
\endequat

\bigskip

\begin{rema}\label{deployee ou pas}
Il se peut que $K$ soit diff\'erent de $\Kbov_\CG$. 
Cependant, la proposition~\ref{verma} s'applique aussi au cas o\`u $K=\Kbov_\CG$ et 
donc en particulier, $\Kbov_\CG\Hbov$ est une $\Kbov_\CG$-alg\`ebre d\'eploy\'ee. 
Cela implique que, dans la plupart des 
arguments qui vont suivre, on pourra se ramener au cas o\`u $K=\Kbov_\CG$~: en d'autres 
termes, les objets construits ici d\'ependent plus de l'id\'eal premier $\CG$ que 
du choix de l'extension $K$ de $\Kbov_\CG$. 

Par exemple les idempotents primitifs centraux de $K\Hbov$ 
appartiennent \`a $\Zrm(\Kbov_\CG \Hbov)$ (et donc \`a $\Kbov_\CG \Zba$ d'apr\`es la 
proposition~\ref{muller}). D'autre part, l'application 
$\groth(\Kbov_\CG\Hbov) \to \groth(K\Hbov)$, $\isomorphisme{M}_{\Kbov_\CG\Hbov} 
\mapsto \isomorphisme{K \otimes_{\Kbov_\CG} M}_{K\Hbov}$ 
est un isomorphisme.\finl
\end{rema}

\bigskip

\section{Caract\`eres centraux}\label{section:centraux-restreinte}

\medskip

Si $z \in Z$ et si $\chi \in \Irr W$, alors $z$ agit sur $\LCov_K(\chi)$ 
par multiplication par un scalaire $\O_\chi^K(z) \in K$ 
(d'apr\`es le lemme de Schur et le fait que $K\Hbov$ est d\'eploy\'ee). 
Nous noterons parfois $\O_\chi^\CG(z)$ l'\'el\'ement $\O_\chi^K(z)$~: en effet, 
$\O_\chi^K(z)$ appartient \`a $\Kbov_\CG$ et ne d\'epend que de $\CG$ 
(voir la remarque~\ref{deployee ou pas}). 
Pour simplifier, on posera $\O_\chi=\O_\chi^\Kbov$. Puisque $Z$ est entier sur $P$, 
$\O_\chi(z)$ est entier sur $P/\pGba =\kb[\CCB]$ donc $\O_\chi(z) \in \kb[\CCB]$, 
car $\kb[\CCB]$ est int\'egralement clos. On a ainsi d\'efini un morphisme de 
$\kb[\CCB]$-alg\`ebres
$$\O_\chi : Z \longto \kb[\CCB].$$
Notons la caract\'erisation suivante de $\O_\chi^K(z)$, qui d\'ecoule 
imm\'ediatement de la proposition~\ref{verma}~:

\bigskip

\begin{lem}\label{omega nilpotent}
Si $z \in Z$ et $\chi \in \Irr W$, alors $\O_\chi^K(z)$ est l'unique \'el\'ement 
$\kappa \in K$ tel que $z-\kappa$ agisse de fa\c{c}on nilpotente sur 
$K\MCov(\chi)$. 
\end{lem}

\bigskip

\begin{coro}\label{theta omega}
$\O_\chi^K = \th_K \circ \O_\chi$.
\end{coro}

\bigskip

L'\'el\'ement d'Euler est un \'el\'ement particulier de $Z$. La formule, bien 
connue~\cite[\S 3.1(4)]{ggor}, d\'ecrivant son image par $\O_\chi$ est donn\'ee par la 
proposition suivante, dont nous rappelons une preuve pour le cas g\'en\'erique~:

\bigskip

\begin{prop}\label{action euler verma}
 Si $\chi \in \Irr(W)$, alors
\begin{eqnarray*}
\O_\chi(\euler) &=& \DS{\frac{1}{\chi(1)} 
\sum_{s \in \Ref(W)} \e(s)\chi(s)~C_s\quad \in \kb[\CCB]}\\
&=& \DS{\sum_{(\O,j) \in \Omeb_W^\circ} ~\frac{m_{\O,-j}^\chi |\O|e_\O}{\chi(1)}}\cdot K_{\O,j} .
\end{eqnarray*}
\end{prop}

\begin{proof}
Rappelons (voir la section~\ref{section:eulertilde}) que, si $(x_1,\dots,x_n)$ d\'esigne une $\kb$-base de $V^*$ et si 
$(y_1,\dots,y_n)$ d\'esigne sa base duale, alors 
$$\euler=\sum_{i=1}^n x_i y_i + \sum_{s \in \Ref(W)} \e(s)C_ss.$$
Soit $h \otimes_{(\Kbov\Hbov^\moins)} v \in (\Kbov\Hbov) \otimes_{(\Kbov\Hbov^\moins)} (\Kbov V_\chi)^{(-)}=
\MCov(\chi)$. Alors 
$$\euler \cdot (h \otimes_{(\Kbov\Hbov^\moins)} v) = (\euler  h) \otimes_{(\Kbov\Hbov^-)} v =
h\euler \otimes_{(\Kbov\Hbov^-)} v .$$
Mais $y_i \in \Hbov^-$, donc 
$$\euler \cdot (h \otimes_{(\Kbov\Hbov^\moins)} v)= h \otimes_{(\Kbov\Hbov^\moins)} 
\Bigl( \sum_{s \in \Ref(W)} \e(s)C_s s \cdot v \Bigr).$$
Mais $\sum_{s \in \Ref(W)} \e(s)C_s s$ agit sur $\Kbov V_\chi$ par multiplication par 
$$\frac{1}{\chi(1)} \sum_{s \in \Ref(W)} \e(s)\chi(s)~C_s.$$
D'o\`u la premi\`ere \'egalit\'e.

La deuxi\`eme d\'ecoule imm\'ediatement de la d\'efinition des $K_{H,i}$ et 
de manipulations \'el\'ementaires.
\end{proof}

\bigskip

Il d\'ecoule de la preuve de la proposition~\ref{action euler verma} que~:

\bigskip

\begin{coro}\label{coro:action-euler-verma}
Si $m \in K\MCov(\chi)$, alors $\euler \cdot m=\O_\chi^K(\euler) m$.
\end{coro}



\bigskip

\section{Familles de Calogero-Moser}\label{section:familles CM} 

\medskip


\subsection{D\'efinition}  
Si $b \in \blocs(\Zrm(K\Hbov))$, on notera $\Irr_\Hb(W,b)$   \indexnot{I}{\Irr_\Hb(W,b)}  l'ensemble des caract\`eres 
irr\'eductibles $\chi$ de $W$ tels que $\LCov_K(\chi)$ appartienne \`a $\Irr K\Hbov b$. 
Il d\'ecoule de la proposition~\ref{muller} que
$$\blocs(K\Zba)=\blocs(\Zrm(K\Hbov)).$$
Ainsi,
\equat\label{cm familles}
\Irr(W) = \coprod_{b \in \blocs(K\Zba)} \Irr_\Hb(W,b).
\endequat
On appelle {\it $\CG$-famille de Calogero-Moser} (ou {\it $K$-famille de Calogero-Moser}) 
toute partie de $\Irr W$ de la forme 
$\Irr_\Hb(W,b)$, o\`u $b \in \blocs(K\Zba)$. 
Le lemme suivant d\'ecoule du corollaire~\ref{coro:r-blocs} 
et de la remarque~\ref{deployee ou pas}~:

\bigskip

\begin{lem}\label{caracterisation blocs CM}
Soient $\chi$, $\chi' \in \Irr W$. Alors $\chi$ et $\chi'$ sont dans la 
m\^eme $\CG$-famille de Calogero-Moser si et seulement si $\O_\chi^\CG=\O_{\chi'}^\CG$.
D'autre part, l'application
\equat\label{...}
\fonction{\Th_\CG}{\Irr W}{\Upsilon^{-1}(\pGba_\CG)}{\chi}{\Ker~\O_\chi^\CG}   \indexnot{ty}{\Th_\CG}
\endequat
est surjective et ses fibres sont les $\CG$-familles de Calogero-Moser.
\end{lem}

\bigskip

\begin{coro}\label{Q extention kc}
Si $\zG$ est un id\'eal premier de $Z$ au-dessus de $\pGba_\CG$, alors l'inclusion $P \injto Z$ 
induit un isomorphisme $P/\pGba_\CG \longiso Z/\zG$.
\end{coro}

\begin{proof}
D'apr\`es le lemme~\ref{caracterisation blocs CM}, il existe $\chi \in \Irr(W)$ tel que 
$\zG=\Ker(\O_\chi^\CG)$. Or, d'apr\`es le corollaire~\ref{theta omega}, 
$\O_\chi^\CG : Z \to K$ se factorise \`a travers 
un morphisme surjectif $Z \to P/\pGba_\CG$, donc $Z/\zG \simeq P/\pGba_\CG$. 
\end{proof}

\bigskip

\begin{exemple}\label{cm generique}
On appelle {\it famille de Calogero-Moser g\'en\'erique} toute 
$\CG$-famille de Calogero-Moser, o\`u $\CG=0$. Dans ce cas, l'application $\Th_\CG$ sera simplement 
not\'ee $\Th$.  \indexnot{ty}{\Th}  
Toute $\CG$-famille de Calogero-Moser est une union de familles de Calogero-Moser g\'en\'eriques.\finl
\end{exemple}

\bigskip

\begin{exemple}\label{cm c}
Si $c \in \CCB$, on appelle {\it $c$-famille de Calogero-Moser} 
toute $\CG_c$-famille de Calogero-Moser. Dans ce cas, $\O_\chi^{\CG_c}$ sera not\'ee 
$\O_\chi^c$   \indexnot{oz}{\O_\chi^c}  et $\Th_{\CG_c}$ sera  \indexnot{ty}{\Th_c}  not\'ee $\Th_c$.\finl
\end{exemple}

\bigskip

Soit $b \in \blocs(K\Zba)$. Nous noterons $\O_b^K : Z \to K$  \indexnot{oz}{\O_b^K, \O_b,\O_b^c}  le morphisme de 
$\kb[\CCB]$-alg\`ebres $\O_\chi^K$, o\`u $\chi$ est n'importe quel caract\`ere 
dans $\Irr_\Hb(W,b)$. Si $\CG=0$ (i.e. si $K=\kb(\CCB)$), alors $\O_b^K$ sera not\'e $\O_b$ tandis que, 
si $\CG=\CG_c$ (avec $c \in \CCB$), alors $\O_b^K$ sera not\'e $\O_b^c$. 

\bigskip

\section{Caract\`eres lin\'eaires et familles de Calogero-Moser}\label{section:lineaire CM}

\medskip

%
%
%

Si $M$ est un $\Hbov$-module et si 
$\t \in \kb^\times \times \kb^\times \times (W^\wedge \rtimes \NC)$, 
nous noterons $\lexp{\t}{M}$ 
le $\Hbov$-module obtenu de la fa\c{c}on suivante. 
Comme $\kb$-espace vectoriel, $\lexp{\t}{M} \simeq M$ (nous noterons 
$M \to \lexp{\t}{M}$, $m \mapsto \lexp{\t}{m}$ l'isomorphisme en question). 
Si $h \in \Hbov$ et $m \in M$, alors 
$$h \cdot \lexp{\t}{m} = \lexp{\t}{\bigl(\lexp{\t^{-1}}{h} \cdot m\bigr)}.$$
Cela nous d\'efinit un foncteur 
$$\t : \Hbov\module \longto \Hbov\module$$
et, plus g\'en\'eralement, une action de 
$\kb^\times \times \kb^\times \times (W^\wedge \rtimes \NC)$ sur la cat\'egorie 
$\Hbov\module$. De m\^eme, on peut d\'efinir un foncteur
$$\t : \Ab^0\module \longto \Ab^0\module$$
et une action de $\kb^\times \times \kb^\times \times (W^\wedge \rtimes \NC)$ sur la cat\'egorie 
$\Ab^0\module$. Le lemme~\ref{tout stable} montre que le diagramme 
\equat\label{M lineaire}
\diagram
\Ab^0\module \rrto^{\DS{\MC}} \ddto_{\DS{\t}} && \Hbov\module 
\ddto^{\DS{\t}}\\
&& \\
\Ab^0\module \rrto_{\DS{\MC}} && \Hbov\module 
\enddiagram
\endequat
est commutatif. La proposition suivante est imm\'ediate~:

\bigskip

\begin{prop}\label{tau M}
Si $\chi \in \Irr(W)$ et si 
$\t=(\xi,\xi',\g \rtimes g) \in \kb^\times \times \kb^\times \times (W^\wedge \rtimes \NC)$, alors 
$$\lexp{\t}{\MCov(\chi)} \simeq \MCov(\lexp{g}{\chi} \g^{-1}).$$
\end{prop}


\bigskip

\begin{coro}\label{coro:action-tau-L}
Si $\chi \in \Irr(W)$ et si 
$\t=(\xi,\xi',\g \rtimes g) \in \kb^\times \times \kb^\times \times (W^\wedge \rtimes \NC)$ 
stabilise $\CG$, alors 
$$\lexp{\t}{\LCov_K(\chi)} \simeq \LCov_K(\lexp{g}{\chi}\g^{-1}).$$
\end{coro}

\bigskip

\begin{coro}\label{omega lineaire}
Si $\chi \in \Irr(W)$ et si 
$\t=(\xi,\xi',\g \rtimes g) \in \kb^\times \times \kb^\times \times (W^\wedge \rtimes \NC)$, alors 
$$\O_\chi(\lexp{\t}{z})=\lexp{\t}{\bigl(\O_{\lexp{g}{\chi} \g^{-1}}(z)\bigr)}$$
pour tout $z \in Z$. 
\end{coro}

\bigskip

\begin{coro}\label{omega gradue}
Si $\chi \in \Irr(W)$, alors $\O_\chi : Z \to \kb[\CCB]$ est un morphisme 
gradu\'e. En particulier, $\Ker(\O_\chi)$ est un id\'eal homog\`ene de $Z$. 
\end{coro}

\bigskip

\begin{coro}\label{gamma c familles}
Soit $c \in \CCB$, soit $\g$ un caract\`ere lin\'eaire de $W$ et 
soit $\FC$ une $c$-famille de Calogero-Moser. Alors $\FC \g$ est une 
$\g \cdot c$-famille de Calogero-Moser.
\end{coro}

\begin{proof}
Il suffit d'appliquer le corollaire~\ref{omega lineaire} avec $\xi=\xi'=1$ 
et de composer avec le morphisme $\th_c : \kb[\CCB] \to \Kbov_c \simeq \kb$, $C_s \mapsto c_s$~: 
en effet, $\th_c \circ \t = \th_{\g \cdot c}$. 
\end{proof}

\bigskip

\begin{coro}\label{familles lineaires}
Soit $\t=(\xi,\xi',\g \rtimes g) \in \kb^\times \times \kb^\times \times (W^\wedge \rtimes \NC)$ 
et soit $\FC$ une $\CG$-famille de Calogero-Moser. 
Si $\t$ stabilise $\CG$, alors $\FC \g$ est une $\CG$-famille de 
Calogero-Moser.
\end{coro}

\begin{proof}
Cela d\'ecoule imm\'ediatement du lemme~\ref{caracterisation blocs CM}, 
et des corollaires~\ref{omega lineaire} et~\ref{theta omega}.
\end{proof}

\bigskip

\begin{coro}\label{familles lineaires generiques}
Soient $\g$ un caract\`ere lin\'eaire de $W$ 
et soit $\FC$ une famille de Calogero-Moser {\bfit g\'en\'erique}. Alors  
$\FC\g$ est une famille de Calogero-Moser g\'en\'erique.
\end{coro}

\bigskip

\begin{coro}\label{ordre 2}
Supposons que toutes les r\'eflexions de $W$ sont d'ordre $2$. Soit $\FC$ une 
$\CG$-famille de Calogero-Moser. Alors $\FC \e$ est une $\CG$-famille de 
Calogero-Moser (rappelons que $\e$ est le d\'eterminant).
\end{coro}

\begin{proof}
Posons $\t=(-1,1,\e \rtimes 1) \in \kb^\times \times \kb^\times \times (W^\wedge \rtimes \NC)$. 
Alors $\t$ agit trivialement sur $\kb[\CCB]$. 
On peut donc appliquer le corollaire~\ref{familles lineaires}.
\end{proof}

\bigskip

\begin{exemple}[Familles g\'en\'eriques
et caract\`eres lin\'eaires]\label{lineaire}
Soient $\g \in W^\wedge$ et $\chi \in \Irr(W)$ dans la m\^eme famille de Calogero-Moser 
{\it g\'en\'erique}. Alors $\O_\chi(\euler)=\O_\g(\euler)$, ce qui implique que, pour tout 
$s \in \Ref(W)$, $\chi(s)=\g(s) \chi(1)$. En d'autres termes, toutes les r\'eflexions 
de $W$ sont dans le centre de $\chi$ (c'est-\`a-dire le sous-groupe distingu\'e 
de $W$ form\'e des \'el\'ements qui agissent sur $V_\chi$ comme une homoth\'etie), 
et donc le centre de $\chi$ est $W$ tout entier. Cela force $\chi=\g$. 

Par cons\'equent, un caract\`ere lin\'eaire est seul dans sa famille de Calogero-Moser 
g\'en\'erique. Ce r\'esultat s'applique en particulier \`a $\unb_W$ et $\e$ 
et est compatible avec le corollaire~\ref{familles lineaires generiques}.\finl
\end{exemple}

\bigskip

%
%
%

\section{Dimension gradu\'ee, $\bb$-invariant}\label{section:dim graduee}

\medskip

D'apr\`es la proposition~\ref{graduation idem}, les \'el\'ements de 
$\blocs(K\Zba)$ sont de $\ZM$-degr\'e $0$. En particulier, si $b \in \blocs(K\Zba)$, 
alors $K\Zba b$ est un $K$-espace vectoriel gradu\'e. Le but de cette section est 
d'\'etudier cette graduation. 

\bigskip

\begin{theo}\label{dim graduee bonne}
Soit $b \in \blocs(K\Zba)$ et posons $\FC=\Irr_\Hb(W,b)$. Alors~:
\begin{itemize}
\itemth{a} $\DS{\dim_k^\grad bK\Zba =\sum_{\chi \in \FC} f_{\chi}(\tb^{-1}) ~f_{\chi}(\tb)}$.

\itemth{b} Il existe un unique $\chi \in \FC$ de $\bb$-invariant minimal~: 
notons-le $\chi_\FC$.

\itemth{c} Le coefficient de $\tb^{\bb_{\chi_\FC}}$ dans 
$f_{\chi_\FC}(\tb)$ est \'egal \`a $1$.

\itemth{d} $bK\Hbov e$ est une enveloppe projective de $\LCov_K(\chi_\FC)$.

\itemth{e} L'alg\`ebre $\End_{K\Hbov}(K\MCov(\chi_\FC))$ est un quotient de $bK\Zba$~: 
en particulier, elle est commutative.
\end{itemize}
\end{theo}

\bigskip

En vertu de~\ref{chi 1}, on obtient la cons\'equence imm\'ediate suivante~:

\bigskip

\begin{coro}\label{dim bonne}
Soit $b \in \blocs(K\Zba)$. Alors
$$\dim_K bK\Zba= \sum_{\chi \in \Irr_\Hb(W,b)} \chi(1)^2.$$
\end{coro}

\bigskip

\begin{rema}\label{generalisation gordon}
Le th\'eor\`eme~\ref{dim graduee bonne} 
g\'en\'eralise~\cite[th\'eor\`eme 5.6]{gordon} et le corollaire~\ref{dim bonne} 
g\'en\'eralise~\cite[corollaire 5.8]{gordon}. Il est \`a noter que, comme 
nous l'a fait remarquer Iain Gordon, une erreur s'est gliss\'ee 
dans~\cite[th\'eor\`eme 5.6]{gordon}~: avec les notations en vigueur dans~\cite{gordon}, 
il y est \'ecrit que $p_\chi(\tb)=\tb^{\bb_{\chi^*}-\bb_\chi}f_{\chi}(\tb)f_{\chi^*}(\tb^{-1})$. 
En fait, la formule correcte est, toujours avec les notations de~\cite{gordon}, 
$p_\chi(\tb)=f_{\chi^*}(\tb)f_{\chi^*}(\tb^{-1})$~: la diff\'erence avec 
notre r\'esultat provient du fait que nous avons pris une $\ZM$-graduation 
oppos\'ee \`a celle de~\cite[\S 4.1]{gordon}, ce qui revient \`a \'echanger $V$ 
avec $V^*$ et donc \`a remplacer, dans la formule, $\chi$ par $\chi^*$. 
Nous remercions Iain Gordon pour nous avoir apport\'e ces pr\'ecisions.\finl
\end{rema}

%

\bigskip

\begin{proof}[D\'emonstration du th\'eor\`eme~\ref{dim graduee bonne}]
Le $K\Hbov^\moins$-module gradu\'e $K\Hbov^\moins e$ est isomorphe 
\`a $K[V^*]^\cow$, la sous-alg\`ebre $K[V^*]^\cow$ agissant par multiplication 
\`a gauche et $W$ agissant de fa\c{c}on naturelle. Notons $M_i$ le sous-$K\Hbov^\moins$-module 
gradu\'e de $K[V^*]^\cow$ form\'e des \'el\'ements de degr\'e $\le -i$. Notons $N$ 
le plus grand entier tel que $M_N \neq 0$. Alors 
$$0 = M_{N+1} \varsubsetneq M_N \subset M_{N-1} \subset \cdots 
\subset M_1 \varsubsetneq M_0=K[V^*]^\cow.$$
\'Ecrivons $f_\chi(\tb)=\sum_{i=0}^N \gamma_{\chi,i} \tb^i$ avec $\gamma_{\chi,i} \in \NM$. 
Alors
$$M_i/M_{i+1} \simeq \mathop{\bigoplus}_{\chi \in \Irr(W)} 
\bigl(KV_\chi^{(-)}\langle i \rangle\bigr)^{\g_{\chi,i}} $$
comme $K\Hbov^\moins$-module gradu\'e.
Posons $P_i=K\Hbov \otimes_{K\Hbov^\moins} M_i$. Puisque $K\Hbov$ est un $K\Hbov^\moins$-module 
libre de rang $|W|$, le foncteur $K\Hbov \otimes_{K\Hbov^\moins} -$ est exact et donc 
$$0 = P_{N+1} \varsubsetneq P_N \subset P_{N-1} \subset \cdots \subset 
P_1 \varsubsetneq P_0=K\Hbov e,$$
avec 
$$P_i/P_{i+1} \simeq \mathop{\bigoplus}_{\chi \in \Irr(W)} 
\bigl(K\MCov(\chi)\langle i \rangle\bigr)^{\g_{\chi,i}}.$$
En tronquant par l'idempotent central $b$, on obtient 
$$0 = bP_{N+1} \subset bP_N \subset bP_{N-1} \subset \cdots \subset 
bP_1 \subset bP_0=bK\Hbov e,\leqno{(*)}$$
avec 
$$bP_i/bP_{i+1} \simeq \mathop{\bigoplus}_{\chi \in \FC} 
\bigl(K\MCov(\chi)\langle i \rangle\bigr)^{\g_{\chi,i}}.\leqno{(**)}$$
En tronquant \`a nouveau par l'idempotent $e$, on obtient 
$$0 = ebP_{N+1} \subset ebP_N \subset ebP_{N-1} \subset \cdots \subset 
ebP_1 \subset ebP_0=ebK\Hbov e,$$
avec 
$$ebP_i/ebP_{i+1} \simeq \mathop{\bigoplus}_{\chi \in \FC} 
\bigl(eK\MCov(\chi)\langle i \rangle\bigr)^{\g_{\chi,i}}.$$
Ainsi,
$$\dim_K^\grad ebK\Hbov e = \sum_{\chi \in \FC} \sum_{i=0}^N \g_{\chi,i} \tb^{-i} 
\dim_K^\grad(eK\MCov(\chi))= \sum_{\chi \in \FC} f_\chi(\tb^{-1})  
\dim_K^\grad(eK\MCov(\chi)).$$
Puisque $ebK\Hbov e \simeq bK\Zba$ (comme espace vectoriels gradu\'es), l'\'enonc\'e (a) d\'ecoule 
du lemme facile suivant~:

\medskip

\begin{quotation}
\begin{lem}\label{lem:dim-e-verma}
$\dim_K^\grad(eK\MCov(\chi)) = f_\chi(\tb)$. 
\end{lem}

\begin{proof}
D'apr\`es~(\ref{AA moins}), on a 
\equat\label{m}
\isomorphisme{K\MCov(\chi)}_{KW}^\grad = \isomorphisme{K[V]^\cow}_{KW}^\grad \cdot \chi.
\endequat
D'o\`u le r\'esultat. 
\end{proof}
\end{quotation}

\medskip

Montrons maintenant les autres assertions. L'alg\`ebre d'endomorphismes de $bK\Hbov e$ est 
isomorphe \`a $bK\Zba$ (voir le corollaire~\ref{coro:endo-bi-c}) et donc est une alg\`ebre commutative 
locale. Cela montre donc que $bK\Hbov e$ est ind\'ecomposable (et bien s\^ur projectif), 
donc il admet un unique quotient simple. Compte tenu de la filtration $(*)$ et de la 
d\'ecomposition $(**)$, cela montre (b), (c) et (d). 

\medskip

Pour finir, $K\MCov(\chi_\FC)$ est un quotient du module projectif $bK\Hbov e$ (notons 
$M$ le noyau de ce quotient) et, 
puisque les endomorphismes du $K\Hbov$-module $bK\Hbov e$ sont induits par multiplication 
par les \'el\'ements de $K\Zba$, cela montre que $M$ est stable sous l'alg\`ebre d'endomorphismes 
de $bK\Hbov e$. On obtient donc un morphisme de $K$-alg\`ebres 
$bK\Zba \to \End_{K\Hb}(K\MCov(\chi_\FC))$ qui est surjectif en raison de la projectivit\'e 
de $bK\Hbov e$. 
\end{proof}

\bigskip
%
%

\begin{coro}\label{polynome caracteristique}
Soit $z \in Z$ et notons $\carac_z(\tb) \in P[\tb]$ le polyn\^ome caract\'eristique de 
la multiplication par $z$ dans le $P$-module $Z$. Alors
$$\carac_z(\tb) \equiv \prod_{\chi \in \Irr(W)} (\tb - \O_\chi(z))^{\chi(1)^2} \mod \pGba.$$
\end{coro}

\begin{proof}
Si $b \in \blocs(\kb(\CCB)\Zba)$, alors $z-\O_b(z)$ est un endomorphisme 
nilpotent de $b\kb(\CCB)\Zba$. Donc le polyn\^ome caract\'eristique de 
$z$ sur $b\kb(\CCB)\Zba$ est $(\tb - \O_b(z))^{\dim_{\kb(\CCB)} b\kb(\CCB)\Zba}$. 
Par cons\'equent, 
$$\carac_z(\tb) \equiv \prod_{b \in \blocs(\kb(\CCB)\Zba)} 
(\tb - \O_b(z))^{\dim_{\kb(\CCB)} b\kb(\CCB)\Zba } \mod \pGba.$$
Puisque $\O_b(z)=\O_\chi(z)$ pour tout $\chi \in \Irr_\Hb(W,b)$, le r\'esultat 
d\'ecoule de~(\ref{cm familles}) et du corollaire~\ref{dim bonne}.
\end{proof}

\bigskip

\bigskip

\begin{coro}\label{multiplicite 1}
Soit $\g : W \longto \kb^\times$ un caract\`ere lin\'eaire. Alors 
$Z$ est nette sur $P$ en $\Ker(\O_\g)$.
\end{coro}

\begin{proof}
En effet, si on note $b_\g$ l'idempotent primitif de $\kb(\CCB)\Zba$ associ\'e \`a $\g$, 
alors $\Irr_\Hb(W,b_\g)=\{\g\}$ d'apr\`es l'exemple~\ref{lineaire}, ce qui implique, 
d'apr\`es le corollaire~\ref{dim bonne}, que 
$\dim_{\kb(\CCB)}(b_\g \kb(\CCB)\Zba)=1$. 

Posons $\zG_\g=\Ker(\O_\g)$ (on a $\zG_\g \cap P = \pGba$). Alors $Z/\zG_\g \simeq \kb[\CCB]$ et 
donc $Z_{\zG_\g}/\zG_\g Z_{\zG_\g}\simeq \kb(\CCB)$. Mais, d'autre part, 
$Z_{\zG_\g}/\pGba Z_{\zG_\g} = b_\g \kb(\CCB) \Zba$. Donc 
$\dim_{\kb(\CCB)}(Z_{\zG_\g}/\pGba Z_{\zG_\g})=1$, ce qui implique que 
$\pGba Z_{\zG_\g} = \zG_\g Z_{\zG_\g}$, ce qui est exactement le r\'esultat attendu.
\end{proof}

\bigskip

\section{G\'eom\'etrie}\label{section:geometrie CM}

\medskip

La composition 
\equat\label{section upsilon}
\diagram
 \kb[\CCB] \xyinj[rr] && Z \xysur[rr]^{\DS{\O_b}} && \kb[\CCB]
\enddiagram
\endequat
est l'identit\'e, ce qui signifie que le morphisme de $\kb$-vari\'et\'es 
$\O_b^\sharp : \CCB \longto \ZCB$  \indexnot{oz}{\O_b^\sharp}  induit par $\O_b$ est une section du morphisme 
$\pi \circ \Upsilon : \ZCB \longto \CCB$ (voir le diagramme~\ref{diagramme geometrie}). 
Le lemme~\ref{caracterisation blocs CM} nous dit que l'application 
$$\fonctio{\blocs(\kb(\CCB)\Zba)}{\Upsilon^{-1}(\pGba)}{b}{\Ker(\O_b)}$$
est une bijection ou, en termes g\'eom\'etriques, que 
les composantes irr\'eductibles de $\Upsilon^{-1}(\CCB \times 0)$ 
sont en bijection avec $\blocs(\kb(\CCB)\Zba)$, via l'application $b \mapsto \O_b^\sharp(\CCB)$. 
Par cons\'equent~:

\bigskip

\begin{prop}\label{generique particulier}
Soit $c \in \CCB$. Alors les conditions suivantes sont \'equivalentes~:
\begin{itemize}
\itemth{1} $|\blocs(\kb(\CCB)\Zba)|=|\blocs(\Kbov_c\Zba)|$.

\itemth{2} $|\Upsilon^{-1}_c(0)|$ est \'egal au nombre de composantes irr\'eductibles de 
$\Upsilon^{-1}(\CCB \times 0)$.

\itemth{3} Tout \'el\'ement de $\Upsilon^{-1}_c(0)$ appartient 
\`a une unique composante irr\'eductible de $\Upsilon^{-1}(\CCB \times 0)$. 

\itemth{4} Si $b$ et $b'$ sont deux \'el\'ements distincts de $\blocs(\kb(\CCB)\Zba)$, 
alors $\th_c \circ \O_b \neq \th_c \circ \O_{b'}$.
\end{itemize}
\end{prop}

\bigskip

Si $c \in \CCB$, nous dirons que 
$c$ est {\it g\'en\'erique} s'il v\'erifie l'une des conditions \'equivalentes 
de la proposition~\ref{generique particulier}. Il sera dit {\it particulier} 
dans le cas contraire. 
Nous noterons $\CCB_\gen$ (respectivement $\CCB_\parti$) l'ensemble 
des $c \in \CCB$ g\'en\'eriques (respectivement particuliers). 

\bigskip

\begin{coro}\label{particulier ferme}
$\CCB_\gen$ est un ouvert de Zariski non vide de $\CCB$. 
D'autre part, $\CCB_\parti$ est un ferm\'e de Zariski de $\CCB$. Si $W \neq 1$, 
il est purement de codimension $1$ et contient $0$. 

De plus, $\CCB_\gen$ et $\CCB_\parti$ sont stables sous l'action de 
$\kb^\times \times \kb^\times \times (W^\wedge \rtimes \NC)$. 
\end{coro}

\begin{proof}
La stabilit\'e sous l'action de 
$\kb^\times \times \kb^\times \times (W^\wedge \rtimes \NC)$ est \'evidente. 
Le fait que $\CCB_\gen$ (respectivement $\CCB_\parti$) soit ouvert 
(respectivement ferm\'e) d\'ecoule de la proposition~\ref{codimension un}(2). 
Lorsque $W \neq 1$, le caract\`ere trivial est seul dans sa famille de Calogero-Moser g\'en\'erique 
(voir l'exemple~\ref{lineaire}) alors que sa $0$-famille de Calogero-Moser est $\Irr(W)$. 
Cela montre que $0 \in \CCB_\parti$ et, en vertu de la proposition~\ref{codimension un}(1), 
purement de codimension $1$. 
%
%
\end{proof}

\bigskip

On d\'eduit de l'exemple~\ref{lineaire} que~:

\bigskip

\begin{coro}\label{generique lineaire}
Si $c \in \CCB$ est g\'en\'erique, alors tout caract\`ere lin\'eaire de $W$ 
est seul dans sa $c$-famille de Calogero-Moser.
\end{coro}

\bigskip


\begin{coro}\label{multiplicite 1 generique}
Soit $\g : W \longto \kb^\times$ un caract\`ere lin\'eaire et supposons $c$ g\'en\'erique. 
Alors $Z$ est nette sur $P$ en $\Ker(\O_\g^c)$.
\end{coro}

\bigskip


Terminons par une courte \'etude de la lissit\'e de $\ZCB$.
Soit $b \in \blocs(K\Zba)$ et notons $\bar{\zG}_b$  \indexnot{z}{\bar{\zG}_b}  l'id\'eal premier 
de $Z$ \'egal au noyau de $\O_b^K : Z \to K$. D'apr\`es~\cite[th\'eor\`eme~5.6]{gordon}, on a~:

\bigskip

\begin{prop}[Gordon]\label{prop lissite}
Si l'anneau $Z$ est r\'egulier en $\bar{\zG}_b$, alors $|\Irr_\Hb(W,b)|=1$. 
De plus, 
$$K\Hbov b \simeq \Mat_{|W|}(bK\Zba)$$
et $bK\Zba$ est une $K$-alg\`ebre locale de dimension finie et de corps r\'esiduel $K$.
\end{prop}

\bigskip
%
%
%
%

\bigskip

Soit maintenant $c \in \CCB$ et supposons que $\CG=\CG_c$. 
Notons $z_b$ le point de $\Upsilon_c^{-1}(0) \subset \ZCB_c \subset \ZCB$ 
correspondant \`a $b$. 

\bigskip

\begin{prop}\label{lissite generique}
Avec les notations ci-dessus, les assertions suivantes sont \'equivalentes~:
\begin{itemize}
\itemth{1} $\ZCB$ est lisse en $z_b$.

\itemth{2} $\ZCB_c$ est lisse en $z_b$.
\end{itemize}
\end{prop}

\begin{proof}
Rappelons pour commencer le lemme suivant~:

\bigskip

\begin{quotation}
{\small
\begin{lem}\label{somme tangents}
Soit $\ph : \YCB \to \XCB$ un morphisme de $\kb$-vari\'et\'es (pas forc\'ement r\'eduites), 
soit $y \in \YCB$ et soit $x=\ph(y)$. 
On suppose qu'il existe un morphisme de $\kb$-vari\'et\'es 
$\s : \XCB \to \YCB$ tel que $y=\s(x)$ et $\ph \circ \s = \Id_\XCB$. Alors 
$$\TC_y(\YCB)=\TC_y(\ph^*(x)) \oplus \TC_y(\s(\XCB)).$$
Ici, $\TC_y(\YCB)$ d\'esigne l'espace tangent \`a la $\kb$-vari\'et\'e $\YCB$ et 
$\ph^*(x)$ d\'esigne la fibre (sch\'ematique) 
en $x$ de $\ph$, vue comme $\kb$-sous-vari\'et\'e ferm\'ee 
de $\YCB$, non n\'ecessairement r\'eduite.
\end{lem}}
\end{quotation}

\bigskip

Soit $\chi \in \Irr_\Hb(W,b)$. Le morphisme de vari\'et\'es 
$\O_\chi^\sharp : \CCB \to \ZCB$ qui est une section du morphisme 
$\pi \circ \Upsilon : \ZCB \to \CCB$. D'autre part, par hypoth\`ese, 
$z_b = \O_\chi^\sharp(c)$. En vertu du lemme~\ref{somme tangents} ci-dessus, 
on a
$$\TC_{z_b}(\ZCB) = \TC_{z_b}(\ZCB_c) \oplus \TC_{z_b}(\O_\chi^\sharp(\CCB)).$$
Puisque $\TC_{z_b}(\O_\chi^\sharp(\CCB)) \simeq \TC_c(\CCB)$, la proposition d\'ecoule de la lissit\'e de 
$\CCB$ et du fait que 
$\dim(\ZCB)=\dim(\ZCB_c)+\dim(\CCB)$.
\end{proof}

\bigskip

\`A la suite des travaux d'Etingof-Ginzburg~\cite{EG}, Ginzburg-Kaledin~\cite{GK},
Gordon~\cite{gordon} et Bellamy~\cite{bellamy g4}, une classification compl\`ete 
des groupes de r\'eflexions complexes $W$ tels qu'il existe $c \in \CCB$ 
tel que $\ZCB_c$ soit lisse a \'et\'e obtenue. Notons que les assertions 
{\it ``Il existe $c \in \CCB$ tel que $\ZCB_c$ soit lisse''} et 
{\it ``l'anneau $\kb(\CCB) \otimes_{\kb[\CCB]} Z=\kb(\CCB)Z$ est r\'egulier''} 
sont \'equivalentes. Rappelons ici le r\'esultat, 
sachant que l'on peut se ramener tr\`es facilement au cas o\`u $W$ est 
irr\'eductible~:

\bigskip

\begin{theo}[Etingof-Ginzburg, Ginzburg-Kaledin, Gordon, Bellamy]\label{les lisses}
Supposons $W$ irr\'eductible. Alors $\kb(\CCB)Z$ est un anneau r\'egulier 
si et seulement si on est dans un des deux cas suivants~:
\begin{itemize}
 \itemth{1} $W$ est de type $G(d,1,n)$, avec $d$, $n \ge 1$.

\itemth{2} $W$ est le groupe not\'e $G_4$ dans la classification de 
Shephard-Todd.
\end{itemize}
\end{theo}

\bigskip

Il d\'ecoule des travaux d'Etingof-Ginzburg~\cite{EG}, Gordon~\cite{gordon}, 
Bellamy~\cite{bellamy g4} que la proposition suivante est vraie~:

\bigskip

\begin{prop}\label{lissite en 0}
Les assertions suivantes sont \'equivalentes~:
\begin{itemize}
\itemth{1} Il existe $c \in \CCB$ tel que $\ZCB_c$ est lisse.

\itemth{2} Il existe $c \in \CCB$ tel que les points de 
$\Upsilon_c^{-1}(0)$ sont lisses dans $\ZCB_c$.
\end{itemize}
\end{prop}

\bigskip

\`A l'heure actuelle, la preuve de ce fait repose sur la classification de 
Shephard-Todd des groupes de r\'eflexions complexes. 

\section{Probl\`emes, questions}\label{section:problemes 3}

\medskip

Ce chapitre soul\`eve plusieurs questions num\'eriques concr\`etes, 
qu'il n'est pas facile de r\'esoudre (voir cependant~\cite{gordon B}, 
\cite{gordon martino},~\cite{bellamy},~\cite{martino 2},~\cite{thiel}).

\bigskip

\begin{probleme}\label{probleme generique}
Calculer les familles de Calogero-Moser g\'en\'erique.
\end{probleme}

\bigskip

\begin{probleme}\label{probleme c}
Calculer les $c$-familles de Calogero-Moser.
\end{probleme}

\bigskip

Un probl\`eme beaucoup plus facile est le suivant~:

\bigskip

\begin{probleme}\label{probleme particulier}
D\'eterminer $\CCB_\parti$.
\end{probleme}

\bigskip

Les exemples connus jusqu'\`a pr\'esent sugg\`erent la question suivante~:

\bigskip

\begin{question}\label{question hyperplan}
Est-ce que $\CCB_\parti$ est une r\'eunion d'hyperplans~? 
Si oui, sont-ce les {\bfit hyperplans essentiels} d\'efinis par 
Chlouveraki~\cite[\S 3]{chlouveraki}~?
\end{question}

\bigskip

C'est vrai si $W$ est ab\'elien (par produits directs, on se ram\`ene au cas o\`u $\dim_\kb(V)=1$, 
qui est trait\'e dans le chapitre~\ref{chapitre:rang 1}). 
D'apr\`es le corollaire~\ref{particulier ferme}, $\CCB_\parti$ est un ferm\'e 
de $\CCB$, de codimension $1$ et stable par homoth\'eties~: ce ne peut alors \^etre qu'une r\'eunion 
d'hyperplans si $\dim_\kb(\CCB) \in \{1,2\}$. D'autres probl\`emes sont soulev\'es par la construction des b\'eb\'es modules de Verma.

\bigskip

\begin{probleme}\label{dec verma}
D\'eterminer l'image de $K\MCov(\chi)$ dans $\groth(K\Hbov)\simeq \ZM\Irr(W)$.
\end{probleme}

\bigskip

Rappelons que
le $K\Hbov$-module simple 
$\LCov_K(\chi)$ h\'erite d'une graduation, et donc d'une structure de $KW$-module 
gradu\'e~:

\bigskip

\begin{probleme}\label{dim L}
D\'eterminer la classe $\isomorphisme{\LCov_K(\chi)}_{KW}^\grad \in \groth(KW)[\tb,\tb^{-1}]$, 
ou plus simplement $\isomorphisme{\LCov_K(\chi)}_{KW} \in \groth(KW)$, 
ou encore plus simplement $\dim_K \LCov_K(\chi)$.
\end{probleme}

\bigskip

Terminons par quelques questions de nature g\'eom\'etriques, en \'echo 
\`a la proposition~\ref{lissite generique} et au corollaire~\ref{lissite en 0}~:

\bigskip

\begin{question}\label{question:zc-z}
Soit $z \in \ZCB_c \subset \ZCB$. Est-ce que $z$ est lisse dans $\ZCB$ si et seulement si 
il est lisse dans $\ZCB_c$~? Plus g\'en\'eralement, est-ce que les singularit\'es de 
$\ZCB$ et $\ZCB_c$ en $z$ sont \'equivalentes~?
\end{question}

\bigskip

La proposition~\ref{lissite generique} (et sa preuve) montre que, si $z \in \Upsilon^{-1}(0)$, 
alors les singularit\'es de $\ZCB$ et $\ZCB_c$ en $z$ sont \'equivalentes.

La question suivante est sugg\'er\'ee par le corollaire~\ref{lissite en 0}~: il serait 
souhaitable de l'\'etudier de mani\`ere ind\'ependante de la classification.

\bigskip

\begin{question}\label{question:lissite-0}
Est-ce que $\ZCB_c$ est lisse si et seulement si $\ZCB_c$ est lisse en tout point de 
$\Upsilon^{-1}(0)$~?
\end{question}

\chapter{Conjecture de Martino}\label{chapter:martino}

L'objet de ce chapitre est de rappeler l'\'enonc\'e de la conjecture de Martino~\cite{martino} 
reliant les familles de Calogero-Moser et celles de Hecke (voir la d\'efinition~\ref{defi:famille-rouquier}), 
de rappeler l'\'etat des connaissances dans sa preuve et d'apporter quelques arguments th\'eoriques en sa faveur.

\bigskip

\boitegrise{
{\it Dans ce chapitre, on suppose que $\kb=\CM$ et on fixe une fonction invariante par conjugaison $c : \Ref(W) \to \CM$ 
(ainsi, $c \in \CCB$). On d\'efinit alors une famille $k=(k_{\O,j})_{(\O,j) \in \Omeb_W^\circ}$ par 
les \'egalit\'es du~\S\ref{section:hyperplans-variables} et {\bfit on suppose dans tout ce chapitre 
que $k_{\O,j} \in \RM$ pour tout $(\O,j) \in \Omeb_W^\circ$} (ainsi, $k \in \CCB_\RM$). \\
\hphantom{A} Notons $k^\sharp=(k_{\O,j}^\sharp)_{(\O,j) \in \Omeb_W^\circ}$ l'\'el\'ement de $\CCB_\RM$ d\'efini par 
$k_{\O,j}^\sharp=k_{\O,-j}$ (les indices $j$ \'etant vus modulo $e_\O$). Pour finir, {\bfit on suppose 
que l'hypoth\`ese~\libertesymetrie~ est v\'erifi\'ee} (voir~\S\ref{sub:hecke}).}}{0.75\textwidth}

\bigskip

\section{\'Enonc\'e et cas connus}

\medskip

Rappelons ici l'\'enonc\'e donn\'e dans~\cite[conjecture~2.7]{martino}~:

\bigskip

\begin{conjecturem}[Martino]
Si $b \in \blocs(\Zba_c)$, alors il existe un idempotent 
central $b^\HC$ de $\OC^\cyclo[\qb^\RM]\HC_W^\cyclo(k^\sharp)$ v\'erifiant les deux propri\'et\'es suivantes~:
\begin{itemize}
\itemth{a} $\Irr_\Hb(W,b)=\Irr_\HC(W,b^\HC)$~;

\itemth{b} $\dim_\CM(\Zba b) = \dim_{F(\qb^\RM)}\bigl(F(\qb^\RM)\HC_W^\cyclo(k^\sharp)b^\HC\bigr)$.
\end{itemize}
En particulier, toute $c$-famille de Calogero-Moser est une r\'eunion de $k^\sharp$-familles de Hecke. 
\end{conjecturem}

\bigskip

Cette conjecture a \'et\'e v\'erifi\'ee dans certains cas en calculant s\'epar\'ement les partitions 
en familles de Hecke et en familles de Calogero-Moser. Aucun lien th\'eorique n'a pour l'instant 
\'et\'e \'etabli permettant d'imaginer une preuve de cette conjecture ind\'ependante de la classification. 

\bigskip

\begin{theo}[Bellamy, Chlouveraki, Gordon, Martino]
Supposons que $W$ soit de type $G(de,e,n)$ et supposons que, si $n=2$, alors $e$ est impair ou $d=1$. 
Alors la conjecture~\FAM~est vraie.
\end{theo}

\bigskip

La preuve de ce th\'eor\`eme d\'ecoule des travaux suivants~:
\begin{itemize}
\item M. Chlouveraki a calcul\'e la partition en familles de Hecke dans~\cite{chlouveraki B} et~\cite{chlouveraki D}.

\item Lorsque $e=1$, la partition en familles de Calogero-Moser a \'et\'e calcul\'ee par I. Gordon~\cite{gordon B} 
pour des valeurs rationnelles de $k$ (en utilisant des sch\'emas de Hilbert). Ce r\'esultat a \'et\'e \'etendu 
\`a toutes les valeurs de $k$ par M. Martino~\cite{martino 2}  par des m\'ethodes alg\'ebriques. 

\item La combinatoire de M. Chlouveraki et celle de I. Gordon a \'et\'e compar\'ee par M. Martino~\cite{martino} 
pour d\'emontrer la conjecture~\FAM~lorsque $e=1$.

\item Lorsque $e$ est quelconque, la partition en familles de Calogero-Moser a \'et\'e calcul\'ee 
par G. Bellamy~\cite{bellamy} pour les valeurs rationnelles de $k$, car ce calcul reposait sur le r\'esultat 
de I. Gordon. La m\'ethode se g\'en\'eralisait toutefois \`a toutes les valeurs de $k$, une fois 
le r\'esultat de Martino~\cite{martino 2} \'etabli.
\end{itemize}

\bigskip

\begin{rema}\label{rema:generique-niet}
Il \'etait conjectur\'e par M. Martino que, lorsque $c$ est g\'en\'erique, alors 
les partitions en $c$-familles de Calogero-Moser et en $k^\sharp$-familles de Hecke co\"{\i}ncident. 
Un contre-exemple \`a cela a \'et\'e r\'ecemment trouv\'e par U. Thiel~\cite{thiel}. 
Gr\^ace \`a ses calculs, U. Thiel a aussi obtenu quelques nouveaux cas de la conjecture~\FAM~parmi 
les groupes exceptionnels~: une publication d'U. Thiel viendra bient\^ot dresser une liste 
pr\'ecise des cas connus.

Il faut noter que M. Chlouveraki a calcul\'e les partitions en familles de Hecke des 
groupes exceptionnels~\cite{chlouveraki LNM} dans tous les cas et que les programmes 
d'U. Thiel calculent la partition en familles de Calogero-Moser.\finl
\end{rema}

\bigskip

\section{Arguments th\'eoriques} 

\medskip

Le corollaire~\ref{dim bonne} montre que~:

\bigskip

\begin{prop}\label{prop:a-b}
Dans la conjecture~\FAM, l'assertion (a) implique l'assertion (b).
\end{prop}

\bigskip

\begin{proof}
Reprenons les notations de l'\'enonc\'e de la conjecture~\FAM~($b$, $b^\HC$,...). 
Puisque l'alg\`ebre $F(\qb^\RM)\HC_W^\cyclo(k^\sharp)$ est d\'eploy\'ee, 
on a 
$$\dim_{F(\qb^\RM)}\bigl(F(\qb^\RM)\HC_W^\cyclo(k^\sharp)b^\HC\bigr)=\sum_{\chi \in \Irr_\HC(W,b^\HC)} \chi(1)^2.$$
Mais d'autre part, il d\'ecoule du corollaire~\ref{dim bonne} que
$$\dim_\CM(\Zba_c b)=\sum_{\chi \in \Irr_\Hb(W,b)} \chi(1)^2.$$
D'o\`u le r\'esultat.
\end{proof}

\bigskip

\begin{rema}\label{rema:c-constant}
L'argument th\'eorique le plus probant en faveur de la conjecture~\FAM~est le suivant. 
Il a \'et\'e montr\'e que, si $\chi$ et $\chi'$ sont dans la m\^eme $c$-famille de Calogero-Moser 
(respectivement $k^\sharp$-famille de Hecke), alors $\O_\chi^c(\euler)=\O_{\chi'}^c(\euler)$ 
(respectivement $c_\chi(k^\sharp)=c_{\chi'}(k^\sharp)$)~: voir le lemme~\ref{caracterisation blocs CM} 
(respectivement le lemme~\ref{lem:c-contant}). Or, il d\'ecoule de la proposition~\ref{action euler verma} 
et de la d\'efinition de $c_\chi(k^\sharp)$ (voir~\S\ref{sub:cas-cyclo}) que
\equat\label{eq:o=c}
\O_\chi^c(\euler_c)=c_\chi(k^\sharp).
\endequat
Si cet invariant num\'erique n'est pas suffisant pour d\'eterminer en g\'en\'eral 
la partition en familles de Calogero-Moser, il est quand m\^eme assez fin.\finl
\end{rema}

\bigskip

Un dernier argument est donn\'e par la proposition suivante, qui fait une synth\`ese du lemme~\ref{lem:rouquier-ordre-2} 
et du corollaire~\ref{ordre 2}~:

\bigskip

\begin{prop}\label{prop:ordre 2}
{\bfit Supposons que toutes les r\'eflexions de $W$ soient d'ordre $2$}. Si $\FC$ est une 
$c$-famille de Calogero-Moser (respectivement une $k^\sharp$-famille de Hecke), alors $\FC \e$ est une 
$c$-famille de Calogero-Moser (respectivement une $k^\sharp$-famille de Hecke).
\end{prop}


\chapter{Cellules de Calogero-Moser bilat\`eres}\label{chapter:bilatere}

\bigskip

Conform\'ement aux notations en vigueur dans cette partie, nous fixons un id\'eal premier 
$\CG$ de $\kb[\CCB]$. Le but de ce chapitre est d'\'etudier les $\CG$-cellules de Calogero-Moser bilat\`eres, 
en lien avec la th\'eorie des b\'eb\'es module de Verma rappel\'ee dans le chapitre 
pr\'ec\'edent. 

\bigskip

\section{Choix}\label{section:premiers}

\medskip

Esp\'erant relier les cellules de Calogero-Moser bilat\`eres avec celles de Kazdahn-Lusztig, 
nous sommes contraints de pr\'eciser la fa\c{c}on de choisir un id\'eal premier $\rGba_\CG$ 
de $R$ au-dessus de $\pGba_\CG$~: nous n'avons pas de r\'eponses pr\'ecises \`a cette question, 
mais nous allons donner ici quelques \'el\'ements pour guider ce choix.


\medskip

Rappelons que $\pGba$ d\'esigne l'id\'eal premier de $P$ correspondant 
\`a la sous-vari\'et\'e ferm\'ee irr\'eductible $\CCB \times \{0\} \times \{0\}$ 
de $\PCB$ (voir la section~\ref{section:def hbar}). 
Il y a alors plusieurs id\'eaux premiers de $Z$ au-dessus de $\pGba$~: 
ils sont d\'ecrits dans le lemme~\ref{caracterisation blocs CM}, 
qui nous dit qu'ils sont en bijection avec l'ensemble des familles de Calogero-Moser 
g\'en\'eriques de $W$. Rappelons aussi que, d'apr\`es l'exemple~\ref{lineaire}, 
le caract\`ere trivial $\unb_W$ de $W$ forme \`a lui seul une 
famille de Calogero-Moser g\'en\'erique~: nous noterons $\zGba$  \indexnot{z}{\zGba}  
l'id\'eal premier de $Z$ qui lui est associ\'e~:
$$\zGba=\Ker(\O_{\unb_W}).$$
Nous poserons ensuite $\qGba=\copie(\zGba)$.  \indexnot{q}{\qGba}  

\bigskip

\begin{lem}\label{qba}
$\qGba$ est l'unique id\'eal premier de $Q$ au-dessus de $\pGba$ contenant 
$\eulerq-\sum_{H \in \AC} e_H K_{H,0}$ (rappelons que $\eulerq=\copie(\euler)$). 
L'alg\`ebre $Q$ est \'etale sur $P$ au-dessus de $\qGba$.
\end{lem}

\begin{proof}
Il d\'ecoule du corollaire~\ref{multiplicite 1} que $Q$ est nette sur $P$ au-dessus de $\qGba$~: 
comme de plus $Q$ est un $P$-module libre, donc plat et que, en caract\'eristique 
nulle, toutes les extensions de corps sont s\'eparables, on en d\'eduit 
que $Q$ est \'etale sur $P$ au-dessus de $\qGba$. 
Le fait que $\euler -\sum_{H \in \AC} e_H K_{H,0} \in \zGba$ d\'ecoule de 
la proposition~\ref{action euler verma}. 
Maintenant, si $\zGba'$ est un id\'eal 
premier de $Z$ au-dessus de $\pGba$ et contenant $\euler-\sum_{H \in \AC} e_H K_{H,0}$, 
alors il existe $\chi \in \Irr(W)$ tel que $\zGba'=\Ker(\O_\chi)$. 
En particulier, $\O_\chi(\euler)=\sum_{H \in \AC} e_H K_{H,0}$ et 
donc $\O_\chi(\euler)=\O_{\unb_W}(\euler)$. Or, il est d\'emontr\'e dans 
l'exemple~\ref{lineaire} que cela implique que $\chi=\unb_W$.
\end{proof}

\bigskip

Reprenons maintenant les notations du chapitre pr\'ec\'edent~\ref{chapter:bebe-verma}.
%
Comme pr\'ec\'edemment, 
le lemme~\ref{caracterisation blocs CM}, nous dit que l'ensemble 
des id\'eaux premiers de $Q$ au-dessus de $\pGba_\CG$ est en bijection 
avec les $\CG$-familles de Calogero-Moser. Nous noterons $\zGba_\CG$ 
l'id\'eal maximal correspondant \`a la $\CG$-famille contenant 
le caract\`ere trivial de $W$~:
$$\zGba_\CG=\Ker(\O_{\unb_W}^{\Kbov_\CG})$$  \indexnot{z}{\zGba_\CG}  
et nous posons $\qGba_\CG=\copie(\zGba_\CG)$.   \indexnot{q}{\qGba_\CG}  

\bigskip

\begin{lem}\label{qba c}
$\qGba_\CG=\qGba + \CG Q$.
\end{lem}

\begin{proof}
Le morphisme $\O_{\unb_W} : Z \to \kb[\CCB]$ induit un isomorphisme 
$Z/\zGba \simeq \kb[\CCB]$ et, 
puisque $\zGba_\CG$ contient $\CG$, $\zGba+\CG Z$ 
est bien un id\'eal premier de $Z$, correspondant au sous-sch\'ema ferm\'e 
irr\'eductible $\CCB_\CG$ de $\CCB$. 
\end{proof}

\bigskip

\begin{coro}\label{qba c plus}
$\qGba_\CG$ est l'unique id\'eal premier de $Q$ au-dessus de $\pGba_\CG$ contenant 
$\eulerq-\sum_{H \in \AC} e_H K_{H,0}$. 
\end{coro}

\bigskip

\section{Cellules bilat\`eres}\label{cellules et familles}

\bigskip

\boitegrise{{\bf Hypoth\`ese.} 
{\it Dans la suite de cette partie, 
nous fixons un id\'eal premier $\rGba_\CG$  \indexnot{r}{\rGba_\CG}  de $R$ au-dessus de $\qGba_\CG$. 
Rappelons que $\Kbov_\CG=k_P(\pGba_\CG)=k_{\kb[\CCB]}(\CG)$ et nous posons 
$\Lbov_\CG=k_Q(\qGba_\CG)$  \indexnot{L}{\Lbov_\CG, \Lbov, \Lbov_c}  et $\Mbov_\CG=k_R(\rGba_\CG)$.   \indexnot{M}{\Mbov_\CG, \Mbov, \Mbov_c}  
Par coh\'erence avec 
des notations d\'ej\`a utilis\'ees, nous noterons $\Kbov$, $\Lbov$ et $\Mbov$ 
(respectivement $\Kbov_c$, $\Lbov_c$ et $\Mbov_c$) les corps $\Kbov_\CG$, $\Lbov_\CG$ et  
$\Mbov_\CG$ lorsque $\CG=0$ (respectivement $\CG=\CG_c$ avec $c \in \CCB$).\\
\hphantom{A} Le groupe de d\'ecomposition (respectivement d'inertie) de $\rGba_\CG$ sera not\'e 
$\Dba_\CG$  \indexnot{D}{\Dba_\CG, \Dba,\Dba_c}  (respectivement $\Iba_\CG$)~:  \indexnot{I}{\Iba_\CG,\Iba, \Iba_c}  
on d\'efinit de m\^eme les notations $\Dba$, $\Iba$, $\Dba_c$ et $\Iba_c$.
}}{0.75\textwidth}

\bigskip

\subsection{Th\'eorie de Galois}\label{sub:galois-bilatere}
Rappelons que, d'apr\`es le corollaire~\ref{Q extention kc}, l'injection canonique 
$P/\pGba_\CG \injto Q/\qGba_\CG$ est un isomorphisme, ce qui signifique que 
$\Kbov_\CG=\Lbov_\CG$. Puisque $\Gal(\Mbov_\CG/\Lbov_\CG)=\Dba_\CG/\Iba_\CG$ (voir le th\'eor\`eme~\ref{bourbaki}), on 
en d\'eduit que
\equat\label{eq:dba-h}
(\Dba_\CG \cap H)/(\Iba_\CG \cap H) \simeq \Dba_\CG/\Iba_\CG.
\endequat
En particulier,
\equat\label{eq:iba-h}
(\Dba_\CG \cap H) \Iba_\CG =\Dba_\CG.
\endequat
Par ailleurs, puisque l'alg\`ebre $Q$ est nette au-dessus de $P$ en $\qGba$ (voir le lemme~\ref{qba}), 
il d\'ecoule du th\'eor\`eme~\ref{raynaud} que $\Iba \subset H$. Combin\'e \`a~\ref{eq:iba-h}, on obtient
\equat\label{eq:iba-dba-h}
\Iba \subset \Dba \subset H.
\endequat
Comme le montre~\ref{exemple:dba-0}, ce dernier r\'esultat ne se g\'en\'eralise pas forc\'ement \`a $\Iba_\CG$.

Pour terminer avec les propri\'et\'es galoisiennes imm\'ediates, notons que, d'apr\`es 
les corollaires~\ref{Q extention kc} et~\ref{dgh igh}, on a
\equat\label{eq:dba-g-h}
\Iba_\CG\backslash G /H = \Dba_\CG \backslash G/H.
\endequat

\bigskip

\subsection{Cellules bilat\`eres et graduation} 
Soit $\CGt$ l'id\'eal homog\`ene maximal contenu dans $\CG$ (i.e. $\CGt=\bigoplus_{i \ge 0} \CG \cap \kb[\CCB]^\NM[i]$). 
Alors $\CGt$ est un id\'eal premier de $\kb[\CCB]$ (voir le lemme~\ref{lem:homogeneise-premier}). Notons 
$\rGba_\CGt$ l'id\'eal homog\`ene maximal contenu dans $\rGba_\CG$~: c'est un id\'eal premier de $R$ au-dessus de 
$\qGba_\CGt$ 
(voir le corollaire~\ref{coro:homogeneise-premier}). 
Il d\'ecoule alors de la proposition~\ref{prop:cellules-homogeneise} que~:

\bigskip

\begin{lem}\label{lem:cellules-bilateres-homogeneise}
On a $\Iba_\CG=\Iba_\CGt$. 
Les $\CG$-cellules de Calogero-Moser bilat\`eres et les $\CGt$-cellules de Calogero-Moser bilat\`eres 
co\"{\i}ncident.
\end{lem}

\bigskip

\subsection{Cellules bilat\`eres et familles} 
Par construction, $\Mbov_\CG$ (qui est une extension 
finie de $\Kbov_\CG=k_P(\pGba_\CG)=\Frac(\kb[\CCB]/\CG)$) est une $\kb[\CCB]$-alg\`ebre et 
$\Mbov_\CG\Hb=\Mbov_\CG\Hbov$. On peut donc parler de 
$\Mbov_\CG$-familles de Calogero-Moser (qui sont les $\CG$-familles de Calogero-Moser)~: 
le th\'eor\`eme~\ref{theo:calogero} 
nous dit qu'il y a une bijection entre les $\rGba_\CG$-cellules de Calogero-Moser 
et les $\CG$-familles de Calogero-Moser~: si $b \in \blocs(R_{\rGba_\CG} Z)$, 
alors cette bijection fait correspondre $\calo_{\rGba_\CG}(b)$ avec 
$\Irr_\Hb(W,\bba)$, o\`u $\bba$ d\'esigne l'image de $b$ dans $\Mbov_\CG\Hbov$ 
(notons que, en raison du d\'eploiement de $\Kbov_\CG\Hbov$, 
l'\'el\'ement $\bba$ appartient en fait \`a $\Kbov_\CG\Hbov$). 

\bigskip

\boitegrise{\noindent{\bf Terminologie, notation.} 
{\it 
Si $b \in \blocs(\Mbov_\CG Z)$, nous dirons que la $\CG$-cellule de Calogero-Moser bilat\`ere 
$\calo_{\rGba_\CG}(b)$ {\bfit recouvre} la $\CG$-famille de Calogero-Moser $\Irr_\Hb(W,b)$. 
Si $\G$ est une $\CG$-cellule de Calogero-Moser bilat\`ere, nous noterons 
$\Irr_\G^\calo(W)$  \indexnot{I}{\Irr_\G^\calo(W)}  la $\CG$-famille de Calogero-Moser recouverte par $\G$. 
L'ensemble des $\CG$-cellules de Calogero-Moser bilat\`eres 
sera not\'e  \indexnot{C}{{{{^\calo{{\mathrm{Cell}}_{LR}^\CG(W)}}}}}  $\cmcellules_{LR}^\CG(W)$.}}{0.75\textwidth}

\bigskip

\begin{rema}
Il est bien s\^ur clair que la notion de $\CG$-cellule de Calogero-Moser 
bilat\`ere d\'epend du choix de l'id\'eal premier $\rGba_\CG$ 
au-dessus de $\qGba_\CG$. Cependant, si $\rGba_\CG'$ est un autre id\'eal 
premier de $R$ au-dessus de $\qGba_\CG$, alors il existe $h \in H$ tel que 
$\rGba_\CG'=h(\rGba_\CG)$ et les $\rGba_\CG'$-cellules de Calogero-Moser 
se d\'eduisent des $\rGba_\CG$-cellules de Calogero-Moser en faisant agir $h$ 
par translation sur $W  \stackrel{\sim}{\longleftrightarrow} G/H$. 
Le choix de l'id\'eal $\rGba_\CG$ sera donc implicite tout au long 
de ce chapitre.\finl 
\end{rema}

\bigskip

Les liens entre $\CG$-cellules de Calogero-Moser bilat\`eres et $\CG$-familles sont renforc\'es par 
le th\'eor\`eme suivant~:

\bigskip

\begin{theo}\label{theo cellules familles}
Soient $w \in W$, $b \in \blocs(R_{\rGba_\CG} Q)$ et $\chi \in \Irr(W)$. Alors~:
\begin{itemize}
\itemth{a} Le groupe de d\'ecomposition $\Dba_\CG$ agit trivialement sur $\cmcellules_{LR}^\CG(W)$. 

%

\itemth{b} La $\CG$-cellule de Calogero-Moser bilat\`ere de $w$ est associ\'ee 
\`a la $\CG$-famille de Calogero-Moser de $\chi$ si et seulement si 
$w^{-1}(\rGba_\CG) \cap Q = \copie(\Ker(\O_\chi^\CG))$.

\itemth{c} La $\CG$-cellule de Calogero-Moser bilat\`ere de $w$ est associ\'ee 
\`a la $\CG$-famille de Calogero-Moser de $\chi$ si et seulement si 
$(w(q) \mod \rGba_\CG) = \O_\chi^\CG(\copie^{-1}(q)) \in \Mbov_\CG=k_R(\rGba_\CG)$ pour tout $q \in Q$. 

\itemth{d} Si $\G$ est une $\CG$-cellule de Calogero-Moser bilat\`ere, 
alors $|\G| = \sum_{\chi \in \Irr_\G^\calo(W)} \chi(1)^2$. 
\end{itemize}
\end{theo}

\begin{proof}
(a) d\'ecoule de~\ref{eq:dba-g-h}.

\medskip

(b) Notons $\omeba_w : Q \to R/\rGba_\CG$ le morphisme de $P$-alg\`ebres qui, 
\`a $q \in Q$, associe l'image de $\o_w(q)=w(q) \in R$ dans $R/\rGba_\CG$. 
Alors la $\CG$-cellule de Calogero-Moser bilat\`ere de $w$ est associ\'ee 
\`a la $\CG$-famille de Calogero-Moser de $\chi$ si et seulement si 
$\omeba_w = \O_\chi$. Or, d'apr\`es le lemme~\ref{caracterisation blocs CM}, 
ceci est \'equivalent \`a dire que $\Ker(\omeba_w) = \copie(\Ker(\O_\chi^\CG))$. 
Puisque $\Ker(\omeba_w)=w^{-1}(\rGba_\CG) \cap Q$, le r\'esultat en d\'ecoule.

\medskip

(c) d\'ecoule de (b) et du fait que $Q=(w^{-1}(\rG) \cap Q) + \kb[\CCB]$ 
(voir le corollaire~\ref{Q extention kc}) tandis que 
(d) d\'ecoule des corollaires~\ref{dim bonne} et~\ref{dim centre r}. 
\end{proof}

\bigskip

\begin{coro}\label{coro:famille-semicontinu}
Soit $\CG'$ un id\'eal premier de $\kb[\CCB]$ contenu dans $\CG$ et soit $\rGba_{\CG'}$ un id\'eal premier de $R$ 
au-dessus de $\pGba_{\CG'}$ et contenu dans $\rGba_{\CG'}$. Alors toute $\CG$-cellule de Calogero-Moser bilat\`ere 
est une r\'eunion de $\CG'$-cellules de Calogero-Moser bilat\`eres. De plus, si $\G$ est une 
$\CG$-cellule de Calogero-Moser bilat\`ere et si $\G=\G_1 \coprod \cdots \coprod \G_r$ o\`u les 
$\G_i$ sont des $\CG'$-cellules de Calogero-Moser bilat\`eres, alors 
$$\Irr_\G^{\calo,\CG}(W)= \Irr_{\G_1}^{\calo,\CG'}(W) \coprod \cdots \coprod \Irr_{\G_r}^{\calo,\CG'}(W).$$
\end{coro}

\bigskip


\bigskip

\begin{coro}\label{w0 epsilon}
Supposons que toutes les r\'eflexions de $W$ sont d'ordre $2$ et que 
$w_0=-\Id_V \in W$. Si $\G$ est une $\CG$-cellule de Calogero-Moser bilat\`ere 
recouvrant la $\CG$-famille de Calogero-Moser $\FC$, alors 
$w_0\G=\G w_0$ est une $\CG$-cellule de Calogero-Moser bilat\`ere recouvrant la 
$\CG$-famille de Calogero-Moser $\e\FC$.
\end{coro}

\begin{proof}
Tout d'abord, $w_0\G=\G w_0$ est une $\CG$-cellule de Calogero-Moser bilat\`ere 
d'apr\`es l'exemple~\ref{cellules w0} tandis que $\e \FC$ est une 
$\CG$-famille de Calogero-Moser d'apr\`es le corollaire~\ref{ordre 2}. 

Soient $w \in \G$, $\chi \in \FC$ et $q \in Q$. D'apr\`es 
le th\'eor\`eme~\ref{theo cellules familles}(c), il suffit de montrer que 
$(ww_0(q) \mod \rGba_\CG) = \O_{\chi\e}^{k_R(\rGba_\CG)}(q)$. Posons 
$\t_0=(-1,1,\e) \in \kb^\times \times \kb^\times \times W^\wedge$. 
D'apr\`es la proposition~\ref{tau 0 in G} on a 
$w_0(q)=\lexp{\t_0}{q}$ pour tout $q \in Q$. De plus, d'apr\`es le corollaire~\ref{omega lineaire}, 
$\O_{\chi\e}(q)=\lexp{\t_0}{\O_\chi(\lexp{\t_0}{q})}$ (car $\t_0$ est d'ordre $2$). 
Puisque $\t_0$ agit trivialement sur $\kb[\CCB]$, on a $\O_{\chi\e}(q)=\O_\chi(\lexp{\t_0}{q})$. 
Il suffit donc de montrer que 
$(w(\lexp{\t_0}{q}) \mod \rGba_\CG) = \O_\chi^{k_R(\rGba_\CG)}(\lexp{\t_0}{q})$. Mais cela 
d\'ecoule encore du th\'eor\`eme~\ref{theo cellules familles}(c). 
\end{proof}

\bigskip

%
%
%

\begin{exemple}[Lissit\'e]\label{cellules lisses}
Si l'anneau $Q$ est r\'egulier en $\qGba_b$ et si $\chi$ d\'esigne 
l'unique \'el\'ement de $\Irr_\Hb(W,b)$ (en vertu de la proposition~\ref{prop lissite}), 
alors $|\calo_\rG(b)|=\chi(1)^2$ (th\'eor\`eme~\ref{theo cellules familles}(d).\finl
\end{exemple}

\bigskip

\begin{rema}\label{rem:gen-spe}
Si $\rGba_\CG$ et $\rGba$ sont choisis de telle sorte que $\rGba \subset \rGba_\CG$, 
alors $\Iba \subset \Iba_\CG$ et donc 
toute $\CG$-cellule de Calogero-Moser bilat\`ere  est r\'eunion de 
cellules de Calogero-Moser bilat\`eres g\'en\'eriques.
C'est le pendant ``cellulaire'' de l'\'enonc\'e \'equivalent sur les familles.\finl
\end{rema} 

\bigskip

Si $\g$ est un caract\`ere lin\'eaire, alors 
il est seul dans sa famille de Calogero-Moser g\'en\'erique (voir l'exemple~\ref{lineaire}) 
et la cellule de Calogero-Moser bilat\`ere g\'en\'erique correspondante ne contient qu'un 
seul \'el\'ement (en vertu du th\'eor\`eme~\ref{theo cellules familles}(d)). 
Notons $w_\g$ cet unique \'el\'ement. Alors, d'apr\`es le th\'eor\`eme 
\ref{theo cellules familles}(b), on a 
\equat\label{wgamma}
w_\g^{-1}(\rGba) \cap Q = \Ker(\O_\g).
\endequat

\begin{coro}\label{nettete triviale}
On a $w_{\unb_W}=1$. En d'autres termes, $1$ est seul dans sa cellule de Calogero-Moser 
g\'en\'erique et recouvre la famille de Calogero-Moser g\'en\'erique 
du caract\`ere trivial $\unb_W$ (qui est un singleton).
\end{coro}

\begin{proof}
D'apr\`es le th\'eor\`eme~\ref{theo cellules familles} et~(\ref{wgamma}), $w_{\unb_W}$ est 
l'unique \'el\'ement $w \in W$ tel que $w^{-1}(\rGba) \cap Q=\Ker(\O_{\unb_W})=\qGba$. 
Puisque $\rGba \cap Q =\qGba$, on a $w_{\unb_W}=1$.
\end{proof}

%

\bigskip

\begin{prop}\label{nettete lineaire}
Soit $\g \in W^\wedge$. Alors $\Iba \subset w_\g H w_\g^{-1}$. 
\end{prop}

\begin{proof}
Nous allons donner deux preuves de ce fait. Tout d'abord, 
$w_\g$ est seul dans sa cellule de Calogero-Moser g\'en\'erique, donc 
$\Iba w_\g H = w_\g H$. D'o\`u le r\'esultat.

\medskip

Donnons maintenant une deuxi\`eme preuve. D'apr\`es le corollaire~\ref{multiplicite 1}, 
$Q$ est nette sur $P$ en $\Ker(\O_\g)=w_\g^{-1}(\rGba) \cap Q$. Donc, d'apr\`es le 
th\'eor\`eme~\ref{raynaud}, $I_{w_\g^{-1}(\rGba)}\subset H$, 
ce qui est le r\'esultat annonc\'e car $I_{w_\g^{-1}(\rGba)}=w_\g^{-1} \Iba w_\g$.
\end{proof}

\bigskip

\begin{rema}
L'action de $H$ sur $W \stackrel{\sim}{\longleftrightarrow} G/H$ stabilise 
l'\'el\'ement neutre (c'est-\`a-dire que 
$H$ stabilise $\euler$). Cela montre que l'\'enonc\'e du corollaire~\ref{nettete triviale} 
ne d\'epend pas du choix de $\rGba$, comme il se doit.\finl
\end{rema}

\bigskip

\section{Probl\`emes, questions}\label{section:problemes 4} 

\medskip

Les notions d\'efinies dans ce chapitre requi\`erent une bonne connaissance 
de l'anneau $R$, du groupe de Galois $G$ mais aussi la connaissance de groupes 
d'inertie. D'apr\`es le corollaire~\ref{Q extention kc}, 
$Q/\qGba' \simeq \kb[\CCB]$ pour tout id\'eal premier 
$\qGba'$ de $Q$ au-dessus de $\pGba$. On sait aussi que $\rGba$ est un 
id\'eal $(\NM \times \NM)$-homog\`ene de $R$ (voir le corollaire~\ref{premier homogene}), 
et que la composante $(\NM \times \NM)$-homog\`ene de $R/\rGba$ de bi-degr\'e 
$(0,0)$ est \'egale \`a $\kb$. Il est donc naturel de se poser la question 
suivante~:

\bigskip

\begin{question}\label{R/r}
Est-ce que $R/\rGba \simeq \kb[\CCB]$~?
\end{question}

\bigskip

Puisque l'extension de corps $k_R(\rGba)/\kb(\CCB)$ est galoisienne de 
groupe de Galois $\Dba/\Iba$, 
cette question est \'equivalente \`a~: est-ce que $\Dba=\Iba$~? 
Nous verrons dans le chapitre~\ref{chapitre:rang 1} que c'est vrai 
si $\dim_\kb(V)=1$ (voir~\ref{eq:iso-r}). Par produits 
directs, cela montre le r\'esultat lorsque $W$ est ab\'elien. 

Si la r\'eponse \`a la question~\ref{R/r} est positive, alors $R_c/\rGba_c \simeq \kb$ et 
$\Dba_c=\Iba_c$ pour tout $c \in \CCB$. Ceci est bien s\^ur vrai lorsque $\kb$ est alg\'ebriquement 
clos, mais un tel r\'esultat \'eviterait de devoir se ramener au cas 
alg\'ebriquement clos.

\bigskip

\cbstart

\begin{question}
Supposons que $W$ soit un groupe de Coxeter. Est-il possible de choisir $\rGba$ de sorte 
que $\s_\Hb(\rGba)=\rGba$~?
\end{question}

\bigskip

Rappelons ici que $\s_\Hb$ est l'automorphisme de $\Hb$ induit par l'isomorphisme de 
$W$-modules $\s : V \longiso V^*$ et que l'on note encore $\s_\Hb$ l'automorphisme de $R$ qu'il induit.

\cbend

\chapter{Cellules de Calogero-Moser \`a gauche et \`a droite}\label{chapter:gauche}

Comme dans toute cette partie, $\CG$ d\'esigne un id\'eal premier de $\kb[\CCB]$. 
L'objectif de ce chapitre est d'\'etudier les $\CG$-cellules de Calogero-Moser \`a gauche (ou \`a droite), 
en lien avec les repr\'esentations de l'alg\`ebre $\Hb^\gauche$. Dans ce cadre, tout comme avec les 
cellules bilat\`eres, nous aurons \`a notre disposition des {\it modules de Verma \`a gauche} (ou {\it \`a droite}). 
Ces modules de Verma nous permettront de d\'efinir les notions de $\calo$-caract\`eres 
$\CG$-cellulaires, dont nous esp\'erons qu'ils co\"{\i}ncident avec les $\kl$-caract\`eres cellulaires 
(lorsque $W$ est un groupe de Coxeter). 

L'essentiel de ce chapitre traitera des $\CG$-cellules de Calogero-Moser {\it \`a gauche} et des modules de Verma 
{\it \`a gauche}~: les r\'esultats se transposent sans probl\`eme dans le contexte {\it \`a droite}.

\bigskip

\section{Choix}\label{section:choix-gauche}

\medskip

Les questions de choix d'id\'eal premier de $R$ au-dessus de $\pG_\CG^\gauche$ sont tout autant 
cruciales que dans le cas des cellules bilat\`eres. Nous allons d\'efinir et utiliser les modules de Verma 
{\it \`a gauche} pour guider ce choix.

Rappelons que $P^\gauche \simeq \kb[\CCB \times V/W]$. Il d\'ecoule de la d\'ecomposition PBW 
que l'on a un isomorphisme de $P^\gauche$-modules
\equat\label{PBW left}
\Hb^\gauche \simeq \kb[\CCB \times V] \otimes \kb W \otimes \kb[V^*]^\cow.
\endequat
On retrouve alors l'alg\`ebre $\Hbov^\moins = \kb[\CCB] \otimes (W \ltimes \kb[V^*]^\cow)$ comme sous-$\kb$-alg\`ebre de $\Hb^\gauche$. 
Cela nous permet de d\'efinir des {\it modules de Verma \`a gauche}~: si $\chi \in \Irr(W)$, on pose 
$$\MC^\gauche(\chi)=\Hb^\gauche \otimes_{\Hbov^\moins} (\kb[\CCB] \otimes V_\chi)^{(-)}.$$\indexnot{Ma}{\MC^\gauche(\chi)}
Alors $\MC^\gauche(\chi)$ est un $\Hb^\gauche$-module \`a gauche et 
\equat\label{dimension left}
\text{\it  $\MC^\gauche(\chi)$ est libre de rang $|W|\chi(1)$ comme $P^\gauche$-module.}
\endequat

\bigskip

\begin{lem}\label{absolue simplicite}
Si $\g$ est un caract\`ere {\bfit lin\'eaire} de $W$, alors 
$\Kb^\gauche\MC^\gauche(\g)$ est un $\Kb^\gauche\Hb^\gauche$-module 
absolument simple.
\end{lem}

\begin{proof}
Notons que, d'apr\`es l'exemple~\ref{exemple lisse}, $\qG_\singulier \cap P \not\subset \pG^\gauche$. Ainsi, 
le th\'eor\`eme~\ref{lissite et simples} s'applique et donc, 
si $\rG$ d\'esigne un id\'eal premier de $R$ au-dessus de $\pG^\gauche$, alors les 
$k_R(\rG)\Hb^\gauche$-modules simples sont absolument simples et de dimension $|W|$. 
Donc, pour des raisons de dimension (voir~\ref{dimension left}), 
$k_R(\rG)\MC^\gauche(\g)$ est forc\'ement absolument simple et le lemme en d\'ecoule.
\end{proof}

\bigskip

Fixons donc un caract\`ere lin\'eaire $\g$ de $W$. Par le lemme~\ref{absolue simplicite}, 
l'alg\`ebre d'endomorphismes de $\Kb^\gauche\MC^\gauche(\g)$ est \'egale \`a 
$\Kb^\gauche$, ce qui 
induit un morphisme de $P$-alg\`ebres $\O_\g^\gauche : Z \to \Kb^\gauche$ \indexnot{oz}{\O_\g^\gauche}
dont la restriction \`a $P$ est le morphisme canonique $P \to P^\gauche$. 
Puisque $Z$ est entier sur $P$, l'image de $\O_\g^\gauche$ est enti\`ere sur $P^\gauche$ et contenue dans 
$\Kb^\gauche=\Frac(P^\gauche)$. Puisque $P^\gauche \simeq \kb[\CCB \times V/W]$, $P^\gauche$ est int\'egralement clos, 
cela force $\O_\g^\gauche$ \`a se factoriser \`a travers $P^\gauche$. 
Nous poserons 
$$\zG^\gauche = \Ker(\O_1^\gauche)\quad\text{et}\quad \qG^\gauche=\copie(\zG^\gauche).$$\indexnot{za}{\zG^\gauche}\indexnot{qa}{\qG^\gauche}
Alors~:

\bigskip

\begin{prop}\label{q left}
L'id\'eal $\qG^\gauche$ de $Q$ v\'erifie les propri\'et\'es suivantes~:
\begin{itemize}
 \itemth{a} $\qG^\gauche$ est un id\'eal premier de $Q$ au-dessus de $\pG^\gauche$.

\itemth{b} $\qG^\gauche \subset \qGba$.

\itemth{c} $P^\gauche=P/\pG^\gauche \simeq Q/\qG^\gauche$.
\end{itemize}
\end{prop}

\begin{proof}
Puisque $\Kb^\gauche$ est int\`egre, $\qG^\gauche$ est premier. Puisque la restriction de $\O_1^\gauche$ \`a $P$ 
est le morphisme canonique $P \to P^\gauche$, $\qG^\gauche \cap P = \pG^\gauche$. Cela montre (a). 

\medskip

Par construction, $\MC^\gauche(\g)/\pGba \MC^\gauche(\g)=\MCov(\g)$ et donc 
le morphisme $\O_\g : Z \to \Pba=P/\pGba$ se factorise \`a travers les morphismes $\O_\g^\gauche : Z \to P^\gauche$ et 
$P^\gauche \to \Pba$, d'o\`u (b).

\medskip

Pour finir, l'isomorphisme (c) d\'ecoule du fait que l'image de $\O_\g$ est $P^\gauche$.
\end{proof}

\bigskip

La proposition~\ref{q left} nous permet donc de faire un choix d'id\'eal premier de 
$Q$ au-dessus de $\pG^\gauche$ coh\'erent avec notre choix de $\qGba$. Le 
lemme suivant montre l'unicit\'e de ce choix~:

\bigskip

\begin{lem}\label{unicite qgauche plus}
On a $\pG^\gauche Q_\qGba = \qG^\gauche Q_\qGba$.
\end{lem}

\begin{proof}
Il suffit de montrer que $\pG^\gauche Q_{\qGba}$ est un id\'eal premier de $Q_{\qGba}$. 
D'apr\`es le lemme~\ref{qba}, le morphisme local d'anneaux locaux 
$P_{\pGba} \to Q_{\qGba}$ est \'etale. De plus, $P/\pG^\gauche \simeq \kb[\CCB \times V^*/W]$ 
est int\'egralement clos (c'est une alg\`ebre de polyn\^omes) et donc 
$P_{\pGba}/\pG^\gauche P_{\pGba}$ l'est aussi. Par changement de base, le morphisme 
d'anneaux $P_{\pGba}/\pG^\gauche P_{\pGba} \injto Q_{\qGba}/\pG^\gauche Q_{\qGba}$ 
est \'etale, ce qui implique que $Q_{\qGba}/\pG^\gauche Q_{\qGba}$, 
qui est un anneau local (donc connexe), est aussi 
normal (en vertu de~\cite[expos\'e I, corollaire 9.11]{sga}) et donc int\`egre 
(car connexe). Cela montre que $\pG^\gauche Q_{\qGba}$ est 
un id\'eal premier de $Q_{\qGba}$, comme souhait\'e.
\end{proof}

\bigskip

\begin{coro}\label{unicite qgauche}
L'id\'eal $\qG^\gauche$ est l'unique id\'eal premier de $Q$ au-dessus de $\pG^\gauche$ 
et contenu dans $\qGba$. De plus, $Q$ est \'etale sur $P$ en $\qG^\gauche$. 
\end{coro}

\bigskip

Puisque $Q/\qG^\gauche \simeq P/\pG^\gauche=\kb[\CCB \times V/W]$, on obtient que 
$Q/(\qG^\gauche +\CG\, Q) \simeq \kb[\CCB]/\CG \otimes \kb[V/W]$ et donc 
$\qG^\gauche + \CG\, Q$ est un id\'eal premier de $Q$. Nous le noterons 
$\qG_\CG^\gauche$.\indexnot{qa}{\qG_\CG^\gauche}

\bigskip

\begin{coro}\label{q gauche unique}
On a $Q/\qG^\gauche_\CG \simeq P/\pG^\gauche_\CG$. De plus, 
$\qG_\CG^\gauche$ est l'unique id\'eal premier de $Q$ au-dessus de 
$\pG_\CG^\gauche$ et contenu dans $\qGba$. 
.
\end{coro}

\bigskip

\begin{proof}
La premi\`ere assertion est imm\'ediate et la deuxi\`eme en d\'ecoule.
\end{proof}

\bigskip

%
%
%
%
%
%
%

\bigskip

\boitegrise{\noindent{\bf Choix, notations.} 
{\it Dor\'enavant, et ce jusqu'\`a la fin de cette partie, nous fixons un id\'eal 
premier $\rG_\CG^\gauche$  \indexnot{ra}{\rG_\CG^\gauche, \rG^\gauche, \rG_c^\gauche}  
de $R$ au-dessus de $\qG_\CG^\gauche$ et contenu dans $\rGba_\CG$. 
Nous noterons $\Kb^\gauche_\CG=k_P(\pG_\CG^\gauche)$,  \indexnot{K}{\Kb^\gauche_\CG, \Kb^\gauche, \Kb^\gauche_c}  
$\Lb^\gauche_\CG=k_Q(\qG_\CG^\gauche)$ \indexnot{L}{\Lb^\gauche_\CG, \Lb^\gauche, \Lb^\gauche_c}  
et $\Mb^\gauche_\CG=k_R(\rG_\CG^\gauche)$. \indexnot{M}{\Mb^\gauche_\CG, \Mb^\gauche, \Mb^\gauche_c}  
Le groupe de d\'ecomposition (respectivement d'inertie) de $\rG_\CG^\gauche$ sera not\'e 
$D_\CG^\gauche$ \indexnot{D}{D^\gauche_\CG, D^\gauche, D^\gauche_c}  
(respectivement $I_\CG^\gauche$).  \indexnot{I}{I^\gauche_\CG, I^\gauche, I^\gauche_c}  \\
\hphantom{A} Lorsque $\CG=0$ (respectivement $\CG=\CG_c$ avec $c \in \CCB$), les objets 
$\rG_\CG^\gauche$, $\Kb_\CG^\gauche$, $\Lb_\CG^\gauche$, $\Mb_\CG^\gauche$, $D_\CG^\gauche$ et 
$I_\CG^\gauche$ seront not\'es respectivement $\rG^\gauche$, $\Kb^\gauche$, $\Lb^\gauche$, $\Mb^\gauche$, 
$D^\gauche$ et $I^\gauche$ (respectivement $\rG_c^\gauche$, $\Kb_c^\gauche$, $\Lb_c^\gauche$, 
$\Mb_c^\gauche$, $D_c^\gauche$ et $I_c^\gauche$). }}{0.75\textwidth}

\bigskip

\begin{rema}\label{R rgauche}
Il a \'et\'e montr\'e dans le corollaire~\ref{Q extention kc} que, si 
$\qGba_\star$ est un id\'eal premier de $Q$ au-dessus de $\pGba$, alors 
$Q/\qGba_\star \simeq P/\pGba$. Bien que $Q/\qG^\gauche \simeq P/\pG^\gauche$, 
nous verrons dans le chapitre~\ref{chapitre:b2} que ceci ne s'\'etend 
pas en toute g\'en\'eralit\'e aux id\'eaux premiers de $Q$ au-dessus de 
$\pG^\gauche$~: en effet, si $W$ est de type $B_2$, alors il existe un 
id\'eal premier $\qG_\star^\gauche$ de $Q$ au-dessus de $\pG^\gauche$ 
tel que $P/\pG^\gauche$ soit un sous-anneau propre de $Q/\qG^\gauche_\star$ 
(voir le lemme~\ref{lem:qgauche-b2}(c)). Donc, en g\'en\'eral, 
$\Kb^\gauche \varsubsetneq \Mb^\gauche$.\finl
\end{rema}

\bigskip

\begin{coro}\label{dleft}
On a $I_\CG^\gauche \subset D_\CG^\gauche \subset H$. 
Si de plus $\rG_\CG^\gauche$ contient $\rG^\gauche$, 
alors $I^\gauche \subset I_\CG^\gauche$. 
\end{coro}

\begin{proof}
D'apr\`es le corollaire~\ref{q gauche unique} et le th\'eor\`eme~\ref{raynaud}, 
$\Iba_\CG^\gauche \subset H$ et 
$k_P(\pG_\CG^\gauche)=k_Q(\qG_\CG^\gauche)$. Ainsi, 
$(D_\CG^\gauche \cap H)/(I_\CG^\gauche \cap H) \simeq D_\CG^\gauche/I_\CG^\gauche$, 
d'o\`u la premi\`ere suite d'inclusions.

La derni\`ere inclusion est triviale.
\end{proof}

\bigskip

Nous allons montrer que $\rG_\CG^\gauche$ d\'etermine $\rGba_\CG$~: pour cela, nous utiliserons la 
$\ZM$-graduation sur $R$ d\'efinie dans~\S\ref{section:graduation R}. Posons 
$$R_{<0} = \mathop{\bigoplus}_{i < 0} R^\ZM[i]\quad\text{et}\quad R_{>0} = \mathop{\bigoplus}_{i > 0} R^\ZM[i].$$
Alors~:

\bigskip

\begin{prop}\label{prop:rgauche-rbarre}
$\rGba_\CG = \rG_\CG^\gauche + \langle R_{>0},R_{<0} \rangle$. 
\end{prop}

\bigskip

\begin{proof}
Notons $I$ l'id\'eal de $R$ engendr\'e par $R_{>0}$ et $R_{<0}$. 
L'id\'eal $\pGba_\CG$ de $P$ est $\ZM$-homog\`ene (il n'est pas forc\'ement $(\NM \times \NM)$-homog\`ene) 
donc l'id\'eal $\rGba_\CG$ de $R$ est lui aussi $\ZM$-homog\`ene (voir le corollaire~\ref{premier homogene}). 
L'extension $R/\rGba_\CG$ de $P/\pGba_\CG$ est enti\`ere et, puisque $P/\pGba_\CG$ a sa $\ZM$-graduation concentr\'ee 
en degr\'e $0$, il en d\'ecoule que la $\ZM$-graduation de $R/\rGba_\CG$ est concentr\'ee en degr\'e $0$ 
(voir la proposition~\ref{prop:graduation-positive}). En particulier, $I \subset \rGba_\CG$ et donc 
$$\rG_\CG^\gauche + I \subset \rGba_\CG.$$
D'autre part, $(\rG_\CG^\gauche + I) \cap P$ contient $\CG$ et 
$\pGba=\langle P_{<0},P_{>0} \rangle$, donc 
$$\pGba_\CG \subset (\rG_\CG^\gauche + I) \cap P.$$
Il suffit donc de montrer que $\rG_\CG^\gauche + I$ est un id\'eal premier.

\medskip

Posons $I_0 = I \cap R_0$. Alors l'application naturelle $R_0 \injto R \surto R/I$ induit un isomorphisme 
$R_0/I_0 \longiso R/I$. Par cons\'equent, $R/(\rG_\CG^\gauche + I)$ est isomorphe \`a 
$R_0 / ((\rG_\CG^\gauche + I) \cap R_0)$. Il nous suffit donc de montrer que $(\rG_\CG^\gauche + I) \cap R_0$ est 
un id\'eal premier de $R_0$. En fait, nous allons montrer que
$(\rG_\CG^\gauche + I) \cap R_0 = \rG_\CG^\gauche \cap R_0,$
ce qui terminera la d\'emonstration. 

Tout d'abord, remarquons que, puisque $\rG_\CG^\gauche$ et $I$ sont $\ZM$-homog\`enes, on a 
$(\rG_\CG^\gauche + I) \cap R_0= (\rG_\CG^\gauche \cap R_0) + I_0$. Il nous suffit donc de montrer 
que 
$$I_0 \subset \rG_\CG^\gauche.\leqno{(*)}$$
Puisque $R/\rG_\CG^\gauche$ est une extension enti\`ere de $P/\pG_\CG^\gauche$ et 
que $P/\pG_\CG^\gauche$ est $\NM$-gradu\'e, on en d\'eduit que $R/\rG_\CG^\gauche$ est $\NM$-gradu\'e 
(voir la proposition~\ref{prop:graduation-positive}). Donc $R_{<0} \subset \rG_\CG^\gauche$ et 
$I_0 = R_0 \cap (R \cdot R_{>0}) = R_0 \cap (R \cdot R_{<0}) = R_0 \cap (R_{<0} \cdot R_{>0}) \subset \rG_\CG^\gauche$. 
\end{proof}

\bigskip

\begin{coro}\label{coro:dgauche-dbarre}
$D_\CG^\gauche \subset \Dba_\CG$.
\end{coro}

\bigskip

\begin{proof}
Cela d\'ecoule imm\'ediatement de la proposition~\ref{prop:rgauche-rbarre} et du fait que $R_{<0}$ et $R_{>0}$ sont $G$-stables 
(voir la proposition~\ref{bigraduation sur R}).
\end{proof}

\bigskip

\begin{rema}\label{rema:bb cache}
La preuve alg\'ebrique de la proposition~\ref{prop:rgauche-rbarre} pr\'esent\'ee ci-dessus est en fait une traduction 
d'un fait de nature g\'eom\'etrique, comme on le verra dans le chapitre~\ref{chapter:bb}.\finl
\end{rema}

\bigskip

%

\begin{prop}\label{coro:app-lim}
Soit $\zG_*^\gauche$ un id\'eal premier de $Z$ au-dessus de $\pG_\CG^\gauche$. 
Alors il existe un unique id\'eal premier de $Z$ au-dessus de $\pGba_\CG$ contenant $\zG_*^\gauche$~: 
il est \'egal \`a $\zG_*^\gauche + \langle Z_{<0}, Z_{>0} \rangle$. 
\end{prop}

\bigskip

\begin{proof}
Puisque $Z$ est entier sur $P$, la preuve de la proposition~\ref{prop:rgauche-rbarre} s'applique mot pour mot 
dans cette situation pour fournir la m\^eme conclusion.
\end{proof}

\bigskip

La proposition~\ref{coro:app-lim} fournit une application surjective 
$$\fonction{\limitegauche}{\Upsilon^{-1}(\pG_\CG^\gauche)}{\Upsilon^{-1}(\pGba_\CG)}{\zG_*^\gauche}{\zG_*^\gauche + 
\langle Z_{<0}, Z_{>0} \rangle.}$$
La notation $\limitegauche$ sera justifi\'ee dans le chapitre~\ref{chapter:bb}. 

\bigskip

\section{Cellules \`a gauche}

\medskip

\subsection{D\'efinitions} 
Rappelons les d\'efinitions pos\'ees en pr\'eambule \`a cette partie~:

\medskip

\begin{defi}\label{defi:gauche}
On appellera {\bfit $\CG$-cellule de Calogero-Moser \`a gauche} toute $\rG_\CG^\gauche$-cellule de Calogero-Moser. 
Si $c \in \CCB$, on appellera {\bfit $c$-cellule de Calogero-Moser \`a gauche} toute $\rG_c^\gauche$-cellule de Calogero-Moser. 
On appellera {\bfit cellule de Calogero-Moser \`a gauche g\'en\'erique} toute $\rG^\gauche$-cellule de Calogero-Moser. 

L'ensemble des $\CG$-cellules de Calogero-Moser \`a gauche sera not\'e $\cmcellules_L^\CG(W)$.  
\indexnot{C}{{{{^\calo{\mathrm{Cell}}_L^\CG(W), ^\calo{\mathrm{Cell}}_L(W), ^\calo{\mathrm{Cell}}_L^c(W)}}}}
Si $\CG=0$ (respectivement $\CG=\CG_c$, avec $c \in \CCB$), 
cet ensemble sera not\'e $\cmcellules_L(W)$ (respectivement $\cmcellules_L^c(W)$). 
\end{defi}

\bigskip

Comme d'habitude, la notion de $\CG$-cellule de Calogero-Moser \`a gauche d\'epend du choix de l'id\'eal 
premier $\rG_\CG^\gauche$. La proposition suivante est imm\'ediate~:

\bigskip

\begin{prop}\label{semicontinuite gauche}
Si $\CG'$ est un id\'eal premier de $\kb[\CCB]$ contenu dans $\CG$ et si $\rG_{\CG'}^\gauche$ est contenu dans $\rG_\CG^\gauche$, 
alors toute $\CG$-cellule de Calogero-Moser \`a gauche est une r\'eunion de 
$\CG'$-cellules de Calogero-Moser \`a gauche.
\end{prop}

\bigskip

D'autre part, puisque $\rG_\CG^\gauche \subset \rGba_\CG$, on obtient aussi~:

\bigskip

\begin{prop}\label{bilatere gauche}
Toute $\CG$-cellule de Calogero-Moser bilat\`ere est une r\'eunion de cellules de Calogero-Moser 
\`a gauche.
\end{prop}

\bigskip

Pour finir, 
soit $\CGt$ l'id\'eal homog\`ene maximal contenu dans $\CG$ (i.e. $\CGt=\bigoplus_{i \ge 0} \CG \cap \kb[\CCB]^\NM[i]$). 
Alors $\CGt$ est un id\'eal premier de $\kb[\CCB]$ (voir le lemme~\ref{lem:homogeneise-premier}). Notons 
$\rG^\gauche_\CGt$ l'id\'eal homog\`ene maximal contenu dans $\rG^\gauche_\CG$~: c'est un id\'eal premier de $R$ au-dessus de 
$\qG^\gauche_\CGt$ 
(voir le corollaire~\ref{coro:homogeneise-premier}). 
On d\'eduit de la proposition~\ref{prop:cellules-homogeneise} le r\'esultat suivant~:

\bigskip

\begin{prop}\label{lem:cellules-gauches-homogeneise}
On a $I_\CG^\gauche=I_\CGt^\gauche$. En particulier, 
les $\CG$-cellules de Calogero-Moser \`a gauche et les $\CGt$-cellules de Calogero-Moser \`a gauche 
co\"{\i}ncident.
\end{prop}

\bigskip

\subsection{Cellules \`a gauche et modules simples}
D'apr\`es l'exemple~\ref{exemple lisse}, on a $\qG_\singulier \cap P \not\subset \pG^\gauche$. Par cons\'equent, tous 
les r\'esultats de la section~\ref{section:cellules et lissite} s'appliquent. Rappelons ici quelques cons\'equences~:

\bigskip

\begin{theo}\label{theo:gauche}
Si $\CG$ est un id\'eal premier de $\kb[\CCB]$, alors~:
\begin{itemize}
\itemth{a} L'alg\`ebre $\Mb_\CG^\gauche \Hb^\gauche$ est d\'eploy\'ee et ses modules simples sont de dimension $|W|$.

\itemth{b} Chaque bloc de $\Mb_\CG^\gauche \Hb^\gauche$ admet un unique module simple.
\end{itemize}
\end{theo}

\bigskip

Si $C \in \cmcellules_L^\CG(W)$, nous noterons $\LC_\CG^\gauche(C)$ l'unique  
\indexnot{L}{\LC_\CG^\gauche(C), \LC^\gauche(C)\LC_c^\gauche(C)}  
$\Mb_\CG^\gauche \Hb^\gauche$-module simple appartenant au bloc de $\Mb_\CG^\gauche \Hb^\gauche$ associ\'e 
\`a $C$. Si $\CG=0$ (respectivement $\CG=\CG_c$, avec $c \in \CCB$), le module $\LC_\CG^\gauche(C)$ 
sera not\'e $\LC^\gauche(C)$ (respectivement $\LC_c^\gauche(C)$). 

\bigskip

\def\DD{{\SSS{D}}}

\subsection{Cellules \`a gauche et cellules bilat\`eres}\label{subsection:gauche-bilatere}
Nous fixons ici une $\CG$-cellule de Calogero-Moser bilat\`ere $\G$ ainsi qu'une $\CG$-cellule de Calogero-Moser \`a gauche $C$
contenue dans $\G$. Puisque $\Dba_\CG$ stabilise $\G$ (voir le th\'eor\`eme~\ref{theo cellules familles}(a)) 
et puisque $D_\CG^\gauche \subset \Dba_\CG$ (voir le corollaire~\ref{coro:dgauche-dbarre}), 
le groupe $D_\CG^\gauche$ stabilise $\G$ (et permute les cellules gauches qui sont contenues dans $\G$). 
Posons
$$C^\DD=\bigcup_{d \in D_\CG^\gauche} d(C).$$\indexnot{C}{C^\DD} 
Soit $w \in C^\DD$. On pose $\qGba_\CG(\G)=w^{-1}(\rGba_\CG) \cap Q$ et   \indexnot{qa}{\qGba_\CG(\G), \qG_\CG^\gauche(C^\DD)}
$\qG_\CG^\gauche(C^\DD)=w^{-1}(\rG^\gauche_\CG) \cap Q$. On pose aussi $\zGba_\CG(\G)=\copie^{-1}(\qGba_\CG(\G))$  
\indexnot{za}{\zGba_\CG(\G), \zG_\CG^\gauche(C^\DD), \zG_\CG^\gauche(C)}  
et $\zG^\gauche_\CG(C^\DD)=\copie^{-1}(\qG_\CG^\gauche(C^\DD))$. 
Il d\'ecoule de la proposition~\ref{reduction} 
que $\zGba_\CG(\G)$ ne d\'epend que de $\G$ et non du choix de $C$ et de $w$, tandis que 
$\zG_\CG^\gauche(C^\DD)$ ne d\'epend que de $C^\DD$ et non du choix de $w$. On pose
$$\deg_\CG(C^\DD)=[k_Z(\zG_\CG^\gauche(C^\DD)):k_P(\pG^\gauche_\CG)].$$ \indexnot{da}{\deg_\CG(C^\DD), \deg_\CG(C)}
Nous noterons quelquefois $\zG_\CG^\gauche(C)$ ou $\deg_\CG(C)$ \`a la place de 
$\zG_\CG^\gauche(C^\DD)$ ou $\deg_\CG(C^\DD)$. 

\bigskip

\begin{prop}\label{prop:gauche-bilatere}
Soit $w \in C^\DD$. Alors~:
\begin{itemize}
\itemth{a} $\zGba_\CG(\G)=\limitegauche(\zG_\CG^\gauche(C^\DD))$.

\itemth{b} $\DS{\deg_\CG(C^\DD)=\frac{|C^\DD|}{|C|}=\frac{|D_\CG^\gauche|}{|(D_\CG^\gauche \cap \lexp{w}{H})I_\CG^\gauche|}}$.

\itemth{c} L'application $D_\CG^\gauche \backslash \G \longto \lim_\gauche^{-1}(\zGba_\CG(\G))$, 
$C^\DD \longmapsto \zG_\CG^\gauche(C^\DD)$ est bijective.
\end{itemize}
\end{prop}

\bigskip

\begin{proof}
(a) Notons que $\zGba_\CG(\G) \in \Upsilon^{-1}(\pGba_\CG)$, $\zG_\CG^\gauche(C^\DD) \in \Upsilon^{-1}(\pG_\CG^\gauche)$ et 
$\zG^\gauche_\CG(C^\DD) \subset \zGba_\CG(\G)$, d'o\`u le r\'esultat d'apr\`es la proposition~\ref{coro:app-lim}.

\medskip

(b) Par transport \`a travers $w$, 
l'extension $k_R(w^{-1}(\rG_\CG^\gauche))/k_P(\pG^\gauche_\CG)$ est galoisienne de groupe de Galois 
$\lexp{w^{-1}}{D_\CG^\gauche}/\lexp{w^{-1}}{I_\CG^\gauche}$ tandis que l'extension $k_R(w^{-1}(\rG_\CG^\gauche))/k_Q(\qG^\gauche_\CG(C^\DD))$ 
est galoisienne de groupe de Galois $(\lexp{w^{-1}}{D_\CG^\gauche} \cap H)/(\lexp{w^{-1}}{I_\CG^\gauche} \cap H)$. 
D'o\`u 
$$\deg_\CG(C^\DD)=\frac{|D_\CG^\gauche|}{|(D_\CG^\gauche \cap \lexp{w}{H})I_\CG^\gauche|}.$$
Par ailleurs, $|C^\DD|/|C|$ est \'egal \`a l'indice du stabilisateur de $C$ dans $D^\gauche_\CG$~: mais ce stabilisateur 
est exactement $(D_\CG^\gauche \cap \lexp{w}{H})I_\CG^\gauche$. 

\medskip

(c) d\'ecoule essentiellement de la commutativit\'e du diagramme
$$\diagram
D_\CG^\gauche\backslash G/H \rrto^{\DS{\sim}} \ddto && \Upsilon^{-1}(\pG_\CG^\gauche) \ddto^{\DS{\limitegauche}}\\
&&\\
\Dba_\CG \backslash G/H \rrto^{\DS{\sim}} && \Upsilon^{-1}(\pGba_\CG),
\enddiagram$$
o\`u la fl\`eche verticale de gauche est l'application canonique (car $D_\CG^\gauche \subset \Dba_\CG$). 
Il faut simplement rappeler que $\Dba_\CG \backslash G/H=\Iba_\CG \backslash G/H$ (voir~(\ref{eq:dba-g-h})).
\end{proof}

\bigskip

\section{Caract\`eres cellulaires}

\medskip

\subsection{Multiplicit\'es} 
Les $\Mb_\CG^\gauche \Hb^\gauche$-modules simples \'etant param\'etr\'es par $\cmcellules_\CG^\gauche(W)$, 
il existe une unique famille d'entiers naturels 
$(\mult_{C,\chi}^\calo)_{C \in \cmcellules_\CG^\gauche(W),\chi \in \Irr(W)}$  
\indexnot{m}{\mult_{C,\chi}^\calo}  telle que
$$\isomorphisme{\Mb_\CG^\gauche\MC^\gauche(\chi)}_{\Mb_\CG^\gauche \Hb^\gauche} = \sum_{C \in \cmcellules_L^\CG(W)} 
\mult_{C,\chi}^\calo ~\cdot ~\isomorphisme{\LC_\CG^\gauche(C)}_{\Mb_\CG^\gauche \Hb^\gauche}$$
pour tout $\chi \in \Irr(W)$. 

\bigskip

\begin{prop}\label{multiplicite cm}
Avec les notations ci-dessus, on a~:
\begin{itemize}
\itemth{a} Si $\chi \in \Irr(W)$, alors $\sum_{C \in \cmcellules_L^\CG(W)} \mult_{C,\chi}^\calo = \chi(1)$.
 
\itemth{b} Si $C \in \cmcellules_L^\CG(W)$, alors $\sum_{\chi \in \Irr(W)} \mult_{C,\chi}^\calo~\chi(1) = |C|$.

\itemth{c} Si $C \in \cmcellules_L^\CG(W)$, si $\G$ est l'unique $\CG$-cellule de Calogero-Moser bilat\`ere 
contenant $C$ et si $\chi \in \Irr(W)$ est tel que $\mult_{C,\chi}^\calo \neq 0$, alors 
$\chi \in \Irr_\G^\calo(W)$.
\end{itemize}
\end{prop}

\begin{proof}
(a) d\'ecoule du calcul de la dimension des modules de Verma \`a gauche (voir~(\ref{dimension left})). 

\medskip

Montrons maintenant (b). Tout d'abord, remarquons que, gr\^ace \`a l'\'equivalence de Morita du th\'eor\`eme~\ref{theo:KH-mat}, on a 
$$\isomorphisme{\Mb\Hb e}_{\Mb\Hb} = \sum_{w \in W} \isomorphisme{\LC_w}_{\Mb\Hb}.$$
En appliquant $\dec_\CG^\gauche$ \`a cette \'egalit\'e, on obtient
$$\isomorphisme{\Mb_\CG^\gauche \Hb^\gauche e}_{\Mb_\CG^\gauche \Hb^\gauche} 
= \sum_{C \in \cmcellules_L^\CG(W)} 
|C|~\cdot~\isomorphisme{\LC_\CG^\gauche(C)}_{\Mb_\CG^\gauche \Hb^\gauche}.$$
Pour montrer (b), il suffit de montrer que 
\equat\label{he verma}
\isomorphisme{\Mb_\CG^\gauche \Hb^\gauche e}_{\Mb_\CG^\gauche \Hb^\gauche}
= \sum_{\chi \in \Irr(W)} 
\chi(1)~\cdot~\isomorphisme{\Mb_\CG^\gauche\MC^\gauche(\chi)}_{\Mb_\CG^\gauche \Hb^\gauche}.
\endequat
Puisque $\Mb_\CG^\gauche \Hb^\gauche$ est un $\Hbov^\moins$-module libre, le foncteur 
$\Mb_\CG^\gauche \Hb^\gauche \otimes_{\Hbov^\moins} -$ est exact et il suffit donc de montrer que 
$$\isomorphisme{\Hbov^\moins e}_{\Hbov^\moins} = \sum_{\chi \in \Irr(W)} \isomorphisme{V_\chi^{(-)}}_{\Hbov^\moins},$$
ce qui d\'ecoule du fait que $\Hbov^\moins e$ est isomorphe, comme $\kb W$-module, \`a $\kb[V^*]^\cow$, c'est-\`a-dire 
\`a la repr\'esentation r\'eguli\`ere de $W$. D'o\`u (b).

\medskip

(c) est imm\'ediat, car la r\'eduction modulo $\pGba$ du module de Verma \`a gauche est le b\'eb\'e module de Verma 
correspondant, 
et est donc ind\'ecomposable comme $\Kbov_\CG\Hbov$-module.
\end{proof}

\bigskip

\subsection{Premi\`ere d\'efinition}
Si $C$ est une $\CG$-cellule de Calogero-Moser \`a gauche, on pose
\equat\label{eq:cellulaire}
\isomorphisme{C}_\CG^\calo = \sum_{\chi \in \Irr(W)} \mult_{C,\chi}^\calo \cdot \chi.  
\indexnot{C}{\isomorphisme{C}_\CG^\calo, \isomorphisme{C}^\calo, \isomorphisme{C}_c^\calo}
\endequat

\bigskip

\begin{defi}\label{defi:cellulaires}
On appelle {\bfit $\calo$-caract\`ere $\CG$-cellulaire} de $W$ tout caract\`ere 
de $W$ de la forme $\isomorphisme{C}_\CG^\calo$, o\`u $C$ est une 
$\CG$-cellule de Calogero-Moser \`a gauche. Si $\CG=0$ (respectivement $\CG=\CG_c$ avec $c \in \CCB$), 
on parlera de {\bfit $\calo$-caract\`ere cellulaire g\'en\'erique} (respectivement {\bfit $\calo$-caract\`ere 
$c$-cellulaire}) et on notera $\isomorphisme{C}^\calo$ (respectivement $\isomorphisme{C}_c^\calo$).
\end{defi}

\bigskip

Compte tenu de la proposition~\ref{multiplicite cm}(c), l'ensemble des caract\`eres irr\'eductibles 
apparaissant avec une multiplicit\'e non nulle dans un $\calo$-caract\`ere $\CG$-cellulaire  
est contenu dans une $\CG$-famille de Calogero-Moser. On parlera alors de $\calo$-caract\`ere 
$\CG$-cellulaire {\it associ\'e \`a la famille en question}.

\bigskip

\begin{rema}\label{rema:cellulaire-independant}
Nous allons montrer ici que l'ensemble des $\calo$-caract\`eres $\CG$-cellulaires (associ\'es \`a une $\CG$-famille) 
ne d\'epend pas du choix de l'id\'eal premier $\rG_\CG^\gauche$. En effet, soit $\rG$ un id\'eal premier de $R$ au-dessus 
de $\pG_\CG^\gauche$. Alors il existe $g \in G$ tel que $g(\rG_\CG^\gauche)=\rG$. 
Ainsi, $g$ induit un isomorphisme $R/\rG_\CG^\gauche \longiso R/\rG$, qui induit un isomorphisme 
$\Mb_\CG^\gauche \longiso k_R(\rG)$.

Notons $\cmcellules_L^{\CG,*}(W)$ l'ensemble des $\rG$-cellules de Calogero-Moser et, si $C \in \cmcellules_L^{\CG,*}(W)$, 
notons $\LC_{\CG,*}^\gauche(C)$ le $k_R(\rG)\Hb^\gauche$-module simple correspondant. Alors, 
$g^{-1}(C) \in \cmcellules_L^{\CG}(W)$ et l'isomorphisme $\Mb_\CG^\gauche \longiso k_R(\rG)$, qui 
se prolonge en un isomorphisme $\Mb_\CG^\gauche\Hb^\gauche \longiso k_R(\rG)\Hb^\gauche$, envoie le 
module simple $\LC_\CG^\gauche(g^{-1}(C))$ sur $\LC_{\CG,*}^\gauche(C)$. Le module de Verma \`a gauche 
\'etant d\'efini sur $\Kb^\gauche_\CG$, la multiplicit\'e de $\LC_\CG^\gauche(\lexp{g^{-1}}{C})$ dans 
$\Mb_\CG^\gauche\MC^\gauche(\chi)$ est donc \'egale \`a celle de $\LC_{\CG,*}^\gauche(C)$ dans 
$k_R(\rG)\MC^\gauche(\chi)$. D'o\`u le r\'esultat.\finl
\end{rema}

\bigskip

Si $d \in D_\CG^\gauche$ et $C$ est une $\CG$-cellule de Calogero-Moser \`a gauche, alors $d(C)$ 
est aussi une $\CG$-cellule de Calogero-Moser \`a gauche. La 
remarque~\ref{rema:cellulaire-independant} montre que les $\calo$-caract\`eres $\CG$-cellulaires 
associ\'es \`a $C$ et $d(C)$ co\"{\i}ncident~:

\bigskip

\begin{coro}\label{coro:dec-cellulaire}
Si $d \in D_\CG^\gauche$ et $C$ est une $\CG$-cellule de Calogero-Moser \`a gauche, alors 
$$\isomorphisme{d(C)}_\CG^\calo=\isomorphisme{C}_\CG^\calo.$$
\end{coro}

\bigskip

\begin{rema}\label{rema:decleft-decba}
Le corollaire~\ref{coro:dec-cellulaire} pr\'ec\'edent montre en particulier que 
$d(C)$ est contenue dans la m\^eme cellule bilat\`ere que $C$, ce qui 
a d\'ej\`a \'et\'e d\'emontr\'e autrement (voir le d\'ebut de la sous-section~\ref{subsection:gauche-bilatere}).\finl
\end{rema}

\bigskip

\subsection{Deuxi\`eme d\'efinition} 
La $\kb$-alg\`ebre $\Hbov^\moins \simeq \kb[\CCB] \otimes(\kb[V^*]^\cow \rtimes W)$ 
est une sous-$\kb$-alg\`ebre de $\Hb^\gauche$ qui nous a servi \`a construire les modules de Verma \`a gauche. 
Ainsi, $\Mb_\CG^\gauche \otimes \Hbov^\moins$ est une sous-$\Mb_\CG^\gauche$-alg\`ebre de 
$\Mb_\CG^\gauche\Hb^\gauche$ de dimension $|W|^2$, dont le groupe de Grothendieck s'identifie \`a $\ZM\Irr(W)=\groth(\kb W)$. 

Si $C$ est une $\CG$-cellule de Calogero-Moser \`a gauche, on notera $\PC^\gauche_\CG(C)$  
\indexnot{P}{\PC^\gauche_\CG(C)}  une enveloppe projective 
du $\Mb_\CG^\gauche \Hb^\gauche$-module simple $\LC_\CG^\gauche(C)$. 
\begin{prop}\label{eq:socle-cellulaire}
On a 
$$\isomorphisme{{\mathrm{Soc}}(\Res_{\Mb_\CG^\gauche \otimes \Hbov^\moins}^{\Mb_\CG^\gauche\Hb^\gauche} 
\PC_\CG^\gauche(C))}_{\Mb_\CG^\gauche \otimes \Hbov^\moins} 
= \sum_{\chi \in \Irr(W)} \mult_{C,\chi}^\calo \cdot \chi.$$
\end{prop}

\begin{proof}
Soit $\chi \in \Irr(W)$. Puisque l'alg\`ebre $\Hb$ est sym\'etrique (voir~(\ref{symetrique})), $\PC_\CG^\gauche(C)$ est aussi 
une enveloppe injective de $\LC_\CG^\gauche(C)$. On a donc 
$$\mult_{C,\chi}^\calo = \dim_{\Mb_\CG^\gauche} \Hom_{\Mb_\CG^\gauche\Hb^\gauche}\bigl(
\Mb_\CG^\gauche \MC^\gauche(\chi),\PC_\CG^\gauche(C) \bigr).$$
Mais $\Mb_\CG^\gauche \MC^\gauche(\chi) = \Ind_{\Mb_\CG^\gauche\Hbov^\moins}^{\Mb_\CG^\gauche\Hb^\gauche} 
(\Mb_\CG^\gauche \otimes V_\chi^{(-)})$. Par cons\'equent,
$$\mult_{C,\chi}^\calo = \dim_{\Mb_\CG^\gauche}\Hom_{\Mb_\CG^\gauche\Hbov^\moins}\bigl( V_\chi^{(-)},
\Res_{\Mb_\CG^\gauche \otimes \Hbov^\moins}^{\Mb_\CG^\gauche\Hb^\gauche} \PC_\CG^\gauche(C)\bigr).$$
D'o\`u le r\'esultat.
\end{proof}

\bigskip

\subsection{Troisi\`eme d\'efinition} 
D'apr\`es la remarque~\ref{rema:cellulaire-independant}, l'ensemble des $\calo$-caract\`eres $\CG$-cellulaires 
ne d\'epend pas du choix de l'id\'eal $\rG_\CG^\gauche$ au-dessus de $\pG_\CG^\gauche$. 
Nous allons tirer parti du fait que le module de Verma \`a gauche est d\'efini sur $\Hb_\CG^\gauche$ et du fait que 
le $\calo$-caract\`ere $\CG$-cellulaire associ\'e \`a une $\CG$-cellule de Calogero-Moser \`a gauche $C$ ne 
d\'epend que de la $D_\CG^\gauche$-orbite de $C$ pour donner une d\'efinition qui ne passe pas par 
l'extension des scalaires \`a $\Mb_\CG^\gauche$ (et donc qui ne passe pas par le choix de $\rG_\CG^\gauche$). 

Pour cela, si $\zG$ est un id\'eal premier de $Z$ et si $M$ est un $Z$-module de dimension de Krull 
inf\'erieure ou \'egale \`a celle de $Z/\zG$, alors $M_\zG$ est un $Z_\zG$-module de longueur finie et 
nous noterons $\longueur_{Z_\zG}(M_\zG)$  \indexnot{L}{\longueur_{Z_\zG}(M_\zG)}  cette longueur. 

\medskip

%

%
%
%
%

\begin{prop}\label{prop:multiplicite-Z}
Soit $C$ une $\CG$-cellule de Calogero-Moser \`a gauche, soit $\zG=\zG_\CG^\gauche(C^\DD)$ et soit 
$M$ un $P_\CG^\gauche \Hb$-module de type fini. Alors $eM_\zG$ est un $Z_\zG$-module de longueur 
finie et $\longueur_{Z_\zG}(eM_\zG)$ est \'egale \`a la multiplicit\'e de $\LC_\CG^\gauche(C)$ 
dans $\Mb_\CG^\gauche \otimes_{P_\CG^\gauche} M$. 
\end{prop}

\bigskip

\begin{proof}
Notons $l$ la longueur du $Z_\zG$-module $eM_\zG$ et $m$ la multiplicit\'e de $\LC_\CG^\gauche(C)$ 
dans $\Mb_\CG^\gauche \otimes_{P_\CG^\gauche} M$. 
Ainsi, le $Z_\zG$-module $eM_\zG$ 
admet une filtration
$$0=M_0 \subset M_1 \subset \cdots \subset M_l=eM_\zG$$
telle que $M_i/M_{i-1} \simeq Z_\zG/\zG Z_\zG=k_Z(\zG)$ pour tout $i$. Alors le 
$(\Mb_\CG^\gauche \otimes_{P_\CG^\gauche} Z_\zG)$-module $\Mb_\CG^\gauche \otimes_{P_\CG^\gauche}  eM_\zG$ 
admet une filtration 
$$0=M_0' \subset M_1' \subset \cdots \subset M_l'=\Mb_\CG^\gauche \otimes_{P_\CG^\gauche}  eM_\zG$$
telle que $M_i'/M_{i-1}' \simeq \Mb_\CG^\gauche \otimes_{P_\CG^\gauche}  k_Z(\zG)$ pour tout $i$. 

D'autre part, \`a travers l'\'equivalence de Morita du corollaire~\ref{morita lissite r}, 
le module simple $\LC^\gauche_\CG(C)$ devient $e\LC_\CG^\gauche(C)$, qui 
est un $(\Mb_\CG^\gauche \otimes_{P_\CG^\gauche} Z)$-module, de dimension $1$ comme 
$\Mb_\CG^\gauche$-espace vectoriel, et sur lequel 
un \'el\'ement $z \in Z$ agit par multiplication par $w(\copie(z)) \mod \rG$, o\`u $w \in C$. 
Pour montrer que $l=m$, il suffit donc de v\'erifier que la multiplicit\'e de $e\LC_\CG^\gauche(C)$ 
dans le $(\Mb_\CG^\gauche \otimes_{P_\CG^\gauche} Z)$-module 
$\Mb_\CG^\gauche \otimes_{P_\CG^\gauche} k_Z(\zG)$ est \'egale \`a $1$, ce qui 
d\'ecoule de la proposition~\ref{iso galois} et de la d\'efinition de l'id\'eal $\zG$.
\end{proof}

\bigskip

\begin{coro}\label{coro:cellulaire-multiplicite}
Soit $C$ une $\CG$-cellule de Calogero-Moser \`a gauche et soit $\zG=\zG_\CG^\gauche(C^\DD)$. Alors
$\mult_{C,\chi}^\calo = \longueur_{Z_\zG}\bigl(eP_\CG^\gauche \MC^\gauche(\chi)\bigr)_\zG$.
\end{coro}

\bigskip

\begin{rema}\label{rema:cellulaires-pas-choix}
Le corollaire~\ref{coro:cellulaire-multiplicite} montre que 
$$\sum_{\chi \in \Irr(W)} \longueur_{Z_\zG}\bigl(eP_\CG^\gauche \MC^\gauche(\chi)\bigr)_\zG \cdot \chi$$
est un $\calo$-caract\`ere $\CG$-cellulaire de $W$ et tous sont obtenus ainsi, pour $\zG \in \Upsilon^{-1}(\pG_\CG^\gauche)$. 

Cela donne une d\'efinition de l'ensemble des $\calo$-caract\`eres $\CG$-cellulaires qui n'utilise 
\`a aucun moment le choix d'un id\'eal premier $\rG_\CG^\gauche$ de $R$ au-dessus de $\pG_\CG^\gauche$.\finl
\end{rema}

\bigskip

\begin{coro}\label{coro:cardinal-C}
Soit $C$ une $\CG$-cellule de Calogero-Moser \`a gauche et soit $\zG=\zG_\CG^\gauche(C^\DD)$. Alors 
$|C|=\longueur_{Z_\zG}(Z/\pG_\CG^\gauche Z)_\zG$. 
\end{coro}

\bigskip

\begin{proof}
En effet, via l'\'equivalence de Morita du corollaire~\ref{morita lissite p}, le module $Z/\pG_\CG^\gauche Z$ 
correspond au module $P_\CG^\gauche \Hb e$~: ce dernier est filtr\'e par les $P_\CG^\gauche \MC^\gauche(\chi)$, 
chaque $P_\CG^\gauche \MC^\gauche(\chi)$ apparaissant $\chi(1)$ fois (comme dans le th\'eor\`eme~\ref{dim graduee bonne}). 
Ainsi, 
$$\longueur_{Z_\zG}(Z/\pG_\CG^\gauche Z)_\zG=\sum_{\chi \in \Irr(W)} \mult_{C,\chi}^\calo \cdot \chi(1)=|C|,$$
en utilisant la proposition~\ref{multiplicite cm}(b).
\end{proof}

\bigskip

\begin{coro}\label{coro:cellulaire-ordre-2}
Suposons que toutes les r\'eflexions de $W$ soient d'ordre $2$ et notons 
$\t_0=(-1,1,\e) \in \kb^\times \times \kb^\times \times W^\wedge$. Soit $\zG$ un id\'eal premier de $Z$ 
et soient $C$ et $C_\e$ deux $\CG$-cellules de Calogero-Moser \`a gauche telles que 
$\zG_\CG^\gauche(C^\DD)=\zG$ et $\zG_\CG^\gauche(C_\e^\DD)=\t_0(\zG)$. Alors 
$$\isomorphisme{C_\e}_\CG^\calo = \e \cdot \isomorphisme{C}_\CG^\calo.$$
Si de plus $w_0=-\Id_V \in W$, alors on peut prendre $C_\e=Cw_0$ et donc 
$$\isomorphisme{C w_0}_\CG^\calo = \e \cdot \isomorphisme{C}_\CG^\calo.$$
\end{coro}

\begin{proof}
La premi\`ere assertion d\'ecoule du fait que $\lexp{\t_0}{\MC^\gauche(\chi)} \simeq \MC^\gauche(\chi \e)$ tandis que 
la deuxi\`eme se d\'emontre comme le corollaire~\ref{w0 epsilon}.
\end{proof}

\bigskip

\subsection{Caract\`eres cellulaires et $\bb$-invariant} 
Le th\'eor\`eme suivant est un analogue du th\'eor\`eme~\ref{dim graduee bonne} (\'enonc\'es (b) et (c)). 

\bigskip

\begin{theo}\label{theo:b-minimal-cellulaire}
Soit $C$ une $\CG$-cellule de Calogero-Moser \`a gauche. Alors il existe un unique 
caract\`ere $\chi$ de $\bb$-invariant minimal tel que $\mult_{C,\chi}^\calo \neq 0$ (notons-le $\chi_C$). 
De plus, le coefficient de $\tb^{\bb_{\chi_C}}$ dans $f_{\chi_C}(\tb)$ est \'egal \`a $1$. 
\end{theo}

\begin{proof}
Notons $b_C$ l'idempotent primitif central de $\Mb_\CG^\gauche\Hb^\gauche$ associ\'e \`a $C$. 
L'alg\`ebre d'endomorphismes de $b_C \Mb_\CG^\gauche\Hb^\gauche e$ est \'egale \`a 
$(\Mb_\CG^\gauche \otimes_P Z)b_C$ et cette alg\`ebre (commutative) est locale. Cela montre 
que le module projectif $b_C \Mb_\CG^\gauche\Hb^\gauche e$ admet un unique quotient simple. 

La preuve se conclut alors comme dans le th\'eor\`eme~\ref{dim graduee bonne}, en remarquant 
que $b_C \Mb_\CG^\gauche\Hb^\gauche e$ est filtr\'e par des modules de la forme $b_C \Mb_\CG^\gauche \MC^\gauche(\chi)$. 
\end{proof}

\section{Probl\`emes, questions}

\medskip

Les premi\`eres questions th\'eoriques g\'en\'erales concernent les 
id\'eaux premiers $\rG_\CG^\gauche$ et le corps $k_R(\rG_\CG^\gauche)$. 
En effet, la remarque~\ref{R rgauche} montre qu'en g\'en\'eral 
$\Kb^\gauche \varsubsetneq \Mb^\gauche$. N\'eanmoins, 
la question de d\'eterminer $\Mb^\gauche$ est int\'eressante~:

\bigskip

\begin{probleme}\label{M gauche}
Calculer $R/\rG^\gauche$ ou, du moins, d\'eterminer l'extension 
$\Kb^\gauche \subset \Mb^\gauche$ ou son groupe de Galois (qui est \'egal 
\`a $D^\gauche/I^\gauche$).
\end{probleme}

\bigskip

\begin{question}\label{rgauche c}
Est-ce que $\rG^\gauche + \CG \, R$ est un id\'eal premier de $R$~?
\end{question}

\bigskip

Si la r\'eponse \`a la question~\ref{rgauche c} est positive, alors 
on peut prendre $\rG_\CG^\gauche = \rG^\gauche + \CG\, R$. Cela impliquerait 
entre autres que $D^\gauche \subset D_\CG^\gauche$. Une variante plus faible 
de la question~\ref{rgauche c} est la suivante~:

\bigskip

\begin{question}\label{d gauche}
Peut-on choisir $\rG_\CG^\gauche$ de sorte que $\rG^\gauche \subset \rG_\CG^\gauche$ et 
$D^\gauche \subset D_\CG^\gauche$~?
\end{question}

\bigskip

On a $\pG_\CG^\gauche + \pG_\CG^\droite = \pGba_\CG$. Il est naturel de se demander si cela se rel\`eve \`a $R$~:

\bigskip

\begin{question}\label{r bil}
Est-il vrai que $\rG_\CG^\gauche + \rG_\CG^\droite = \rGba_\CG$~?
\end{question}

\bigskip

Bien s\^ur, $I_\CG^\gauche$ et $I_\CG^\droite$ sont contenus dans $\Iba_\CG$. 

\bigskip

\begin{question}\label{semicontinuite droite gauche}
Est-ce que $\Iba_\CG = \langle I_\CG^\gauche , I_\CG^\droite \rangle$~?
\end{question}

\bigskip

Une r\'eponse positive \`a la question~\ref{semicontinuite droite gauche} impliquerait que toute partie de 
$W$ qui est \`a la fois une union de $\CG$-cellules de Calogero-Moser \`a gauche et une union de $\CG$-cellules 
\`a droite serait aussi une union de $\CG$-cellules bilat\`eres.

\chapter{Matrices de d\'ecomposition}

\section{Cadre g\'en\'eral}\label{part:decomposition}


\medskip

\def\propdec{({\DC\!\acute{e}\!c})}

Soit $R_1$ une $R$-alg\`ebre commutative int\`egre et soit $\rG_1$ un id\'eal premier de $R_1$. 
On pose $R_2=R_1/\rG_1$, $K_1=\Frac(R_1)$ et $K_2=\Frac(R_2)=k_{R_1}(\rG_1)$. 
Nous dirons que le couple $(R_1,\rG_1)$ v\'erifie la propri\'et\'e 
$\propdec$ si les trois assertions suivantes sont satisfaites~:
\begin{quotation}
\begin{itemize}
\item[(D1)] $R_1$ est noeth\'erien, int\`egre.

\item[(D2)] Si $h \in R_1\Hb$ et si $\LC$ est un $K_1\Hb$-module simple, alors 
le polyn\^ome caract\'eristique de $h$ (pour son action sur $\LC$) est \`a coefficients dans $R_1$ 
(notons que cette assertion est automatiquement satisfaite si $R_1$ est int\'egralement clos).

\item[(D3)] Les alg\`ebres $K_1\Hb$ et $K_2\Hb$ sont d\'eploy\'ees.
\end{itemize}
\end{quotation}
Dans ce contexte, totalement similaire \`a celui 
de la section~\ref{section:decomposition} (voir l'appendice~\ref{appendice: blocs}), 
l'application de d\'ecomposition 
$$\dec_{R_2\Hb}^{R_1\Hb} : \groth(K_1\Hb) \longto \groth(K_2\Hb)$$
est bien d\'efinie (voir la proposition~\ref{prop:geck-rouquier}).

Si $\rG_2$ est un id\'eal premier de $R_1$ contenant $\rG_1$, notons 
$R_3=R_1/\rG_2 =R_2/(\rG_2/\rG_1)$, $K_3=\Frac(R_3)=k_{R_1}(\rG_2)=k_{R_2}(\rG_2/\rG_1)$ 
et supposons que $(R_2,\rG_2)$ v\'erifie $\propdec$. 
Ainsi, les applications $\dec_{R_2\Hb}^{R_1\Hb}$, $\dec_{R_3\Hb}^{R_1\Hb}$ et 
$\dec_{R_3\Hb}^{R_2\Hb}$ sont bien d\'efinies et, 
d'apr\`es le corollaire~\ref{geck rouquier}, le diagramme 
\equat\label{transitivite decomposition}
\diagram
\groth(K_1\Hb) \rrto^{\DS{\dec_{R_2\Hb}^{R_1\Hb}}} \ddrrto_{\DS{\dec_{R_3\Hb}^{R_1\Hb}}} 
&& \groth(K_2\Hb) \ddto^{\DS{\dec_{R_3\Hb}^{R_2\Hb}}}\\
&&\\
&&\groth(K_3\Hb)\\
\enddiagram
\endequat
est commutatif.

\bigskip

\begin{exemple}[Sp\'ecialisation]
\label{subsection:specialisation cellules} 
Soit $c \in \CCB$. Rappelons que $\qG_c=\pG_c Q$ est premier et notons 
$\rG_c$ un id\'eal premier de $R$ au-dessus de $\qG_c$. Reprenons les notations 
de l'exemple~\ref{specialisation} et de la sous-section~\ref{subsection:specialisation galois}. 
Alors $R/\rG_c$ est une $R$-alg\`ebre, de corps des fractions $\Mb_c$. 
Comme dans la preuve du th\'eor\`eme~\ref{theo:KH-mat}, on d\'eduit du corollaire~\ref{morita lissite p}
un isomorphisme de $\Kb_c$-alg\`ebres 
$$\Hb_c \longiso \Mat_{|W|}(\Lb_c)$$
qui induit, comme dans le cas g\'en\'erique (voir~\S\ref{section:deploiement}), un isomorphisme de $\Mb_c$-alg\`ebres 
$$\Mb_c \Hb_c 
\stackrel{\sim}{\longto} \prod_{d (D_c \cap H) \in D_c /(D_c \cap H)} \Mat_{|W|}(\Mb_c).$$
Ainsi, la $\Mb_c$-alg\`ebre $\Mb_c\Hb_c$ est d\'eploy\'ee, tout comme $\Mb \Hb$, 
et ses modules simples sont index\'es par $D_c/(D_c \cap H)$~: ce dernier 
ensemble est en bijection naturelle avec $W$ (corollaire~\ref{Dc W}). 
L'application de d\'ecomposition 
$\dec_{R_c\Hb}^{R\Hb}$ est donc bien d\'efinie~: nous la noterons $\dec_c$.  \indexnot{da}{\dec_c}  
On peut de plus identifier 
$\groth(\Mb_c \Hb_c)$ avec le $\ZM$-module $\ZM W$ et, \`a travers cette identification, 
le diagramme
\equat\label{dec c}
\diagram
\groth(\Mb\Hb) \rrto^{\DS{\dec_c}} \ar@{=}[d]&& \groth(\Mb_c\Hb_c) \ar@{=}[d] \\
\ZM W \rrto^{\DS{\Id_{\ZM W}}} && \ZM W
\enddiagram
\endequat
est commutatif. Cela d\'ecoule du fait que l'\'equivalence de Morita entre 
$\Hb_c$ et $\Lb_c$ est la ``sp\'ecialisation en $c$'' de l'\'equivalence de Morita 
entre $\Hb$ et $\Lb$.\finl
\end{exemple}

\bigskip

%
%
%

\section{Cellules et matrices de d\'ecomposition}

\medskip

Soit $\rG$ un id\'eal premier de $R$. Nous noterons 
$D_\rG$ le groupe de d\'ecomposition de $\rG$ dans $G$ et $I_\rG$ son groupe d'inertie.
Le groupe de Galois $G$ (respectivement le groupe de d\'ecomposition $D_\rG$) 
agit naturellement sur le groupe de Grothendieck $\groth(\Mb\Hb)$ (respectivement $\groth(k_R(\rG)\Hb)$). 
Alors~:

\bigskip

\begin{lem}\label{lem:dec-constant}
Supposons que la $k_R(\rG)$-alg\`ebre $k_R(\rG)\Hb$ est d\'eploy\'ee. Alors~:
\begin{itemize}
\itemth{a} L'application de d\'ecomposition $\dec_{(R/\rG)\Hb}^{R\Hb}$ est bien d\'efinie 
(nous la noterons $\dec_\rG : \groth(\Mb\Hb) \longto \groth(k_R(\rG)\Hb)$).  \indexnot{da}{\dec_\rG}  

\itemth{b} L'application de d\'ecomposition $\dec_\rG$ est $D_\rG$-\'equivariante.

\itemth{c} Le groupe $I_\rG$ agit trivialement sur $\groth(k_R(\rG)\Hb)$.

\itemth{d} Si $w$ et $w'$ sont dans la m\^eme $\rG$-cellule de Calogero-Moser, alors 
$\dec_\rG (\LC_w)=\dec_\rG(\LC_{w'})$. 
\end{itemize}
\end{lem}

\begin{proof}
Puisque $R$ est int\'egralement clos, dire que la $k_R(\rG)$-alg\`ebre $k_R(\rG)\Hb$ est d\'eploy\'ee 
\'equivaut \`a dire que $(R,\rG)$ v\'erifie $\propdec$. 
Les applications de d\'ecomposition \'etant calcul\'ees par r\'eduction des polyn\^omes caract\'eristiques, 
l'assertion (b) est imm\'ediate. Le groupe $I_\rG$ agissant trivialement sur $k_R(\rG)$ par d\'efinition, 
(c) est clair. L'assertion (d) d\'ecoule alors de (b) et (c) car les $\rG$-cellules de Calogero-Moser 
sont les $I_\rG$-orbites.
\end{proof}

\bigskip

Le lemme~\ref{lem:dec-constant} dit que, restreinte \`a un $\rG$-bloc, l'application de d\'ecomposition 
$\dec_\rG$ est de rang $1$.

\bigskip

\begin{exemple}\label{exemple:lissite-decomposition}
Posons $\pG = \rG \cap P$ et supposons dans cet exemple, et seulement dans cet exemple, 
que $\qG_\singulier \cap P \not\subset \pG$. 
Alors le th\'eor\`eme~\ref{lissite et simples}(a) nous dit que 
la $k_R(\rG)$-alg\`ebre $k_R(\rG)\Hb$ est d\'eploy\'ee. Par cons\'equent, l'application 
de d\'ecomposition $\dec_\rG : \groth(\Mb\Hb) \to \groth(k_R(\rG)\Hb)$ 
est bien d\'efinie. Puisque les modules simples de $\Mb\Hb$ sont de dimension $|W|$, 
tout comme les $k_R(\rG)\Hb$-modules simples, l'application de d\'ecomposition envoie 
la classe d'un $\Mb\Hb$-module simple sur la classe d'un $k_R(\rG)\Hb$-module simple. 
Ainsi, $\dec_\rG$ d\'efinit une application surjective 
\equat\label{dec irr}
\dec_\rG : W \longto \Irr(k_R(\rG)\Hb)
\endequat
dont les fibres sont les $\rG$-cellules de Calogero-Moser 
(en vertu du lemme~\ref{lem:dec-constant}).\finl 
\end{exemple}

\bigskip

\begin{rema}\label{rema:gauche-droite-decomposition}
L'exemple pr\'ec\'edent s'applique en particulier au cas o\`u $\rG=\rG_\CG^\gauche$ ou $\rG_\CG^\droite$.\finl
\end{rema}

%
%

\bigskip

\section{Cellules \`a gauche, \`a droite, bilat\`eres et matrices de d\'ecomposition}

\medskip

Pour pouvoir d\'efinir des matrices de d\'ecomposition, il faut v\'erifier que 
certaines hypoth\`eses sont satisfaites (voir les conditions (D1), (D2) et (D3) pr\'ec\'edentes). 
C'est l'objet de la proposition suivante que de v\'erifier 
ces hypoth\`eses dans les cas qui nous int\'eressent~:

\bigskip

\begin{prop}\label{prop:d1-d2-d3}
Soit $\rG$ un id\'eal premier de $R$ parmi $\rG_\CG$, $\rG_\CG^\gauche$, $\rG_\CG^\droite$ ou 
$\rGba_\CG$. Alors~:
\begin{itemize}
\itemth{a} La $k_R(\rG)$-alg\`ebre $k_R(\rG)\Hb$ est d\'eploy\'ee.

\itemth{b} Supposons ici $\rG \neq \rGba_\CG$ ou $\CG=0$ ou $\CG_c$ pour un $c \in \CCB$. 
Si $\LC$ est un $k_R(\rG)\Hb$-module simple et si $h \in \Hb/\rG\Hb=(R/\rG)\Hb$, alors 
le polyn\^ome caract\'eristique de $h$ (pour son action sur $\LC$) est \`a coefficients 
dans $R/\rG$.
\end{itemize}
\end{prop}

\begin{proof}
(a) a \'et\'e d\'emontr\'e pour $\rG=\rG_c$ dans l'exemple~\ref{subsection:specialisation cellules}, 
pour $\rG=\rG_\CG^\gauche$ dans le th\'eor\`eme~\ref{theo:gauche}(a) et pour $\rG=\rGba_\CG$ dans la proposition~\ref{verma}.

\medskip

Montrons maintenant (b). Tout d'abord, si $\rG=\rG_\CG$ ou $\rG_\CG^\gauche$ ou $\rG_\CG^\droite$, 
alors les images dans le groupe de Grothendieck $\groth(k_R(\rG)\Hb)$ des $k_R(\rG)\Hb$-modules 
simples sont les images des $\Mb\Hb$-modules simples par l'application de d\'ecomposition 
(voir l'exemple~\ref{exemple:lissite-decomposition} et la remarque~\ref{rema:gauche-droite-decomposition}). 
Ainsi, si $h$ est l'image dans $\Hb/\rG\Hb$ de $h' \in \Hb$, alors le polyn\^ome caract\'eristique 
de $h'$ agissant sur un $\Mb\Hb$-module simple est \`a coefficients dans $R$ (car $R$ 
est int\'egralement clos) et donc le polyn\^ome caract\'eristique de $h$ est \`a coefficient 
dans $R/\rG$ (c'est la r\'eduction modulo $\rG$ de celui de $R$).

Maintenant, si $\rG=\rGba$ ou $\rGba_c$, alors les $k_R(\rG)\Hb$-modules simples 
sont obtenus par extension des scalaires \`a partir des $k_P(\rG \cap P)\Hb$-modules simples, 
et le r\'esultat d\'ecoule alors du fait que $P/\pGba \simeq \kb[\CCB]$ et $P/\pGba_c \simeq \kb$ 
sont int\'egralement clos.
\end{proof}

\bigskip

Compte tenu de la proposition~\ref{prop:d1-d2-d3}, on peut donc d\'efinir 
des applications de d\'ecomposition
$$\dec_\CG : \groth(\Mb\Hb) \longiso \groth(\Mb_\CG\Hb),$$
$$\dec_\CG^\gauche : \groth(\Mb\Hb) \longto \groth(\Mb_\CG^\gauche \Hb^\gauche),$$
\indexnot{d}{\dec_\CG,\dec_\CG^\gauche,\dec_\CG^\droite\dec_\CG^\res}
$$\dec_\CG^\droite : \groth(\Mb\Hb) \longto \groth(\Mb_\CG^\droite \Hb^\droite),$$
$$\decba_\CG : \groth(\Mb\Hb) \longto \groth(\Mbov_\CG \Hbov),$$\indexnot{d}{\decba_\CG^\gauche,\decba_\CG^\droite}
$$\decba_\CG^\gauche : \groth(\Mb_\CG^\gauche\Hb^\gauche) \longto \groth(\Mbov_\CG \Hbov),$$
$$\decba_\CG^\droite : \groth(\Mb_\CG^\droite\Hb^\droite) \longto \groth(\Mbov_\CG \Hbov)$$
$$\dec_\CG^\res : \groth(\Mbov \Hbov) \longto \groth(\Mbov_\CG \Hbov).\leqno{\text{et}}$$
Comme d'habitude, l'indice $\CG$ sera omis lorsque $\CG=0$ ou sera remplac\'e par $c$ si $\CG=\CG_c$ 
(avec $c \in \CCB$). Rappelons que $\dec_\CG$ est un isomorphisme (d'apr\`es l'exemple~\ref{subsection:specialisation cellules}, 
qui s'\'etend facilement au cas o\`u $\CG_c$ est remplac\'e par un id\'eal premier $\CG$ quelconque 
de $\kb[\CCB]$)  et que 
$$\groth(\Mb_\CG\Hb) \simeq \ZM W\qquad\text{et}\qquad \groth(\Mbov_\CG\Hbov) \simeq \ZM \Irr(W).$$
Notons qu'en revanche, 
$\dec_\CG^\res : \groth(\Mbov\hskip1mm \Hbov) \simeq \ZM \Irr(W) \longto \groth(\Mbov_\CG\Hbov) \simeq \ZM \Irr(W)$ 
n'est pas un isomorphisme en g\'en\'eral. 
Des formules de transitivit\'e d\'ecoulent de~\ref{transitivite decomposition}.

\section{Classes des b\'eb\'es modules de Verma}

\medskip

Les modules de Verma $\MC^\gauche(\chi)$ \'etant d\'efinis sur l'anneau $P$, les propri\'et\'es fondamentales 
des applications de d\'ecomposition montrent que 
\equat\label{eq:dec M}
\decba_\CG^\gauche \isomorphisme{\Mb_\CG^\gauche\MC^\gauche(\chi)}_{\Mb_\CG^\gauche\Hb^\gauche}=
\isomorphisme{\Mbov_\CG\MCov(\chi)}_{\Mbov_\CG\Hbov}.
\endequat
Les multiplicit\'es $\mult_{C,\chi}^\calo$ sont d\'efinies \`a partir de l'image 
de $\Mb_\CG^\gauche\MC^\gauche(\chi)$ dans le groupe de Grothendieck 
$\groth(\Mb_\CG^\gauche\Hb^\gauche)$. Nous allons maintenant nous 
int\'eresser \`a l'image de $\Mbov_\CG\MCov(\chi)$ dans le groupe de Grothendieck 
$\groth(\Mbov_\CG\Hbov)$~:

Fixons une $\CG$-cellule de Calogero-Moser bilat\`ere $\G$ et posons 
$L_\CG(\G)=\decba_\CG \isomorphisme{\LC_w}_{\Mb\Hb}$, pour $w \in \G$. 
Notons que $L_\CG(\G)$ ne d\'epend pas du choix de $w \in \G$ d'apr\`es le lemme~\ref{lem:dec-constant}.

\bigskip 

\begin{prop}\label{prop:verma-rang-1}
Si $\chi \in \Irr_\G^\calo(W)$, alors 
$$\isomorphisme{\Mbov_\CG\MCov(\chi)}_{\Mbov_\CG\Hbov} = \chi(1) L_\CG(\G).$$
\end{prop}

\medskip

\begin{rema}\label{rema:dec-rang-1}
La proposition~\ref{prop:verma-rang-1} dit que, \`a l'int\'erieur d'une famille de Calogero-Moser donn\'ee, 
la matrice de d\'ecomposition des b\'eb\'es modules de Verma dans la base des modules 
simples est de rang $1$, r\'esultat conjectur\'e par U. Thiel~\cite{thiel}.\finl
\end{rema}

\medskip

\begin{proof}
Soit $C$ une $\CG$-cellule de Calogero-Moser \`a gauche. Alors $\mult_{C,\chi}^\calo=0$ si $C$ n'est pas 
contenue dans $\G$ (voir la proposition~\ref{multiplicite cm}(c)). Ainsi, d'apr\`es~(\ref{eq:dec M}), on a 
$$\isomorphisme{\Mbov_\CG\MCov(\chi)}_{\Mbov_\CG\Hbov}=\sum_{\substack{C \in \cmcellules^\CG_L(W) \\ C \subset \G}} 
\mult_{C,\chi}^\calo \cdot \decba^\gauche \isomorphisme{\LC_\CG(C)}_{\Mb^\gauche_\CG\Hb^\gauche}.$$
Mais, si $C \subset \G$, alors 
$\isomorphisme{\LC_\CG(C)}_{\Mb^\gauche_\CG\Hb^\gauche}=\dec_\CG^\gauche \isomorphisme{\LC_w}_{\Mb\Hb}$ 
o\`u $w \in C$. Alors, par la transitivit\'e des applications de d\'ecomposition, on a 
$$\decba^\gauche \isomorphisme{\LC_\CG(C)}_{\Mb^\gauche_\CG\Hb^\gauche} = \LC_\CG(\G)$$
(d'apr\`es le lemme~\ref{lem:dec-constant}). Le r\'esultat d\'ecoule alors 
de la proposition~\ref{multiplicite cm}(a).
\end{proof}

\bigskip

Nous terminons par un r\'esultat comparant les $\calo$-caract\`eres $\CG$-cellulaires pour diff\'erents 
id\'eaux premiers $\CG$~:

\bigskip

\begin{prop}\label{prop:cellulaire-semicontinu}
Soit $\CG'$ un id\'eal premier de $\kb[\CCB]$ contenu dans $\CG$ et choisissons un id\'eal premier 
$\rG_{\CG'}^\gauche$ au-dessus de $\qG_{\CG'}^\gauche$ et contenu dans $\rG_\CG^\gauche$. Soit $C$ une $\CG$-cellule de 
Calogero-Moser \`a gauche et \'ecrivons $C=C_1 \coprod \cdots \coprod C_r$, o\`u les $C_i$ sont des 
$\CG'$-cellules de Calogero-Moser \`a gauche (voir la proposition~\ref{semicontinuite gauche}). Alors 
$$\isomorphisme{C}_{\CG}^\calo = \isomorphisme{C_1}_{\CG'}^\calo + \cdots + \isomorphisme{C_r}_{\CG'}^\calo.$$
\end{prop}

\begin{proof}
D'apr\`es la proposition~\ref{prop:d1-d2-d3}, l'application de d\'ecomposition 
$\db : \groth(\Mb_{\CG'}^\gauche\Hb^\gauche) \longto \groth(\Mb^\gauche_\CG \Hb^\gauche)$ est bien d\'efinie et 
elle v\'erifie la propri\'et\'e de transitivit\'e $\db \circ \dec_{\CG'}^\gauche = \dec_\CG^\gauche$. 
De plus, on a 
$$\db~\isomorphisme{\Mb_{\CG'}^\gauche\MC^\gauche(\chi)}_{\Mb_{\CG'}^\gauche\Hb^\gauche} = 
\isomorphisme{\Mb_\CG^\gauche\MC^\gauche(\chi)}_{\Mb_\CG^\gauche\Hb^\gauche}.$$
Le r\'esultat d\'ecoule alors du fait que $\db~\isomorphisme{\LC_{\CG'}(C_i)}_{\Mb_{\CG'}^\gauche\Hb^\gauche}=
\isomorphisme{\LC_\CG^\gauche(C)}_{\Mb_{\CG}^\gauche\Hb^\gauche}$ pour tout $i$ 
(voir l'exemple~\ref{exemple:lissite-decomposition}).
\end{proof}

\bigskip

\section{Probl\`emes, questions}

\medskip

Le calcul des matrices de d\'ecomposition pose aussi de nombreuses questions. 
Une solution au probl\`eme~\ref{probleme c} 
peut \^etre approch\'ee si on sait r\'esoudre le probl\`eme~\ref{probleme generique} 
et calculer la matrice de d\'ecomposition $\dec_c^\res$~:

\bigskip

\begin{probleme}\label{probleme dec}
Calculer $\decba$, $\decba_c$ et $\dec_c^\res$ pour tout $c \in \CCB$.
\end{probleme}

\bigskip

\begin{question}\label{dec c generique}
Soit $c \in \CCB$. Les assertions suivantes sont-elles \'equivalentes~?
\begin{itemize}
 \itemth{1} $c \in \CCB_\gen$.

\itemth{2} $\dec_c^\res = \Id_{\ZM\Irr(W)}$.
\end{itemize}
\end{question}

\bigskip

\chapter{Cellules de Bialinicky-Birula de $\ZCB_c$}\label{chapter:bb}

\boitegrise{{\bf Hypoth\`ese.} 
{\it Dans ce chapitre, nous supposons que $\kb=\CM$ et nous fixons un \'el\'ement $c \in \CCB$.}}{0.75\textwidth}

\bigskip

Le groupe $\CM^\times$ agit sur la vari\'et\'e alg\'ebrique $\ZCB_c$. Nous allons ici r\'einterpr\'eter 
g\'eom\'etriquement diverses notions introduites dans ce m\'emoire (familles, caract\`eres cellulaires,...) en termes 
de cette action (points fixes, ensemble attractifs ou r\'epulsifs,...). Le r\'esultat principal de ce chapitre 
(et peut-\^etre de ce m\'emoire) concerne le cas d'une famille associ\'ee \`a un point lisse de $\ZCB_c$~: 
nous montrerons qu'alors les caract\`eres cellulaires associ\'es \`a cette famille sont irr\'eductibles. 
Ce r\'esultat sera vu comme \'etant de nature g\'eom\'etrique. En effet, la lissit\'e du point fixe 
implique que les ensembles attractifs et r\'epulsifs sont des espaces affines s'intersectant 
proprement et transversalement~; un calcul de multiplicit\'e d'une intersection 
conclura la preuve (voir le th\'eor\`eme~\ref{theo:cellulaire-lisse}).

\bigskip

\def\action{{\hskip0.2mm\SS{\bullet}\hskip0.5mm}}

\section{G\'en\'eralit\'es sur les actions de $\CM^\times$}

\medskip

Soit $\XCB$ une vari\'et\'e alg\`ebrique {\it affine} 
munie d'une $\CM^\times$-action r\'eguli\`ere $\CM^\times \times \XCB \to \XCB$, $(\xi,x) \mapsto \xi \action x$. 
On notera $\XCB^{\CM^\times}$ la sous-vari\'et\'e ferm\'ee form\'ee des points fixes sous l'action de 
$\CM^\times$. 
Si $x \in \XCB$, nous dirons que {\it $\lim_{\xi \to 0} \xi \action x$ existe et vaut $x_0$} s'il existe un 
morphisme de vari\'et\'es $\ph : \CM \to \XCB$ tel que, si $\xi \in \CM^\times$, alors $\ph(\xi)=\xi \action x$ 
et $\ph(0)=x_0$. Il est alors clair que $x_0 \in \XCB^{\CM^\times}$. 
De m\^eme, nous dirons que {\it $\lim_{\xi \to 0} \xi^{-1} \action x$ existe et vaut $x_0$} s'il existe un 
morphisme de vari\'et\'es $\ph : \CM \to \XCB$ tel que, si $\xi \in \CM^\times$, alors $\ph(\xi)=\xi^{-1} \action x$ 
et $\ph(0)=x_0$. 

Nous noterons $\XCB^\attractif$ (respectivement $\XCB^\repulsif$) l'ensemble des $x \in \XCB$ 
tel que $\lim_{\xi \to 0} \xi \action x$ (respectivement $\lim_{\xi \to 0} \xi^{-1} \action x$) existe. 
C'est une sous-vari\'et\'e ferm\'ee de $\XCB$ et les applications 
$$\fonction{\limiteattractive}{\XCB^\attractif}{\XCB^{\CM^\times}}{x}{\lim_{\xi \to 0} \xi \action x}$$
$$\fonction{\limiterepulsive}{\XCB^\repulsif}{\XCB^{\CM^\times}}{x}{\lim_{\xi \to 0} \xi^{-1} \action x}
\leqno{\text{et}}$$
sont des morphismes de vari\'et\'es (qui sont \'evidemment surjectifs~: une section est donn\'ee par 
l'identit\'e sur $\XCB^{\CM^\times} \subset \XCB^\attractif \cap \XCB^\repulsif$), car 
$\XCB$ est affine par hypoth\`ese (cela d\'ecoule du fait que c'est vrai pour la vari\'et\'e 
$\CM^N$ muni d'une action lin\'eaire de $\CM^\times$ et que $\XCB$ peut-\^etre vue comme une sous-vari\'et\'e 
ferm\'ee $\CM^\times$-stable d'un tel $\CM^N$). Notons que cela n'est plus vrai si $\XCB$ n'est pas affine, 
comme le montre l'exemple de $\Pb^1(\CM)$ muni de l'action $\xi \action [x;y]=[\xi x;y]$. 

Pour finir, si $x_0 \in \XCB^{\CM^\times}$, nous noterons $\XCB^\attractif(x_0)$ (respectivement 
$\XCB^\repulsif(x_0)$) l'image inverse de $x_0$ par l'application $\limiteattractive$ 
(respectivement $\limiterepulsive$). Alors $\XCB^\attractif(x_0)$ (respectivement $\XCB^\repulsif(x_0)$) 
sera appel\'e l'{\it ensemble attractif} (respectivement l'{\it ensemble r\'epulsif}) de $x_0$~: 
c'est une sous-vari\'et\'e ferm\'ee de $\XCB$. Rappelons le fait classique suivant, d\^u 
\`a Bialynicki-Birula~\cite{bialynicki}~:

\bigskip

\begin{prop}\label{prop:bb-lisse}
Si $x_0$ est un point lisse de $\XCB$, alors il existe $N \ge 0$ tel que $\XCB^\attractif(x_0) \simeq \CM^N$. 
En particulier, $\XCB^\attractif(x_0)$ est lisse et irr\'eductible.

Les m\^emes assertions valent pour $\XCB^\repulsif(x_0)$.
\end{prop}

\bigskip

Nous allons dans ce chapitre revisiter les notions introduites dans les chapitres pr\'ec\'edents (familles, 
CM-caract\`eres cellulaires) \`a travers les notions de points fixes et d'ensembles attractifs attach\'es 
\`a l'action de $\CM^\times$ sur $\ZCB_c$. 

\bigskip

\section{Points fixes et familles} 

\medskip

Ayant le choix pour l'action de $\CM^\times$ sur nos vari\'et\'es ($\PCB$, $\ZCB$, $\RCB$,...), 
nous ferons le choix de l'action qui induit la $\ZM$-graduation de 
l'exemple~\ref{Z graduation-1}. En d'autres termes, un \'el\'ement $\xi \in \CM^\times$ agit 
sur $\Hb$ comme l'\'el\'ement $(\xi^{-1},\xi,1 \rtimes 1)$ de $\CM^\times \times \CM^\times \times (\Hom(W,\CM^\times) \rtimes \NC)$. 
Ainsi, pour l'action sur $\Hb$, $\xi$ agit trivialement sur $\CM[\CCB] \otimes \CM W$, agit avec des poids positifs ou nuls 
sur $\CM[V]$, avec des poids n\'egatifs ou nuls sur $\CM[V^*]$. On r\'ecup\`ere ainsi une action sur 
$P$ et $Z_c$, qui induisent des actions r\'eguli\`eres de $\CM^\times$ sur les vari\'et\'es $\PC_\bullet \simeq V/W \times V^*/W$ 
et $\ZCB_c$ rendant le morphisme naturel
$$\Upsilon_c : \ZCB_c \longto \PCB_\bullet = V/W \times V^*/W$$
$\CM^\times$-\'equivariant. Si $\xi \in \CM^\times$ et $z \in \ZCB_c$, l'image de $z$ \`a travers cette action de $\xi$ sera not\'ee 
$\xi \action z$. L'unique point fixe de $\PCB_\bullet$ est $(0,0)$~:
\equat\label{eq:fixe-P}
\PCB_\bullet^{\CM^\times} = (0,0).
\endequat
Puisque $\Upsilon_c$ est un morphisme fini, on en d\'eduit que
\equat\label{eq:fixe-Z}
\ZCB_c^{\CM^\times} = \Upsilon_c^{-1}(0,0).
\endequat
On obtient ainsi une bijection naturelle entre $\ZCB_c^{\CM^\times}$ et les $c$-familles de Calogero-Moser.

\bigskip

\section{Ensembles attractifs et caract\`eres cellulaires} 

\medskip

Tout d'abord, il est clair que 
\equat\label{eq:P-attractif}
\PC_\bullet^\attractif = V/W \times 0 \subset V/W \times V^*/W\quad\text{et}\quad 
\PC_\bullet^\repulsif = 0 \times V^*/W \subset V/W \times V^*/W.
\endequat
En d'autres termes, $\PC_\bullet^\attractif$ est la sous-vari\'et\'e irr\'eductible de $\PCB_\bullet$ associ\'ee 
\`a l'id\'eal premier $\pG_c^\gauche$. 
De plus, puique $\Upsilon_c$ est un morphisme fini, on a 
\equat\label{eq:Z-attractif}
\ZCB_c^\attractif = \Upsilon_c^{-1}(V/W \times 0)\quad\text{et}\quad 
\ZCB_c^\repulsif = \Upsilon_c^{-1}(0 \times V^*/W).
\endequat
\begin{proof}
Il suffit de montrer la premi\`ere \'egalit\'e, l'autre se montrant de m\^eme. 
Tout d'abord, il est clair que $\Upsilon_c(\ZCB_c^\attractif) \subset \PCB_\bullet^\attractif = V/W \times 0$. 
Pour montrer l'inclusion inverse, il suffit de montrer que, si $\r : Z_c \to \CM[\tb,\tb^{-1}]$ 
est un morphisme de $\CM$-alg\`ebres tel que $\r(P_\bullet) \subset \CM[\tb]$, alors 
$\r(Z_c) \subset \CM[\tb]$. Mais cela d\'ecoule du fait que $Z_c$ est entier sur $P_\bullet$ 
(et donc $\r(Z_c)$ est entier sur $\r(P_\bullet)$) et du fait que $\CM[\tb]$ est 
int\'egralement clos.
\end{proof}

\bigskip

Il d\'ecoule de \ref{eq:Z-attractif} que les composantes irr\'eductibles de $\ZCB_c^\attractif$ 
sont en bijection avec l'ensemble $\Upsilon_c^{-1}(\pG_c^\gauche)$ des id\'eaux premiers de $Z_c$ 
au-dessus de $\pG_c^\gauche$. Si $\zG$ est un tel id\'eal premier, nous noterons 
$\ZCB_c^\attractif[\zG]$ la composante irr\'eductible associ\'ee. 
Puisque $\limiteattractive : \ZCB_c^\attractif \to \ZCB_c^{\CM^\times}$ est un morphisme de vari\'et\'es, 
l'image de $\ZCB_c^\attractif[\zG]$ est irr\'eductible~: $\ZCB_c^{\CM^\times}$ \'etant fini, 
$\limiteattractive(\ZCB_c^\attractif[\zG])$ est donc r\'eduit \`a un point. Ainsi, le morphisme de 
vari\'et\'es $\limiteattractive : \ZC_c^\attractif \to \ZC_c^{\CM^\times}$ induit 
une application surjective $\Upsilon_c^{-1}(\pG_c^\gauche) \longto \Upsilon_c^{-1}(\pGba_c)$ 
qui n'est autre que l'application $\limitegauche$ d\'efinie dans la section~\ref{section:premiers}. 

On obtient ainsi une application entre les composantes irr\'eductibles de $\ZCB_c^\attractif$ et 
l'ensemble des CM-caract\`eres $c$-cellulaires (voir la remarque~\ref{rema:cellulaires-pas-choix}). 

%
%
%
%
%
%
%
%

\bigskip

\section{Le cas lisse}

\medskip

\boitegrise{{\bf Hypoth\`eses et notations.} 
{\it Nous fixons dans cette section, et dans cette section seulement, un point $z_0 \in \ZCB_c^{\CM^\times}$ 
que nous supposerons {\bfit lisse} dans $\ZCB_c$. On note $\chi$ l'unique caract\`ere irr\'eductible 
de la $c$-famille de Calogero-Moser associ\'ee \`a $z_0$. Nous notons $\G$ la $c$-cellule de Calogero-Moser 
bilat\`ere de $W$ associ\'ee \`a $z_0$ et nous fixons une $c$-cellule de Calogero-Moser \`a gauche $C$ 
contenue dans $\G$.}}{0.75\textwidth}

\bigskip

Le but de cette section est de d\'emontrer le r\'esultat suivant~:

\bigskip

\begin{theo}\label{theo:cellulaire-lisse}
Avec les hypoth\`eses et notations ci-dessus, on a~:
\begin{itemize}
\itemth{a} $|\G|=\chi(1)^2$.

\itemth{b} $\bigcup_{d \in D_c^\gauche} \lexp{d}{C}=\G$. 

\itemth{c} $|C|=\chi(1)$.

\itemth{d} $\isomorphisme{C}_c^\calo = \chi$.

\itemth{e} $\deg_c(C)=\chi(1)$. 
\end{itemize}
\end{theo}

\bigskip

\noindent{\sc Notation - } Si $A$ est un anneau commutatif local, d'id\'eal maximal $\mG$ et si $M$ est un $A$-module 
de type fini, on notera $e_\mG(M)$ la {\it multiplicit\'e} de $M$ pour l'id\'eal $\mG$, 
telle qu'elle est d\'efinie dans~\cite[chapitre~\MakeUppercase{\romannumeral 5},~\S{A.2}]{serre}. 

Si $A$ est un anneau commutatif r\'egulier (pas forc\'ement local), si $M$ et $N$ sont deux $A$-modules de type fini 
tels que $M \otimes_A N$ soit de longueur fini, et si $\aG$ est un id\'eal premier de $A$, on posera
$$\chi_\aG(M,N)=\sum_{i=0}^{\dim A} (-1)^i \longueur_{A_\aG}(\Tor_i^A(M,N)_\aG),$$
conform\'ement \`a~\cite[chapitre~\MakeUppercase{\romannumeral 5},~\S{B},~th\'eor\`eme~1]{serre} .\finl

\bigskip

\begin{proof}
(a) d\'ecoule du th\'eor\`eme~\ref{theo cellules familles}(d).

\medskip

(b) L'ensemble des composantes irr\'eductibles de $\limiteattractiveinverse(z_0)$ est en bijection 
avec $\G/D_c^\gauche$ (voir la proposition~\ref{prop:gauche-bilatere}(c)). Puisque $z_0$ est lisse (et isol\'e), alors 
$$\ZC_c^\attractif(z_0) \simeq \CM^n,$$
et donc que $\ZC_c^\attractif(z_0)$ est lisse et irr\'eductible (voir la proposition~\ref{prop:bb-lisse}). Cela 
montre que $\limiteattractiveinverse(z_0)$ est irr\'eductible et isomorphe \`a un espace affine. Ainsi $|\G/D_c^\gauche|=1$, 
ce qui montre (b). En d'autres termes, avec les notations introduites dans le chapitre pr\'ec\'edent, on a $C^\DD=\G$. 

\medskip

(c) Notons $\zG_L=\zG_c^\gauche(C)$~: alors $\zG_L$ est l'id\'eal de d\'efinition 
de $\ZC_c^\attractif(z_0)$, mais nous le verrons comme un id\'eal premier de $Z_c$ 
par passage au quotient. Notons $\zGba=\zGba_c(\G)$~: alors $\zGba$ est l'id\'eal de d\'efinition 
du point $z_0$ (nous le verrons aussi comme un id\'eal de $Z_c$). On d\'efinit de m\^eme $\zG_R$ comme 
\'etant l'id\'eal de d\'efinition de $\ZC_c^\repulsif(z_0)$~: on note $C'$ une $c$-cellule \`a droite 
contenue dans $\G$ (de sorte que $\zG_R=\zG_c^\droite(C')$). Le corollaire~\ref{coro:cardinal-C} montre que 
$$|C|=\longueur_{Z_{c,\zG_L}}(Z_c/\pG_c^\gauche Z_c)\quad\text{et}\quad
|C'|=\longueur_{Z_{c,\zG_R}}(Z_c/\pG_c^\droite Z_c).\leqno{(\clubsuit)}$$
Si on note $m_L=\mult_{C,\chi}^\calo$, alors 
$$|C|=m_L \chi(1) \quad\text{et}\quad \isomorphisme{C}_c^\calo = m_L\chi.\leqno{(\diamondsuit_L)}$$
Sym\'etriquement, on peut d\'efinir un entier naturel $m_R$ v\'erifiant
$$|C'|=m_R \chi(1).\leqno{(\diamondsuit_R)}$$

\medskip

Commen\c{c}ons par calculer la multiplicit\'e du $Z_{c,\zGba}$-module $Z_{c,\zGba}/\pG_c^\gauche Z_{c,\zGba}$ 
pour $\zGba Z_{c,\zGba}$. La dimension de Krull de ce module est $n=\dim_\kb V$. 
Par la formule d'additivit\'e~\cite[chapitre~\MakeUppercase{\romannumeral 5},~\S{A.2}]{serre}, 
$$e_{\zGba Z_{c,\zGba}}(Z_{c,\zGba}/\pG_c^\gauche Z_{c,\zGba}) = \sum_{{\mathrm{coht}}(\zG)=n} 
\longueur_{Z_{c,\zG}}(Z_{c,\zG}/\pG_c^\gauche Z_{c,\zG}) e_{\zGba Z_{c,\zGba}}(Z_{c,\zGba}/\zG Z_{c,\zGba}).$$
Ici, ${\mathrm{coht}}(\zG)$ d\'esigne la cohauteur de l'id\'eal premier $\zG$ de $Z_{c,\zGba}$. 
Mais, puisque $\ZC_c^\attractif(z_0)$ est irr\'eductible de dimension $n$, il n'y a qu'un seul id\'eal 
premier de $Z_{c,\zGba}$ de cohauteur $n$ qui contienne $\pG_c^\gauche Z_{c,\zGba}$ (et donc tel que 
$\longueur_{Z_{c,\zG}}(Z_{c,\zG}/\pG_c^\gauche Z_{c,\zG})$ soit non nul), c'est l'id\'eal premier $\zG_L$. 
De plus, puisque $Z_{c,\zGba}/\zG_L Z_{c,\zGba}$ est un anneau r\'egulier (car $\ZC_c^\attractif(z_0)$ est lisse), 
la multiplicit\'e $e_{\zGba Z_{c,\zGba}}(Z_{c,\zGba}/\zG Z_{c,\zGba})$ vaut $1$ 
(voir~\cite[chapitre~\MakeUppercase{\romannumeral 4}]{serre}). Ainsi, il d\'ecoule de $(\clubsuit)$ et 
$(\diamondsuit_L)$ que 
$$e_{\zGba Z_{c,\zGba}}(Z_{c,\zGba}/\pG_c^\gauche Z_{c,\zGba})=m_L \chi(1).\leqno{(\heartsuit_L)}$$
Sym\'etriquement, 
$$e_{\zGba Z_{c,\zGba}}(Z_{c,\zGba}/\pG_c^\droite Z_{c,\zGba})=m_R \chi(1).\leqno{(\heartsuit_R)}$$

\medskip

D'autre part, $P/\pG_c^\gauche$ est un anneau de polyn\^ome et $Z_c/\pG_c^\gauche Z_c$ est un 
$P/\pG_c^\gauche$-module libre de rang $|W|$. Donc $Z_{c,\zGba}/\pG_c^\gauche Z_{c,\zGba}$ est un 
$Z_{c,\zGba}$-module de Cohen-Macaulay de dimension $n$. De m\^eme, $Z_{c,\zGba}/\pG_c^\droite Z_{c,\zGba}$ est un 
$Z_{c,\zGba}$-module de Cohen-Macaulay de dimension $n$. Puisque $Z_c$ est de dimension $2n$, 
il d\'ecoule de~\cite[chapitre~\MakeUppercase{\romannumeral 5},~\S{B},~corollaire~du~th\'eor\`eme~1]{serre}
que 
$$\chi_{\zGba}(Z_{c,\zGba}/\pG_c^\gauche Z_{c,\zGba},Z_{c,\zGba}/\pG_c^\droite Z_{c,\zGba}) 
= \longueur_{Z_{c,\zGba}}(Z_{c,\zGba}/\pG_c^\gauche Z_{c,\zGba} \otimes_{Z_{c,\zGba}} Z_{c,\zGba}/\pG_c^\droite Z_{c,\zGba}) 
= \chi(1)^2 > 0.\leqno{(\spadesuit)}$$
Par cons\'equent~\cite[chapitre~\MakeUppercase{\romannumeral 5},~\S{B},~th\'eor\`eme~1]{serre}, 
$$e_{\zGba Z_{c,\zGba}}(Z_{c,\zGba}/\pG_c^\gauche Z_{c,\zGba}) \cdot 
e_{\zGba Z_{c,\zGba}}(Z_{c,\zGba}/\pG_c^\droite Z_{c,\zGba}) \le 
\chi_{\zGba}(Z_{c,\zGba}/\pG_c^\gauche Z_{c,\zGba},Z_{c,\zGba}/\pG_c^\droite Z_{c,\zGba}).$$
De cette derni\`ere \'egalit\'e, on d\'eduit, en utilisant $(\heartsuit_L)$, $(\heartsuit_R)$ et $(\spadesuit)$ 
que 
$$m_L m_R \le 1.$$
Ainsi, $m_L=m_R=1$, ce qui prouve (c)

\medskip

(d) d\'ecoule imm\'ediatement de (c).

\medskip

(e) d\'ecoule de (b) et de la proposition~\ref{prop:gauche-bilatere}(b).
\end{proof}

\bigskip

\part{Groupes de Coxeter~: Calogero-Moser vs Kazhdan-Lusztig}\label{part:coxeter}

\cbstart

\boitegrise{{\bf Hypoth\`eses.} {\it Tout au long de cette partie, nous supposerons 
que $W$ est un groupe de Coxeter, que $\kb=\CM$ et que $\kb_\RM=\RM$. Nous fixons aussi un 
\'el\'ement $c \in \CCB_\RM \subset \CCB$.}}{0.75\textwidth}

\bigskip

Dans ce cas, nous avons rappel\'e dans la section~\ref{section:cellules-kl} la d\'efinition des 
$c$-cellules de Kazhdan-Lusztig \`a droite, \`a gauche et bilat\`eres, des $c$-familles de Kazhdan-Lusztig, 
et des $\kl$-caract\`eres $c$-cellulaires. De m\^eme, les notions de $c$-cellule de Calogero-Moser 
\`a droite, \`a gauche et bilat\`ere, de $c$-famille de Calogero-Moser 
et de $\calo$-caract\`ere $c$-cellulaire ont \'et\'e d\'efinies et longuement \'etudi\'ees 
dans la partie~\ref{part:verma}. Il est assez tentant de conjecturer 
que ces notions co\"{\i}ncident. Le but de cette partie est d'\'enoncer des conjectures aussi  
pr\'ecises que possible et de donner des arguments en faveur de ces conjectures.

\chapter{Conjectures}\label{chapter:conjectures}

\section{Cellules et caract\`eres} 

\medskip

La premi\`ere conjecture concerne les cellules bilat\`eres et leur famille associ\'ee~:

\bigskip

\begin{conjecturebil}
Soit $c \in \CCB_\RM$. Alors il existe un choix de l'id\'eal premier $\rGba_c$ au-dessus de $\qGba_c$ tel que~:
\begin{itemize}
\itemth{a} La partition de $W$ en $c$-cellules de Calogero-Moser bilat\`eres co\"{\i}ncide avec la 
partition en $c$-cellules de Kazhdan-Lusztig bilat\`eres.

\itemth{b} Supposons que $c_s \ge 0$ pour tout $s \in \Ref(W)$. 
Si $\G \in \cmcellules_{LR}^c(W)=\klcellules_{LR}^c(W)$, alors $\Irr_\G^\calo(W)=\Irr_\G^\kl(W)$. 
\end{itemize}
\end{conjecturebil}

\bigskip

Du c\^ot\'e des cellules \`a gauche, nous proposons une conjecture similaire, qui se 
d\'ecline aussi au niveau des caract\`eres cellulaires~: 

\bigskip

\begin{conjectureleft}
Soit $c \in \CCB_\RM$. Alors il existe un choix de l'id\'eal premier $\rG^\gauche_c$ au-dessus de $\qG^\gauche_c$ tel que~:
\begin{itemize}
\itemth{a} La partition de $W$ en $c$-cellules de Calogero-Moser \`a gauche co\"{\i}ncide avec la 
partition en $c$-cellules de Kazhdan-Lusztig \`a gauche.

\itemth{b} Supposons que $c_s \ge 0$ pour tout $s \in \Ref(W)$. 
Si $C \in \cmcellules_L^c(W)=\klcellules_L^c(W)$, alors $\isomorphisme{C}_c^\calo=\isomorphisme{C}_c^\kl$. 
\end{itemize}
\end{conjectureleft}

\bigskip

On a bien s\^ur une conjecture identique pour les cellules \`a droite. En outre, les conjectures~\BIL et~\GAUCHE~devraient 
avoir des r\'eponses positives avec $\rG_c^\gauche \subset \rGba_c$. 

\bigskip

\section{Caract\`eres}

\medskip

%
%
%
%
%
%

Tout d'abord, remarquons que l'ensemble des $c$-familles de Calogero-Moser est ind\'ependant du choix de 
$\rGba_c$. De m\^eme, l'ensemble des $\calo$-caract\`eres $c$-cellulaires ne d\'epend pas du choix de 
l'id\'eal $\rG_c^\gauche$. Au niveau des caract\`eres, les \'enonc\'es (b) des conjectures \BIL~ et \GAUCHE~ 
impliquent les \'enonc\'es suivants, plus simples, qui ne font pas r\'ef\'erence au choix d'un id\'eal premier 
de $R$. 

\bigskip

\begin{conjecturecar}
Soit $c \in \CCB_\RM$. Alors~:
\begin{itemize}
\itemth{a} La partition de $\Irr(W)$ en $c$-familles de Calogero-Moser co\"{\i}ncide avec la 
partition en $c$-familles de Kazhdan-Lusztig (Gordon-Martino).

\itemth{b} Si $c_s \ge 0$ pour tout $s \in \Ref(W)$, alors l'ensemble des 
$\calo$-caract\`eres $c$-cellulaires co\"{\i}ncide avec l'ensemble des 
$\kl$-caract\`eres   $c$-cellulaires. 
\end{itemize}
\end{conjecturecar}

\bigskip

Notons que le point (a) ci-dessus a \'et\'e conjectur\'e par Gordon et 
Martino~\cite[conjecture 1.3(1)]{gordon martino}. La conjecture \BIL~ rel\`eve 
donc la conjecture de Gordon-Martino au niveau des cellules bilat\`eres. 

\bigskip

\noindent{\sc Commentaire - } 
Si le choix des id\'eaux premiers $\rGba_c$ ou $\rG_c^\gauche$ n'a aucune importance 
pour la conjecture \CAR, il n'en est pas de m\^eme pour les conjectures \BIL~ et \GAUCHE. 
En effet, remplacer $\rGba_c$ par un autre id\'eal $\rGba_c'$ transforme 
l'ensemble des cellules bilat\`eres \`a travers l'action d'un \'el\'ement $g \in G \subset \SG_W$ 
(tel que $\rGba_c'=g(\rGba_c)$). Or, dans certains cas, $G$ est le groupe $\SG_W$ tout entier...

D'autre part, la partition en $c$-cellules de Kazhdan-Lusztig d\'epend elle aussi 
fortement du choix d'un ensemble $S$ de r\'eflexions simples (changer cet ensemble 
reviendrait \`a conjuguer les cellules par un \'el\'ement de $W$). Il est donc imp\'eratif 
de lier le choix de $\rGba_c$ (ou $\rG_c^\gauche$) au choix de $S$, ce que nous ne 
savons pas faire pour le moment.\finl

\chapter{Arguments}\label{chapter:arguments}

Comme cela sera expliqu\'e dans la partie~\ref{part:exemples}, toutes les conjectures \'enonc\'ees 
dans le chapitre~\ref{chapter:conjectures} sont v\'erifi\'ees si $W$ est de type $A_1$ ou $B_2$~: 
le cas du type $A_2$ peut aussi \^etre trait\'e mais nous n'avons pas inclus les calculs 
dans ce m\'emoire. Le cas du type $B_2$ est trait\'e dans le chapitre~\ref{chapitre:b2}. 
Cependant, la difficult\'e des calculs ne nous permet pas, en l'\'etat actuel de nos connaissances, 
d'allonger cette liste d'exemples. Le but de ce chapitre est plut\^ot 
d'avancer des arguments th\'eoriques en faveur de ces conjectures, qui seront 
distill\'es sous forme de remarques successives.

\bigskip

\section{Le cas ${\boldsymbol{c=0}}$}

\medskip

Les faits suivants seront d\'emontr\'es 
dans le chapitre~\ref{chapitre nul}~:

\bigskip

\begin{prop}\label{lem:nul}
Il n'y a qu'une seule $0$-cellule de Calogero-Moser bilat\`ere, \`a gauche ou \`a droite~: c'est $W$ tout 
entier. De plus,
$$\Irr_W^\calo(W)=\Irr(W)\qquad\text{et}\qquad 
\isomorphisme{W}_0^\calo=\isomorphisme{\kb W}_{\kb W} = \sum_{ \chi \in \Irr(W)} \chi(1) \chi.$$
\end{prop}

\bigskip

\begin{coro}\label{coro:nul}
Les conjectures~L et~LR sont v\'erifi\'ees lorsque $c=0$.
\end{coro}

\begin{proof}
Cela d\'ecoule de la comparaison de~\cite[corollaires~2.13~et~2.14]{bonnafe continu} 
avec la proposition~\ref{lem:nul}.
\end{proof}

\bigskip

\section{Caract\`eres constructibles, familles de Lusztig} 

\medskip

Dans la suite de ce chapitre, nous nous int\'eresserons au cas des param\`etres strictement positifs~: 
nous ne savons pas traiter le cas o\`u seulement certains param\`etres sont nuls (pour 
comparer avec~\cite[corollaires~2.13~et~2.14]{bonnafe continu}). 

\bigskip

\boitegrise{Fixons dans cette section un \'el\'ement $c \in \CCB_\RM$ tel que $c_s > 0$ pour 
tout $s \in \Ref(W)$.}{0.75\textwidth}

\bigskip

\noindent{\sc Convention - } Si $(W,S)$ est de type $B_n$, alors nous noterons 
$S=\{t,s_1,s_2,\dots,s_{n-1}\}$ avec la convention que $t$ n'est conjugu\'e \`a 
aucun des $s_i$. En d'autres termes, le diagramme de Dynkin est 
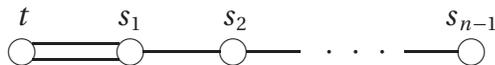
\begin{center}
\begin{picture}(220,30)
\put( 40, 10){\circle{10}}
\put( 44,  7){\line(1,0){33}}
\put( 44, 13){\line(1,0){33}}
\put( 81, 10){\circle{10}}
\put( 86, 10){\line(1,0){29}}
\put(120, 10){\circle{10}}
\put(125, 10){\line(1,0){20}}
\put(155,  7){$\cdot$}
\put(165,  7){$\cdot$}
\put(175,  7){$\cdot$}
\put(185, 10){\line(1,0){20}}
\put(210, 10){\circle{10}}
\put( 38, 20){$t$}
\put( 76, 20){$s_1$}
\put(116, 20){$s_2$}
\put(200, 20){$s_{n{-}1}$}
\end{picture}
\end{center}
Dans ce cas, nous poserons $b=c_t$ et $a=c_{s_1}=c_{s_2}=\cdots = c_{s_{n-1}}$.\finl

\bigskip

Lusztig~\cite[\S{22}]{lusztig} a d\'efini la notion de {\it caract\`ere constructible} de $W$, 
que nous appellerons ici {\it caract\`ere $c$-constructible}. On peut alors d\'efinir un graphe $\GC_c(W)$ 
ainsi~: 
\begin{itemize}
\item[$\bullet$] L'ensemble des sommets de $\GC_c(W)$ est $\Irr(W)$.

\item[$\bullet$] Deux caract\`eres irr\'eductibles de $W$ sont reli\'es dans $\GC_c(W)$ s'ils 
apparaissent dans un m\^eme caract\`ere $c$-constructible. 
\end{itemize}
On d\'efinit alors les {\it $c$-familles de Lusztig} comme les composantes connexes de $\GC_c(W)$. 

\bigskip

\begin{prop}\label{prop:cellulaire-constructible}
Supposons que l'une des assertions suivantes soit satisfaite~:
\begin{itemize}
\itemth{1} $c$ est constante~;

\itemth{2} $|S| \le 2$~;

\itemth{3} $(W,S)$ est de type $F_4$~;

\itemth{4} $(W,S)$ est de type $B_n$, $a \neq 0$ et $b/a \in \{1/2,1,3/2,2\} \cup ]n-1, + \infty)$. 
\end{itemize}
Alors~:
\begin{itemize}
\itemth{a} Les caract\`eres $c$-constructibles et les KL-caract\`eres $c$-cellulaires co\"{\i}ncident.

\itemth{b} Les $c$-familles de Lusztig et les $c$-familles de Kazhdan-Lusztig co\"{\i}ncident. 
\end{itemize}
\end{prop}

\begin{proof}
Lusztig~\cite[conjectures~14.2]{lusztig} a \'enonc\'e une s\'erie de conjectures 
(num\'erot\'ees P1, P2,\dots, P15) portant sur les cellules de Kazhdan-Lusztig. 
Elles ont \'et\'e d\'emontr\'ees~:
\begin{itemize}
\itemth{1} si $c$ est constante dans~\cite[chapitre~15]{lusztig}~;

\itemth{2} si $|S| \le 2$ dans~\cite[chapitre~17]{lusztig}~;

\itemth{3} si $(W,S)$ est de type $F_4$ dans~\cite{geck f4}~;

\itemth{4} si $(W,S)$ est de type $B_n$ et $a=0$ ou bien $a \neq 0$ et $b/a \in \{1/2,1,3/2,2\}$ 
dans~\cite[chapitre~16]{lusztig}~;

\itemth{4'} si $(W,S)$ est de type $B_n$, $a \neq 0$ et $b/a > n-1$ 
dans~\cite{bonnafe iancu},~\cite{bonnafe two} et~\cite{geck iancu}.
\end{itemize}
D'autre part, il est d\'emontr\'e dans~\cite[lemme~22.2]{lusztig} et~\cite[\S{6}~et~\S{7}]{geck plus} que 
ces conjectures impliquent que les caract\`eres $c$-constructibles et les KL-caract\`eres $c$-cellulaires 
co\"{\i}ncident. Cela montre (a). L'\'enonc\'e (b) est alors d\'emontr\'e dans~\cite[corollaire~1.8]{bonnafe geck}.
\end{proof}

\bigskip

\section{Conjectures sur les caract\`eres} 

\medskip

\subsection{Familles} 
Les caract\`eres $c$-constructibles (et donc les $c$-familles de Lusztig) 
ont \'et\'e calcul\'e(e)s dans tous les cas par Lusztig~\cite{lusztig}. 
On d\'eduit de la proposition~\ref{prop:cellulaire-constructible} ce que sont les $c$-familles 
de Kazhdan-Lusztig dans les cas (1), (2), (3) et (4). Or, le calcul explicite des 
$c$-familles de Calogero-Moser en type classique a \'et\'e effectu\'e par Bellamy, Gordon et Martino 
dans la s\'erie d'articles~\cite{bellamy these}, \cite{bellamy}, 
\cite{gordon}, \cite{gordon B}, \cite{gordon martino}, 
\cite{martino 2}. Il en r\'esulte le th\'eor\`eme suivant~:

\bigskip
 
\begin{theo}\label{theo:gordon-martino-bellamy}
Supposons que l'une des assertions suivantes soit satisfaite~:
\begin{itemize}
\itemth{1} $|S| \le 2$.

\itemth{2} $(W,S)$ est de type $A_n$, $D_n$ ou $F_4$.

\itemth{3} $(W,S)$ est de type $B_n$ et $a=0$ ou bien $a \neq 0$ et $b/a \in \{1/2,1,3/2,2\} \cup ]n-1, + \infty)$. 
\end{itemize}
Alors la conjecture~\CAR(a) est v\'erifi\'ee.
\end{theo}

\bigskip

\subsection{Caract\`eres cellulaires} 
Si $(W,S)$ est de type $A$ ou si $(W,S)$ est de type $B_n$ avec $a \neq 0$ et 
$b/a \in \{1/2,3/2\} \cup ]n-1, + \infty)$, alors il d\'ecoule des r\'esultats pr\'ec\'edents 
que les KL-caract\`eres $c$-cellulaires sont les caract\`eres irr\'eductibles. 
De plus, il d\'ecoule aussi des travaux de Gordon et Martino que, 
toujours dans le m\^eme cas, l'espace de Calogero-Moser $\ZCB_c$ est lisse. 
Le th\'eor\`eme suivant r\'esulte alors du th\'eor\`eme~\ref{theo:cellulaire-lisse}.

\bigskip

\begin{theo}\label{theo:cellulaire-conjecture}
Supposons que l'on est dans l'un des cas suivants~:
\begin{itemize}
\itemth{1} $(W,S)$ est de type $A$~;

\itemth{2} $(W,S)$ est de type $B_n$ avec $a \neq 0$ et 
$b/a \in \{1/2,3/2\} \cup ]n-1, + \infty)$.
\end{itemize}
Alors la conjecture~\CAR(b) est vraie (et les caract\`eres $c$-cellulaires sont irr\'eductibles).
\end{theo}

\bigskip

\subsection{Autres arguments} 
Tout d'abord, remarquons que, si l'on suppose les conjectures de Lusztig~P1, P2,\dots, P15 vraies 
(voir~\cite[conjectures 14.2]{lusztig}), alors les raisonnements pr\'ec\'edents impliquent 
que la conjecture~\CAR(b) est vraie en type $B$ et la conjecture~\CAR(a) est vraie 
en type $B$ avec $a \neq 0$ et $b/a \not\in \{1,2,\dots,n-1\}$ (car alors les 
caract\`eres $c$-constructibles sont les caract\`eres irr\'eductibles et l'espace de Calogero-Moser 
est lisse).  


\bigskip

\begin{rema}\label{rema:epsilon}
Si $\FC$ est une $c$-famille de Calogero-Moser (respectivement Kazhdan-Lusztig), alors 
$\FC\e$ est une $c$-famille de Calogero-Moser (respectivement Kazhdan-Lusztig)~: 
voir le corollaire~\ref{ordre 2} et~(\ref{eq:familles-cellulaire-w0}).

De m\^eme, si $\chi$ est un CM-caract\`ere (respectivement un KL-caract\`ere) 
$c$-cellulaire, alors $\chi \e$ est un CM-caract\`ere (respectivement un KL-caract\`ere) 
$c$-cellulaire~: voir le corollaire~\ref{coro:cellulaire-ordre-2} 
et~(\ref{eq:caractere-cellulaire-w0}).\finl
\end{rema}

\bigskip

\begin{rema}\label{rema:b-invariant}
Si $\FC$ est une $c$-famille de Calogero-Moser (respectivement de Lusztig), alors 
il existe un unique caract\`ere $\chi \in \FC$ de $\bb$-invariant minimal~: 
voir le th\'eor\`eme~\ref{dim graduee bonne}(b) (respectivement~\cite{bonnafe b}, 
ou~\cite[theor\`eme~5.25~et~sa~preuve]{lusztig orange} 
dans le cas o\`u $c$ est constant).

De m\^eme, si $\chi$ est un CM-caract\`ere $c$-cellulaire (respectivement un caract\`ere 
$c$-constructible), alors il existe une unique composante irr\'eductible de $\chi$ de $\bb$-invariant minimal~: 
voir le th\'eor\`eme~\ref{theo:b-minimal-cellulaire} (respectivement~\cite{bonnafe b}, 
ou~\cite[theor\`eme~5.25~et~sa~preuve]{lusztig orange} dans le cas o\`u $c$ est constant).\finl
\end{rema}

\bigskip

\section{Cellules} 

\medskip

\subsection{Cellules bilat\`eres} 
Le premier argument en faveur le la conjecture \BIL~ vient de la comparaison du cardinal des cellules, et
du fait que la conjecture~\CAR(a) a \'et\'e d\'emontr\'ee dans de nombreux cas. 

\bigskip

\begin{rema}\label{rem:cardinal-cellules-kl-cm}
Supposons ici que $(W,c)$ v\'erifie l'une des hypoth\`eses du th\'eor\`eme~\ref{theo:gordon-martino-bellamy}. 
Soit $\FC$ une $c$-famille de Calogero-Moser 
(c'est-\`a-dire une $c$-famille de Kazhdan-Lusztig en vertu du th\'eor\`eme~\ref{theo:gordon-martino-bellamy}). 
Notons $\G_\calo$ (respectivement $\G_\kl$) la $c$-cellule de Calogero-Moser (respectivement Kazhdan-Lusztig) 
bilat\`ere associ\'ee. Alors, il d\'ecoule du th\'eor\`eme~\ref{theo cellules familles}(d)  que 
$$|\G_\calo| = \sum_{\chi \in \FC} \chi(1)^2$$
et il d\'ecoule de~(\ref{eq:cardinal-cellule-kl})  que 
$$|\G_\kl|=\sum_{\chi \in \FC} \chi(1)^2.$$
Ainsi, 
$$|\G_\calo|=|\G_\kl|.$$
Ce n'est \'evidemment pas suffisant en g\'en\'eral pour montrer que $\G_\calo=\G_\kl$. En revanche, cela d\'emontre 
la conjecture~\BIL~  dans le cas o\`u $G=\SG_W$ (ce qui pourra\^{\i}t \^etre le cas si $W$ est de type $A_n$)~: 
en effet, quitte \`a remplacer $\rGba_c$ par $g(\rGba_c)$ pour un certain $g \in G=\SG_W$, 
on pourrait s'arranger pour que $\G_\calo=\G_\kl$ (et ce pour toute famille $\FC$). 
Cela montre aussi qu'il faudrait savoir pr\'eciser le choix de $\rGba_c$ dans la conjecture~\BIL.\finl
\end{rema}

\bigskip

\begin{rema}\label{rema:bil-w0}
Soit $\G_\calo$ (respectivement $\G_\kl$) une $c$-cellule de Calogero-Moser (respectivement de Kazhdan-Lusztig) 
bilat\`ere. Notons $w_0$ l'\'el\'ement le plus long de $W$. Alors~:
\begin{itemize}
 \item D'apr\`es~(\ref{eq:sim-w0}) et~(\ref{eq:w0gw0}), 
$w_0 \G_\kl=\G_\kl w_0$ est une $c$-cellule de Kazhdan-Lusztig bilat\`ere et 
$\Irr_{w_0\G_\kl}^\kl(W)=\Irr_{\G_\kl}^\kl(W) \e$.

\item Puisque toutes les r\'eflexions de $W$ sont d'ordre $2$, il a \'et\'e 
montr\'e dans le corollaire~\ref{w0 epsilon} que, {\it si $w_0$ est central dans $W$}, alors 
$w_0 \G_\calo=\G_\calo w_0$ est une $c$-cellule de Calogero-Moser bilat\`ere et 
$\Irr_{w_0\G_\calo}^\calo(W)=\Irr_{\G_\calo}^\calo(W) \e$.
\end{itemize}
Ces r\'esultats montrent une certaine analogie {\it lorsque $w_0$ est central} dans $W$. 
Pour le deuxi\`eme \'enonc\'e, il n'est pas raisonnable d'esp\'erer que ce soit vrai 
lorsque $w_0$ n'est pas central (comme le montre l'exemple du type $A_2$) sans 
avoir fait un choix judicieux de l'id\'eal premier $\rGba_c$.\finl
\end{rema}

\bigskip

\subsection{Cellules \`a gauche} 
Commen\c{c}ons en rappelant que l'exp\'erience montre que de 
nombreuses $c$-cellules de Kazhdan-Lusztig donnent lieu au m\^eme $\kl$-caract\`ere $c$-cellulaire. 
Sur le versant Calogero-Moser, le corollaire~\ref{coro:dec-cellulaire} montre que, si $d \in D_c^\gauche$, alors 
les $\calo$-caract\`eres $c$-cellulaires $\isomorphisme{C}_c^\calo$ et 
$\isomorphisme{\lexp{d}{C}}_c^\calo$ sont \'egaux. Ainsi, 
de nombreuses $c$-cellules de Calogero-Moser donnent lieu au m\^eme $\calo$-caract\`ere 
$c$-cellulaire (voir par exemple le th\'eor\`eme~\ref{theo:cellulaire-lisse} 
dans le cas lisse).

\bigskip

\begin{rema}\label{rema:gauche-w0}
Soit $C_\calo$ (resp. $C_\kl$) une $c$-cellule de Calogero-Moser 
(resp. de Kazhdan-Lusztig) 
\`a gauche. Notons $w_0$ l'\'el\'ement le plus long de $W$. Alors~:
\begin{itemize}
 \item Il d\'ecoule de~(\ref{eq:sim-w0}) et~(\ref{eq:caractere-cellulaire-w0})  
que $w_0 C_\kl$ et $C_\kl w_0$ sont des $c$-cellules de Kazhdan-Lusztig \`a gauche et 
que $\isomorphisme{w_0C_\kl}_c^\kl= \isomorphisme{C_\kl w_0}_c^\kl=\isomorphisme{C_\kl}_c^\kl \e$.

\item Puisque toutes les r\'eflexions de $W$ sont d'ordre $2$, il d\'ecoule du corollaire~\ref{coro:cellulaire-ordre-2} 
que, {\it si $w_0$ est central dans $W$}, alors 
$w_0 C_\calo=C_\calo w_0$ est une $c$-cellule de Calogero-Moser \`a gauche et que 
$\isomorphisme{w_0C_\calo}_c^\calo= \isomorphisme{C_\calo w_0}_c^\calo=\isomorphisme{C_\calo}_c^\calo \e$.\finl
\end{itemize}
\end{rema}

\bigskip

\begin{rema}\label{rema: lignes-colonnes}
Notons aussi l'analogie des \'egalit\'es num\'eriques suivantes~: si $C$ est une 
$c$-cellule de Calogero-Moser (respectivement de Kazhdan-Lusztig) \`a gauche et si $\chi \in \Irr(W)$, alors 
$$
\begin{cases}
|C|=\DS{\sum_{\psi \in \Irr(W)} \mult_{C,\psi}^\calo \psi(1),}\\
~\\
\chi(1)=\DS{\sum_{C' \in \cmcellules_L(W)} \mult_{C',\chi}^\calo }\\
\end{cases}
$$
(respectivement
$$
\begin{cases}
|C|=\DS{\sum_{\psi \in \Irr(W)} \mult_{C,\psi}^\kl \psi(1),}\\
~\\
\chi(1)=\DS{\sum_{C' \in \klcellules_L(W)} \mult_{C',\chi}^\kl \quad ).}\\
\end{cases}
$$
Il serait int\'eressant d'\'etudier si d'autres propri\'et\'es num\'eriques des cellules de Kazhdan-Lusztig \`a gauche 
(comme par exemple~\cite[lemme~4.6]{geck plus}) sont aussi v\'erifi\'ees par les cellules de Calogero-Moser \`a gauche.\finl
\end{rema}

\bigskip
%
%
%

\bigskip

Outre les exemples en petit rang, notre argument le plus probant en 
faveur des conjectures~\GAUCHE~et~\BIL~ est le suivant~:

\bigskip

\begin{theo}\label{theo:gauche-presque}
Supposons que l'on soit dans l'un des cas suivants~:
\begin{itemize}
\itemth{1} $(W,S)$ est de type $A$ et $c \neq 0$~;

\itemth{2} $(W,S)$ est de type $B$ avec $a \neq 0$ et $b/a \in \{1/2,3/2\} \cup ]n-1,+\infty)$. 
\end{itemize}
Alors il existe une bijection $\ph : W \to W$ telle que~:
\begin{itemize}
\itemth{a} Si $\G$ est une cellule de Kazhdan-Lusztig bilat\`ere, alors $\ph(\G)$ est une cellule de 
Calogero-Moser bilat\`ere et $\Irr_\G^\kl(W)=\Irr^\calo_{\ph(\G)}(W)$.

\itemth{b} Si $C$ est une cellule de Kazhdan-Lusztig \`a gauche, alors $\ph(C)$ est une cellule de 
Calogero-Moser \`a gauche et $\isomorphisme{C}_c^\kl=\isomorphisme{\ph(C)}_c^\calo$.
\end{itemize}
\end{theo}

\begin{proof}
Sous les hypoth\`eses (1) ou (2), l'espace de Calogero-Moser $\ZCB_c$ est lisse (voir~\cite[th\'eor\`eme~1.24]{EG} 
dans le cas (1) et~\cite[lemme~4.3~et~sa~preuve]{gordon}) dans le cas (2)) et donc le 
th\'eor\`eme~\ref{theo:cellulaire-lisse} s'applique \`a toutes les cellules de Calogero-Moser de $W$. 
Le r\'esultat d\'ecoule alors d'une comparaison de cardinaux.
\end{proof}

\bigskip

\cbend

\part{Exemples}\label{part:exemples}

\chapter{Un exemple assez nul~: le cas ${\boldsymbol{c=0}}$}\label{chapitre nul}

\bigskip

\section{Cellules bilat\`eres, familles}

\medskip

Rappelons que $R_+$ d\'esigne l'unique id\'eal bi-homog\`ene maximal de $R$ et que
$$R/R_+ \simeq \kb$$
(voir le corollaire~\ref{r0}). Rappelons aussi que $D_+$ (respectivement $I_+$) d\'esigne son groupe 
de d\'ecomposition (respectivement d'inertie) et que
$$D_+=I_+=G$$
(voir le corollaire~\ref{r0 DI}). 

\bigskip

\begin{prop}\label{prop:rba0}
$R_+$ est l'unique id\'eal premier de $R$ au-dessus de $\pGba_0$.
\end{prop}

\begin{proof}
En effet, $\pGba_0=P_+$ et donc $R_+$ est un id\'eal premier de $R$ au-dessus de $\pGba_0$~: 
il est stabilis\'e par $G$ ce qui termine la preuve de l'unicit\'e.
\end{proof}

\bigskip

Notons donc $\rGba_0$ l'unique id\'eal premier de $R$ au-dessus de $\pGba_0=P_+$, 
$\Dba_0$ son groupe de d\'ecomposition et $\Iba_0$ son groupe d'inertie. Alors
\equat\label{exemple:dba-0}
\rGba_0=R_+\qquad\text{et}\qquad \Dba_0=\Iba_0=G.
\endequat
Ainsi~:

\bigskip

\begin{coro}\label{coro:familles-nulles}
$W$ ne contient qu'une seule $0$-cellule de Calogero-Moser bilat\`ere, \`a savoir $W$, et
$$\Irr_W^\calo(W)=\Irr(W).$$
\end{coro}

\bigskip

Une des particularit\'es de la sp\'ecialisation en $0$ est que l'alg\`ebre 
$\Hbov_0$ h\'erite de la $(\NM \times \NM)$-graduation, et donc de la $\NM$-graduation. 
Si on note 
$$\Hbov_{0,+} = \mathop{\bigoplus}_{i \ge 1} \Hbov_0^\NM[i],$$
alors $\Hbov_{0,+}$ est un id\'eal bilat\`ere nilpotent de $\Hbov_0$ et, puisque 
$\Hbov_0^\NM[0] = \kb W$, on obtient le r\'esultat suivant~:

\bigskip

\begin{prop}\label{rad h0}
$\Rad(\Hbov_0) = \Hbov_{0,+}$ et $\Hbov_0/\Rad(\Hbov_0) \simeq \kb W$.
\end{prop}

\bigskip

En particulier, 
\equat\label{l 0}
\isomorphisme{\LCov_{\Kbov_0}(\chi)}_{\kb W}^\grad = \chi \in \groth(\kb W)[\tb,\tb^{-1}]
\endequat
et
\equat\label{m 0}
\isomorphisme{\Kbov_0\MCov(\chi)}_{\Hbov_0}
= \chi(1)~ \isomorphisme{\kb W}_{\kb W} \in \ZM\Irr(W)\simeq \groth(\Hbov_0).
\endequat

\bigskip

\section{Cellules \`a gauche, caract\`eres cellulaires}

\medskip

Rappelons que, dans~\S\ref{subsection:specialisation galois 0}, il a \'et\'e fix\'e un 
id\'eal premier $\rG_0$ de $R$ au-dessus de $\qG_0=\CG_0 Q$ ainsi qu'un isomorphisme de corps
$$\iso : \kb(V \times V^*)^{\D\Zrm(W)} \longiso \Mb_0=k_R(\rG_0)$$
dont la restriction \`a $\kb(V \times V^*)^{\D W}$ est l'isomorphisme canonique 
$\kb(V \times V^*)^{\D W} \longiso \Frac(Z_0) \longiso \Lb_0$. Ainsi, 
$R/\rG_0 \subset \iso(\kb[V \times V^*]^{\D\Zrm(W)})$ et ces deux anneaux ont le m\^eme corps des fractions, \`a savoir 
$\Mb_0$. Rappelons aussi que nous ne savons pas si ces deux anneaux sont \'egaux, ce qui 
est \'equivalent \`a savoir si $R/\rG_0$ est int\'egralement clos (question~\ref{question:r0}).

\bigskip

\begin{prop}\label{prop:unicite-r0gauche}
Il existe un unique id\'eal premier de $R$ au-dessus de $\pG_0^\gauche$ et contenant $\rG_0$.
\end{prop}

\begin{proof}
Notons $\pG^*=\iso^{-1}(\pG_0^\gauche/\pG_0)$. Alors $\kb[V \times V^*]^{W \times W}/\pG^* \simeq \kb[V \times 0]^{W \times W}$. 
Donc il n'y a qu'un seul id\'eal premier $\rG^*$ de $\kb[V \times V^*]$ au-dessus de $\pG^*$~: c'est l'id\'eal 
de d\'efinition de la sous-vari\'et\'e ferm\'ee irr\'eductible $V \times 0$ de $V \times V^*$. En d'autres termes, 
$$\kb[V \times V^*]/\rG^*=\kb[V \times 0].$$
Par cons\'equent, le seul id\'eal premier $\rG_0^\gauche$ de $R$ au-dessus de $\pG_0^\gauche$ et contenant $\rG_0$ est 
d\'efini par $\rG_0^\gauche/\rG_0 = \iso(\rG^* \cap \kb[V \times V^*]^{\D\Zrm(W)}) \cap (R/\rG_0)$. 
\end{proof}

\bigskip

Notons $\rG_0^\gauche$ l'unique id\'eal premier de $R$ au-dessus de $\qG_0^\gauche$ et contenant $\rG_0$ 
(voir la proposition~\ref{prop:unicite-r0gauche}) et notons 
$D_0^\gauche$ (respectivement $I_0^\gauche$) son groupe de d\'ecomposition (respectivement d'inertie). 
Alors~:

\bigskip

\begin{prop}\label{prop:d0-left}
\begin{itemize}
\itemth{a} $\iota(W \times W) \subset D_0^\gauche$ et $\iota(W \times 1) \subset I_0^\gauche$.

\itemth{b} L'application canonique $\bar{\iota} : W \times W \to D_0^\gauche/I_0^\gauche$ est 
surjective et son noyau contient $W \times \Zrm(W)$.

\itemth{c} $D_0^\gauche/I_0^\gauche$ est un quotient de $W/\Zrm(W)$.

\itemth{d} Si $R/\rG_0$ est int\'egralement clos (i.e. si $R/\rG_0 \simeq \kb[V \times V^*]^{\D\Zrm(W)}$), alors 
$\Ker(\bar{\iota}) = W \times \Zrm(W)$ et $D_0^\gauche/I_0^\gauche \simeq W/\Zrm(W)$.
\end{itemize}
\end{prop}

\begin{proof}
La premi\`ere assertion de (a) d\'ecoule de l'unicit\'e de $\rG_0^\gauche$ (voir la proposition~\ref{prop:unicite-r0gauche}). 
Pour la deuxi\`eme assertion, reprenons les notations de la preuve de la proposition~\ref{prop:unicite-r0gauche}, 
et remarquons que $W \times 1$ agit trivialement sur $\kb[V \times V^*]/\rG^*$.

\medskip

Notons $B_0$ l'image inverse de $R/\rG_0$ dans $\kb[V \times V^*]$ via $\iso$. Alors 
$\kb[V \times 0]^{W \times W} \subset B_0/\rG^* \subset \kb[V \times 0]^{\D \Zrm(W)}=\kb[V \times 0]^{W \times \Zrm(W)} 
\subset \kb[V \times 0]$. (b), (c) et (d) d\'ecoulent alors de ces observations.
\end{proof}

Apr\`es cette \'etude des groupes de d\'ecomposition et d'inertie, on peut imm\'ediatement en d\'eduire~:

\bigskip

\begin{coro}\label{coro:0-cm}
$W$ ne contient qu'une seule $0$-cellule de Calogero-Moser \`a gauche, \`a savoir $W$, et
$$\isomorphisme{W}_0^\calo = \isomorphisme{\kb W}_{\kb W} = \sum_{\chi \in \Irr(W)} \chi(1).$$
\end{coro}

\begin{proof}
La premi\`ere assertion d\'ecoule de la proposition~\ref{prop:d0-left}(a) tandis que la deuxi\`eme 
d\'ecoule de la proposition~\ref{multiplicite cm}(a).
\end{proof}

\bigskip

Terminons avec une remarque facile, mais qui, combin\'ee avec la 
propositoin~\ref{prop:d0-left}, montre que le couple $(I_0^\gauche,D_0^\gauche)$ est assez surprenant~:

\bigskip

\begin{prop}\label{prop:dc-d0}
Soit $\CG$ un id\'eal premier de $\kb[\CCB]$. Alors il existe $h \in H$ tel que $\lexp{h}{I_\CG^\gauche} \subset I_0^\gauche$.
\end{prop}

\begin{proof}
Notons $\CGt$ l'id\'eal homog\`ene maximal de $\kb[\CCB]$ contenu dans 
$\CG$. D'apr\`es la proposition~\ref{lem:cellules-gauches-homogeneise}, on a 
$I_\CG^\gauche=I_\CGt^\gauche$. Cela signifie que l'on peut supposer $\CG$ homog\`ene. 
En particulier, $\CG \subset \CG_0$. Donc $\qG_\CG^\gauche \subset \qG_0^\gauche$ et 
il existe $h \in H$ tel que $h(\rG_\CG^\gauche) \subset \rG_0^\gauche$. 
Par suite, $\lexp{h}{I_\CG^\gauche} \subset I_0^\gauche$. 
\end{proof}

\bigskip

Il serait tentant de penser, apr\`es la proposition~\ref{prop:d0-left}, que $D_0^\gauche=\iota(W \times W)$ et 
$I_0^\gauche=\iota(W \times \Zrm(W))$. Cependant, ceci entre en contradiction avec la proposition 
pr\'ec\'edente~\ref{prop:dc-d0}, 
surtout si l'on esp\`ere que les conjectures~\BIL~et~\GAUCHE~soient valides~: en effet, $I_0^\gauche$ doit 
donc contenir des conjugu\'es de sous-groupes admettant pour orbites les cellules de Calogero-Moser 
\`a gauche. 
Nous verrons dans le chapitre~\ref{chapitre:rang 1} 
que si $\dim_\kb(V)=1$, alors $D_0^\gauche=G$. 

%
%
%
%
%
%
%
%
%
%
%
%

\chapter{Groupes de rang 1}\label{chapitre:rang 1}

\bigskip
 
\boitegrise{{\bf Hypoth\`eses et notation.} {\it Dans ce chapitre, 
et seulement dans ce chapitre, nous supposons que $\dim_\kb V=1$, nous fixons un
\'el\'ement non nul $y$ de $V$ et nous noterons $x$ l'unique 
\'el\'ement de $V^*$ tel que $\langle y,x\rangle = 1$. 
Fixons un entier $d \ge 2$ et supposons que $\kb$ contienne une racine primitive 
$d$-i\`eme de l'unit\'e $\z$. 
Notons $s$ l'automorphisme de $V$ d\'efini par $s(y)=\z y$ (on a alors 
$s(x)=\z^{-1} x$) et 
supposons de plus que $W=\langle s \rangle$~: $s$ est une r\'eflexion (!) et $W$ 
est cyclique d'ordre $d$.}}{0.75\textwidth}

\bigskip

\section{L'alg\`ebre $\Hb$}\label{section:H rang 1}

\medskip

\subsection{D\'efinition} 
Bien s\^ur, $\Ref(W)=\{s^i~|~1 \le i \le d-1\}$. 
Pour $1 \le i \le d-1$, nous noterons $C_i$ l'ind\'etermin\'ee $C_{s^i}$, 
de sorte que $\kb[\CCB]=\kb[C_1,C_2,\dots,C_{d-1}]$.
La relation suivante est v\'erifi\'ee dans $\Hb$~:
\equat\label{relation d}
[y,x] = \sum_{1 \le i \le d-1} (\z^i-1) C_i~s^i.
\endequat
Nous poserons $C_0=C_{s^0}=0$. 
Puisque l'arrangement d'hyperplans $\AC$ est r\'eduit \`a un \'el\'ement, 
et donc $\AC/W$ aussi (\'ecrivons $\AC/W=\{\O\}$), 
nous poserons pour simplifier $K_j=K_{\O,j}$ (pour $0 \le j \le d-1$). 
Rappelons que la famille $(K_j)_{0 \le j \le d-1}$ 
est d\'etermin\'ee par les relations
\equat\label{relations K}
\forall~0 \le i \le d-1,~C_i=\sum_{j=0}^{d-1} \z^{i(j-1)} K_j.
\endequat
Pour simplifier certains \'enonc\'es de cette partie, nous posons 
$$K_{di+j}=K_j$$
pour tout $i \in \ZM$ et $j \in \{0,1,\dots,d-1\}$. Rappelons que
$$K_0+K_1+\cdots + K_{d-1}=0\qquad\text{(c'est-\`a-dire}\quad K_1+K_2+\cdots+K_d=0).$$

\bigskip

\subsection{Calcul de ${\boldsymbol{(V \times V^*)/W}}$}\label{subsection:quotient rang 1}
Posons $X=x^d$, $Y=y^d$ et rappelons que $\euler_0 = xy$. Ainsi
$$\kb[V \times V^*]^W = \kb[X,Y,\euler_0]$$
et la relation suivante 
\equat\label{relation centre d}
\euler_0^d = XY
\endequat
est satisfaite. Il est facile de v\'erifier que cette relation engendre l'id\'eal 
des relations.

\bigskip

\section{L'alg\`ebre ${\boldsymbol{Z}}$}\label{section:Q rang 1}

\medskip

Rappelons que $\euler=yx + \sum_{i=1}^{d-1} C_i~s^i$ (de sorte que son image dans 
$\Hb_0$ est $\euler_0$) et que $\e : W \to \kb^\times$ est le d\'eterminant~: 
il est caract\'eris\'e par $\e(s)=\z$. On a $\e^d=1$ et 
$$\Irr W=\{1,\e,\e^2,\dots,\e^{d-1}\}.$$
L'image de l'\'el\'ement d'Euler par $\O_\chi$ se calcule 
gr\^ace \`a la proposition~\ref{action euler verma}~:
\equat\label{euler cyclique}
\O_{\e^i}(\euler) = d K_{-i}
\endequat
pour tout $i \in \ZM$. 
%
%
Le r\'esultat suivant est certainement bien connu~:

\bigskip

\begin{theo}\label{centre rang 1}
On a $Z=P[\euler]=\kb[C_1,\dots,C_{d-1},X,Y,\euler]=\kb[K_1,\dots,K_{d-1},X,Y,\euler]$ 
et l'id\'eal des relations est engendr\'e par
$$\prod_{i=1}^d (\euler - d K_i) = XY.$$
\end{theo}



\begin{proof}
Soit $z=\prod_{i=1}^d (\euler-dK_i)-XY$. On a $z\in k[\CCB]_+Z$ et $z$ est bihomog\`ene
de degr\'e $(d,d)$. On en d\'eduit que $z=\sum_{j=0}^{d-1}z_j d^{-j}\euler^j$, o\`u
$z_j\in k[\CCB]$ est bihomog\`ene de bidegr\'e $(d-j,d-j)$. On a 
$\Omega_{\e^i}(z)=0=\sum_{j=0}^{d-1}z_jK_{-i}^j$. Les $z_j$ sont solutions
d'un syst\`eme de Vandermonde de d\'eterminant $\prod_{1\le j<j'\le d}
(K_j-K_{j'})\not=0$, donc $z_0=\cdots=z_{d-1}=0$.

\medskip

%

Puisque le polyn\^ome minimal de $\euler$ sur $P$ est de degr\'e $|W|=d$
(voir le corollaire~\ref{eq:minimal-euler}), on en d\'eduit que 
$$\prod_{i=1}^d (\tb - d K_i) - XY$$
est le polyn\^ome minimal de $\euler$ sur $P$, ce qui termine 
la preuve du th\'eor\`eme.
\end{proof}

\bigskip

\begin{coro}\label{inter 1}
La $\kb$-alg\`ebre $Z$ est d'intersection compl\`ete.
\end{coro}

\bigskip

Nous noterons $F_\euler(\tb) \in P[\tb]$ le polyn\^ome minimal de $\euler$ sur $P$. 
D'apr\`es le th\'eor\`eme~\ref{centre rang 1}, on a 
\equat\label{polynome minimal euler rang 1}
F_\euler(\tb)=\prod_{i=1}^d (\tb - d K_i) - XY.
\endequat

\bigskip

\section{L'anneau ${\boldsymbol{R}}$, le groupe ${\boldsymbol{G}}$}\label{section:G rang 1} 

\medskip

\subsection{Polyn\^omes sym\'etriques} 
Pour tirer parti du fait que le polyn\^ome minimal de l'\'el\'ement d'Euler est 
sym\'etrique en les variables $K_i$, nous rappelons quelques r\'esultats classiques 
sur les polyn\^omes sym\'etriques. 
Si $T_1$, $T_2$,\dots, $T_d$ sont des ind\'etermin\'ees et si $1 \le i \le d$, notons 
$\s_i(\Tb)$ la $i$-i\`eme fonction sym\'etrique \'el\'ementaire 
$$\s_i(\Tb)=\s_i(T_1,\dots,T_d)=\sum_{1 \le j_1 < \cdots < j_i \le d} T_{j_1}\cdots T_{j_i}.$$
Rappelons la formule bien connue
\equat\label{eq:jacobien}
\det\Bigl(\frac{\partial \s_i(\Tb)}{\partial T_j}\Bigr)_{1 \le i,j \le d} = 
\prod_{1 \le i < j \le d} (T_j-T_i).
\endequat
Le groupe $\SG_d$ agit sur $\kb[T_1,\dots,T_d]$ par permutation des ind\'etermin\'ees.
Rappelons le r\'esultat classique suivant (cas 
particulier du th\'eor\`eme~\ref{chevalley})~:

\bigskip

\begin{prop}\label{prop:polynomes-symetriques}
Les polyn\^omes $\s_1(\Tb)$,\dots, $\s_d(\Tb)$ sont alg\'ebriquement ind\'ependants et 
$\kb[T_1,\dots,T_d]^{\SG_d}=\kb[\s_1(\Tb),\dots,\s_d(\Tb)]$. De plus, la $\kb$-alg\`ebre 
$\kb[T_1,\dots,T_d]$ est un $\kb[\s_1(\Tb),\dots,\s_d(\Tb)]$-module libre de rang $d!$
\end{prop}

Rappelons aussi que $\s_1(\Tb)=T_1+\cdots + T_d$~:

\begin{coro}\label{coro:polynomes-symetriques}
On a $\bigl(\kb[T_1,\dots,T_d]/\langle \s_1(\Tb) \rangle \bigr)^{\SG_d} \simeq \kb[\s_2(\Tb),\dots, \s_d(\Tb)]$ 
et la $\kb$-alg\`ebre 
$\kb[T_1,\dots,T_d]/\langle\s_1(\Tb)\rangle$ est un $\kb[\s_2(\Tb),\dots, \s_d(\Tb)]$-module libre 
de rang $d!$
\end{coro}

\bigskip

Comme cons\'equence de la proposition~\ref{prop:polynomes-symetriques}, 
il existe un unique polyn\^ome $\D_d$ en $d$ variables tel que 
\equat\label{eq:def-discriminant}
\prod_{1 \le i < j \le d} (T_j-T_i)^2=\D_d(\s_1(\Tb),\s_2(\Tb),\dots,\s_d(\Tb)).
\endequat

\bigskip

\subsection{Pr\'esentation de ${\boldsymbol{R}}$} 
Notons $\s_i(\Kb)=\s_i(K_1,\dots,K_d)$ (en particulier, $\s_1(\Kb)=0$). 
Alors, d'apr\`es le corollaire~\ref{coro:polynomes-symetriques}, 
$P_\sym=\kb[\s_2(\Kb),\dots,\s_d(\Kb),X,Y]$ est l'anneau des invariants, 
dans $P$, du groupe $\SG_d$ agissant par permutations des $K_i$. 
De plus, 
\equat\label{eq:p-psym}
\text{\it $P$ est un $P_\sym$-module libre de rang $d!$}
\endequat
Introduisons une nouvelle famille d'ind\'etermin\'ees $E_1$,\dots, $E_{d-1}$, et d\'efinissons 
$E_d=-(E_1+\cdots+E_{d-1})$ et $\s_i(\Eb)=\s_i(E_1,\dots,E_d)$ (en particulier 
$\s_1(\Eb)=0$). Notons $R_\sym=\kb[E_1,\dots,E_{d-1},X,Y]=\kb[E_1,\dots,E_d,X,Y]/\langle \s_1(\Eb)\rangle$, 
sur lequel le groupe sym\'etrique $\SG_d$ agit par permutations des $E_i$. L'anneau $R_\sym^{\SG_d}$ 
est donc encore une alg\`ebre de polyn\^omes \'egale \`a $\kb[\s_2(\Eb),\dots,\s_d(\Eb),X,Y]$ 
(toujours gr\^ace au corollaire~\ref{coro:polynomes-symetriques}).

\medskip

\boitegrise{{\bf Identification.} {\it Nous identifions les $\kb$-alg\`ebres $P_\sym$ et 
$R_\sym^{\SG_d}$ \`a travers les \'egalit\'es 
$$\begin{cases}
\s_1(d\Kb)=\s_1(\Eb)=0 \\
\forall~2 \le i \le d-1,~\s_i(d\Kb)=\s_i(\Eb)\\
\s_d(d\Kb)=\s_d(\Eb)+(-1)^d XY
\end{cases}$$
Rappelons que $\s_i(d\Kb)=d^i\s_i(\Kb)$.}}{0.75\textwidth}

\medskip

Ainsi,
\equat\label{eq:rsym-psym}
\text{\it $R_\sym$ est un $P_\sym$-module libre de rang $d!$}
\endequat

\bigskip

\begin{lem}\label{lem:integre-normal}
L'anneau $P \otimes_{P_\sym} R_\sym$ est int\`egre et int\'egralement clos.
\end{lem}

\begin{proof}
Notons tout d'abord que nous pouvons, et nous le ferons, supposer dans cette preuve 
que $\kb$ est alg\'ebriquement clos. Posons $\Rti=P \otimes_{P_\sym} R_\sym$. 
Alors $\Rti$ admet la pr\'esentation suivante~:
$$\begin{cases}
\text{G\'en\'erateurs~:} & K_1, K_2,\dots, K_d, E_1,E_2,\dots,E_d,X,Y\\
\text{Relations~:} & 
\begin{cases}
\s_1(d\Kb)=\s_1(\Eb)=0 \\
\forall~2 \le i \le d-1,~\s_i(d\Kb)=\s_i(\Eb)\\
\s_d(d\Kb)=\s_d(\Eb)+(-1)^d XY
\end{cases}
\end{cases}\leqno{(\PC)}$$

La pr\'esentation $(\PC)$ de $\Rti$ montre que l'on peut graduer $\Rti$ de sorte 
que $\deg(K_i)=\deg(E_i)=2$ et $\deg(X)=\deg(Y)=d$. Ainsi, la composante 
de degr\'e $0$ de $\Rti$ est isomorphe \`a $\kb$, ce qui montre que 
$$\text{\it $\Rti$ est connexe.}\leqno{(\clubsuit)}$$

D'autre part, il d\'ecoule de~(\ref{eq:p-psym}) et~(\ref{eq:rsym-psym}) que 
$\Rti$ est un $P$-module libre de rang $d!$ (et un $P_\sym$-module libre de 
rang $(d!)^2$), et donc que
$$\text{\it $\Rti$ est un anneau de Cohen-Macaulay purement de dimension $d+1$.}\leqno{(\diamondsuit)}$$
Mieux, la pr\'esentation $(\PC)$ montre que 
$$\text{\it $\Rti$ est un anneau d'intersection compl\`ete.}\leqno{(\heartsuit)}$$

\def\Jac{{\mathrm{Jac}}}

Montrons maintenant que
$$\text{\it $\Rti$ est r\'egulier en codimension $1$.}\leqno{(\spadesuit)}$$
Pour cela, notons $\RCt$ la sous-vari\'et\'e ferm\'ee de $\AM^{2d+2}(\kb)$ 
form\'ee des \'el\'ements $r=(k_1,\dots,k_d,e_1,\dots,e_d,x,y)$ satisfaisant 
aux \'equations $(\PC)$. Le Jacobien $\Jac(r)$ du syst\`eme d'\'equations $(\PC)$ en $r \in \RCt$ 
est donn\'e par 
{\scriptsize$${\Jac(r)=
\begin{pmatrix}
d &  \cdots & d & 0 &  \cdots & 0 & 0 & 0 \\
&&&&&&&\\
0 &  \cdots & 0 & -1 &  \cdots & -1 & 0 & 0 \\
&&&&&&&\\
\DS{\frac{\partial \s_2(d\Kb)}{\partial K_1}}(r) & 
\cdots & \DS{\frac{\partial \s_2(d\Kb)}{\partial K_d}}(r) & 
-\DS{\frac{\partial \s_2(\Eb)}{\partial E_1}}(r) & 
\cdots & -\DS{\frac{\partial \s_2(\Eb)}{\partial E_d}}(r) & 0 & 0 \\
&&&&&&&\\
\vdots & & \vdots & \vdots & & \vdots & \vdots & \vdots \\
&&&&&&&\\
\DS{\frac{\partial \s_{d-1}(d\Kb)}{\partial K_1}}(r) & 
\cdots & \DS{\frac{\partial \s_{d-1}(d\Kb)}{\partial K_d}}(r) & 
-\DS{\frac{\partial \s_{d-1}(\Eb)}{\partial E_1}}(r) & 
\cdots & -\DS{\frac{\partial \s_{d-1}(\Eb)}{\partial E_d}}(r) & 0 & 0 \\
&&&&&&&\\
\DS{\frac{\partial \s_{d}(d\Kb)}{\partial K_1}}(r) & 
\cdots & \DS{\frac{\partial \s_{d-1}(d\Kb)}{\partial K_d}}(r) & 
-\DS{\frac{\partial \s_{d}(\Eb)}{\partial E_1}}(r) & 
\cdots & -\DS{\frac{\partial \s_{d}(\Eb)}{\partial E_d}}(r) & (-1)^{d+1} y & (-1)^{d+1} x \\
\end{pmatrix}}$$}
Puisque $\RCt$ est purement de dimension $d+1$ et d'intersection compl\`ete, 
$r$ est un point lisse si et seulement si le rang de $\Jac(r)$ est \'egal \`a $(2d+2)-(d+1)=d+1$. 
Or, si le rang de $\Jac(r)$ est inf\'erieur ou \'egal \`a $d$, cela implique que 
$$\det\Bigl(\frac{\partial \s_i(d\Kb)}{\partial K_j}(k_1,\dots,k_d)\Bigr)_{1 \le i , j \le d} = 
\det\Bigl(\frac{\partial \s_i(\Eb)}{\partial E_j}(e_1,\dots,e_d)\Bigr)_{1 \le i , j \le d} = 0.$$
D'apr\`es~(\ref{eq:jacobien}), cela signifie que 
$$\prod_{1 \le i < j \le d}(k_j-k_i)=\prod_{1 \le i < j \le d}(e_j-e_i)=0$$
En particulier, 
$$\D_d(\s_1(k_1,\dots,k_d),\dots,\s_d(k_1,\dots,k_d))=\D_d(\s_1(e_1,\dots,e_d),\dots,\s_d(e_1,\dots,e_d))=0.$$
Compte tenu du fait que $r \in \RCt$ v\'erifie les \'equations $(\PC)$, 
cela montre que la projection du lieu singulier de $\RCt$ sur la vari\'et\'e 
$\PC_\sym \simeq \AM^{d+1}(\kb)$ est contenu dans l'ensemble des $(d+2)$-uplets 
$(a_2,\dots,a_d,x,y) \in \AM^{d+1}(\kb)$ tels que 
$$\D_d(0,a_2,\dots,a_{d-1},a_d)=\D_d(0,a_2,\dots,a_{d-1},a_d+(-1)^d xy)=0.\leqno{(*)}$$
Il est bien connu que $\D_d(0,U_2,\dots,U_d)$ est un polyn\^ome irr\'eductible en 
les ind\'etermin\'ees $U_1$,\dots, $U_d$. Par cons\'equent, pour montrer que 
la sous-vari\'et\'e de $\AM^{d+1}(\kb)$ d\'efinie par les \'equations $(*)$ est de codimension $\ge 2$, 
il suffit de montrer qu'il existe $(a_2,\dots,a_d,x,y) \in \AM^{d+1}(\kb)$ tel que 
$\D_d(0,a_2,\dots,a_{d-1},a_d)=0$ et $\D_d(0,a_2,\dots,a_{d-1},a_d+(-1)^d xy)\neq 0$. 
Il suffit pour cela de prendre $a_2=\cdots = a_d=0$ et $x=y=1$. Cela termine 
la d\'emonstration de $(\spadesuit)$.

\medskip

En conclusion, d'apr\`es $(\diamondsuit)$ et $(\spadesuit)$, $\Rti$ est normal 
(voir~\cite[\S{\MakeUppercase{\romannumeral 4}.D},~th\'eor\`eme~11]{serre}). 
C'est donc un produit direct d'anneaux int\`egres int\'egralement clos mais, 
\'etant connexe d'apr\`es~$(\clubsuit)$, cela signifie que 
$\Rti$ est int\`egre et int\'egralement clos.
\end{proof}


\bigskip

Comme cons\'equence du lemme pr\'ec\'edent, nous obtenons~:

\bigskip

\begin{theo}\label{theo:r-cyclique}
L'anneau $R$ v\'erifie les propri\'et\'es suivantes~:
\begin{itemize}
\itemth{a} $R$ est isomorphe \`a $P \otimes_{P_\sym} R_\sym$. Il admet la pr\'esentation suivante~:
$$\begin{cases}
\text{G\'en\'erateurs~:} & K_1, K_2,\dots, K_d, E_1,E_2,\dots,E_d,X,Y\\
\text{Relations~:} & 
\begin{cases}
\s_1(d\Kb)=\s_1(\Eb)=0 \\
\forall~2 \le i \le d-1,~\s_i(d\Kb)=\s_i(\Eb)\\
\s_d(d\Kb)=\s_d(\Eb)+(-1)^d XY
\end{cases}
\end{cases}\leqno{(\PC)}$$

\itemth{b} $R$ est un anneau d'intersection compl\`ete et $R$ est un $P$-module libre de rang $d!$ 
(en particulier, $R$ est un anneau de Cohen-Macaulay).

\itemth{c} Il existe un unique morphisme de $P$-alg\`ebres $\copie : Z \to R$ tel que $\copie(\euler)=E_d$. 
Ce morphisme est injectif (on note $Q$ son image).

\itemth{d} Pour l'action du groupe $\SG_d$ par permutation des $E_i$, 
on a $R^{\SG_d}=P$ et $R^{\SG_{d-1}}=Q$.

\itemth{e} $G=\SG_W \simeq \SG_d$~; si $\s \in \SG_d$ et $1 \le i \le d$, alors $\s(E_i)=E_{\s(i)}$. 
L'\'el\'ement $\eulerq$ de $Q$ s'identifie \`a $E_d$. 

\itemth{f} $G$ est un groupe de r\'eflexions pour son action sur $R_+/(R_+)^2$.
\end{itemize}
\end{theo}

\begin{proof}
Reprenons la notation $\Rti=P \otimes_{P_\sym} R_\sym$ utilis\'ee dans 
la preuve du lemme~\ref{lem:integre-normal}. 
Les relations $(\PC)$ montrent que, dans l'anneau de polyn\^omes $\Rti[\tb]$, on a 
l'\'egalit\'e
$$\prod_{i=1}^d(\tb-dK_i)-XY = \prod_{i=1}^d (\tb - E_i).$$
Ainsi, $E_d$ annule le polyn\^ome minimal de $\euler$ sur $P$ ce qui, compte tenu 
du th\'eor\`eme~\ref{centre rang 1}, montre qu'il existe un unique morphisme 
de $P$-alg\`ebres $\copie : Z \to \Rti$ tel que $\copie(\euler)=E_d$. 
Si on note $\zG=\Ker(\copie)$, alors $\zG \cap P=0$ car $P \subset \Rti$ et, 
puisque $Z$ est int\`egre et est un $P$-module de type fini, cela force $\zG=0$. 
Donc $\copie : Z \to \Rti$ est injectif (on note $Q$ son image).

\medskip

Notons $\Mbt$ le corps des fractions de $\Rti$ (rappelons que $\Rti$ est int\`egre 
en vertu du lemme~\ref{lem:integre-normal}). 
Par construction, $\Rti$ est $P$-libre de rang $d!$ et, 
d'apr\`es le corollaire~\ref{coro:polynomes-symetriques}, $\Rti^{\SG_d}=P$. 
Donc l'extension $\Mbt/\Kb$ est galoisienne, contient $\Lb$ (le corps des fractions 
de $Q$) et v\'erifie $\Gal(\Mbt/\Kb)=\SG_d$. De plus, toujours par construction, 
$\Gal(\Mbt/\Lb)=\SG_{d-1}$ (car $\SG_{d-1}$ est le stabilisateur de $E_d$ dans $\SG_d$). 
Puisque le seul sous-groupe distingu\'e de $\SG_d$ contenu dans $\SG_{d-1}$ est 
le groupe trivial, cela montre donc que $\Mbt/\Kb$ est une cl\^oture galoisienne de $\Lb/\Kb$. 
Donc $\Mbt \simeq \Mb$.

\medskip

Puisque $\Rti$ est int\'egralement clos (voir le lemme~\ref{lem:integre-normal}) et entier 
sur $P$, cela implique que $\Rti \simeq R$. Tous les \'enonc\'es du th\'eor\`eme~\ref{theo:r-cyclique} 
se d\'eduisent de ces observations (pour l'\'enonc\'e (f), on peut cependant utiliser (b),  et 
la proposition~\ref{intersection complete R} car $\SG_d$ agit trivialement sur les relations, 
ou alors v\'erifier directement en remarquant que $R_+/(R_+)^2$ est le $\kb$-espace vectoriel 
de dimension $2d$ engendr\'e par 
$K_1$,\dots, $K_d$, $E_1$,\dots, $E_d$, $X$, $Y$, avec les relations 
$K_1+\cdots+K_d=0$ et $E_1+ \cdots + E_d=0$~: cela montre que, comme repr\'esentation de 
$\SG_d$, $R_+/(R_+)^2$ est la somme directe de la repr\'esentation irr\'eductible 
de r\'eflexion et de $d+1$ copies de la repr\'esentation triviale).
\end{proof}

\bigskip

\subsection{Choix de l'id\'eal ${\boldsymbol{\rG_0}}$} 
Notons $\rG'$ l'id\'eal de $R$ engendr\'e par les \'el\'ements $E_i-\z^i E_d$. 
On a alors
$$\s_1(\Eb) \equiv \s_2(\Eb) \equiv \cdots \equiv \s_{d-1}(\Eb) \equiv 0 \mod \rG'.$$
Nous choisissons pour $\rG_0$ l'id\'eal de $R$ \'egal \`a 
$\rG_0=\rG' + \langle K_1,\dots, K_d\rangle_R$. Alors $R/\rG_0$ admet la pr\'esentation suivante~:
$$\begin{cases}
\text{G\'en\'erateurs~:} & E_d,X,Y\\
\text{Relation~:} & 
E_d^d = XY
\end{cases}\leqno{(\PC_0)}$$
Rappelons que $\Zrm(W)=W$. 
Ainsi, comme pr\'evu, $R/\rG_0 \simeq Q/\qG_0 \simeq \kb[V \times V^*]^{\D W}$, 
\`a travers l'isomorphisme qui envoie $\eulerq=E_d$ sur $\euler_0=yx \in \kb[V \times V^*]^{\D W}$. 
Rappelons qu'un \'el\'ement $w \in W$, vu comme un \'el\'ement du groupe de Galois $G=\SG_W\simeq \SG_d$, 
est caract\'eris\'e par l'\'egalit\'e 
$$(w(\eulerq) \mod \rG_0)\equiv w(y)x \in \kb[V \times V^*]^{\D W}.$$
Or, $s^i(y)=\z^i y$, donc
\equat\label{eq:action-w-cyclique}
s^i(\eulerq) = E_i.
\endequat
Pour l'action de $G =\SG_W \simeq \SG_d$, cela revient \`a identifier les ensembles 
$\{1,2,\dots,d\}$ et $W$ \`a travers la bijection $i \mapsto s^i$, ce qui est naturel. 

\bigskip

\bigskip

\subsection{Choix des id\'eaux ${\boldsymbol{\rG^\gauche}}$, ${\boldsymbol{\rG^\droite}}$ et 
${\boldsymbol{\rGba}}$}\label{section:choix rang 1}
Notons $\rG''$ l'id\'eal de $R$ engendr\'e par les $E_i-dK_i$. Ainsi
$$\forall~1 \le i \le d,~\s_i(d\Kb) \equiv \s_i(\Eb) \mod \rG''.$$
En particulier, $XY \in \rG''$. 
Nous choisissons $\rG^\gauche=\rG'' + \langle Y \rangle_R$, $\rG^\droite=\rG''+\langle X \rangle_R$ 
et $\rGba=\rG'' + \langle X,Y \rangle_R$. Alors
\equat\label{eq:iso-r}
\begin{cases}
 R/\rG^\gauche \simeq \kb[K_1,\dots,K_{d-1},X] =P/\pG^\gauche,\\
 R/\rG^\droite \simeq \kb[K_1,\dots,K_{d-1},Y] =P/\pG^\droite,\\
 R/\rGba \simeq \kb[K_1,\dots,K_{d-1}]=\kb[\CCB] =P/\pGba.\\
\end{cases}
\endequat
Il en r\'esulte imm\'ediatement la proposition suivante~:

\bigskip

\begin{prop}\label{prop:dec-cyclique}
$D^\gauche=I^\gauche=D^\droite=I^\droite=\Dba=\Iba=1$.
\end{prop}

\bigskip

\section{Cellules, familles, caract\`eres cellulaires}

\medskip

\boitegrise{\noindent{\bf Notation.} {\it Nous fixons dans cette section un 
id\'eal premier $\CG$ de $\kb[\CCB]$ et nous notons $k_i$ l'image de 
$K_i$ dans $\kb[\CCB]/\CG$.}}{0.75\textwidth}

\medskip

Compte tenu de~(\ref{eq:iso-r}), on a 
\equat\label{eq:rc-cyclique}
\rG_\CG^\gauche=\rG^\gauche + \CG R,\quad \rG_\CG^\droite=\rG^\droite + \CG R
\quad\text{et}\quad \rGba_\CG=\rGba + \CG R
\endequat
et
\equat\label{eq:r-c-cyclique}
\begin{cases}
R/\rG_\CG^\gauche=\kb[\CCB]/\CG \otimes \kb[X]=P/\pG_\CG^\gauche,\\
R/\rG_\CG^\droite=\kb[\CCB]/\CG \otimes \kb[Y]=P/\pG_\CG^\droite,\\
R/\rGba_\CG=\kb[\CCB]/\CG=P/\pGba_\CG.
\end{cases}
\endequat
On notera $\SG[\CG]$ le sous-groupe de $\SG_d$ form\'e des permutations 
stabilisant les fibres de l'application $\{1,2,\dots,d\} \to \kb[\CCB]/\CG$, $i \mapsto k_i$. 
En d'autres termes,
$$\SG[\CG]=\{\s \in \SG_d~|~\forall~1 \le i \le d,~k_{\s(i)}=k_i\}.$$
Alors~:

\bigskip

\begin{prop}\label{prop:dec-c-cyclique}
$D^\gauche_\CG=I^\gauche_\CG=D^\droite_\CG=I^\droite_\CG=\Dba_\CG=\Iba_\CG=\SG[\CG]$.
\end{prop}

%

\medskip

\begin{coro}\label{coro:cellules}
Soient $i$, $j \in \ZM$. Alors $s^i$ et $s^j$ sont dans la m\^eme $\CG$-cellule de Calogero-Moser 
bilat\`ere (resp. \`a gauche, resp. \`a droite) si et seulement si $k_i=k_j$.
\end{coro}

\bigskip

Terminons avec la description des familles et des caract\`eres cellulaires.

\bigskip

\begin{coro}\label{coro:familles-cellulaires-cyclique}
Soient $i$, $j \in \ZM$. Alors $\e^{-i}$ et $\e^{-j}$ sont dans la m\^eme $\CG$-cellule de Calogero-Moser 
bilat\`ere (resp. \`a gauche, resp. \`a droite) si et seulement si $k_i=k_j$.

L'application $\o \mapsto \sum_{i \in \o} \e^{-i}$ induit une bijection entre l'ensemble 
des $\SG[\CG]$-orbites dans $\{1,2,\dots,d\}$ (c'est-\`a-dire l'ensemble des fibres 
de l'application $i \mapsto k_i$) et l'ensemble des $\calo$-caract\`eres $\CG$-cellulaires.
\end{coro}

\begin{proof}
Puisque $Z=P[\euler]$, on d\'eduit que $\e^{-i}$ et $\e^{-j}$ sont dans la m\^eme $\CG$-famille de Calogero-Moser 
si et seulement si $\O_{\e^{-i}}^{\Kbov_\CG}(\euler)=\O_{\e^{-j}}^{\Kbov_\CG}(\euler)$. 
Donc la premi\`ere assertion d\'ecoule de~\ref{euler cyclique}.

\medskip

Pour la deuxi\`eme, on remarque que $\euler$ agit sur $\LC_{s^i}$ par multiplication par $s^i(\eulerq)=E_i$. 
Donc, modulo $\rG_\CG^\gauche$ (ou $\rGba_\CG$), l'\'el\'ement $s^i(\eulerq)$ est congru \`a 
$dk_{i}=\O_{\e^{-i}}^{\Kbov_\CG}(\euler)$. Ainsi, si $\o$ est une $\SG[\CG]$-orbite dans 
$\{1,2,\dots,d\}$, alors $C=\{s^i~|~i \in \o\}$ est une $\CG$-cellule de Calogero-Moser 
\`a gauche, \`a droite ou bilat\`ere (voir le corollaire~\ref{coro:cellules}) et, 
en tant que cellule bilat\`ere, elle recouvre la $\CG$-famille de Calogero-Moser 
$\{\e^{-i}~|~i \in \o\}$. Or, $\Mb_\CG^\gauche \MC^\gauche(\e^{-i})$ est un 
$\Mb_\CG^\gauche\Hb^\gauche$-module (absolument) simple (car il est de dimension 
$|W|$)~: c'est forc\'ement $\LC_\CG^\gauche(C)$. Ceci montre que 
$\isomorphisme{C}_\CG^\calo=\sum_{i \in \o} \e^{-i}$.
\end{proof}

\bigskip

\section{Compl\'ements}

On va s'int\'eresser ici aux propri\'e\'t\'es g\'eom\'etriques de $\ZCB$ 
(lissit\'e, ramification) et aux propri\'et\'es du groupe $D_c$. Pour simplifier les \'enonc\'es, nous ferons 
l'hypoth\`ese suivante~:

\bigskip
\def\ram{{\mathrm{ram}}}

\boitegrise{{\bf Hypoth\`ese et notations.} {\it Dans cette section, et seulement dans cette section, 
nous supposerons que $\kb$ est {\bfit alg\'ebriquement clos}. Nous identifierons la vari\'et\'e $\ZCB$ avec 
$$\ZCB=\{(k_1,\dots,k_d,x,y,e) \in \AM^{d+3}(\kb)~|~k_1+\cdots+k_d=0\text{ et } \prod_{i=1}^d(e-dk_i)=xy\}.$$
De m\^eme, $\PCB$ (respectivement $\CCB$) sera identifi\'e avec l'hyperplan 
$$\PCB=\{(k_1,\dots,k_d,x,y) \in \AM^{d+2}(\kb)~|~
k_1+\cdots +k_d = 0\}$$
(respectivement
$$\CCB=\{(k_1,\dots,k_d) \in \AM^{d}(\kb)~|~
k_1+\cdots +k_d = 0\}\quad),$$
ce qui permet de red\'efinir 
$$\fonction{\Upsilon}{\ZCB}{\PCB}{(k_1,\dots,k_d,x,y,e)}{(k_1,\dots,k_d,x,y).}$$
\!\!\! Pour finir, nous noterons $\ZCB_\singulier$ le lieu singulier de $\ZCB$ 
et $\ZCB_\ram$ le lieu de ramification du morphisme $\Upsilon$.
}}{0.8\textwidth}

\bigskip

\subsection{Lissit\'e} 
Commen\c{c}ons par la description des singularit\'es de la vari\'et\'e $\ZCB$~:

\bigskip

\begin{prop}\label{prop:zsing-cyclique}
Si $1 \le i < j \le d$, notons $\ZCB_{i,j} = \{(k_1,\dots,k_d,x,y,e) \in \ZCB~|~
e=k_i=k_j$ et $x=y=0\}$. Alors 
$$\ZCB_\singulier=\bigcup_{1 \le i < j \le d} \ZCB_{i,j}.$$
De plus, $\ZCB_{i,j} \simeq \AM^{d-2}(\kb)$ est une composante irr\'eductible de $\ZCB_\singulier$
et $\ZCB_\singulier$ est purement de codimension $3$.
\end{prop}

\begin{proof}
La vari\'et\'e $\ZCB$ \'etant d\'ecrite comme une hypersurface dans l'espace affine 
$\{(k_1,\dots,k_d,x,y,e) \in \AM^{d+3}(\kb)~|~k_1+\cdots+k_d=0\} \simeq \AM^{d+2}(\kb)$, un point 
$z=(k_1,\dots,k_d,x,y,e) \in \ZCB$ est singulier si et seulement si 
la matrice jacobienne de cette \'equation s'annule en $z$. Cela est \'equivalent 
au syst\`eme d'\'equations~:
$$
\begin{cases}
x=y=0,\\
\forall~1 \le i \le d, \prod_{j \neq i} (e-dk_j) = 0,\\
\sum_{i=1}^d \prod_{j \neq i} (e-dk_j) = 0.
\end{cases}
$$
La derni\`ere \'equation \'etant impliqu\'ee par la deuxi\`eme famille d'\'equations, 
il est alors facile de v\'erifier que $\ZCB_\singulier$ est bien d\'ecrit 
comme l'\'enonc\'e le pr\'etend.

Les derni\`eres assertions sont imm\'ediates.
\end{proof}

\bigskip

\begin{coro}\label{coro:zsing-cyclique}
Si $c \in \CCB$ et $z \in \ZCB_c$, alors $z$ est singulier dans $\ZCB$ si et seulement si il 
l'est dans $\ZCB_c$.
\end{coro}

\bigskip

\subsection{Ramification} 
La vari\'et\'e $\ZCB$ \'etant normale, la vari\'et\'e $\PCB$ \'etant lisse et le morphisme 
$\Upsilon : \ZCB \to \PCB$ \'etant plat et fini, le th\'eor\`eme de puret\'e du lieu 
de ramification~\cite[Expos\'e~\MakeUppercase{\romannumeral 10},~th\'eor\`eme~3.1]{sga} 
nous dit que le lieu de ramification de $\Upsilon$ est purement de codimension $1$. 
Il est en fait ais\'ement calculable~:

\bigskip

\begin{prop}\label{prop:ramification-cyclique}
Soit $z=(k_1,\dots,k_d,x,y,e) \in \ZCB$ et $p=(k_1,\dots,k_d,x,y)=\Upsilon(z) \in \PCB$. 
On note $F_{\euler,p}(\tb) \in \kb[\tb]$ la sp\'ecialisation de $F_\euler(\tb)$ en $p$. 
Alors $\Upsilon$ est ramifi\'e en $z$ si et seulement si $F_{\euler,p}'(e) =0$, c'est-\`a-dire 
si et seulement si $e$ est une racine multiple de $F_{\euler,p}$.
\end{prop}

\begin{proof}
Puisque $Z=P[\tb]/\langle F_\euler(\tb)\rangle$ (voir le th\'eor\`eme~\ref{centre rang 1}), 
cela d\'ecoule imm\'ediatement de~\cite[Expos\'e~\MakeUppercase{\romannumeral 1},~corollaire~7.2]{sga}.
\end{proof}

\bigskip

\begin{coro}\label{coro:ramifiction-cyclique-p}
Soit $c = (k_1,\dots,k_d) \in \CCB$ et $(x,y) \in \AM^2(\kb)$ (de sorte que $(c,x,y) \in \PCB$). 
Alors $(c,x,y) \in \Upsilon(\ZCB_\ram)$ si et seulement si 
$\D_d(0,\s_2(c),\dots,\s_{d-1}(c),\s_d(c)-(-1)^dxy)=0$.
\end{coro}

\bigskip

\subsection{Sur les groupes ${\boldsymbol{D_c}}$}\label{rema:dc-cyclique} 
Parmi les groupes $D_c$, le seul dont nous ayons besoin est le groupe $D_0$. 
Dans cette sous-section, nous allons montrer que, m\^eme lorsque $n=\dim_\kb(V)=1$, 
les groupes $D_c$ ont un comportement extr\^emement difficile \`a cerner~: c'est 
une chance de ne pas avoir \`a les utiliser.

Les faits rapport\'es dans cette sous-section nous ont \'et\'e expliqu\'es par G. Malle 
(les erreurs \'eventuelles nous \'etant dues). Nous le remercions chaleureusement pour son aide. 
Fixons $c \in \CCB$, et notons $F_\euler^c(\tb)$ la sp\'ecialisation de $F_\euler(\tb)$ en $c$. 
C'est un polyn\^ome appartenant \`a $\kb[X,Y][\tb]$ et 
$D_c$ est le groupe de Galois de $F_\euler^c(\tb) \in \kb[X,Y][\tb]$, vu comme 
polyn\^ome \`a coefficients dans le corps $\kb(X,Y)$. 
Notons qu'en fait $F_\euler^c(\tb) \in \kb[XY][\tb]$~: nous noterons $T=XY$, de sorte 
que $F_\euler^c(\tb) \in \kb[T][\tb]$. 
Le r\'esultat suivant permet de simplifier le calcul de $D_c$~:

\bigskip

\begin{lem}\label{lem:dc-t}
$D_c$ est le groupe de Galois de $F_\euler^c(\tb)$ vu comme polyn\^ome \`a coefficients 
dans le corps $\kb(T)$.
\end{lem}

\begin{proof}
Si $L$ est un corps de d\'ecomposition du polyn\^ome $F_\euler^c(\tb)$ sur $\kb(T)$, alors 
le corps $L(Y)$ des fractions rationnelles est un corps de d\'ecomposition du m\^eme 
polyn\^ome sur $\kb(T,Y)=\kb(X,Y)$. Le r\'esultat en d\'ecoule. 
\end{proof}

\bigskip

\begin{coro}\label{coro:dc-t}
Le sous-groupe $D_c$ de $G=\SG_d$ contient un cycle de longueur $d$.
\end{coro}

\begin{proof}
Gr\^ace au lemme~\ref{lem:dc-t}, voyons $F_\euler^c(\tb)$ comme appartenant \`a $\kb[T][\tb]$. 
Puisque $\kb$ est de caract\'eristique z\'ero et $\kb[T]$ est r\'egulier de dimension $1$, 
le groupe d'inertie \`a l'infini $I$ est cyclique. 
Puisque $d \ge 2$, le polyn\^ome $F_\euler^c(\tb)$ est totalement ramifi\'e \`a l'infini, 
ce qui implique que $I$ agit transitivement sur $\{1,2,\dots,d\}$. D'o\`u le r\'esultat.
\end{proof}

\bigskip

Le corollaire~\ref{coro:dc-t} impose des restrictions s\'ev\`eres 
sur le groupe $D_c$. Notons celles-ci (la premi\`ere est due \`a Schur, la deuxi\`eme \`a Burnside)~:

\bigskip

\begin{coro}\label{coro:schur-burnside}
\begin{itemize}
 \itemth{a} Si $d$ n'est pas premier et $D_c$ est primitif, alors $D_c$ est $2$-transitif.

\itemth{b} Si $d$ est premier, alors $D_c$ est $2$-transitif ou bien $D_c$ contient un 
$d$-sous-groupe de Sylow distingu\'e.
\end{itemize}
\end{coro}

\bigskip

De plus, $\kb(T)[\tb]/\langle F_\euler^c(\tb)\rangle \simeq \kb(\tb)$, donc la courbe projective 
lisse de corps des fractions 
$\kb(T)[\tb]/\langle F_\euler^c(\tb)\rangle$ est isomorphe \`a $\PM^1(\kb)$~: elle est de genre z\'ero. 
La conjecture de Guralnick-Thompson, dont la preuve a \'et\'e compl\'et\'e 
par Frohardt et Magaard~\cite{gt-conj}, implique alors le r\'esultat suivant~:

\bigskip

\begin{theo}[Conjecture de Guralnick-Thompson]\label{theo:gt-conj}
Il existe un ensemble fini $\EC$ de groupes finis simples tels que, pour tout $d \ge 2$ et pour tout $c \in \CCB$, 
tout facteur de composition de $D_c$ non ab\'elien et non altern\'e appartienne \`a $\EC$. 
\end{theo}

\bigskip

\def\frobenius{{\mathrm{Fr}}}

Nous allons montrer ici par quelques exemples que le calcul de $D_c$ en g\'en\'eral 
peut cependant se r\'ev\'eler extr\^emement compliqu\'e~: il peut m\^eme d\'ependre du choix du corps de base $\kb$, 
comme le montre le tableau suivant, qui donne le groupe $D_c$ selon que $\kb=\QM$ ou $\kb=\CM$ 
(dans ce tableau, $\frobenius_{pr}$ d\'esigne le groupe de Frobenius 
$(\ZM/r\ZM) \ltimes (\ZM/p\ZM)$, vu comme sous-groupe de $\SG_p$, o\`u $p$ est premier et $r$ divise $p-1$). 

\bigskip

\begin{centerline}{
\begin{tabular}{@{{{\vrule width 1.2pt}\,\,\,}}c@{{\,\,\,{\vrule width 1.2pt}\,\,\,}}c|c@{{\,\,\,{\vrule width 1.2pt}}}}
\hlinewd{1.2pt}
\petitespace $F_\euler(\tb)$ & $\kb=\QM$ & $\kb=\CM$ \\
\hlinewd{1.2pt}
\petitespace $(\tb^2+20\tb+180)(\tb^2-5\tb-95)^4 - XY$ & $\Aut(\AG_6)$ & $\Aut(\AG_6)$ \\
\hline
\petitespace $(\tb+1)^4(\tb-2)^2(\tb^3-3\tb-14) - XY$ & $(\SG_3 \wr \SG_3) \cap \AG_9$ & $(\SG_3 \wr \SG_3) \cap \AG_9$ \\
\hline
\petitespace 
$\tb(\tb^8+6\tb^4+25) - XY$ & $\AG_9$ & $\AG_9$ \\
\hline\petitespace 
$\tb(\tb^4+6\tb^2+25)^2 - XY$ & $\AG_9$ & $\AG_9$ \\
\hline
\petitespace $\tb^9-9\tb^7+27\tb^5-30\tb^3+9\tb -XY$
& $\SG_3 \ltimes (\ZM/3\ZM)^2$ & $\SG_3 \ltimes (\ZM/3\ZM)^2$ \\
\hline 
\petitespace
$\tb^{11}-11\tb^9+44\tb^7-77\tb^5+55\tb^3-11\tb - XY$ & $\frobenius_{110}$ & $\frobenius_{22}$  \\ 
\hline
\petitespace
$\tb^{13}-13\tb^{11}+65\tb^9-156\tb^7+182\tb^5-91\tb^3+13\tb - XY$ & $\frobenius_{156}$ & $\frobenius_{26}$ \\ 
\hlinewd{1.2pt}
\end{tabular}
}
\end{centerline}

\bigskip

\`A propos de ce tableau, quelques commentaires s'imposent. Pour pr\'eciser sur quel 
corps on travaille, notons $D_c^\kb$ le groupe de Galois $D_c$. Commen\c{c}ons par rappeler 
quelques faits classiques~:

\medskip

\begin{itemize}
\itemth{a} $D_c^\CM$ est un sous-groupe distingu\'e de $D_c^\QM$.

\medskip

\itemth{b} Le calcul de $D_c^\QM$ dans tous les exemples ci-dessus peut \^etre effectu\'e 
gr\^ace au logiciel {\tt MAGMA} (voir~\cite{magma}).

\medskip

\itemth{c} Puisqu'il n'existe pas de rev\^etement non ramifi\'e de la droite affine, 
$D_c^\CM$ est engendr\'e par ses sous-groupes d'inertie.

\medskip

\itemth{d} Si $z \in \CM$, notons $F_\euler^{c,z}(\tb)$ la sp\'ecialisation en $T \mapsto z$ 
de $F_\euler^c(\tb)$. Si $\a \in \CM$ est une racine de $F_\euler^{c,z}(\tb)$ de multiplicit\'e $m$, 
alors $D_c^\CM$ contient un \'el\'ement d'ordre $m$.
\end{itemize}

\medskip

\`A partir de ces faits-l\`a, le tableau ci-dessus s'obtient de la fa\c{c}on suivante 
(notons $\D_c(T) \in \kb[T]$ le discriminant du polyn\^ome $F_\euler^c(\tb)$)~:

\medskip

\begin{itemize}
\itemth{1} 
Le calcul de $D_c^\QM$ dans le premier exemple est effectu\'e dans~\cite[th\'eor\`eme~I.9.7]{mama}. 
Notons que, pour retrouver le polyn\^ome de~\cite[th\'eor\`eme~I.9.7]{mama}, il faut remplacer 
dans le polyn\^ome ci-dessus $\tb$ par $2\tb-5$, et renormaliser~: cette op\'eration nous permet d'obtenir un polyn\^ome 
dont le coefficient de $\tb^9$ est nul, comme cela doit \^etre le cas pour $F_\euler^c(\tb)$. 
Le passage de $\QM$ \`a $\CM$ en d\'ecoule, car cet exemple provient d'un {\it triplet rigide}.

\medskip

\itemth{2} 
Dans le deuxi\`eme exemple, le logiciel {\tt MAGMA} nous dit que $D_c^\QM=(\SG_3 \wr \SG_3) \cap \AG_9$. 
D'autre part, $D_c^\CM$ est un sous-groupe distingu\'e de $D_c^\QM$ contenant un \'el\'ement d'ordre $9$. 
De plus, $-1$ est une racine de $F_\euler^{c,0}(\tb)$ de multiplicit\'e $4$, donc 
$D_c^\CM$ contient un \'el\'ement d'ordre $4$ (voir~(d)). Il contient aussi un \'el\'ement 
d'ordre $9$ (voir le corollaire~\ref{coro:dc-t}). Or $D_c^\QM$ ne 
contient qu'un seul sous-groupe distingu\'e contenant un \'el\'ement d'ordre $9$ et un 
\'el\'ement d'ordre $4$, c'est lui-m\^eme. D'o\`u $D_c^\CM=D_c^\QM$.

\medskip

\itemth{3\text{-}4} 
Dans les troisi\`eme et quatri\`eme exemples, le fait que $D_c^\QM=\AG_9$ est obtenu 
gr\^ace au logiciel {\tt MAGMA}. Le fait que $D_c^\CM=\AG_9$ d\'ecoule 
de ce que $D_c^\CM$ est distingu\'e dans $D_c^\QM$ et contient un \'el\'ement d'ordre $9$.

\medskip

\itemth{5} Une fois le calcul de $D_c^\QM$ via {\tt MAGMA} effectu\'e, remarquons que 
$$F_\euler^{c,2}(\tb)=(\tb-2)(\tb+1)^2(\tb^3-3\tb+1)^2$$
ce qui permet d'affirmer, gr\^ace \`a~(d) et au corollaire~\ref{coro:dc-t}, 
que $D_c^\CM$ contient un \'el\'ement d'ordre $9$ 
et un \'el\'ement d'ordre $2$. Donc $18$ divise $|D_c^\CM|$. 
Or $D_c^\QM$ ne contient pas de sous-groupe distingu\'e d'indice $3$ et contenant 
un \'el\'ement d'ordre $9$.

\medskip

\itemth{6\text{-}7} Les deux derniers exemples se traitent de la m\^eme mani\`ere. Nous ne traiterons ici 
que le dernier. Dans ce cas, gr\^ace au logiciel {\tt MAGMA}, 
on obtient 
$$\D_c(T)=13^{13} (T-2)^6(T+2)^6.$$
Ce discriminant n'est pas un carr\'e dans $\QM(T)$, alors que c'est un carr\'e dans $\CM(T)$. 
Les groupes d'inertie non triviaux de $D_c^\CM$, en dehors du groupe d'inertie \`a l'infini, 
se trouvent au-dessus des id\'eaux $\langle T-2 \rangle$ et $\langle T+2 \rangle$. Or,
$$F_\euler^{c,2}(\tb)=(\tb-2)(\tb^6 + \tb^5 - 5\tb^4 - 4\tb^3 + 6\tb^2 + 3\tb - 1)^2$$
et
$$F_\euler^{c,-2}(\tb)=(\tb+2)(\tb^6 - \tb^5 - 5\tb^4 + 4\tb^3 + 6\tb^2 - 3\tb - 1)^2.$$
Donc les groupes d'inertie sont d'ordre $2$, ce qui montre que $D_c^\CM$ est engendr\'e par 
ses \'el\'ements d'ordre $2$. D'o\`u $D_c^\CM=\frobenius_{26}$.


\end{itemize}

\chapter{Le type ${\boldsymbol{B_2}}$}\label{chapitre:b2}

\bigskip

\boitegrise{{\bf Hypoth\`eses et notations.} {\it Dans ce chapitre, 
et seulement dans ce chapitre, nous supposons que $\dim_\kb V=2$, nous fixons une 
$\kb$-base $(x,y)$ de $V$ et nous noterons $(X,Y)$ sa $\kb$-base duale. 
Notons $s$ et $t$ les deux r\'eflexions de $\GL_\kb(V)$ dont les matrices dans la base 
$(x,y)$ sont donn\'ees par
$$s = \begin{pmatrix} 0 & 1 \\ 1 & 0 \end{pmatrix}
\qquad\text{et}\qquad t = \begin{pmatrix} -1 & 0 \\ 0 & 1 \end{pmatrix}.$$
Nous supposons de plus que $W=\langle s, t \rangle$~: c'est un groupe de Weyl de type $B_2$.}}{0.75\textwidth}

\bigskip

\section{L'alg\`ebre $\Hb$}\label{section: H B2}

\medskip

Posons alors $w=st$, $w'=ts$, $s'=tst$, $t'=sts$ et $w_0=stst=tsts=-\Id_V$. Alors
$$W=\{1,s,t,w,w',s',t',w_0\}\qquad\text{et}\qquad \Ref(W)=\{s,t,s',t'\}.$$
De plus,
$$\refw=\{\{s,s'\},\{t,t'\}\}.$$
Les matrices des \'el\'ements $w$, $w'$, $s'$, $t'$ et $w_0$ dans 
la base $(x,y)$ sont donn\'ees par
\equat\label{elements B2}
\begin{cases}
w = st=s't' = \begin{pmatrix} 0 & 1 \\ -1 & 0 \end{pmatrix},&\\
w' = ts=t's' = \begin{pmatrix} 0 & -1 \\ 1 & 0 \end{pmatrix},&\\
s' = tst =  \begin{pmatrix} 0 & -1 \\ -1 & 0 \end{pmatrix},&\\
t' = sts = \begin{pmatrix} 1 & 0 \\ 0 & -1 \end{pmatrix},&\\
w_0 = ss'=tt' = \begin{pmatrix} -1 & 0 \\ 0 & -1 \end{pmatrix}.&\\
\end{cases}
\endequat
Nous poserons $A=C_s$ et $B=C_t$ de sorte que, dans $\Hb$, les relations suivantes 
sont v\'erifi\'ees~:
\equat\label{relations B2}
\begin{cases}
[x,X] = -A(s+s') - 2B t, & \\
[x,Y] = A(s-s'), & \\
[y,X] = A(s-s'), & \\
[y,Y] = -A(s+s') - 2B t'. & \\
\end{cases}
\endequat
On en d\'eduit par exemple 
\equat\label{relations B2 plus}
\begin{cases}
[x,X^2] = -A(s+s') X  - A(s-s')Y, & \\
[x,XY] = -2B t Y, & \\
[x,Y^2] = A (s+s') X + A(s-s') Y, & \\
[y,X^2] = A(s+s') Y + A (s-s') X, & \\
[y,XY] = -2B t' X, &\\
[y,Y^2] = -A(s+s') Y - A(s-s') X. &\\
\end{cases}
\endequat
Notons enfin que $\Hb$ est muni d'un automorphisme $\eta$ correspondant \`a $\begin{pmatrix} -1 & 1\\ 1 & 1 \end{pmatrix} \in \NC$~:
$$\eta(x)=y-x,\quad\eta(y)=x+y,\quad \eta(X)=\frac{Y-X}{2},\quad \eta(Y)=\frac{X+Y}{2},$$
$$\eta(A)=B,\quad\eta(B)=A,\quad \eta(s)=t\quad \text{et}\quad \eta(s)=t.$$

\bigskip

\section{Caract\`eres irr\'eductibles}

\medskip

Notons $\e_s$ (respectivement $\e_t$) 
l'unique caract\`ere lin\'eaire de $W$ tel que $\e_s(s)=-1$ et $\e_s(t)=1$ (respectivement 
$\e_t(s)=1$ et $\e_t(t)=-1$). Remarquons que $\e_s\e_t=\e$. 
Notons $\chi$ le caract\`ere de la repr\'esentation $V$ de $W$. Alors
$$\Irr(W)=\{1,\e_s,\e_t,\e,\chi\}$$
et la table des caract\`eres de $W$ est donn\'ee dans la table~\ref{table b2}.
\begin{table}\refstepcounter{theo}
\centerline{
\begin{tabular}{@{{{\vrule width 2pt}\,\,\,}}c@{{\,\,\,{\vrule width 2pt}\,\,\,}}c|c|c|c|c@{{\,\,\,{\vrule width 2pt}}}}
\hlinewd{2pt}
\petitespace
$g$ & $~1~$ & $w_0$ & $s$ & $t$ & $w$
\vphantom{$\DS{\frac{a}{A_{\DS{A}}}}$}\\
\hlinewd{1pt}
\petitespace $|\classe_W(g)|$ & $1$ & $1$ & $2$ & $2$ & $2$ \\
\hline
\petitespace $o(g)$ & $1$ & $2$ & $2$ & $2$ & $4$ \\
\hline
\petitespace $\Crm_W(g)$ & $~W~$ & $~W~$ & $\langle s,s'\rangle$ & $\langle t,t'\rangle$ & 
$\langle w \rangle$  \\
\hlinewd{2pt}
\petitespace $1$& $1$ & $1$  & $1$ & $1$ & $1$ \\
\hline
\petitespace $\e_s$& $1$ & $1$  & $-1$ & $1$ & $-1$ \\
\hline
\petitespace $\e_t$& $1$ & $1$  & $1$ & $-1$ & $-1$ \\
\hline
\petitespace $\e$& $1$ & $1$  & $-1$ & $-1$ & $1$ \\
\hline
\petitespace $\chi$& $2$ & $-2$  & $0$ & $0$ & $0$ \\
\hlinewd{2pt}
\end{tabular}}

\bigskip

\caption{Table des caract\`eres de $W$}\label{table b2}
\end{table}
Les degr\'es fant\^omes sont donn\'es par
\equat\label{fantome b2}
\begin{cases}
f_1(\tb)=1, & \\
f_{\e_s}(\tb)=\tb^2, & \\
f_{\e_t}(\tb)= \tb^2, & \\
f_\e(\tb)=\tb^4, \\
f_\chi(\tb)=\tb+\tb^3.
\end{cases}
\endequat

\bigskip

\section{Calcul de ${\boldsymbol{(V \times V^*)/W}}$}\label{section:quotient B2}

\medskip

Avant de calculer le centre $Z$ de $\Hb$, nous allons calculer sa sp\'ecialisation $Z_0$ en 
$(A,B) \mapsto (0,0)$. D'apr\`es l'exemple~\ref{exemple zero-0}, 
$$Z_0=\kb[V \times V^*]^W.$$
Gr\^ace \`a~(\ref{fantome b2}) et \`a la proposition~\ref{dim bigrad Q0 fantome}. 
la s\'erie de Hilbert bi-gradu\'ee de $Z_0$ est donn\'ee par
\equat\label{hilbert}
\dim_\kb^\bigrad(Z_0)=\frac{1+\tb\ub+ \tb\ub^3 + 2 \tb^2\ub^2 + 
\tb^3\ub + \tb^3\ub^3 + \tb^4\ub^4}{(1-\tb^2)(1-\tb^4)(1-\ub^2)(1-\ub^4)}.
\endequat
Posons
$$\s=x^2+y^2,\quad \pi=x^2y^2,\quad \Sig = X^2+Y^2\quad \text{et}\quad 
\Pi=X^2Y^2.$$
Alors
\equat\label{invariants B2}
\kb[V^*]^W = \kb[\s,\pi]\qquad\text{et}\qquad\kb[V]^W=\kb[\Sig,\Pi].
\endequat
La s\'erie de Hilbert bi-gradu\'ee de  
$P_\bullet = \kb[V \times V^*]^{W \times W} = 
\kb[V]^W\otimes\kb[V^*]^W = \kb[\s,\pi,\Sig,\Pi]$ 
est donc donn\'ee par
\equat\label{hilbert double}
\dim_\kb^\bigrad(P_\bullet)=\frac{1}{(1-\tb^2)(1-\tb^4)(1-\ub^2)(1-\ub^4)}.
\endequat
Posons maintenant 
$$\euler_0=xX + yY,\hskip2mm \euler_0'=(xY+yX) XY,\hskip2mm 
\euler_0''=xy(xY+yX),\hskip2mm \euler_0'''=xy(xX+yY)XY,$$
$$\delb_0=xyXY,\quad\delb_0'=(x^2-y^2)(X^2-Y^2)\quad
\text{et}\quad \Delb_0=xy(x^2-y^2)XY(X^2-Y^2).$$
Il est alors facile de v\'erifier que la famille 
$(1,\euler_0,\euler_0',\euler_0'',\euler_0''',\delb_0,\delb_0',\Delb_0)$ 
est $P_\bullet$-lin\'eairement 
ind\'ependante et est contenue dans $\kb[V \times V^*]^W$. 
D'autre part, $1$, $\euler_0$, $\euler_0'$, $\euler_0''$, $\euler_0'''$, 
$\delb_0$, $\delb_0'$ et $\Delb_0$ ont respectivement pour bi-degr\'e 
$(0,0)$, $(1,1)$, $(1,3)$, $(3,1)$, $(3,3)$, $(2,2)$, $(2,2)$ et $(4,4)$. Donc 
la s\'erie de Hilbert bi-gradu\'ee du $P_\bullet$-module 
libre de base $(1,\euler_0,\euler_0',\euler_0'',\euler_0''',\delb_0,\delb_0',\Delb_0)$ 
est \'egale \`a celle de $Z_0$ (voir~(\ref{hilbert}) et~(\ref{hilbert double})). 
Ainsi
\equat\label{invariants diagonaux B2}
\kb[V \times V^*]^W = P_\bullet \oplus P_\bullet\, \euler_0 \oplus 
P_\bullet\, \euler_0' \oplus P_\bullet\, \euler_0'' \oplus P_\bullet\, \euler_0''' 
\oplus P_\bullet\, \delb_0 \oplus P_\bullet\, \delb_0' \oplus P_\bullet\, \Delb_0.
\endequat
Le r\'esultat suivant, d\'ej\`a connu (voir par exemple~\cite{alev}), d\'ecrit 
l'alg\`ebre $Z_0$~:

\bigskip

\begin{theo}\label{theo:invariants B2}
$Z_0=\kb[V \times V^*]^W = \kb[\s,\pi,\Sig,\Pi,\euler_0,\euler_0',\euler_0'',\delb_0]$. 
D'autre part, l'id\'eal des relations est engendr\'e par les relations suivantes~:
$$
\begin{cases}
(1) \hskip2cm &\euler_0 ~\euler_0' = \s\Pi + \Sig ~\delb_0,\\
(2) \hskip2cm &\euler_0 ~\euler_0''= \Sig \pi + \s ~\delb_0,\\
(3) \hskip2cm &\delb_0~\euler_0' = \Pi ~\euler_0'',\\
(4) \hskip2cm &\delb_0~\euler_0''= \pi~\euler_0',\\
(5) \hskip2cm &\delb_0^2=\pi\Pi,\\
(6) \hskip2cm &\euler_0^{\prime 2} = \Pi(4~\delb_0 -\euler_0^2+ \s\Sig ),\\
(7) \hskip2cm &\euler_0^{\prime\prime 2} = \pi(4~\delb_0 -\euler_0^2+ \s\Sig ),\\
(8) \hskip2cm &\euler_0'~\euler_0''=4 \pi \Pi+ \s\Sig~\delb_0 -\delb_0~\euler_0^2,\\
(9) \hskip2cm &\euler_0(4~\delb_0 -\euler_0^2+ \s\Sig )=\s ~\euler_0'+\Sig~\euler_0''.
\end{cases}
$$
De plus, $Z_0=P\oplus P\euler_0\oplus P\euler^2_0 \oplus P\delb_0 \oplus P\delb_0 \euler_0 
\oplus P\delb_0\euler_0^2 \oplus P\euler_0' \oplus P \euler_0''$.
\end{theo}

\begin{proof}
Il est facile de v\'erifier que 
\equat\label{autres invariants}
\delb_0'=2~\euler_0^2 - \s\Sig - 4~\delb_0,\quad
\Delb_0=\delb_0~\delb_0'\quad\text{et}\quad \euler_0''' = \delb_0~\euler_0.
\endequat
En vertu de~(\ref{invariants diagonaux B2}), ces trois relations impliquent imm\'ediatement que 
$\kb[V \times V^*]^W = \kb[\s,\pi,\Sig,\Pi,\euler_0,\euler_0',\euler_0'',\delb_0]$. 
Cela montre la premi\`ere assertion.

\medskip

Les relations pr\'esent\'ees dans l'\'enonc\'e du th\'eor\`eme~\ref{theo:invariants B2} 
r\'esultent de calculs directs. Compte tenu de la relation (5), 
la relation (8) peut se r\'e\'ecrire
$$\euler_0'~\euler_0''=\delb_0(4~\delb_0 + \s\Sig -\euler_0^2),\leqno{(8')}$$
tandis que (6) et (7) impliquent 
$$\pi ~\euler_0^{\prime 2} = \Pi ~\euler_0^{\prime\prime 2}.\leqno{(10)}$$
Soient $E$, $E'$, $E''$ et $D$ des ind\'etermin\'ees sur le corps 
$\kb(\s,\pi,\Sig,\Pi)$ et notons 
$$\rho : \kb[\s,\pi,\Sig,\Pi,E,E',E'',D] \longto \kb[V \times V^*]^W$$
l'unique morphisme de $\kb$-alg\`ebres qui envoie la suite $(\s,\pi,\Sig,\Pi,E,E',E'',D)$ 
sur $(\s,\pi,\Sig,\Pi,\euler_0,\euler_0',\euler_0'',\delb_0)$. 
Alors $\r$ est surjectif. Notons $f_i$ l'\'el\'ement de 
$\kb[\s,\pi,\Sig,\Pi,E,E',E'',D]$ correspondant \`a la relation $(i)$ de l'\'enonc\'e 
(pour $1 \le i \le 9$), en soustrayant le terme de droite au terme de gauche. Posons
$$\IG = \langle f_1,f_2,f_3,f_4,f_5,f_6,f_7,f_8,f_9 \rangle ~\subset \Ker \r.$$
Notons $\Zti_0=\kb[\s,\pi,\Sig,\Pi,E,E',E'',D]/\IG$ et soient $e$, $e'$, $e''$ et $d$ 
les images respectives de $E$, $E'$, $E''$ et $D$ dans $\Zti_0$. Posons
$$\Zti_0'=P_\bullet + P_\bullet\, e + P_\bullet\, e^2 + P_\bullet\, e' 
+ P_\bullet\, e'' + P_\bullet\, d + P_\bullet\, de + P_\bullet\, de^2.$$
Alors $\Zti_0'$ est un sous-$\kb$-espace vectoriel de 
$\Zti_0$. Les relations donn\'ees par les $(f_i)_{1 \le i \le 9}$ 
montrent que $\Zti_0'$ est une sous-$\kb$-alg\`ebre de $\Zti_0$. 
Comme de plus $\s$, $\pi$, $\Sig$, $\Pi$, $e$, $e'$, $e''$ et $d$ appartiennent 
\`a $\Zti_0'$, on en d\'eduit que $\Zti_0=\Zti_0'$. 

Par cons\'equent, $\Zti_0$ est un quotient du $\kb$-espace vectoriel gradu\'e
$$\EC=P_\bullet \oplus P_\bullet[-2] \oplus (P_\bullet[-4])^3 \oplus P_\bullet[-6] 
\oplus P_\bullet[-8],$$
et $Z_0$ est un quotient de $\Zti_0$. Comme la s\'erie de Hilbert du $\kb$-espace vectoriel 
$\EC$ est \'egale \`a la s\'erie de Hilbert de $Z_0$, on en d\'eduit que 
$$\Zti_0'=P_\bullet \oplus P_\bullet e \oplus P_\bullet e^2 \oplus P_\bullet e' 
\oplus P_\bullet e'' \oplus P_\bullet d \oplus P_\bullet de \oplus P_\bullet de^2$$
et que $\Zti_0 \simeq Z_0$, ce qui montre que $\Ker \r = \IG$.
\end{proof}

\bigskip

\begin{coro}\label{intersection complete}
Les relations $(1)$, $(2)$,\dots, $(9)$ forment un syst\`eme {\bfit minimal}. 
En particulier, 
la $\kb$-alg\`ebre $Z_0=\kb[V \times V^*]^W$ n'est pas d'intersection compl\`ete.
\end{coro}

\begin{proof}
En reprenant les notations de la preuve du th\'eor\`eme~\ref{theo:invariants B2}, 
il suffit de montrer que $(f_i)_{1 \le i \le 9}$ est un syst\`eme minimal 
de g\'en\'erateurs de $\IG$. Notons $\Zba_0=Z_0/\langle \s,\pi,\Sig,\Pi\rangle$ 
et notons $e$, $e'$, $e''$ et $d$ les images respectives de $\euler_0$, 
$\euler_0'$, $\euler_0''$ et $\delb_0$ dans $\Zba_0$. Alors il d\'ecoule 
de~\ref{invariants B2} et des relations~\ref{autres invariants} que
$$\Zba_0 = \kb \oplus \kb e \oplus \kb e^2 \oplus \kb e' \oplus \kb e'' \oplus \kb d 
\oplus \kb de \oplus \kb de^2.$$
Notons $\fba_i \in \kb[E,E',E'',D]$ la r\'eduction du polyn\^ome $f_i$ modulo 
$\langle \s,\pi,\Sig,\Pi \rangle$. Il suffit de montrer que $(\fba_i)_{1 \le i \le 9}$ 
est un syst\`eme de g\'en\'erateurs minimal du noyau du morphisme de $\kb$-alg\`ebres 
$$\rhoba : \kb[E,E',E'',D] \longto \Zba_0$$
qui envoie $E$, $E'$, $E''$ et $D$ sur $e$, $e'$, $e''$ et $d$ respectivement. 

\medskip

L'alg\`ebre $N=k[E,E',E'',D]$ est bigradu\'ee, avec
$E$, $E'$, $E''$ et $D$ de bidegr\'es respectifs $(1,1)$, $(1,3)$, $(3,1)$
et $(2,2)$, et les \'el\'ements $\bar{f}_1,\ldots,\bar{f}_9$ sont
homog\`enes de bidegr\'es respectifs $(2,4)$, $(4,2)$, $(3,5)$, $(5,3)$,
$(4,4)$, $(2,6)$, $(6,2)$, $(4,4)$, $(3,3)$. On en d\'eduit que
$$\bigl(\sum_{i=1}^9\mathbf{k}\bar{f_i}\bigr)\cap
\bigl(\sum_{i=1}^9 N_+ \bar{f}_i\bigr)\subset
\bigl(\sum_{i=1}^9\mathbf{k}\bar{f_i}\bigr)
\cap\bigl(\sum_{i=1}^9 \mathbf{k}E\bar{f}_i\bigr)$$
Puisque tous ces espaces sont bigradu\'es, cette intersection est contenue
dans
$$(\mathbf{k}\bar{f}_3)\cap(\mathbf{k}E\bar{f}_1)+
(\mathbf{k}\bar{f}_4)\cap(\mathbf{k}E\bar{f}_2)+
(\mathbf{k}\bar{f}_5+\mathbf{k}\bar{f}_8)\cap (\mathbf{k}Ef_9).$$
Puisque $E$ ne divise pas $\bar{f}_3$ ni $\bar{f}_4$, ni aucun \'el\'ement
non nul de $\mathbf{k}\bar{f}_5+\mathbf{k}\bar{f}_8$, on
conclut que 
$\bigl(\sum_{i=1}^9\mathbf{k}\bar{f_i}\bigr)\cap
\bigl(\sum_{i=1}^9 N_+ \bar{f}_i\bigr)=0$, donc $(\bar{f}_i)_{1\le i\le 9}$
est un syst\`eme minimal de g\'en\'erateurs de $\ker\bar{\rho}$.
\end{proof}

\bigskip

\begin{coro}\label{euler non}
Le polyn\^ome minimal de $\euler_0$ sur $P_\bullet$ est
$$\tb^8 - 2\s\Sig~ \tb^6 + 
\bigl(\s^2\Sig^2 + 2(\s^2\Pi + \Sig^2\pi - 8 \pi\Pi)\bigr)~ \tb^4 - 
2\s\Sig\,(\s^2\Pi + \Sig^2\pi - 8\pi\Pi)~ \tb^2 + (\s^2\Pi - \Sig^2\pi)^2.$$
\end{coro}

\begin{proof}
En multipliant la relation (9) par $\euler_0$ et en utilisant les relations (1) et (2), 
on obtient
$$\euler_0^2(4\delb_0+\s\Sig-\euler_0^2) = \s^2 \Pi + \Sig^2 \pi + 2\s\Sig \delb_0.$$
On en d\'eduit imm\'ediatement que
$$\delb_0 (4\euler_0^2 - 2 \s\Sig) = \euler_0^4 - \s\Sig \euler_0^2 + \s^2\Pi +\Sig^2\pi.$$
En \'elevant au carr\'e cette relation et en utilisant la relation (5), 
on obtient que le polyn\^ome annonc\'e annule $\euler_0$. Le degr\'e du polyn\^ome 
minimal de $\euler_0$ sur $P_\bullet$ \'etant \'egal \`a $|W|=8$, 
la preuve du corollaire est termin\'ee.
\end{proof}

\bigskip

\section{L'alg\`ebre ${\boldsymbol{Z}}$}\label{section:Q B2}

\medskip

Rappelons que 
$$\euler=xX + yY + A(s+s') + B(t+t')$$
et posons
$$\begin{cases}
\euler'=(xY+yX)XY - A(s-s') XY + BtY^2 + Bt'X^2,\\
\euler''=xy(xY+yX) - A xy (s-s') + B y^2 t + B x^2 t',\\
\delb=xyXY + B x t' X + B y t Y + B^2 (1+w_0) + AB (w+w').\\
\end{cases}$$
Un calcul assez fastidieux montre que
\equat\label{Q}
\euler, \euler',\euler'', \delb \in Z=\Zrm(\Hb)
\endequat
et que les relations suivantes sont satisfaites~:
\equat\label{relations Q}
\begin{cases}
(\Zrm 1) \hskip1cm &\euler ~\euler' = \s\Pi + \Sig ~\delb,\\
(\Zrm 2) \hskip1cm &\euler ~\euler''= \Sig \pi + \s ~\delb,\\
(\Zrm 3) \hskip1cm &\delb~\euler' = \Pi ~\euler'' + B^2 \Sig~\euler,\\
(\Zrm 4) \hskip1cm &\delb~\euler''= \pi~\euler' + B^2 \s~\euler,\\
(\Zrm 5) \hskip1cm &\delb^2=\pi\Pi + B^2 ~\euler^2,\\
(\Zrm 6) \hskip1cm &\euler^{\prime 2} = 
\Pi(4~\delb -\euler^2+ \s\Sig + 4A^2-4B^2) + B^2 \Sig^2,\\
(\Zrm 7) \hskip1cm &\euler^{\prime\prime 2} = 
\pi(4~\delb -\euler^2+ \s\Sig + 4A^2-4B^2)+ B^2 \s^2,\\
(\Zrm 8) \hskip1cm &\euler'~\euler''=
\delb(4~\delb -\euler^2+ \s\Sig + 4A^2-4B^2)-B^2\s\Sig,\\
(\Zrm 9) \hskip1cm &\euler(4~\delb -\euler^2+ \s\Sig  + 4A^2-4B^2)=
\s ~\euler'+\Sig~\euler''.
\end{cases}
\endequat
On voit imm\'ediatement que $\euler_0$, $\euler_0'$, $\euler_0''$ et 
$\delb_0$ sont les images respectives, dans $Z_0=Z/\pG_0 Z$, des \'el\'ements 
$\euler$, $\euler'$, $\euler''$ et $\delb$. D'autre part, les relations 
(1), (2),\dots, (9) du th\'eor\`eme~\ref{theo:invariants B2} sont aussi les images, 
modulo $\pG_0$, des relations (Z1), (Z2),\dots, (Z9).

\bigskip

\begin{theo}\label{theo:Q}
La $\kb$-alg\`ebre $Z$ est engendr\'ee par $A$, $B$, $\s$, $\pi$, $\Sig$, $\Pi$, 
$\euler$, $\euler'$, $\euler''$ et $\delb$. L'id\'eal des relations est engendr\'e par 
$(\Zrm 1)$, $(\Zrm 2)$,\dots, $(\Zrm 9)$. 

De plus, $Z=P\oplus P\euler\oplus P\euler^2 \oplus P\delb \oplus P\delb \euler 
\oplus P\delb\euler^2 \oplus P\euler' \oplus P \euler''$.
\end{theo}

\begin{proof}
La preuve suit strictement les m\^emes arguments que ceux 
du th\'eor\`eme~\ref{theo:invariants B2}, bas\'es entre autres sur des 
comparaisons de s\'eries de Hilbert bi-gradu\'ees.
%
%
\end{proof}

\bigskip

\begin{coro}\label{Q:intersection complete}
Les relations $(\Zrm 1)$, $(\Zrm 2)$,\dots, $(\Zrm 9)$ forment un syst\`eme minimal. 
En particulier, la $\kb$-alg\`ebre $Z$ n'est pas d'intersection compl\`ete.
\end{coro}

\begin{proof}
Cela d\'ecoule imm\'ediatement du th\'eor\`eme~\ref{theo:Q} et du corollaire 
\ref{intersection complete}. 
\end{proof}

\bigskip

\begin{coro}\label{euler non pluss}
Le polyn\^ome minimal de $\euler$ sur $P$ est
\begin{multline*}
 \tb^8 - 2(\s\Sig+4A^2+4B^2)~\tb^6 
+ \bigl(\s^2\Sig^2 + 2(\s^2\Pi + \Sig^2\pi -8\pi\Pi) + 8 (A^2+B^2)\s\Sig +16(A^2-B^2)^2\bigr)~\tb^4 \\
- 2\bigl((\s\Sig+4A^2-4B^2)(\s^2\Pi + \Sig^2\pi)-8\s\Sig\pi\Pi+2B^2 \s^2\Sig^2\bigr)~\tb^2 
+ (\s^2\Pi - \Sig^2\pi)^2.
\end{multline*}
\end{coro}

\begin{proof}
La preuve suit exactement le m\^eme sch\'ema que la preuve du corollaire~\ref{euler non}, 
mais en partant des relations $(\Zrm 1)$,\dots, $(\Zrm 9)$ au lieu des relations 
(1),\dots, (9).
\end{proof}

\bigskip

\noindent{\bf Remerciements --- } 
Les calculs ci-dessus (v\'erification du fait que $\euler$, $\euler'$, $\euler''$ et $\delb$ sont centraux et 
v\'erification des relations $(\Zrm 1)$,\dots, $(\Zrm 9)$) ont \'et\'e effectu\'es {\it \`a la main}. 
Malgr\'e toute notre bonne volont\'e, la lourdeur des calculs fait qu'il pourrait \^etre envisageable 
qu'une erreur se soit gliss\'ee insidieusement. Cependant, Ulrich Thiel a \'ecrit des programmes 
(bas\'es sur le logiciel {\tt MAGMA}, voir~\cite{magma}) permettant de calculer dans l'alg\`ebre $\Hb$~: ainsi, il a pu 
v\'erifier {\it ind\'ependamment} que les \'el\'ements $\euler$, $\euler'$, $\euler''$ et $\delb$ 
sont centraux et que les relations $(\Zrm 1)$,\dots, $(\Zrm 9)$ sont satisfaites. Nous tenons 
\`a remercier chaleureusement Ulrich Thiel pour avoir effectu\'e ce travail de v\'erification~: 
il a aussi v\'erifi\'e que le polyn\^ome minimal de $\euler$ est bien donn\'e 
par le corollaire~\ref{euler non pluss}.\finl

\bigskip

\section{Familles de Calogero-Moser}

\medskip

La table~\ref{table omega} donne la valeur de $\O_\psi$ (pour $\psi \in \Irr(W)$) en les g\'en\'erateurs 
de la $P$-alg\`ebre $Z$. Elles sont obtenues soit en calculant effectivement 
l'action des \'el\'ements $\euler$, $\euler'$, $\euler''$ et $\delb$ 
ou alors en utilisant la proposition~\ref{action euler verma} et en 
utilisant ensuite les relations $(\Zrm 1)$,\dots, $(\Zrm 9)$ 
(sachant que $\O_\psi(\s)=\O_\psi(\pi)=\O_\psi(\Sig)=\O_\psi(\Pi)=0$). 

\bigskip

\begin{table}\refstepcounter{theo}
\centerline{
\begin{tabular}{@{{{\vrule width 2pt}\,\,\,}}c@{{\,\,\,{\vrule width 2pt}\,\,\,}}c|c|c|c@{{\,\,\,{\vrule width 2pt}}}}
\hlinewd{2pt}
\petitespace
$z \in Z$ & $~\euler~$ & $\euler'$ & $\euler''$ & $\delb$ 
\vphantom{$\DS{\frac{a}{A_{\DS{A}}}}$}\\
\hlinewd{2pt}
\petitespace $\O_1 $     & $-2(B+A)$ & $0$ & $0$ & $2B(B+A)$ \\
\hline
\petitespace $\O_{\e_s}$ & $-2(B-A)$ & $0$ & $0$ & $2B(B-A)$ \\
\hline
\petitespace $\O_{\e_t}$ & $2(B-A)$ & $0$ & $0$ & $2B(B-A)$ \\
\hline
\petitespace $\O_\e$     &$2(B+A)$ & $0$ & $0$ & $2B(B+A)$ \\
\hline
\petitespace $\O_\chi$   & $0$      & $0$ & $0$ & $0$ \\
\hlinewd{2pt}
\end{tabular}}

\bigskip

\caption{Table des caract\`eres centraux de $\Hbov$}\label{table omega}
\end{table}

\bigskip

Soit maintenant $k$ un corps commutatif et fixons un morphisme $\kb[\CCB] \to k$. 
On note $a$ l'image de $A$ et $b$ l'image de $B$ (dans $k$). Le tableau pr\'ec\'edent 
permet de calculer imm\'ediatement les partitions de $\Irr(W)$ en 
$k$-familles de Calogero-Moser, en fonction des valeurs de $a$ et $b$. 
Le r\'esultat (classique) est donn\'e dans la table~\ref{table calogero}.

\bigskip

\begin{table}\refstepcounter{theo}
\centerline{
\begin{tabular}{@{{{\vrule width 2pt}\,\,\,}}c@{{\,\,\,{\vrule width 2pt}\,\,\,}}c@{{\,\,\,{\vrule width 2pt}}}}
\hlinewd{2pt}
\petitespace
Conditions & $k$-familles
\vphantom{$\DS{\frac{a}{A_{\DS{A}}}}$}\\
\hlinewd{2pt}
\petitespace $a=b=0$     & $\Irr(W)$ \\
\hline
\petitespace $a=0$, $b \neq 0$ & $\{1,\e_s\}$,\quad $\{\e_t,\e\}$ \quad 
et \quad$\{\chi\}$ \\
\hline
\petitespace $a \neq 0$, $b=0$ & $\{1,\e_t\}$,\quad $\{\e_s,\e\}$ \quad 
et \quad $\{\chi\}$ \\
\hline
\petitespace $a=b \neq 0$     &  $\{1\}$,\quad $\{\e\}$ \quad 
et \quad $\{\e_s,\e_t,\chi\}$ \\
\hline
\petitespace $a=-b \neq 0$     &  $\{\e_s\}$,\quad $\{\e_t\}$ \quad 
et \quad $\{1,\e,\chi\}$ \\
\hline
\petitespace $ab \neq 0$, $a^2\neq b^2$     &  
$\{1\}$,\quad $\{\e_s\}$,\quad $\{\e_t\}$,\quad $\{\e\}$ \quad 
et \quad $\{\chi\}$ \\
\hlinewd{2pt}
\end{tabular}}

\bigskip

\caption{$k$-familles de Calogero-Moser}\label{table calogero}
\end{table}

\bigskip

\section{Le groupe ${\boldsymbol{G}}$}

\medskip

Puisque $w_0=-\Id_V$ appartient \`a $W$ (et puisque les r\'eflexions de $W$ sont 
d'ordre $2$), les r\'esultats de la section~\ref{sec:w0} s'appliquent. 
En particulier, si $\t_0=(-1,1,\e) \in \kb^\times \times \kb^\times \times W^\wedge$, 
alors $\t_0$ peut \^etre vu comme l'\'el\'ement $w_0 \in W \injto G$ et est central dans $G$ 
(voir la proposition~\ref{tau 0 in G}). Ainsi, d'apr\`es~(\ref{inclusion w0}), on a 
$$G \subset W_4 ,$$
o\`u $W_4$ est le sous-groupe de $\SG_W$ form\'e des permutations $\s$ de $W$ 
telles que $\s(-x)=-\s(x)$ pour tout $x \in W$. On note $N_4$ le sous-groupe 
(distingu\'e) de $W_4$ form\'e des permutations $\s \in W_4$ 
telles que $\s(x)\in \{x,-x\}$ pour tout $x \in W$. Alors, en notant $\mu_2=\{1,-1\}$, 
$$N_4 \simeq (\mu_2)^4.$$
De plus,
$$|W_4|=384\qquad\text{et}\qquad |N_4|=16.$$
Notons $\e_W : \SG_W \to \mu_2=\{1,-1\}$ la signature et posons $W_4' = W_4 \cap \Ker \e_W$ et 
$N_4'=W_4' \cap N_4$. Alors
$$N_4' =\{(\eta_1,\eta_2,\eta_3,\eta_4) \in (\mu_2)^4~|~\eta_1\eta_2\eta_3\eta_4=1\} \simeq (\mu_2)^3.$$
De plus,
$$|W_4'|=192\qquad\text{et}\qquad |N_4'|=8.$$
Rappelons que $H$ s'identifie au stabilisateur, 
dans $G$, de $1 \in W$. De plus, $G$ contient l'image de $W \times W$ dans $\SG_W$. 
Cette image, isomorphe \`a $(W \times W)/\D\Zrm(W)$, est d'ordre $32$ et 
son intersection avec $H$, isomorphe \`a $\D W/\D\Zrm(W) \simeq W/\Zrm(W)$, 
est d'ordre $4$. Les \'el\'ements $(s,s)$ et $(t,t)$ 
de $W \times W$ s'envoient sur des \'el\'ements distincts de $N_4$. Donc $H \cap N_4$ est un sous-groupe 
de $N_4'$ d'ordre $4$. Puisque $(w_0,1)$ s'envoie aussi sur un \'el\'ement de 
$N_4'$ qui n'est pas dans $H$, on en d\'eduit que 
$$N_4' \subset G.$$

Notons $f(\tb) \in P[\tb]$ l'unique polyn\^ome unitaire de degr\'e $4$ tel que $f(\euler^2)=0$ 
(ce polyn\^ome est donn\'e par le corollaire~\ref{euler non pluss}). En vertu 
de~(\ref{discriminant carre}), on a 
$$\disc(f(\tb^2)) = 256 ~\disc(f)^2 \cdot (\s^2\Pi-\Sig^2\pi)^2,$$
et donc le discriminant du polyn\^ome minimal de $\euler$ est un carr\'e 
dans $P$. Ainsi, 
$$G \subseteq W_4'.$$
Nous allons montrer que cette inclusion est une \'egalit\'e.

\bigskip

\begin{theo}\label{galois B2}
$G=W_4'$.
\end{theo}

\begin{proof}
Il suffit de montrer que $|G|=192$. On sait d\'ej\`a que 
$N_4' \subset G \subset W_4'$, ce qui montre que $G \cap N_4 = N_4'$. 
Pour montrer le th\'eor\`eme, il suffit de montrer que $G/N_4' \simeq \SG_4$. 
Or, $G/N_4'=G/(G \cap N_4)$ est le groupe de Galois du polyn\^ome 
$f$. Il suffit donc de montrer que le groupe de Galois de $f$ sur 
$K$ est $\SG_4$. Notons $\Gba$ ce groupe de Galois. 

Posons $\pG=\langle \s-2, \Sig +2,A-1,B,\Pi-\pi \rangle$. Alors $\pG$ est un id\'eal premier 
et $P/\pG \simeq \kb[\pi]$. Notons $\fba$ la r\'eduction de 
$f$ modulo $\pG$. Alors
$$\fba(\tb)= \tb(\tb^3 + ( 16\,\pi-16\,{\pi}^{2}) \,\tb-64\,{\pi}^{2}).$$
Donc, d'apr\`es~(\ref{discriminant t}), on a 
$$\disc(\fba)=(64 \pi^2)^2 \cdot \left(- 4 ( 16\,\pi-16\,{\pi}^{2})^3 - 27 \cdot (-64\,{\pi}^{2})^2\right)=2^{24} \,\pi^7 (\pi-4)(2\,\pi+1)^2.$$
Donc le discriminant de $\fba$ n'est pas un carr\'e dans $\kb[\pi]$, 
ce qui implique que le discriminant de $f$ n'est pas un carr\'e dans 
$P$. Donc $\Gba$ n'est pas contenu dans le groupe 
altern\'e $\AG_4$.

Puisque $f$ est irr\'eductible, $\Gba$ est un sous-groupe transitif 
de $\SG_4$. En particulier, $4$ divise $|\Gba|$. D'autre part, 
si $c \in \CCB$ est tel que $c_s=c_t=1$, alors, d'apr\`es la table~\ref{table calogero} 
et le th\'eor\`eme~\ref{theo cellules familles}, 
$G$ admet un sous-groupe (le groupe d'inertie de 
$\rGba_c$) qui admet une orbite de longueur $6$. Donc $3$ divise 
$|G|$ et donc $3$ divise aussi $|\Gba|$. Ainsi, 
$12$ divise $|\Gba|$ et, puisque $\Gba \not\subset \AG_4$, cela force 
$\Gba = \SG_4$.
\end{proof}

\begin{rema}
Rappelons que $W_4$ est un groupe de Weyl de type $B_4$ et que 
$W_4'$ est un groupe de Weyl de type $D_4$.\finl
\end{rema}

\section{Cellules de Calogero-Moser, $\Cb\Mb$-caract\`eres cellulaires}

\medskip

\subsection{R\'esultats} 
L'objectif de cette section est de d\'emontrer les conjectures~\BIL~et~\GAUCHE~pour $W$. 
Si $a$ et $b$ sont des nombres r\'eels strictement positifs et si $c_s=a$ et $c_t=b$, 
alors la description des cellules de Kazhdan-Lusztig bilat\`eres, \`a gauche, \`a droite, des 
familles de Kazhdan-Lusztig et des $\kl$-caract\`eres cellulaires est facile et est faite, 
par exemple, dans~\cite{lusztig}. Les diff\'erents cas \`a consid\'erer sont $a > b$, $a=b$ et 
$a < b$~: en utilisant l'automorphisme de $W$ qui \'echange $s$ et $t$, on se ram\`ene 
ais\'ement au cas o\`u $a \ge b > 0$. Les conjectures~\BIL~et~\GAUCHE~d\'ecoulent alors 
de la description des cellules de Calogero-Moser bilat\`eres, \`a gauche, \`a droite, des familles 
de Calogero-Moser et des $\calo$-caract\`eres cellulaires donn\'ee dans la table~\ref{table:conjectures-b2}~:

\medskip

\begin{theo}\label{theo:conjectures-b2}
Soit $c \in \CCB$, posons $a=c_s$ et $b = c_t$ et supposons que $ab\neq 0$. Alors 
il existe un choix d'id\'eaux premiers $\rG_c^\gauche \subset \rGba_c$ tel que les 
$c$-cellules de Calogero-Moser bilat\`eres, \`a gauche, les $c$-familles 
de Calogero-Moser et les $\calo$-caract\`eres $c$-cellulaires soient donn\'es par 
la table~\ref{table:conjectures-b2}.

En cons\'equence, les conjectures~\BIL~et~\GAUCHE~sont v\'erifi\'ees si $W$ est de type $B_2$.
\end{theo}

\medskip

\noindent{\sc Notations - } 
Dans la table~\ref{table:conjectures-b2}, on a pos\'e~:
$$\G_\chi=\{t,st,ts,sts\},~ \G_\chi^+=\{t,st\},~ \G_\chi^-=\{ts,sts\},~
\G_s=\{s,ts,sts\}\text{~et~} \G_t=\{t,st,tst\}.$$
D'autre part~:
\begin{itemize}
\item[$\bullet$] $W_3'=H$ d\'esigne le stabilisateur de $1 \in W$ dans $G=W_4'$ et $W_2'$ d\'esigne le stabilisateur de $s$ dans 
$W_3'$. Notons que $W_3'$ (respectivement $W_2'$) est un groupe de Weyl de type 
$D_3$ (respectivement $D_2=A_1 \times A_1$). 

\item[$\bullet$] $\SG_3$ d\'esigne le sous-groupe de $W_3'$ qui stabilise $\G_s$ 
(c'est aussi le stabilisateur de $\G_t$)~: il est bien isomorphe au groupe sym\'etrique 
de degr\'e $3$. 

\item[$\bullet$] $\ZM/2\ZM$ d\'esigne le stabilisateur, dans $W_2'$, de $\G_\chi^+$ (ou $\G_\chi^-$).\finl
\end{itemize}

\vskip1cm
\begin{table}\refstepcounter{theo}\label{table:conjectures-b2}
{\footnotesize
\centerline{\begin{tabular}{@{{{\vrule width 2pt}}}c@{{{\vrule width 2pt}}}c@{{{\vrule width 1pt}}}c|c|c@{{{\vrule width 1pt}}}c|c
@{{{\vrule width 1.5pt}}}c|c@{{{\vrule width 1pt}}}c|c|c@{{{\vrule width 2pt}}}}
\hlinewd{2pt}
\multirow{2}{*}{~Conditions~} & 
\multirow{2}{*}{~$\Dba_c=\Iba_c$~} & 
\multicolumn{3}{c@{{{\vrule width 1pt}}}}{\trespetitespace Cellules bilat\`eres} & 
\multirow{2}{*}{$\psi$} & \multirow{2}{*}{$\dim_\kb \LC_c(\psi)$~} & 
\multirow{2}{*}{$I_c^\gauche$} & \multirow{2}{*}{$D_c^\gauche$} & 
\multicolumn{3}{c@{{{\vrule width 2pt}}}}{Cellules \`a gauche} \\
\cline{3-5}  \cline{10-12}
&& \trespetitespace$\G$ & $|\G|$ & $\Irr_\G(W)$ & && &&  $C$ & $|C|$ & 
$\isomorphisme{C}_c^\calo$  \\
\hlinewd{2pt}
\multirow{6}{*}{$\stackrel{\DS{a^2 \neq b^2}}{ab \neq 0}$} & \multirow{5}{*}{~$W_2'$~} 
& \trespetitespace$1$ & 1 & $\unb_W$ & ~$\unb_W$~ & $8$ & 
\multirow{6}{*}{~$\ZM/2\ZM$} & \multirow{6}{*}{$W_2'$~} & 1 & 1 & $\unb_W$ \\
\clinewd{0.1pt}{3-7} \clinewd{0.1pt}{10-12}
&& \trespetitespace$w_0$ & 1 & $\e$ & $\e$ & $8$ & && 
$w_0$ & 1 & $\e$  \\
\clinewd{0.1pt}{3-7} \clinewd{0.1pt}{10-12}
&& \trespetitespace$s$ & 1 & $\e_s $  & $\e_s$ & $8$ & && 
$s$ & 1 & $\e_s$ \\
\clinewd{0.1pt}{3-7} \clinewd{0.1pt}{10-12}
&& \trespetitespace$w_0s$ & 1 & $\e_t $  & $\e_t$ & $8$ & && 
$w_0 s$ & 1 & $\e_t$ \\
\clinewd{0.1pt}{3-7} \clinewd{0.1pt}{10-12}
&& \multirow{2}{*}{$\G_\chi$} & \multirow{2}{*}{4} & \multirow{2}{*}{$\chi$} & \multirow{2}{*}{$\chi$} & \multirow{2}{*}{$8$} 
&&& \trespetitespace$\G_\chi^+$ & 2 & $\chi$ 
\\
\clinewd{0.1pt}{10-12}
&&&&&&&&& \trespetitespace$\G_\chi^-$ & 2 & $\chi$ 
\\
\hlinewd{2pt}
\multirow{5}{*}{$\stackrel{\DS{a=b}}{ab \neq 0}$} & \multirow{5}{*}{~$W_3'$~} & \trespetitespace$1$ & 1 & $\unb_W$ & ~$\unb_W$~ & $8$ 
& \multirow{5}{*}{$\SG_3$} & \multirow{5}{*}{$\SG_3$} & 
1 & 1 & $\unb_W$ \\
\clinewd{0.1pt}{3-7} \clinewd{0.1pt}{10-12}
&& \trespetitespace$w_0$ & 1 & $\e$ & $\e$ & $8$ & && 
$w_0$ & 1 & $\e$  \\
\clinewd{0.1pt}{3-7} \clinewd{0.1pt}{10-12}
&& \multirow{3}{*}{~$W \setminus\{1,w_0\}$} & \multirow{3}{*}{6} & 
\multirow{3}{*}{$\e_s, \e_t, \chi$} & \trespetitespace$\e_s$ & $1$ & && 
$\G_s$ & 3 & $\e_s + \chi$  \\
\clinewd{0.1pt}{6-7} \clinewd{0.1pt}{10-12}
&& & && \trespetitespace$\e_t$ & $1$ & && 
$\G_t$ & 3 & $\e_t+\chi$  \\
\clinewd{0.1pt}{6-7} 
&& & && \trespetitespace$\chi$ & $6$ & &&&&\\
\hlinewd{2pt}
\end{tabular}}
}

\vskip0.5cm

\caption{Cellules de Calogero-Moser, familles, caract\`eres cellulaires}
\end{table}

Nous allons maintenant nous consacrer \`a la preuve du th\'eor\`eme~\ref{theo:conjectures-b2}~: 
nous commencerons par \'etudier le cas g\'en\'erique, en descendant la cha\^{\i}ne d'id\'eaux premiers 
$\pGba \supset \pG^\gauche \supset \langle \pi \rangle$. L'utilisation de l'id\'eal premier $\langle \pi \rangle$ 
nous servira \`a lever une ambigu\"{\i}t\'e pour le calcul des cellules de Calogero-Moser \`a gauche. Il est naturel de 
se demander si cette m\'ethode ne peut pas se g\'en\'eraliser, car l'id\'eal premier $\langle \pi \rangle$ 
n'est pas quelconque~: c'est l'id\'eal d'annulation d'une $W$-orbite d'hyperplans de r\'eflexions dans $V^*$. 

Apr\`es avoir \'etudi\'e le cas g\'en\'erique, nous sp\'ecialiserons nos param\`etres pour en d\'eduire le 
th\'eor\`eme~\ref{theo:conjectures-b2}. Le point le plus d\'elicat est la d\'etermination 
des cellules \`a gauche (notamment la proposition~\ref{prop:b2-gauche-enfin}). 

\bigskip

\boitegrise{{\bf Notation.} {\it Si $z \in Z$ (ou $q \in Q$), nous noterons $F_z(\tb)$ (ou $F_q(\tb)$) 
le polyn\^ome minimal de $z$ (ou $q$) sur $P$. Si $F(\tb) \in P[\tb]$, nous noterons $\Fba(\tb)$ (respectivement 
$F^\gauche(\tb)$, respectivement $F^\pi(\tb)$) la r\'eduction de $F(\tb)$ modulo $\pGba$ (respectivement 
$\pG^\gauche$, respectivement $\langle \pi \rangle$).}}{0.75\textwidth}

\bigskip

\subsection{Cellules bilat\`eres g\'en\'eriques} 
Cette section contient le travail pr\'eparatoire \`a la preuve du th\'eor\`eme~\ref{theo:conjectures-b2}, 
qui en d\'ecoulera assez facilement. 
%
%
%
On pose $\eulerq=\copie(\euler)$, $\eulerq'=\copie(\euler')$, $\eulerq''=\copie(\euler'')$ et 
$\d=\copie(\delb)$. Rappelons que $\pGba=\langle \s\,\pi,\Sig,\Pi \rangle_P$, que 
$\zGba = \Ker(\O_1)$ et que $\qGba=\copie(\zGba)$~: en vertu de la table~\ref{table omega}, on a 
\equat\label{eq:qba-b2}
\qGba=\pGba Q + \langle \eulerq + 2(A+B), \d-2B(A+B), \eulerq',\eulerq'' \rangle_Q.
\endequat
On a $Q/\qGba=P/\pGba=\kb[A,B]$. Rappelons que 
\equat\label{fba euler}
\Fba_\euler(\tb)=\tb^4(\tb-2(A+B))(\tb-2(B-A))(\tb+2(A+B))(\tb+2(B-A)).
\endequat
Rappelons aussi que, puisque $w_0=-\Id_V \in W$ et que $W$ est engendr\'e par des r\'eflexions d'ordre $2$, on a 
$\eulerq_{vw_0}=-\eulerq_v$ pour tout $v \in W$ (voir la proposition~\ref{tau 0 in G}). 

\bigskip

\begin{lem}\label{rba b2}
Soit $v \in W \setminus\{1,w_0\}$. Alors il existe un unique id\'eal premier $\rGba$ 
de $R$ au-dessus de $\qGba$ tel que $\eulerq_v \equiv 2(A-B) \mod \rGba$. 
\end{lem}

\begin{proof}
Montrons tout d'abord l'existence. 
Soit $\rGba'$ un id\'eal premier de $R$ au-dessus de $\qGba$~: 
alors $\eulerq \equiv -2(A+B) \mod \rGba'$ et 
$\eulerq_{w_0} \equiv 2(A+B) \mod \rGba'$. 
D'apr\`es~(\ref{fba euler}), il existe un unique \'el\'ement $v' \in W\setminus\{1,w_0\}$ 
tel que $\eulerq_{v'} \equiv 2(A-B) \mod \rGba$. 

Rappelons que $H$ est le stabilisateur de $1 \in W$ dans $G \subset \SG_W$~: 
c'est aussi le stabilisateur de $w_0$. 
Alors $H$ agit transitivement sur $W \setminus \{1,w_0\}$ 
(d'apr\`es le th\'eor\`eme~\ref{galois B2}) et donc il existe $\s \in H$ 
tel que $\s(v')=v$. Posons $\rGba=\s(\rGba')$. Alors 
$\rGba$ est un id\'eal premier de $R$ au-dessus de $\qGba$ (car $\s \in H$) 
et $\eulerq_v \equiv 2(A-B) \mod \rGba$. Cela termine 
la preuve de l'existence.

\medskip

Montrons maintenant l'unicit\'e. Soient donc $\rGba$ et $\rGba'$ deux id\'eaux 
premiers de $R$ au-dessus de $\qGba$ tels que 
$\eulerq_v - 2(A-B) \in \rGba \cap \rGba'$. Alors il existe 
$\s \in H$ tel que $\rGba' =\s(\rGba)$. On a donc 
$\eulerq_v \equiv \eulerq_{\s(v)} \equiv 2(B-A) \mod \rGba$. 
D'apr\`es~(\ref{fba euler}), on sait que $2(A-B)$ est une racine simple de $\fba(\tb)$, 
donc $\s(v)=v$. Par cons\'equent, $\s \in I$, o\`u 
$I$ est le stabilisateur de $v$ dans $H$. 
D'apr\`es le th\'eor\`eme~\ref{galois B2}, $I$ est le groupe 
de Klein agissant sur $W\setminus\{1,w_0,v,vw_0\}$ 
(notons que $|I|=4$). 

Notons $\Dba$ (respectivement $\Iba$) le groupe de d\'ecomposition 
(respectivement d'inertie) de $\rGba$ (dans $G$). D'apr\`es~(\ref{fba euler}), on a 
$\Iba \subset \Dba \subset I$ et il nous reste \`a montrer que $I =\Iba$. 
Or la cellule bilat\`ere g\'en\'erique recouvrant la famille de Calogero-Moser 
g\'en\'erique $\{\chi\}$ a pour cardinal $\chi(1)^2=4$, et c'est une orbite sous l'action de 
$\Iba$. Donc $|\Iba| \ge 4 = |I|$. D'o\`u le r\'esultat.
\end{proof}

\bigskip

Comme cons\'equence de la preuve du lemme pr\'ec\'edent, nous obtenons le r\'esultat suivant~:

\bigskip

\begin{coro}\label{dba=iba}
Soit $v \in W \setminus \{1,w_0\}$. Notons $\rGba$ l'unique id\'eal 
premier de $R$ au-dessus de $\qGba$ et tel que $\eulerq_v \equiv 2(A-B) \mod \rGba$. 
Notons $\Dba$ (respectivement $\Iba$) le groupe de d\'ecomposition (respectivement 
d'inertie) de $\rGba$ dans $G$. Alors~:
\begin{itemize}
\itemth{a} $\Dba=\Iba=\{\t \in G~|~\t(1)=1$ et $\t(v)=v\}
\simeq \ZM/2\ZM \times \ZM/2\ZM$. 

\itemth{b} $R/\rGba = Q/\qGba = P/\pGba \simeq \kb[A,B]$.

\itemth{c} Les cellules de Calogero-Moser bilat\`eres g\'en\'eriques 
(par rapport \`a $\rGba$) sont $\{1\}$, $\{w_0\}$, $\{v\}$, $\{vw_0\}$ 
et $W\setminus \{1,w_0,v,vw_0\}$. De plus, 
$\Irr_{\{1\}}(W)=\{\unb_W\}$, $\Irr_{\{w_0\}}(W)=\{\e\}$, 
$\Irr_{\{v\}}(W)=\{\e_s\}$, $\Irr_{\{vw_0\}}(W)=\{\e_t\}$ et $\Irr_{\{1\}}(W)=\{\chi\}$. 
\end{itemize}
\end{coro}

\bigskip

\boitegrise{{\bf Choix.} 
{\it Dor\'enavant, et ce jusqu'\`a la fin de ce chapitre, nous noterons $\rGba$ 
l'unique id\'eal premier de $R$ au-dessus de $\qGba$ tel que $\eulerq_s \equiv 2(A-B) \mod \rGba$.}}{0.75\textwidth}

\bigskip

Dans ce cas, 
\equat\label{eq:dba-b2}
\Dba=\Iba=W_2',
\endequat
les cellules de Calogero-Moser bilat\`eres g\'en\'eriques sont $\{1\}$, $\{s\}$, $\{w_0s\}$, $\{w_0\}$ et $\G_\chi$ et 
\equat\label{eq:familles-b2}
\begin{cases}
~\Irr_{\{1\}}^\calo(W)=\{\unb_W\},\\
~\Irr_{\{s\}}^\calo(W)=\{\e_s\},\\
~\Irr_{\{w_0s\}}^\calo(W)=\{\e_t\},\\
~\Irr_{\{w_0\}}^\calo(W)=\{\e\},\\
~\Irr_{\G_\chi}^\calo(W)=\{\chi\}.
\end{cases}
\endequat

\bigskip

\subsection{Caract\`eres cellulaires g\'en\'eriques} 
Rappelons que $\pG^\gauche=\langle \Sig,\Pi \rangle_P$. 

\bigskip

\begin{lem}\label{eq:qgauche-b2}
On a 
$\qG^\gauche = \pG^\gauche Q + \langle \eulerq+2(B+A),\eulerq'+B\Sig,\eulerq'',\d-2B(A+B) \rangle_Q$. 
\end{lem}
\begin{proof}
Posons $\qG'=\langle \eulerq+2(B+A),\eulerq'+B\Sig,\eulerq'',\d-2B(A+B) \rangle_Q$. 
Tout d'abord, remarquons que $Q/\pG^\gauche Q$ est la $P/\pG^\gauche=\kb[A,B,\Sig,\Pi]$-alg\`ebre admettant la pr\'esentation
\equat\label{eq:presentation-qgauche-b2}
\begin{cases}
(\Qrm 1^\gauche) \hskip1cm &\eulerq ~\eulerq' = \Sig ~\d,\\
(\Qrm 2^\gauche) \hskip1cm &\eulerq ~\eulerq''= 0,\\
(\Qrm 3^\gauche) \hskip1cm &\d~\eulerq' = \Pi ~\eulerq'' + B^2 \Sig~\eulerq,\\
(\Qrm 4^\gauche) \hskip1cm &\d~\eulerq''= 0,\\
(\Qrm 5^\gauche) \hskip1cm &\d^2= B^2 ~\eulerq^2,\\
(\Qrm 6^\gauche) \hskip1cm &\eulerq^{\prime 2} = 
\Pi(4~\d -\eulerq^2+  4A^2-4B^2) + B^2 \Sig^2,\\
(\Qrm 7^\gauche) \hskip1cm &\eulerq^{\prime\prime 2} = 0,\\
(\Qrm 8^\gauche) \hskip1cm &\eulerq'~\eulerq''=
\d(4~\d -\eulerq^2+ 4A^2-4B^2),\\
(\Qrm 9^\gauche) \hskip1cm &\eulerq(4~\d -\eulerq^2+ 4A^2-4B^2)=\Sig~\eulerq''.
\end{cases}
\endequat
Un calcul imm\'ediat montre que toutes ces relations sont satisfaites dans $Q/\qG'$. 
Posons $\qG'' = \pG^\gauche Q + \qG'$. Alors 
$Q/\qG'' \simeq \kb[\Sig,\Pi,A,B] \simeq P/\pG^\gauche$, donc $\qG''$ est un id\'eal premier 
de $Q$, contenant $\pG^\gauche$ et contenu dans $\qGba$ (d'apr\`es~(\ref{eq:qba-b2})).
Le r\'esultat d\'ecoule alors de l'unicit\'e d'un tel id\'eal premier (voir le corollaire~\ref{unicite qgauche}).
\end{proof}

%
%

\bigskip

Rappelons que l'on peut d\'efinir les $\calo$-caract\`eres cellulaires sans passer par le calcul des 
cellules de Calogero-Moser \`a gauche, en utilisant les id\'eaux premiers de $Z$ (ou $Q$) 
au-dessus de $\pG^\gauche$ (voir la remarque~\ref{rema:cellulaires-pas-choix}). 
Remarquons aussi que
\equat\label{eq:minimal-euler-gauche-b2}
F_\euler^\gauche(\tb)=\tb^4(\tb-2(A+B))(\tb-2(B-A))(\tb+2(A+B))(\tb+2(B-A)).
\endequat
Cette \'egalit\'e nous permet de construire d'autres id\'eaux premiers de $Q$ au-dessus de $\pG^\gauche$~:

\bigskip

\begin{lem}\label{lem:qgauche-b2}
Posons
$$
\begin{cases}
\qG_1^\gauche=\qG^\gauche=\pG^\gauche Q + \langle \eulerq+2(B+A),\eulerq'-B\Sig,\eulerq'',\d-2B(B+A) \rangle_Q,\\
\qG_{\e_s}^\gauche=\pG^\gauche Q + \langle \eulerq+2(B-A),\eulerq'-B\Sig,\eulerq'',\d-2B(B-A) \rangle_Q,\\
\qG_{\e_t}^\gauche=\pG^\gauche Q + \langle \eulerq-2(B-A),\eulerq'+B\Sig,\eulerq'',\d-2B(B-A) \rangle_Q,\\
\qG_{\e}^\gauche=\pG^\gauche Q + \langle \eulerq-2(B+A),\eulerq'+B\Sig,\eulerq'',\d-2B(A+B) \rangle_Q,\\
\qG_\chi^\gauche=\pG^\gauche Q + \langle \eulerq,\eulerq'',\d \rangle_Q.\\
\end{cases}
$$
Alors~:
\begin{itemize}
\itemth{a} Si $\g \in \Irr(W)$, alors $\qG^\gauche_\g$ est un id\'eal premier de $Q$ au-dessus de $\pG^\gauche$. 
Le $\calo$-caract\`ere cellulaire g\'en\'erique associ\'e est $\g$.

\itemth{b} Si $\g \in W^\wedge$, alors $\qG^\gauche_\g=\Ker(\O_\g^\gauche)$ et $Q/\qG_\g^\gauche = P^\gauche$. 

\itemth{c} Si on note $\eulerq_\chi'$ l'image de $\eulerq'$ dans $Q/\qG_\chi^\gauche$, alors 
$Q/\qG^\gauche_\chi=(P/\pG^\gauche)[\eulerq'_\chi]$ et le polyn\^ome minimal de $\eulerq'_\chi$ est 
$\tb^2-\Pi(4A^2-4B^2)-B^2\Sig^2$. En particulier, $[k_Q(\qG_\chi^\gauche) : k_P(\pG^\gauche)] = 2$.

\itemth{d} Si $\qG$ est un id\'eal premier de $Q$ au-dessus de $\qG^\gauche$, alors il existe $\g \in \Irr(W)$ tel que 
$\qG=\qG_\g^\gauche$. 
\end{itemize}
\end{lem}

\begin{proof}
(b) est ais\'ement v\'erifiable par un calcul direct, ce qui implique (a) lorsque $\g$ est un caract\`ere lin\'eaire.

Il d\'ecoule de la pr\'esentation~(\ref{eq:presentation-qgauche-b2}) de $Q/\pG^\gauche Q$ que 
$Q/\qG_\chi^\gauche=(P/\pG^\gauche)[\eulerq'_\chi]$ et que le polyn\^ome minimal de $\eulerq'_\chi$ est 
$\tb^2-\Pi(4A^2-4B^2)-B^2\Sig^2$. Puisque ce dernier polyn\^ome \`a coefficients dans $P/\pG^\gauche=\kb[\Sig,\Pi,A,B]$ 
est irr\'eductible, cela implique que $\qG_\chi^\gauche$ est 
bien un id\'eal premier au-dessus de $\pG^\gauche$. On en d\'eduit (c) et la premi\`ere assertion de (a).  
La deuxi\`eme assertion de (a) r\'esulte du th\'eor\`eme~\ref{theo:cellulaire-lisse}. 

%
%
%

\medskip

(d) d\'ecoule du fait que la somme des $\calo$-caract\`eres cellulaire d\'ej\`a construits est \'egal 
au caract\`ere de la repr\'esentation r\'eguli\`ere de $W$.
\end{proof}

\bigskip

\subsection{Cellules \`a gauche g\'en\'eriques}

\medskip

\begin{coro}\label{coro:rgauche-b2}
Soit $\rG^\gauche$ un id\'eal premier de $R$ au-dessus de $\qG^\gauche$ et contenu dans $\rGba$ et
notons $D^\gauche$ son groupe de d\'ecomposition et $I^\gauche$ son groupe d'inertie.
Alors $D^\gauche=W_2'$ et $|I^\gauche|=2$.
\end{coro}

\begin{proof}
Tout d'abord, d'apr\`es le corollaire~\ref{coro:dgauche-dbarre}, on a $D^\gauche \subset \Dba=W_2'$. 
Il r\'esulte du lemme~\ref{lem:qgauche-b2}(c) que, si $C$ est une cellule \`a gauche g\'en\'erique contenue 
dans la cellule bilat\`ere $\G$ associ\'ee \`a $\chi$, alors $|C|=2$ et $|C^\DD|=4$. En particulier, 
$2$ divise $|I^\gauche|$ et $I^\gauche \varsubsetneq D^\gauche$ d'apr\`es la proposition~\ref{prop:gauche-bilatere}(b). 
D'o\`u le r\'esultat.
\end{proof}

\bigskip

\begin{coro}\label{coro:rgauche-unique-b2}
Il existe un unique id\'eal premier $\rG^\gauche$ de $R$ au-dessus de $\qG^\gauche$ et contenu dans $\rGba$. 
\end{coro}

\begin{proof}
Soient $\rG^\gauche$ et $\rG^\gauche_*$ deux id\'eaux premiers de $R$ au-dessus de $\qG^\gauche$ et contenus dans $\rGba$. 
Alors il existe $h \in H$ tel que $\rG_*^\gauche=h(\rG^\gauche)$. On d\'eduit de la proposition~\ref{prop:rgauche-rbarre} que 
$\rGba=g(\rGba)$. Donc $h$ appartient au groupe de d\'ecomposition de $\rGba$, 
qui est le m\^eme que celui de $\rG^\gauche$ (d'apr\`es le corollaire~\ref{coro:rgauche-b2}). 
Donc $\rG^\gauche_*=\rG^\gauche$.
\end{proof}

\bigskip

\boitegrise{{\it Nous noterons $\rG^\gauche$ l'unique id\'eal premier de 
$R$ au-dessus de $\qG^\gauche$ et contenu dans $\rGba$. Nous noterons $D^\gauche$ son 
groupe de d\'ecomposition et $I^\gauche$ son groupe d'inertie.}}{0.75\textwidth}

\bigskip

Le corollaire~\ref{coro:rgauche-b2} nous dit que
\equat\label{eq:dgauche-igauche-b2}
D^\gauche=W_2'\qquad\text{et}\qquad |I^\gauche|=2.
\endequat

\bigskip

\begin{coro}\label{coro:r-rgauche-b2}
$R/\rG^\gauche \simeq Q/\qG^\gauche_\chi$ est int\'egralement clos.
\end{coro}

\bigskip

\begin{coro}\label{coro:cellules-gauches-b2}
$\{1\}$, $\{s\}$, $\{tst\}$ et $\{w_0\}$ sont des cellules de Calogero-Moser \`a gauche g\'en\'eriques, 
et leurs $\calo$-caract\`eres cellulaires associ\'es sont donn\'es par
$$
\begin{cases}
\isomorphisme{1}_{\kb W}^\calo = \unb_W,\\
\isomorphisme{s}_{\kb W}^\calo = \e_s,\\
\isomorphisme{tst}_{\kb W}^\calo = \e_t,\\
\isomorphisme{w_0}_{\kb W}^\calo = \e.\\
\end{cases}
$$
\end{coro}

\begin{proof}
Cela d\'ecoule du fait que les parties donn\'ees sont aussi des cellules de Calogero-Moser 
bilat\`eres g\'en\'eriques et que les familles de Calogero-Moser g\'en\'eriques associ\'ees sont donn\'ees 
dans le corollaire~\ref{dba=iba}. 
\end{proof}

\bigskip

\begin{coro}\label{coro:euler-modulo-gauche-b2}
Les congruences suivantes sont v\'erifi\'ees dans $R$~:
$$
\begin{cases}
\eulerq \equiv -2(B+A) \mod \rG^\gauche,\\
s(\eulerq) \equiv -2(B-A) \mod \rG^\gauche,\\
tst(\eulerq) \equiv 2(B-A) \mod \rG^\gauche,\\
w_0(\eulerq) \equiv 2(B+A) \mod \rG^\gauche,\\
t(\eulerq)\equiv st(\eulerq) \equiv ts(\eulerq) \equiv sts(\eulerq) \equiv 0 \mod \rG^\gauche.
\end{cases}
$$
\end{coro}

\begin{proof}
D'apr\`es~(\ref{eq:minimal-euler-gauche-b2}), la congruence suivante est v\'erifi\'ee dans 
$R[\tb]$~:
$$\prod_{w \in W} (\tb - w(\eulerq)) \equiv \tb^4(\tb-2(A+B))(\tb-2(B-A))(\tb+2(A+B))(\tb+2(B-A)) \mod \rG^\gauche R[\tb].\leqno{(*)}$$
On sait d\'ej\`a, puisque $\qG^\gauche \subset \rG^\gauche$, que $\eulerq \equiv -2(B+A) \mod \rG^\gauche$. Cela implique que 
$s(\eulerq)$ est congru \`a $-2(B-A)$, $2(B+A)$, $2(B-A)$ ou $0$ modulo $\rG^\gauche$. 
Mais, puisque $s(\eulerq) \equiv -2(B-A) \mod \rGba$ par construction, cela force 
$s(\eulerq) \equiv -2(B-A) \mod \rG^\gauche$.

Les troisi\`eme et quatri\`eme congruences s'obtiennent \`a partir des deux premi\`eres en remarquant 
que $tst(\eulerq)=w_0s(\eulerq)=-s(\eulerq)$ et $w_0(\eulerq)=-\eulerq$. 

La derni\`ere d\'ecoule de $(*)$.
\end{proof}

\bigskip

\begin{coro}\label{coro:congruence-gauche-b2}
Les congruences suivantes sont v\'erifi\'ees dans $R$~:
$$
\begin{cases}
\d \equiv 2B(B+A) \mod \rG^\gauche,\\
s(\d) \equiv 2B(B-A) \mod \rG^\gauche,\\
tst(\d) \equiv 2B(B-A) \mod \rG^\gauche,\\
w_0(\d) \equiv 2B(B+A) \mod \rG^\gauche,\\
t(\d)\equiv st(\d) \equiv ts(\d) \equiv sts(\d) \equiv 0 \mod \rG^\gauche
\end{cases}
$$
$$
\begin{cases}
\eulerq' \equiv -B\Sig \mod \rG^\gauche,\\
s(\eulerq') \equiv -B\Sig \mod \rG^\gauche,\\
tst(\eulerq') \equiv B\Sig \mod \rG^\gauche,\\
w_0(\eulerq') \equiv B\Sig \mod \rG^\gauche,\\
t(\eulerq')^2\equiv st(\eulerq')^2 \equiv ts(\eulerq')^2 \equiv sts(\eulerq')^2 \equiv B^2(\Sig^2-4\Pi) + 4A^2 \Pi \mod \rG^\gauche.
\end{cases}\leqno{\text{\it et}}
$$
Pour finir, $g(\eulerq'') \equiv 0 \mod \rG^\gauche$ pour tout $g \in G$.
\end{coro}

\begin{proof}
Les \'egalit\'es $(Q1^\gauche)$,\dots, $(Q9^\gauche)$ sont \'evidemment satisfaites aussi dans l'alg\`ebre 
$R/\pG^\gauche R$. Puisque $\pG^\gauche R$ est un id\'eal $G$-stable de $R$, on peut appliquer tout \'el\'ement de $G$ 
aux \'egalit\'es $(Q1^\gauche)$,\dots, $(Q9^\gauche)$, et ensuite r\'eduire modulo $\rG^\gauche$. 
On d\'eduit par exemple de $(Q7^\gauche)$ que $g(\eulerq'') \equiv 0 \mod \rG^\gauche$ pour tout $g \in G$, 
comme annonc\'e. 

Lorsque $g(\eulerq) \not\equiv 0 \mod \rG^\gauche$, on d\'eduit de $(Q9^\gauche)$ que 
$4g(\d)\equiv g(\eulerq)^2 - 4A^2+4B^2 \mod \rG^\gauche$, 
ce qui permet de montrer que les congruences de $\d$, $s(\d)$, $tst(\d)$ et $w_0(\d)$ modulo $\rG^\gauche$ 
sont bien celles attendues. D'autre part, si $g \in W\setminus\{1,s,tst,w_0\}$, alors $g(\eulerq) \equiv 0 \mod \rG^\gauche$ 
et on d\'eduit donc de $(Q5^\gauche)$ que $g(\d) \equiv 0 \mod \rG^\gauche$.

Pour finir, lorsque $g(\eulerq) \not\equiv 0 \mod \rG^\gauche$, la congruence de $g(\eulerq')$ modulo $\rG^\gauche$ est 
ais\'ement d\'etermin\'ee par $(Q1^\gauche)$, et est conforme aux pr\'evisions. En revanche, 
lorsque $g(\eulerq) \equiv 0 \mod \rG^\gauche$, alors $g(\d) \equiv 0 \mod \rG^\gauche$ d'apr\`es ce qui pr\'ec\`ede, 
et il d\'ecoule de $(\Qrm 6^\gauche)$ que $g(\eulerq')^2 \equiv B^2(\Sig^2-4\Pi) + 4A^2 \Pi \mod \rG^\gauche$. 
\end{proof}

\bigskip

\begin{lem}\label{lem:red-pol-b2}
$F_{\eulerq'}^\gauche(\tb) = (\tb-B\Sig)^2(\tb+B\Sig)^2(\tb^2 - B^2\Sig^2 - 4A^2 \Pi  + 4B^2\Pi)^2$.
\end{lem}

\begin{proof}
Tout d'abord, en appliquant les \'el\'ements de $W \times W$ \`a $\euler_0' \in \kb[V \times V^*]$, 
on remarque que $\euler_0'$ a huit conjugu\'es, et donc le polyn\^ome minimal de $\euler_0'$ sur $P_\bullet$ 
est de degr\'e $8$. Par cons\'equent, $F_{\eulerq'}(\tb)$ est de degr\'e $8$~: c'est en fait le polyn\^ome caract\'eristique 
de la multiplication par $\euler'$ dans le $P$-module libre $Z$. Ainsi,
$$F_{\eulerq'}(\tb)=\prod_{w \in W} (\tb - w (\eulerq'))$$
et le r\'esultat d\'ecoule alors du corollaire~\ref{coro:congruence-gauche-b2}.
\end{proof}

En conclusion, si $g \in \{t,st,ts,sts\}$ et si $q \in \{\eulerq,\eulerq'',\d\}$ alors
\equat\label{eq:congru-b2}
g(q) \equiv 0 \mod \rG^\gauche
\endequat
et
\equat\label{eq:congru-carre-b2}
g(\eulerq')^2 \equiv B^2(\Sig^2-4\Pi) + 4A^2 \Pi \mod \rG^\gauche.
\endequat
La proposition suivante pr\'ecise la congruence~\ref{eq:congru-carre-b2}~: c'est le point le plus d\'elicat de 
ce chapitre.

\bigskip

\begin{prop}\label{prop:b2-gauche-enfin}
Les congruences suivantes sont vraies dans $R$~:
$$
\begin{cases}
t(\eulerq') \equiv st(\eulerq') \mod \rG^\gauche,\\
ts(\eulerq') \equiv sts(\eulerq') \mod \rG^\gauche,\\
t(\eulerq') \not\equiv ts(\eulerq') \mod \rG^\gauche.\\
\end{cases}
$$
\end{prop}

\begin{proof}
D'apr\`es le corollaire~\ref{coro:congruence-gauche-b2}, 
$$(\tb-t(\eulerq'))(\tb-st(\eulerq'))(\tb-ts(\eulerq'))(\tb-sts(\eulerq'))
\equiv (\tb^2 - B^2\Sig^2 - 4A^2 \Pi  + 4B^2\Pi)^2 \mod \rG^\gauche R[\tb].$$
Cela montre qu'il existe un unique $g_0 \in \{st,ts,sts\}$ tel que $g_0(\eulerq') \equiv t(\eulerq') \mod \rG^\gauche$. 
Puisque $t(\eulerq') \not\equiv 0 \mod \rG^\gauche$ et $sts(\eulerq')=tw_0(\eulerq')=-t(\eulerq')$, l'\'el\'ement 
$g_0$ n'est pas \'egal \`a $sts$. Donc
$$g_0 \in \{st,ts\}.$$
Il nous suffit de montrer que 
$$g_0=st.\leqno{(*)}$$

Notons $F_{\eulerq'}^\pi(\tb)$ la r\'eduction modulo $\pi P$ du polyn\^ome minimal de $\eulerq'$. 
Posons 
\eqna
F^\pi(\tb)&=&\tb^4 + 2 B \Sig  \tb^3 + (-4 A^2 \Pi   + 4 B^2 \Pi   -
        \s  \Sig  \Pi  ) \tb^2 \\ && - (8 A^2 B \Sig  \Pi   +
        2 B^3 \Sig^3 - 8 B^3 \Sig  \Pi   + 2 B \s  \Sig^2 \Pi  
        - 4 B \s  \Pi^2) \tb \\ && - 4 A^2 B^2 \Sig^2 \Pi   - B^4 \Sig^4
        + 4 B^4 \Sig^2 \Pi   - B^2 \s  \Sig^3 \Pi   +
        4 B^2 \s  \Sig  \Pi^2 + \s ^2 \Pi^3.
\endeqna
En utilisant le logiciel {\tt MAGMA} (voir~\cite{magma}), on obtient~:
$$F_{\eulerq'}^\pi(\tb)=F^\pi(\tb) \cdot F^\pi(-\tb).$$
Notons $\rG^\pi$ un id\'eal premier de $R$ au-desus de $\pi P$ et contenu dans $\rG^\gauche$. 
Notons $\rG^\pi_0$ un id\'eal premier de $R$ au-dessus de $\pi P + \pG_0$ et contenu dans $\rG^\pi$. 
On v\'erifie facilement que $F^\pi(\tb)$ est premier avec $F^\pi(-\tb)$ et est s\'eparable, 
donc $F^\pi_{\eulerq'}(\tb)$ admet huit racines distinctes dans $R/\rG^\pi$, qui sont les classes des $g(\eulerq')$, 
o\`u $g$ parcourt $W$. Notons $F^\gauche(\tb)$ la restriction modulo $\pG^\gauche$ de $F^\pi(\tb)$. 
Alors
$$F^\gauche(\tb) = (\tb + B\Sig)^2(\tb^2 - B^2\Sig^2 - 4A^2 \Pi  + 4B^2\Pi).$$
Donc il d\'ecoule du corollaire~\ref{coro:congruence-gauche-b2} que $\eulerq'$ et $s(\eulerq')$ 
sont des racines de $F^\pi(\tb)$ dans $R/\rG^\pi$. D'autre part, puisque $W_2'$ agit transitivement sur 
$\{t,st,ts,sts\}$ et stabilise $\rG^\gauche$, on peut, quitte \`a remplacer $\rG^\pi$ par $g(\rG^\pi)$ pour un $g \in W_2'$, 
supposer que $t(\eulerq')$ est une racine de $F^\pi(\tb)$ modulo $\rG^\pi$. L'autre racine modulo $\rG^\pi$ 
est donc $st(\eulerq')$ ou $ts(\eulerq')$ (ce ne peut pas \^etre $sts(\eulerq')=-t(\eulerq')$, car ce dernier est 
une racine de $F^\pi(-\tb)$ modulo $\rG^\pi$). Notons donc $g_1$ l'unique \'el\'ement de $\{st,ts\}$ 
tel que $g_1(\eulerq')$ soit une racine de $F^\pi(\tb)$ modulo $\rG^\pi$. Par r\'eduction modulo $\rG^\pi_0$, 
on obtient 
$$\eulerq'\cdot s(\eulerq') \cdot t(\eulerq') \cdot g_1(\eulerq') \equiv F^\pi(0) \equiv \s^2\Pi^3 \mod \rG^\pi_0.$$
D'autre part, il existe $g \in G$ tel que $\rG_0 \subset g(\rG^\pi_0)$. Or, puisque $G=W_4'$, il existe des 
signes $\eta_1$, $\eta_2$, $\eta_3$, $\eta_4$ tels que $\{g(1),g(s),g(t),g(g_1)\}=\{\eta_1, \eta_2 s, \eta_3 t,\eta_4 g_1\}$ 
et $\eta_1\eta_2\eta_3\eta_4=1$. Par cons\'equent, 
$$\eulerq' \cdot s(\eulerq') \cdot t(\eulerq') \cdot g_1(\eulerq') \equiv \s^2\Pi^3 \mod \rG_0.$$
Le calcul suivant peut s'effectuer directement \`a l'int\'erieur de $\kb[V \times V^*]^{\D \Zrm(W)} \supset R/\rG_0$~: 
\eqna
\eulerq_0' \cdot s(\eulerq_0') \cdot t(\eulerq_0') \cdot ts(\eulerq_0') &=& 
(xY+yX)(xX+yY)(-xY+yX)(-xX+yY)X^4Y^4 \\
&=& (y^2Y^2-x^2X^2)(y^2X^2-x^2Y^2) \Pi^2 \\
&=& ((x^4+y^4) \Pi-\pi(X^4 + Y^4))\Pi^2 \\
&\equiv& \s^2\Pi^3 \mod \pi \kb[V \times V^*].
\endeqna
Donc $g_1=ts$.

On a donc montr\'e que 
$$(\tb-\eulerq')(\tb-s(\eulerq'))(\tb-t(\eulerq'))(\tb-ts(\eulerq'))
\equiv F^\pi(\tb) \mod \rG^\pi R[\tb].$$
Par r\'eduction modulo $\rG^\gauche$, on obtient 
$$(\tb-t(\eulerq'))(\tb-ts(\eulerq'))\equiv \tb^2 - B^2\Sig^2 - 4A^2 \Pi  + 4B^2\Pi \mod \rG^\gauche R[\tb].$$
Donc $t(\eulerq') \equiv -ts(\eulerq') \mod \rG^\gauche$, ce qui montre que $g_0\neq ts$. Donc $g_0=st$.
\end{proof}

\bigskip

\begin{coro}\label{coro:b2-gauche-enfin}
Pour le choix de $\rG^\gauche$ effectu\'e dans cette section, 
les cellules de Calogero-Moser \`a gauche g\'en\'eriques sont $\{1\}$, $\{s\}$, $\{tst\}$, $\{w_0\}$, 
$\{ts,sts\}$ et $\{t,st\}$.

Notons $g_\gauche$ l'involution de $G$ qui laisse fixe $1$, $s$, $tst$ et $w_0$ et telle que 
$g_\gauche(t)=st$ et $g_\gauche(ts)=sts$. Alors $I^\gauche = \langle g_\gauche\rangle$. 
\end{coro}

\bigskip

\begin{rema}\label{rema:w-w}
On comprend mieux ici la convention choisie pour l'action de $W \times W$ sur $\kb(V \times V^*)$ 
(voir la sous-section~\ref{subsection:specialisation galois 0}). En effet, si on avait choisi l'autre action 
(celle d\'ecrite dans la remarque~\ref{rema:action-dif}), les cellules de Calogero-Moser 
{\it \`a gauche} auraient co\"{\i}ncid\'ees avec les cellules de Kazhdan-Lusztig {\it \`a droite}.\finl
\end{rema}

\bigskip

\subsection{Preuve du th\'eor\`eme~\ref{theo:conjectures-b2}} 
Reprenons les notations du th\'eor\`eme~\ref{theo:conjectures-b2} ($a=c_s$, $b=c_t$). 
Fixons pour l'instant un id\'eal premier $\rG_c^\gauche$ de $R$ contenant $\rG^\gauche$ et $\pG_c R$. 
Puisque $R/\rGba \simeq P/\pGba$, on d\'eduit que $\rGba_c=\rGba + \pGba_c R$ est l'unique id\'eal premier 
de $R$ au-dessus de $\rGba$ et contenant $\rG_c^\gauche$. On notera $D_c^\gauche$ (respectivement 
$I_c^\gauche$) le groupe de d\'ecomposition (respectivement d'inertie) de $\rG_c^\gauche$ 
et $\Dba_c$ (respectivement $\Iba_c$) le groupe de d\'ecomposition (respectivement d'inertie) de $\rGba_c$.

\medskip

Il d\'ecoule des corollaires~\ref{coro:euler-modulo-gauche-b2} et~\ref{coro:congruence-gauche-b2} et 
de la proposition~\ref{prop:b2-gauche-enfin} que~:
$$
\begin{cases}
\eulerq \equiv -2(b+a) \mod \rG_c^\gauche,\\
s(\eulerq) \equiv -2(b-a) \mod \rG_c^\gauche,\\
tst(\eulerq) \equiv 2(b-a) \mod \rG_c^\gauche,\\
w_0(\eulerq) \equiv 2(b+a) \mod \rG_c^\gauche,\\
t(\eulerq)\equiv st(\eulerq) \equiv ts(\eulerq) \equiv sts(\eulerq) \equiv 0 \mod \rG_c^\gauche.
\end{cases}\leqno{(\clubsuit)}
$$
$$
\begin{cases}
\d \equiv 2b(b+a) \mod \rG_c^\gauche,\\
s(\d) \equiv 2b(b-a) \mod \rG_c^\gauche,\\
tst(\d) \equiv 2b(b-a) \mod \rG_c^\gauche,\\
w_0(\d) \equiv 2b(b+a) \mod \rG_c^\gauche,\\
t(\d)\equiv st(\d) \equiv ts(\d) \equiv sts(\d) \equiv 0 \mod \rG_c^\gauche
\end{cases}\leqno{(\diamondsuit)}
$$
$$
\begin{cases}
\eulerq' \equiv -b\Sig \mod \rG_c^\gauche,\\
s(\eulerq') \equiv -b\Sig \mod \rG_c^\gauche,\\
tst(\eulerq') \equiv b\Sig \mod \rG_c^\gauche,\\
w_0(\eulerq') \equiv b\Sig \mod \rG_c^\gauche,\\
t(\eulerq') \equiv st(\eulerq') \equiv -ts(\eulerq') \equiv -sts(\eulerq') \mod \rG_c^\gauche,\\
t(\eulerq')^2\equiv b^2 \Sig^2+4(a^2-b^2) \Pi \mod \rG_c^\gauche.
\end{cases}\leqno{(\heartsuit)}
$$
et
$$g(\eulerq'') \equiv 0 \mod \rG_c^\gauche\leqno{(\spadesuit)}$$
pour tout $g \in G$. Rappelons que l'on suppose que $ab \neq 0$ et que deux \'el\'ements $g$ et $g'$ de $W$ sont 
dans la m\^eme $c$-cellule de Calogero-Moser \`a gauche (respectivement bilat\`ere) 
si et seulement si $g(q) \equiv g'(q) \mod \rG_c^\gauche$ (respectivement $\mod \rGba_c$) 
pour tout $q \in \{\eulerq,\eulerq,\eulerq'',\d\}$ 
(car $Q=P[\eulerq,\eulerq',\eulerq'',\d]$).

\bigskip

\subsubsection*{\bfit Le cas ${\boldsymbol{a^2 \neq b^2}}$} 
Supposons ici que $a^2 \neq b^2$. Il d\'ecoule des congruences $(\clubsuit)$, $(\diamondsuit)$, 
$(\heartsuit)$ et $(\spadesuit)$ que les $c$-cellules de Calogero-Moser \`a gauche 
sont $\{1\}$, $\{s\}$, $\{tst\}$, $\{w_0\}$, $\G_\chi^+=\{t,st\}$ et $\G_\chi^-=\{ts,sts\}$ et 
que les $c$-cellules de Calogero-Moser bilat\`eres sont $\{1\}$, $\{s\}$, $\{tst\}$, $\{w_0\}$ et $\G_\chi$.

Les r\'esultats sur les $c$-familles de Calogero-Moser et les $\calo$-caract\`eres $c$-cellulaires 
donn\'es par la table~\ref{table:conjectures-b2} d\'ecoulent alors du corollaire~\ref{coro:famille-semicontinu}, 
de la proposition~\ref{prop:cellulaire-semicontinu}, de~(\ref{eq:familles-b2}) et 
du corollaire~\ref{coro:cellules-gauches-b2}.

D\'eterminons maintenant $D_c^\gauche$ et $I_c^\gauche$. Notons que $I^\gauche \subset I_c^\gauche$ et que, 
vu la description des $c$-cellules de Calogero-Moser \`a gauche, c'est-\`a-dire des $I_c^\gauche$-orbites, 
cela force l'\'egalit\'e. D'autre part, puique $\rGba_c=\rGba + \pGba_c R$, on a $D^\gauche \subset D_c^\gauche$. 
Puisque les $D_c^\gauche$-orbites sont contenues dans les $c$-cellules de Calogero-Moser bilat\`eres, 
la description de ces derni\`eres force encore l'\'egalit\'e. On montre de m\^eme que $\Dba_c=\Dba=W_2'$ et 
que $\Iba_c=\Iba=W_2'$.

\bigskip

\subsubsection*{\bfit Le cas o\`u ${\boldsymbol{a=b}}$} 
Dans ce cas, la derni\`ere congruence de $(\heartsuit)$ devient 
$$t(\eulerq')^2 \equiv b^2 \Sig^2 \mod \rG_c^\gauche.$$
Donc $t(\eulerq') \equiv b\Sig \mod \rG_c^\gauche$ ou $t(\eulerq') \equiv -b\Sig \mod \rG_c^\gauche$. 
Quitte \`a remplacer $\rG_c^\gauche$ par $g(\rG_c^\gauche)$, o\`u $g \in W_2'=D^\gauche$ 
\'echange $t$ et $sts$, on peut faire le choix suivant~:

\bigskip

\boitegrise{\noindent{\bf Choix de ${\boldsymbol{\rG_c^\gauche}}$.} 
{\it Nous choisissons l'id\'eal premier $\rG_c^\gauche$ de sorte que $t(\eulerq') \equiv b\Sig \mod \rG_c^\gauche$.}}{0.75\textwidth}

\bigskip
 
La famille de congruences $(\clubsuit)$, $(\diamondsuit)$, 
$(\heartsuit)$ et $(\spadesuit)$ montre que les $c$-cellules de Calogero-Moser \`a gauche 
sont $\{1\}$, $\{w_0\}$, $\G_s=\{s,ts,sts\}$ et $\G_t=\{t,st,tst\}$ et 
que les $c$-cellules de Calogero-Moser bilat\`eres sont $\{1\}$, $\{w_0\}$ et $W \setminus \{1,w_0\}$. 

Comme pr\'ec\'edemment, les r\'esultats sur les $c$-familles de Calogero-Moser et les $\calo$-caract\`eres $c$-cellulaires 
donn\'es par la table~\ref{table:conjectures-b2} d\'ecoulent alors du corollaire~\ref{coro:famille-semicontinu}, 
de la proposition~\ref{prop:cellulaire-semicontinu}, de~(\ref{eq:familles-b2}) et 
du corollaire~\ref{coro:cellules-gauches-b2}.

Terminons par la description de $D_c^\gauche$, $I_c^\gauche$, $\Dba_c$ et $\Iba_c$. 
Tout d'abord, $I_c^\gauche$ a deux orbites de longueur $3$ 
($\G_s$ et $\G_t$) donc son ordre est divisible par $3$. De plus, il contient $I^\gauche$ qui est d'ordre $2$. 
Donc son ordre est divisible par $6$. La description des $c$-cellules \`a gauche permet alors de conclure 
que $I_c^\gauche=\SG_3$. D'autre part, $D_c^\gauche$ permute les $c$-cellules \`a gauche qui ont le m\^eme 
$\calo$-caract\`ere $c$-cellulaire associ\'e. Donc $D_c^\gauche$ stabilise $\G_s$ et $\G_t$, ce qui 
force l'\'egalit\'e $D_c^\gauche=I_c^\gauche=\SG_3$.

Du c\^ot\'e des cellules bilat\`eres, rappelons que $\Dba_c=\Iba_c$ car $\Dba_c/\Iba_c$ est un quotient de 
$\Dba/\Iba$. De plus, les inclusions $W_2' \subset \Iba_c$ et $I_c^\gauche \subset \Iba_c$ montrent 
que $W_3' \subset \Iba_c$. L'\'egalit\'e $\Iba_c=W_3'$ s'impose alors.

\part*{Appendices}


\renewcommand\thechapter{\Alph{chapter}}

\def\chaptername{Appendice}\setcounter{chapter}{0}

\chapter{Rappels de th\'eorie de Galois}\label{chapter:galois-rappels}

\bigskip

Soit $R$ un anneau commutatif, $G$ un groupe fini agissant sur $R$ et $H$ 
un sous-groupe de $G$. On pose $Q=R^H$ et $P=R^G$, de sorte que 
$P \subset Q \subset R$. 
Si $\rG$ est un id\'eal premier de $R$, on note $k_R(\rG)$  \indexnot{k}{k_R(\rG)}  le corps des fractions de 
$R/\rG$ (c'est-\`a-dire le quotient $R_\rG/\rG R_\rG$) et $G_\rG^D$   \indexnot{G}{G_\rG^D, G_\rG^I}  
le stabilisateur de $\rG$ dans $G$. Ce sous-groupe de $G$ agit alors 
sur $R/\rG$ et on note $G_\rG^I$ le noyau de cette action. En d'autres termes,
$$G_\rG^I=\{g \in G~|~\forall~r \in R,~g(r) \equiv r \mod \rG\}.$$
Le groupe $G_\rG^D$ (respectivement $G_\rG^I$) est appel\'e le 
{\it groupe de d\'ecomposition} (respectivement le {\it groupe d'inertie}) 
de $G$ en $\rG$. 

Nous fixons dans ce chapitre un id\'eal premier $\rG$ de $R$ et nous posons 
$\qG=\rG \cap Q$ et $\pG=\rG \cap P = \qG \cap P$~:
$$\diagram
\rG \dline & \subset & R \dline \\
\qG \dline & \subset & Q \dline \\
\pG        & \subset & P 
\enddiagram$$
On note 
$$\r_G : \Spec R \to \Spec P,$$
$$\r_H : \Spec R \to \Spec Q$$
$$\Upsilon : \Spec Q \to \Spec P\leqno{\text{et}}$$
les applications induites respectivement par les inclusions $P \subset R$, $Q \subset R$ 
et $Q \subset R$. On a bien s\^ur $\r_G = \Upsilon \circ \r_H$~: en d'autres termes, le diagramme 
$$\xymatrix{
\Spec R \ar[rr]^{\DS{\r_H}} \ar@/_1.2cm/[rrrr]^{\DS{\r_G}}&& \Spec Q \ar[rr]^{\DS{\Upsilon}} 
&& \Spec P.
}$$
est commutatif. Par exemple,
$$\r_G(\rG)=\pG,\qquad\r_H(\rG)=\qG\qquad \text{et}\qquad \Upsilon(\qG)=\pG.$$
Pour finir, posons
$$D=G_\rG^D\qquad \text{et} \qquad I=G_\rG^I.$$

\medskip

\section{Autour du lemme de Dedekind}\label{sub:dedekind} 

\medskip

Rappelons que $(R,\times)$ est un mono\"{\i}de. 
Si $M$ est un autre mono\"\i de, on notera $\Hom_{\text{mon}}(M,R)$ l'ensemble 
des morphismes de mono\"\i des $M \to (R,\times)$. 
Si $A$ est un anneau commutatif, on notera $\Hom_{\mathrm{ann}}(A,R)$ l'ensemble des 
morphismes d'anneaux de $A$ vers $R$~: on peut le voir comme un sous-ensemble de 
$\Hom_{\mathrm{mon}}((A,\times),R)$. Ce sont des sous-ensembles 
de l'ensemble $\FC(M,R)$ des applications de $M$ dans $R$~: notons que $\FC(M,R)$ 
est un $R$-module.

\bigskip

\noindent{\bf Lemme de Dedekind.} 
{\it Si $R$ est {\bfit int\`egre}, alors 
$\Hom_{\mathrm{mon}}(M,R)$ forme une famille $R$-lin\'eairement ind\'ependante 
d'\'el\'ements de $\FC(M,R)$.}

\bigskip

\begin{proof}
Supposons que ce n'est pas le cas et notons $d$ un entier naturel minimal tel qu'il 
existe une relation de d\'ependance lin\'eaire non triviale de longueur $d$ entre 
\'el\'ements de $\Hom_{\text{mon}}(M,R)$. Il existe donc 
des \'el\'ements $\ph_1$,\dots, $\ph_d$ de $\Hom_{\text{mon}}(M,R)$ et 
des \'el\'ements {\it non nuls} $\l_1$,\dots, $\l_d$ de $R$ tels que
$$\forall~m \in M,~\l_1 \ph_1(m) + \cdots + \l_d \ph_d(m) = 0.\leqno{(*)}$$
Notons que $d \ge 2$ car l'int\'egrit\'e de $R$ forcerait $\ph_1(m)=0$ pour tout 
$m \in M$, ce qui est impossible. De plus, par minimalit\'e de $d$, on a $\ph_1\neq \ph_2$. 

Ainsi, il existe $m_0 \in M$ tel que $\ph_1(m_0) \neq \ph_2(m_0)$. 
Par cons\'equent, il r\'esulte de $(*)$ que 
$$ \l_1 \ph_1(m_0 m) + \cdots + \l_d \ph_d(m_0 m) = 0.$$
$$\ph_1(m_0) \cdot \bigl(\l_1 \ph_1(m) + \cdots + \l_d \ph_d(m)\bigr) = 0\leqno{\text{et}}$$
pour tout $m \in M$. 
En soutrayant la deuxi\`eme \'equation \`a la premi\`ere, on obtient~:
$$\forall~m \in M,~
\sum_{i=2}^d \l_i (\ph_i(m_0)-\ph_1(m_0)) \ph_i(m) = 0.$$
Or, puisque $R$ est int\`egre, $\l_2 (\ph_2(m_0)-\ph_1(m_0))$ est non nul, 
donc on a trouv\'e une relation de d\'ependance lin\'eaire non triviale de longueur 
inf\'erieure \`a $d$, ce qui contredit la minimalit\'e de $d$. 
\end{proof}

\bigskip

\begin{coro}\label{coro:dedekind}
Soit $A$ un anneau commutatif et supposons $R$ {\bfit int\`egre}. Alors 
$\Hom_{\mathrm{ann}}(A,R)$ forme une famille $R$-lin\'eairement ind\'ependante 
d'\'el\'ements de $\FC(A,R)$.
\end{coro}

\bigskip

\begin{proof}
En effet, l'ensemble $\Hom_{\text{ann}}(A,R)$ est un sous-ensemble de l'ensemble 
$\Hom_{\text{mon}}((A,\times),R)$~: on applique alors le lemme de Dedekind.
\end{proof}

\bigskip

\section{Groupe de d\'ecomposition, groupe d'inertie}\label{appendice:galois}

\medskip

Nous rappelons quelques r\'esultats plus ou moins classiques dans cette situation.

\begin{prop}\label{galois transitif}
Le groupe $G$ agit transitivement sur les fibres de $\r_G$.
\end{prop}

\begin{rema}\label{H} 
Bien s\^ur, l'\'enonc\'e s'applique \`a $H$~: le groupe $H$ agit transitivement 
sur les fibres de $\r_H$.\finl
\end{rema}

\begin{proof}
Voir~\cite[chapitre 5, \S 2, th\'eor\`eme 2(i)]{bourbaki}.
\end{proof}

\bigskip

\begin{theo}\label{bourbaki}
L'extension de corps $k_R(\rG)/k_P(\pG)$ est normale, 
de groupe de Galois $D/I$ ($=G_\rG^D/G_\rG^I$). 
\end{theo}

\begin{proof}
Voir~\cite[chapitre 5, \S 2, th\'eor\`eme 2(ii)]{bourbaki}.
\end{proof}

\bigskip

\begin{coro}\label{coro bourbaki}
Soit $\rG'$ un id\'eal premier de $R$ contenant $\rG$ et soient $D'=G_{\rG'}^D$ 
et $I'=G_{\rG'}^I$. Alors $D'/I'$ est isomorphe \`a un sous-quotient de $D/I$.
\end{coro}

\begin{proof}
En rempla\c{c}ant $R$ par $R/\rG$, $Q$ par $Q/\qG$ et $P$ par $P/\pG$, 
$D/I$ s'identifie \`a $G$ et le corollaire se d\'eduit imm\'ediatement 
du th\'eor\`eme~\ref{bourbaki}. 
\end{proof}

\bigskip

\begin{theo}\label{raynaud}
Si $Q$ est {\bfit nette} sur $P$ en $\qG$ (i.e. si $\pG Q_\qG=\qG Q_\qG$), 
alors $I$ est contenu dans $H$.
\end{theo}

\begin{proof}
Voir~\cite[chapitre X, th\'eor\`eme 1]{raynaud}.
\end{proof}

\bigskip

\section{Sur la $P/\pG$-alg\`ebre $Q/\pG Q$}\label{section:Q/P} 

\medskip

\subsection{Doubles classes} 
La proposition~\ref{reduction} ci-dessous, 
certainement bien connue (et facile), sera au c\oe ur de notre m\'emoire. Nous 
en donnons donc une preuve pour la commodit\'e au lecteur. Nous aurons besoin de 
quelques notations. Si $g \in G$, le morphisme compos\'e $Q \stackrel{g}{\longto} R \stackrel{{\mathrm{can}}}{\longto} R/\rG$ 
se factorise en un morphisme $\gba : Q/\pG Q \to R/\rG$. Les deux remarques 
suivantes sont \'evidentes.
\equat\label{eq:IGH}
\textit{Si $h \in H$ et $i \in I$, alors $\overline{igh}=\gba$.}
\endequat
On obtient donc une application bien d\'efinie 
\equat\label{def:IGH}
\fonctio{I \backslash G/H}{\Hom_{(P/\pG)\text{-}\alg}(Q/\pG Q,R/\rG)}{IgH}{\gba}.
\endequat
Notons que $\Hom_{(P/\pG)\text{-}\alg}(Q/\pG Q,R/\rG)=\Hom_{P\text{-}\alg}(Q,R/\rG)$. 
Si $\ph \in \Hom_{P/\pG-\text{alg}}(Q/\pG Q,R/\rG)$, alors 
on note $\pht$ la compos\'ee $Q \stackrel{{\mathrm{can}}}{\longto} Q/\pG Q \stackrel{\ph}{\longto} R/\rG R$ 
et il est clair que $\Ker \pht \in \Upsilon^{-1}(\pG)$. Cela nous d\'efinit une application 
\equat\label{def:ker}
\fonctio{\Hom_{(P/\pG)\text{-}\alg}(Q/\pG Q,R/\rG)}{\Upsilon^{-1}(\pG)}{\ph}{\Ker \pht.}
\endequat
Puisque $R/\rG$ est int\`egre, $\Ker \pht$ est un id\'eal premier de $Q$. 

Si $g \in G$, alors $g(\rG) \cap Q \in \Upsilon^{-1}(\pG)$. De plus, si $h \in H$ et $d \in D$, 
alors 
$$hgd(\rG) \cap Q=g(\rG) \cap Q.$$
On a ainsi d\'efini une application 
\equat\label{def:DGH}
\fonctio{D\backslash G / H}{\Upsilon^{-1}(\pG)}{DgH}{g^{-1}(\rG) \cap Q.}
\endequat

\bigskip

\begin{prop}\label{reduction}
L'application $I \backslash G/H \longto \Hom_{(P/\pG)\text{-}\alg}(Q/\pG Q,R/\rG)$ 
d\'efinie en~\ref{def:IGH} est bijective, tout comme 
l'application $D\backslash G / H \to \Upsilon^{-1}(\pG)$ d\'efinie en~\ref{def:DGH}. De plus, 
le diagramme suivant est commutatif~:
$$\diagram
I \backslash G / H \xto[0,3]^{\DS{\sim} \qquad\qquad}_{\DS{IgH \mapsto \gti\quad\quad\quad}} 
\ddto_{\DS{\mathrm{can}}} &&& \Hom_{(P/\pG)\text{-}\alg}(Q/\pG Q,R/\rG) \ddto^{\DS{\ph \mapsto \Ker \pht}} \\
&&&\\
D\backslash G / H \xto[0,3]^{\DS{\sim}}_{\DS{DgH \mapsto g^{-1}(\rG) \cap Q}} &&& \Upsilon^{-1}(\pG). \\
\enddiagram$$
\end{prop}

\bigskip

\begin{proof}
Commen\c{c}ons par montrer l'injectivit\'e de la premi\`ere application. 
Soient $g$ et $g'$ deux \'el\'ements de $G$ tels que $\gba=\gba'$. Cela signifie donc que
$$\forall~q \in Q,~g(q) \equiv g'(q) \mod \rG.$$
Par cons\'equent,
$$\forall~r \in R,~\sum_{h \in H} g h(r) \equiv \sum_{h \in H} g'h(r) \mod \rG.$$
Or, d'apr\`es le lemme de Dedekind, la famille des morphismes 
d'anneaux $R \to R/\rG$ est $R/\rG$-lin\'eairement ind\'ependante. 
Cela signifie donc qu'il existe $h \in H$ tel que
$$\forall~r \in R,~g (r) \equiv g' h(r) \mod \rG,$$
ou encore
$$\forall~r \in R, ~ g' h (g^{-1}(r)) \equiv r \mod \rG.$$
En d'autres termes, $g' h g^{-1} \in I$ et donc 
$g' \in I g H$.

\medskip

Montrons maintenant la surjectivit\'e. Soit $\ph \in \Hom_{P/\pG-\text{alg}}(Q/\pG Q,R/\rG)$ 
et soit $\qG'=\Ker \pht$. Puisque $\ph$ est $(P/\pG)$-lin\'eaire, on a $\qG' \cap P=\pG$. Notons $\rG'$ un 
id\'eal premier de $R$ au-dessus de $\qG'$. Alors il existe $g \in G$ tel que 
$\rG'=g(\rG)$. Donc l'application $g \circ \pht : Q \to R/\rG'$ 
a pour noyau $\qG'=\rG' \cap Q$ et est $Q$-lin\'eaire. D'apr\`es le th\'eor\`eme 
\ref{bourbaki}, il existe donc $d \in G_{\rG'}^D$ tel que 
$g \circ \pht(q) \equiv d(q) \mod \rG'$ pour tout $q \in Q$. 
Ainsi, $\pht(q) \equiv g^{-1}d(q) \mod \rG$, c'est-\`a-dire que 
$\ph = \overline{g^{-1}d}$. 
%
%
%
%

\medskip

Montrons maintenant la bijectivit\'e de la deuxi\`eme application. 
Si $\qG' \in \Upsilon^{-1}(\pG)$, alors il existe $\rG' \in \Spec R$ tel que $\qG' \cap Q=\rG'$. 
En outre, $\rG' \cap P=\qG' \cap P = \pG$ et donc, d'apr\`es la proposition~\ref{galois transitif}, 
il existe $g \in G$ tel que $\rG' = g(\rG)$. Cela montre la surjectivit\'e de 
la fl\`eche horizontale inf\'erieure du diagramme. L'injectivit\'e d\'ecoule encore de la 
proposition~\ref{galois transitif}. 

\medskip

La commutativit\'e du diagramme d\'ecoule des arguments pr\'ec\'edents. 
\end{proof}

\bigskip

\subsection{Corps r\'esiduels}
Soit $g \in [D\backslash G/H]$. On pose pour simplifier $\qG_g=\rG \cap g(Q)$. 
Notons que $\qG_g \cap P=\pG$ et que l'on obtient une suite de morphismes injectifs d'anneaux 
$P/\pG \injto g(Q)/\qG_g \injto R/\rG$. On a donc une suite d'inclusions de corps 
$$k_P(\pG) \quad \subset \quad k_{g(Q)}(\qG_g) \quad \subset \quad k_R(\rG).$$

\bigskip

\begin{lem}\label{gal:bourbaki}
L'extension $k_R(\rG)/k_{g(Q)}(\qG_g)$ est 
normale de groupe de Galois $(D \cap \lexp{g}{H})/(I \cap \lexp{g}{H})$.
\end{lem}

\begin{rema}\label{N} 
Notons que $(D \cap \lexp{g}{H})/(I \cap \lexp{g}{H})$ est naturellement 
un sous-groupe de $D/I$, comme il se doit...\finl 
\end{rema}

\begin{proof}
En effet, cela d\'ecoule du fait que $g(Q)=R^{\lexp{g}{H}}$ et du 
th\'eor\`eme~\ref{bourbaki}.
\end{proof}

\medskip

\begin{coro}\label{dgh igh}
Supposons que, pour tout id\'eal premier $\qG' \in \Upsilon^{-1}(\pG)$, on ait 
$k_Q(\qG')=k_P(\pG)$. Alors $D\backslash G /H=I\backslash G/H$.
\end{coro}

\begin{proof}
D'apr\`es le lemme~\ref{gal:bourbaki} et le th\'eor\`eme~\ref{bourbaki}, 
il d\'ecoule de l'hypoth\`ese que, pour tout $g \in G$, 
$(D \cap \lexp{g}{H})/(I\cap \lexp{g}{H}) \simeq D/I$. En d'autres termes, 
$$\forall~g \in G,~D = I \cdot (D \cap \lexp{g}{H}).$$
Soient donc $g \in G$ et $d \in D$. Alors il existe $i \in I$ et $h \in H$ tels que 
$d = ighg^{-1}$, c'est-\`a-dire $d g = igh$. Donc $DgH=IgH$. 
\end{proof}

\medskip

\begin{lem}\label{inertie}
\'Ecrivons $\Upsilon^{-1}(\pG)=\{\qG_1,\dots,\qG_n\}$ et supposons que $Q$ est nette sur $P$ 
en $\qG_i$ pour tout $i$. Alors $I \subset \bigcap_{g \in G} \lexp{g}{H}$. 
\end{lem}

\begin{proof}
Soit $g \in G$. Alors $g(\rG) \cap Q \in \Upsilon^{-1}(\pG)$ et 
donc il d\'ecoule du th\'eor\`eme~\ref{raynaud} que $\lexp{g}{I} \subset H$ 
(car $\lexp{g}{I}$ est le groupe d'inertie de $g(\rG)$). Ainsi, $I \subset \lexp{g^{-1}}{H}$. 
\end{proof}

\medskip

\begin{prop}\label{cloture galoisienne}
Si $I=\bigcap_{g \in G} \lexp{g}{H} = 1$ et si l'extension $k_R(\rG)/k_P(\pG)$ est s\'eparable, 
alors $k_R(\rG)/k_P(\pG)$ est la cl\^oture galoisienne   
de la famille d'extensions $k_{\lexp{g}{Q}}(\qG_g)/k_P(\pG)$, 
$g \in [D\backslash G / H]$. 
\end{prop}

\noindent{\sc Remarque - } 
Si $R$ est int\`egre (ce qui implique que $P$ et $Q$ le sont) et si $G$ 
agit fid\`element, alors l'hypoth\`ese 
$\bigcap_{g \in G} \lexp{g}{H} = 1$ \'equivaut \`a dire que 
l'extension $\Frac(R)/\Frac(P)$ est la cl\^oture 
galoisienne de $\Frac(Q)/\Frac(P)$.

Notons aussi que l'hypoth\`ese $I=1$ implique que $G$ agit fid\`element.\finl

\bigskip

\begin{proof}
D'apr\`es le th\'eor\`eme~\ref{bourbaki}, 
l'extension $k_R(\rG)/k_P(\pG)$ est normale de groupe de Galois $D$. 
D'apr\`es le lemme~\ref{gal:bourbaki}, l'extension 
$k_R(\rG)/k_{g(Q)}(\qG_g)$ est 
normale de groupe de Galois $D \cap \lexp{g}{H}$. 

Notons $k$ la cl\^oture normale de la famille d'extensions 
$k_{g(Q)}(\qG_g)/k_P(\pG)$, $g \in [H\backslash G / D]$. Alors 
le groupe de Galois $\Gal(k_R(\rG)/k)$ est l'intersection des conjugu\'es, 
dans $D$, des groupes $D \cap \lexp{g}{H}$, $g$ variant dans $[H\backslash G / D]$. 
On a
$$\Gal(k_R(\rG)/k)=\bigcap_{\stackrel{\SS g \in [H\backslash G/D]}{d \in D}} 
\lexp{d}{(D \cap \lexp{g}{H})}=\bigcap_{\stackrel{\SS g \in [D\backslash G/H]}{d \in D}}
D \cap \lexp{dg}{H}.$$
Puisque $\lexp{h}{H}=H$ pour tout $h \in H$, on a donc
$$\Gal(k_R(\rG)/k)=
\bigcap_{\substack{g \in [D\backslash G/H] \\ d \in D \\ h \in H}}
D \cap \lexp{dgh}{H} = \bigcap_{g \in G} \lexp{g}{H}=1,$$
par hypoth\`ese. D'o\`u le r\'esultat.
\end{proof}

\bigskip

\begin{contre}
S'il y a de la ramification, 
alors la proposition ci-dessus est fausse, m\^eme si on suppose que 
$P$, $Q$ et $R$ sont des anneaux de Dedekind (m\^eme de valuation discr\`ete). 
En effet, soit $\sqrt[3]{2}$ une racine cubique de $2$ dans $\CM$, $\z$ une racine primitive 
cubique de l'unit\'e dans $\CM$ et notons $R$ la cl\^oture int\'egrale de 
$\ZM$ dans $M=\QM(\sqrt[3]{2},\z)$. Posons $G=\Gal(M/\QM) \simeq \SG_3$ et 
$H=\Gal(M/\QM(\sqrt[3]{2})) \simeq \ZM/2\ZM$. Alors, $P=\ZM$ et, si $\rG$ est un id\'eal premier 
de $R$ tel que $\rG \cap \ZM = 2\ZM$, alors $D=G$ et $|I|=3$. 

Ainsi, $D\backslash G/H$ est r\'eduit \`a un \'el\'ement et l'extension 
de corps correspondante $k_R(\rG)/k_Q(\qG)$ est galoisienne de groupe de 
Galois $\ZM/2\ZM$ (en vertu du lemme~\ref{gal:bourbaki}), tout comme 
$k_R(\rG)/k_P(\pG)$. Donc $k_P(\pG)=k_Q(\qG) \simeq \FM_{\! 2}$ et 
$k_R(\rG) \simeq \FM_{\! 4}$. Donc $k_R(\rG)$ n'est pas la cl\^oture 
galoisienne de l'extension $k_Q(\qG)/k_P(\pG)$.\finl 
\end{contre}

\bigskip

%

\subsection{Cas des corps}\label{subsection:corps} 
Dans le cas o\`u $R$ est un corps, la situation se simplifie grandement.

\medskip

\boitegrise{{\bf Hypoth\`ese.} 
{\it Dans cette sous-section, et dans cette sous-section seulement, nous reprenons 
les notations de la sous-section pr\'ec\'edente ($P$, $Q$, $R$, $G$, $H$, $D$, $I$...) et nous 
supposons de plus que $R$ est un corps~: nous le noterons $M$. Nous poserons 
$L=Q=M^H$ et $K=P=M^G$. Nous supposerons aussi de plus que $G$ agit 
fid\`element sur $M$. Ainsi, $M/K$ est une extension galoisienne 
de groupe $G$ et $M/L$ est une extension galoisienne de groupe $H$.}}{0.75\textwidth}

\medskip

Il d\'ecoule de l'hypoth\`ese que $\pG=\qG=\rG=0$ et que $D=G$ et $I=1$. 
Ainsi, la proposition~\ref{reduction} fournit une bijection   
$$G/H \stackrel{\sim}{\longleftrightarrow} \Hom_{K-\text{alg}}(L,M).$$
Si $g \in G$, le morphisme de $K$-alg\`ebres $L \to M$, $q \mapsto g(q)$, s'\'etend en un 
morphisme de $M$-alg\`ebres
$$\fonction{g_L}{M \otimes_K L}{M}{m \otimes_K l}{mg(l).}$$

\bigskip

\begin{prop}\label{iso galois}
Le morphisme de $M$-alg\`ebres
$$\sum_{g \in [G/H]} g_L : M \otimes_K L \longto \mathop{\bigoplus}_{g \in [G/H]} M$$
est un isomorphisme.
\end{prop}

\begin{proof}
Puisque $L$ est un $K$-espace vectoriel de dimension $|G/H|$, alors 
$M \otimes_K L$ est un $M$-espace vectoriel de dimension $|G/H|$. 
Il suffit donc de montrer l'injectivit\'e de $\sum_{g \in [G/H]} g_L$, ce qui 
\'equivaut exactement \`a la $M$-ind\'ependance lin\'eaire des applications 
$L \to M$, $q \mapsto g(q)$, lorsque $g$ parcourt $[G/H]$ 
(voir le corollaire~\ref{coro:dedekind} au lemme de Dedekind).
\end{proof}

\bigskip

\section{Rappels sur la cl\^oture int\'egrale}\label{section:cloture integrale}

\medskip

\begin{prop}\label{int clos}
Soit $f \in P[\tb]$, soit $P'$ une $P$-alg\`ebre contenant $P$ et soit $g \in P'[\tb]$. 
On suppose que $f$ et $g$ sont unitaires et que $g$ 
divise $f$ (dans $P'[\tb]$). Alors les coefficients de $g$ sont entiers sur $P$.
\end{prop}

\begin{proof}
Voir~\cite[chapitre 5, \S 1, proposition 11]{bourbaki}.
\end{proof}

\begin{coro}\label{minimal clos}
Si $P$ est {\bfit int\`egre et int\'egralement clos}, 
de corps des fractions $K$, si $A$ est une $K$-alg\`ebre et 
si $x \in A$ est entier sur $P$, alors le polyn\^ome minimal de $x$ sur $K$ 
appartient \`a $P[\tb]$.
\end{coro}

\begin{proof}
 Voir~\cite[chapitre 5, \S 1, corollaire de la proposition 11]{bourbaki}.
\end{proof}

\begin{prop}\label{cloture polynomiale}
Si $P$ est {\bfit int\`egre} et si $f \in P[\tb,\tb^{-1}]$ est entier sur $P$, alors $f \in P$.
\end{prop}

\begin{proof}
Soit $d \ge 1$ et soient $p_0$, $p_1$,\dots, $p_{d-1}$ des \'el\'ements de $P$ 
tels que $p_0 + p_1 f + \cdots + p_{d-1} f^{d-1} = f^d$. Notons $\d$ la $\tb$-valuation de 
$P$ et $\d'$ son degr\'e. Puisque $P$ est int\`egre, le degr\'e 
de $f^d$ est $d\d'$, et donc l'\'egalit\'e ci-dessus ne peut avoir lieu que si 
$\d'=0$. De m\^eme, $\d=0$. Donc $f$ est constant.
\end{proof}

\begin{contre}
L'hypoth\`ese d'int\'egrit\'e est n\'ecessaire dans la proposition pr\'ec\'edente. 
Par exemple, si $p \in P$ est nilpotent, alors $p\tb$ est entier sur $P$.\finl
\end{contre}

\bigskip

\section{Rappels sur les calculs de groupes de Galois\label{sec:calcul}}

\medskip

Soit $K$ un corps commutatif et soit 
$f(\tb)=\tb^d + a_{d-1} \tb^{d-1} + \cdots + a_1 \tb + a_0 \in K[\tb]$. 
On note $M$ un corps de d\'ecomposition de $f$ (sur $K$) et on note
$$\Gal_K(f)=\Gal(M/K).$$
Le groupe $\Gal_K(f)$ est appel\'e le {\it groupe de Galois} de $f$ sur $K$. 

Notons $t_1$,\dots, $t_d$ les \'el\'ements de $M$ tels que
$$f(\tb)=\prod_{i=1}^d (\tb-t_i),$$
de sorte que 
$$M=K[t_1,\dots,t_d]=K(t_1,\dots,t_d).$$
Cette num\'erotation nous fournit un morphisme de groupes injectif
$$\Gal_K(f) ~\injto~ \SG_d.$$

Supposons que $P$ est int\`egre et int\'egralement clos, 
que $K$ est le corps des fractions de $P$, et que $f \in P[\tb]$. 
Notons $R$ la cl\^oture int\'egrale de $P$ dans  $M$ et soit $G=\Gal(M/K)$. 
Alors $P=R^G$ car $P$ est int\'egralement clos. Si $r \in R$, on note $\rba$ son image 
dans $R/\rG$. \'Ecrivons
$$\fba=\prod_{j=1}^l f_j,$$
o\`u $f_j \in k_P(\pG)[\tb]$ est un polyn\^ome irr\'eductible. 
Alors $D/I=\Gal(k_R(\rG)/k_P(\pG))$ d'apr\`es le th\'eor\`eme~\ref{bourbaki}.  
Mais, $R$ contient $t_1$,\dots, $t_d$, donc
$$\fba(\tb)=\prod_{i=1}^d (\tb-\tba_i).$$
On note $\O_j$ la partie de $\{1,2,\dots,d\}$ telle que 
$$f_j(\tb)=\prod_{i \in \O_j} (\tb-\tba_i).$$
Soit $k_j=k_P(\pG)((\tba_i)_{i \in \O_j})$~: c'est un corps de d\'ecomposition de $f_j$ sur $k_P(\pG)$. 
Soit $G_j=\Gal(k_j/k_P(\pG))$, c'est-\`a-dire 
le groupe de Galois de $\fba_j$. Alors, 
\equat\label{eq:galois-surjectif}
\text{\it le morphisme canonique $D/I=\Gal(k_R(\rG)/k_P(\pG)) \to \Gal(k_j/k_P(\pG))=G_j$ est surjectif}
\endequat
pour tout $j$. Comme $G_j$ agit transitivement sur $\O_j$, on obtient en particulier que 
\equat\label{eq:omega-divise-G}
\text{\it $|\O_j|$ divise $|G|$ pour tout $j$.}
\endequat

\bigskip

\section{Quelques calculs de discriminant}

\medskip

Soient $P$ un anneau commutatif et soit $f(\tb) \in P[\tb]$ un polyn\^ome 
unitaire de degr\'e $d$. On notera $\disc(f)$  \indexnot{d}{\disc}  son discriminant. Alors
\equat\label{discriminant carre}
\disc(f(\tb^2))=(-4)^d \disc(f)^2 \cdot f(0).
\endequat
\begin{proof}
Par des arguments faciles de sp\'ecialisation, on peut supposer que $P$ 
est un corps alg\'ebriquement clos. Notons 
$E_1$,\dots, $E_d$ les \'el\'ements de $P$ tels que 
$$f(\tb)=\prod_{i=1}^d (\tb-E_i).$$
Fixons une racine carr\'e $e_i$ de $E_i$ dans $P$. Alors
$$f(\tb^2))=\prod_{1 \le i \le d} \prod_{\e \in \{1,-1\}} (\tb-\e e_i)$$
et le discriminant de $f(\tb^2)$ vaut donc
$$\disc(f(\tb^2)) = 
\Bigl(\prod_{1 \le i < j \le d} \prod_{\e,\e' \in \{1,-1\}} (\e e_i -\e' e_j)^2\Bigr)
\cdot \prod_{i=1}^d (e_i - (-e_i))^2.$$
En d'autres termes, 
$$\disc(f(\tb^2)) = 4^d \cdot 
\Bigl(\prod_{1 \le i < j \le d} (E_i-E_j)^4\Bigr) \cdot \prod_{i=1}^d E_i
= 4^d ~ \disc(f)^2 \cdot (-1)^d f(0),$$
comme annonc\'e.
\end{proof}

\bigskip

Terminons par un r\'esultat tout aussi facile~: 
\equat\label{discriminant t}
\disc(\tb f(\tb))=\disc(f) \cdot f(0)^2.
\endequat
\begin{proof}
Comme pr\'ec\'edemment, on peut supposer que $P$ est un corps alg\'ebriquement 
clos, et on note $E_1$,\dots, $E_d$ les \'el\'ements de $P$ tels que 
$$f(\tb)=\prod_{i=1}^n (\tb-E_i).$$
Alors
$$\disc(\tb f(\tb)) = \Bigl(\prod_{1 \le i < j \le d} (E_i-E_j)^2\Bigr)
\cdot \prod_{i=1}^d (0-E_i)^2.$$
D'o\`u le r\'esultat.
\end{proof}

\chapter{Graduation et extensions enti\`eres}\label{appendice graduation}\setcounter{section}{0}

\bigskip

\section{Idempotents, radical}\label{intro graduation}

\medskip

Soit $A=\mathop{\bigoplus}_{i \in \ZM} A_i$ un anneau $\ZM$-gradu\'e. 
Si $B$ est un anneau contenant $A$ et si $\xi \in B^\times$ commute avec $A$, 
alors il existe un unique morphisme d'anneaux 
$$\mu_A^\xi : A \longto B$$  \indexnot{mz}{\mu_A^\xi}  
tel que $\mu_A^\xi(a)=a\xi^i$ si $a \in A_i$. Notons que, si $A$ est $\NM$-gradu\'e 
(c'est-\`a-dire si $A_i=0$ pour $i < 0$), alors $\mu_A^\xi$ peut \^etre d\'efini 
aussi lorsque $\xi$ n'est pas inversible. 
Par exemple, si $\tb$ est une ind\'etermin\'ee sur $A$, alors 
$$\mu_A^\tb : A \longto A[\tb,\tb^{-1}]$$
est un morphisme d'anneaux. Si on note $\eval_A^\xi : A[\tb,\tb^{-1}] \to B$ 
le morphisme d'\'evaluation en $\xi$, alors
\equat\label{mua}
\mu_A^\xi = \eval_A^\xi \circ \mu_A^\tb.
\endequat
En particulier, si $B=A$ et $\xi=1$, alors 
\equat\label{mua ea}
\mu_A^1 = \Id_A\qquad\text{et}\qquad \eval_A^1 \circ \mu_A^\tb = \Id_A.
\endequat
D'autre part, le morphisme $\mu_A^\xi : A \longto B$ s'\'etend en un morphisme 
$\ZM[\tb,\tb^{-1}]$-lin\'eaire $\mub_A^\xi : A[\tb,\tb^{-1}] \longto B[\tb,\tb^{-1}]$ et
\equat\label{mua mua}
\mub_A^\xi \circ \mu_A^\tb = \mu_A^{\xi \tb}.
\endequat
Comme cas particuliers, on peut prendre $B=A[\ub,\ub^{-1}]$ et $\xi=\ub$, o\`u 
$\ub$ est une autre ind\'etermin\'ee, ou bien on peut prendre $B=A[\tb,\tb^{-1}]$ 
et $\xi = \tb^{-1}$. On obtient les deux \'egalit\'es suivantes~:
\equat\label{mua mub}
\mub_A^\ub \circ \mu_A^\tb = \mu_A^{\tb\ub}
\qquad\text{et}\qquad
\mub_A^{\tb^{-1}} \circ \mu_A^\tb (a) = a \in A[\tb,\tb^{-1}]
\endequat
pour tout $a \in A$. Pour finir, remarquons que
\equat\label{ev mutilde}
\eval_A^1 \circ \mub_A^{\tb^{-1}} = \eval_A^1.
\endequat

\bigskip

\begin{prop}\label{graduation idem}
Supposons $A$ commutatif. Soit $e$ un idempotent de $A$. Alors $e \in A_0$.
\end{prop}

\begin{proof}
Quitte \`a remplacer $A$ par l'anneau engendr\'e par les composantes homog\`enes de $e$, 
on peut supposer que $A$ est noeth\'erien. 
De plus, $\mu_A^{\tb^l}(e)$ est un idempotent de $A[\tb,\tb^{-1}]$ 
pour tout $l \in \ZM$. Puisque $A$ est commutatif et noeth\'erien, 
$A[\tb,\tb^{-1}]$ est aussi commutatif et noeth\'erien, et donc ne 
contient qu'un nombre fini d'idempotents. Par cons\'equent, il existe 
$m > l > 0$ tels que $\mu_A^{\tb^l}(e)=\mu_A^{\tb^m}(e)$, ce qui implique 
que $e \in A_0$.
\end{proof}

\bigskip

\begin{prop}\label{graduation radical}
$\Rad(A)$ est un id\'eal homog\`ene de $A$.
\end{prop}

\begin{proof}
Voir~\cite{puc}.
\end{proof}

\bigskip

\section{Extension de la graduation}\label{section:graduation integrale}

\medskip

\boitegrise{{\bf Notations.} 
{\it Nous fixons dans cette section un anneau commutatif 
$\ZM$-gradu\'e int\`egre $P$. Nous noterons $P_i$ sa composante homog\`ene de degr\'e $i$. 
Soit $Q$ un anneau {\bfit int\`egre} contenant $P$ et {\bfit entier} sur $P$.}}{0.75\textwidth}

\medskip

Le but de cette section est d'\'etudier les graduations de $Q$ qui \'etendent 
celles de $P$. Commen\c{c}ons par r\'egler le probl\`eme de l'unicit\'e~:

\bigskip

\begin{lem}\label{unicite graduation}
Si $Q=\mathop{\bigoplus}_{i \in \ZM} \Qti_i = \mathop{\bigoplus}_{i \in \ZM} \Qhat_i$ 
sont deux graduations de $Q$ \'etendant celle de $P$ (c'est-\`a-dire que 
$P_i=\Qti_i \cap P = \Qhat_i \cap P$ pour tout $i \in \ZM$), alors 
$\Qti_i=\Qhat_i$ pour tout $i \in \ZM$).
\end{lem}

\begin{proof}
Comme dans la section~\ref{intro graduation} (dont on reprend les notations), 
les graduations $Q=\mathop{\bigoplus}_{i \in \ZM} \Qti_i$ et 
$\mathop{\bigoplus}_{i \in \ZM} \Qhat_i$ correspondent \`a des morphismes d'anneaux 
$\mut_Q^\tb : Q \to Q[\tb,\tb^{-1}]$ et $\muh_Q^\tb : Q \to Q[\tb,\tb^{-1}]$ \'etendant 
$\mu_P^\tb : P \to P[\tb,\tb^{-1}]$. Notons 
$$\fonction{\a}{Q}{Q[\tb,\tb^{-1}]}{q}{\mubt_Q^{\tb^{-1}}(\mubh_Q^\tb(q)).}$$
Alors $\a$ est un morphisme d'anneaux et $\a(p)=p$ pour tout $p \in P$ d'apr\`es 
\ref{mua mub}. Par cons\'equent, si $q \in Q$, alors $\a(q) \in Q[\tb,\tb^{-1}]$ est 
entier sur $P$. En particulier, $\a(q)$ est entier sur $Q$. 
Puisque $Q$ est int\`egre, cela implique que $\a(q) \in Q$ (voir la proposition 
\ref{cloture polynomiale}). 

Mais $\eval_Q(\a(q)) = q$ d'apr\`es~(\ref{ev mutilde}), donc $\a(q)=q$, ce qui implique que 
$\mubt_Q^{\tb^{-1}} \circ \mut_Q^\tb = \mubt_Q^{\tb^{-1}} \circ \muh_Q^\tb$, 
toujours d'apr\`es~(\ref{ev mutilde}). Puisque $\mubt_Q^{\tb^{-1}}$ est 
injectif, on en d\'eduit que $\mut_Q^\tb=\muh_Q^\tb$.
\end{proof}

\bigskip

\begin{coro}\label{graduation et automorphisme}
Si $Q=\mathop{\bigoplus}_{i \in \ZM} Q_i$ est une graduation sur $Q$ \'etendant celle de $P$ 
et si $G$ est un groupe agissant sur $Q$, stabilisant $P$ et respectant la graduation de 
$P$, alors $G$ respecte la graduation de $Q$. 
\end{coro}

\begin{proof}
En effet, si $g \in G$, alors $Q=\mathop{\bigoplus}_{i \in \ZM} g(Q_i)$ 
est une graduation sur $Q$ \'etendant celle de $P$. En vertu du lemme~\ref{unicite graduation}, on a $g(Q_i)=Q_i$ pour tout $i$.
\end{proof}

\bigskip

\begin{contre}
On ne peut pas se passer de l'hypoth\`ese d'int\'egrit\'e de $Q$ dans le lemme 
\ref{unicite graduation}. En effet, si $P=P_0$ et si $Q=P \oplus P \e$ avec $\e^2=0$, 
alors on peut munir $Q$ d'une infinit\'e de graduations \'etendant celle de $P$ 
en d\'ecr\'etant que $\e$ est homog\`ene de degr\'e arbitraire.\finl
\end{contre}

\bigskip

\begin{prop}\label{prop:graduation-positive}
Si $Q=\bigoplus_{i \in \ZM} Q_i$ est une $\ZM$-graduation de $Q$ \'etendant celle de $P$ et 
si $P_i = 0$ pour tout $i < 0$, alors $Q_i=0$ pour tout $i < 0$.
\end{prop}

\bigskip

\begin{proof}
Soit $i < 0$ et $q \in Q_i$. Puisque $Q$ est entier sur $P$, il existe $r \ge 0$ et 
$p_0$, $p_1$,\dots, $p_r$ dans $P$ tels que 
$q^{r+1} = p_0 + p_1 q + \cdots + p_r q^r$. Alors 
$q^{r+1} \in Q_{(r+1)i}$ tandis que 
$p_0 + p_1 q + \cdots + p_r q^r \in \bigoplus_{j \ge ri} Q_j$. Donc $q^{r+1}=0$ et donc $q=0$.
\end{proof}

\bigskip

Nous allons maintenant nous int\'eresser \`a la question de l'existence. 
Pour cela, notons $K=\Frac(P)$, $L=\Frac(Q)$ et supposons que l'extension 
$L/K$ est de degr\'e fini. 

\bigskip

\begin{lem}\label{g}
La graduation de $P$ s'\'etend en une graduation de sa cl\^oture int\'egrale dans $K$. 
En outre, si $P$ est en fait $\NM$-gradu\'e (c'est-\`a-dire $P_i=0$ si $i < 0$), 
alors sa cl\^oture int\'egrale dans $K$ l'est aussi.
\end{lem}

\begin{proof}
Voir~\cite[chapitre 5, \S 1, proposition 21]{bourbaki}. 
\end{proof}

\bigskip

Nous allons maintenant montrer que la question de l'existence se transmet 
\`a la cl\^oture normale. Notons $M$ la cl\^oture normale de l'extension 
$L/K$ et $R$ la cl\^oture int\'egrale de $P$ dans $M$. 
Le lemme suivant est facile~:

\medskip

\begin{lem}\label{decomposition gradue}
Graduons $P[\xb]$ en attribuant \`a $\xb$ le degr\'e $d \in \ZM$. 
Soit $F \in P[\xb]$ un polyn\^ome unitaire et {\bfit homog\`ene} pour cette graduation. 
Si $F=F_1\cdots F_r$, avec $F_i \in P[\xb]$ unitaire, alors $F_i$ est homog\`ene 
pour tout $i$.
\end{lem} 

\begin{proof}
Le morphisme d'anneaux $\mu_P^\tb : P \to P[\tb,\tb^{-1}]$ associ\'e \`a la graduation sur $P$ 
(voir la section~\ref{intro graduation}) s'\'etend en un morphisme d'anneaux 
$\mu_{P[\xb]}^\tb : P[\xb] \to P[\xb,\tb,\tb^{-1}]$ envoyant $\xb$ sur $\xb \tb^d$. 
Notons $l$ le degr\'e {\it total} de $F$. Alors 
$$\mu_{P[\xb]}^\tb(F)=F(X) \tb^l = \mu_{P[\xb]}^\tb(F_1)\cdots \mu_{P[\xb]}(F_r).$$
L'anneau $P[\xb]$ \'etant int\`egre, de corps des fractions $K(\xb)$, la factorialit\'e 
de $K(\xb)[\tb,\tb^{-1}]$ implique qu'il existe $\l_1$,\dots, $\l_r \in K(\xb)$ et 
$d_1$,\dots, $d_r \in \ZM$ tels que 
$$\mu_{P[\xb]}^\tb(F_i)=\l_i \tb^{d_i}$$
pour tout $i$. Cela force $F_i$ \`a \^etre homog\`ene de degr\'e $d_i$, 
et $F_i=\l_i$.
\end{proof}

\begin{coro}\label{dec hom}
Graduons $P[\xb]$ en attribuant \`a $\xb$ le degr\'e $d \in \ZM$. 
Soit $F \in P[\xb]$ un polyn\^ome unitaire et {\bfit homog\`ene} pour cette graduation. 
On suppose que $M$ est le corps de d\'ecomposition de $F$ sur $K$. 
Alors $R$ admet une graduation \'etendant celle de $P$.

En outre, si $P$ est $\NM$-gradu\'e et $d \ge 0$, alors $R$ est $\NM$-gradu\'e.
\end{coro}

\begin{proof}
D'apr\`es le lemme~\ref{g}, on peut supposer que $P$ est int\'egralement clos. 
Notons $\d$ le degr\'e de $F$ {\it en la variable $\xb$}. Nous allons montrer 
le r\'esultat par r\'ecurrence sur $\d$, le cas o\`u $\d=1$ \'etant trivial 
(car alors $P=R$).

Supposons donc $\d \ge 2$ et soit $F_1$ un polyn\^ome irr\'eductible unitaire 
de $K[\xb]$ divisant $F$. D'apr\`es la proposition~\ref{int clos}, 
$F_1 \in P[\xb]$. Posons $K'=K[\xb]/< F_1 >$ et notons $x$ l'image de $\xb$ dans $K'$. 
Alors $K'$ est un corps commutatif et il contient l'anneau $P'=P[\xb]/<F_1>$. 
En fait, $K'$ est le corps des fractions de $P$. Puisque $F_1$ est homog\`ene, 
$P'$ est gradu\'e (avec $x$ homog\`ene de degr\'e $d$). D'apr\`es le lemme 
\ref{g}, la cl\^oture int\'egrale $P''$ de $P'$ dans $K'$ 
h\'erite d'une graduation. D'autre part, $K' \subset M$ et $M$ 
est le corps de d\'ecomposition de $F$ sur $K'$. 
Dans $P''[\xb]$, on a 
$$F(\xb)=(\xb-x) F_0(\xb),$$
avec $F_0(\xb) \in P''[\xb]$ homog\`ene, et de degr\'e {\it en la variable $\xb$} 
\'egal \`a $\d-1$. Puisque le corps de d\'ecomposition de $F$ sur $K$ est \'egal 
au corps de d\'ecomposition de $F_0$ sur $K'$, le r\'esultat se 
d\'eduit de l'hypoth\`ese de r\'ecurrence.

L'\'enonc\'e concernant la $\NM$-graduation d\'ecoule de la preuve ci-dessus et 
de l'unicit\'e de l'extension de la graduation (voir le lemme~\ref{unicite graduation}).
\end{proof}

\bigskip

\begin{prop}\label{R gradue}
Supposons $P$ et $Q$ int\'egralement clos. 
Si la graduation de $P$ s'\'etend en une graduation sur $Q$, alors elle 
s'\'etend aussi en une graduation de $R$. 

En outre, si $Q$ est $\NM$-gradu\'e, alors $R$ l'est aussi. 
\end{prop}

\begin{proof}
Soient $q_1$,\dots, $q_r$ des \'el\'ements de $Q$, homog\`enes de degr\'es respectifs 
$d_1$,\dots, $d_r$ et tels que $L=K[q_1,\dots,q_r]$. On note $F_i \in K[\tb]$ 
le polyn\^ome minimal de $q_i$~: en fait, $F_i \in P[\tb]$ en vertu du corollaire 
\ref{minimal clos}. Alors $M$ est le corps de d\'ecomposition de $F_1 \cdots F_r$. 
Quitte \`a raisonner par r\'ecurrence, on peut alors supposer que $r=1$~: 
on \'ecrira alors $q=q_1$, $d=d_1$ et $F=F_1$. 

Si on attribue \`a l'ind\'etermin\'ee $\tb$ le degr\'e $d$, alors on v\'erifie 
facilement que $F$ est un polyn\^ome homog\`ene (pour le degr\'e total sur $P[\tb]$). 
L'existence de l'extension de la graduation d\'ecoule alors du corollaire~\ref{dec hom}.

Le r\'esultat sur la $\NM$-graduation se d\'emontre de m\^eme.
\end{proof}

\bigskip

\begin{lem}\label{lem:homogeneise-premier}
Soit $\pG$ un id\'eal premier de $P$ et notons $\pGt$ l'id\'eal homog\`ene maximal  
de $P$ contenu dans $\pG$ (c'est-\`a-dire $\pGt=\bigoplus_{i \in \ZM} \pG \cap P_i$). 
Alors $\pGt$ est premier.
\end{lem}

\begin{proof}
En effet, $(P/\pG)[\tb,\tb^{-1}]$ est int\`egre et $\pGt$ est le noyau du morphisme 
compos\'e $P \longto P[\tb,\tb^{-1}] \surto (P/\pG)[\tb,\tb^{-1}]$ (ici, la premi\`ere application est $\mu_P^\tb$).
\end{proof}

\bigskip

\begin{lem}\label{premier homogene}
Soit $\qG$ un id\'eal premier de $Q$ et soit $\pG=\qG \cap P$. Supposons que 
la graduation de $P$ s'\'etende \`a $Q$. Alors 
$\pG$ est homog\`ene si et seulement si $\qG$ l'est.
\end{lem}

\begin{proof}
Si $\qG$ est homog\`ene, alors $\pG$ l'est bien \'evidemment. R\'eciproquement, 
supposons $\pG$ homog\`ene. On pose 
$\qG'=\mathop{\bigoplus}_{i \in \ZM} (\qG \cap Q_i)$. Alors $\qG'$ est un id\'eal homog\`ene de $Q$ 
contenu dans $\qG$ et $\qG' \cap P = \pG = \qG \cap P$. D'apr\`es le lemme~\ref{lem:homogeneise-premier}, 
$\qG'$ est un id\'eal premier, donc $\qG'=\qG$ car $Q$ est entier sur $P$.
\end{proof}

\bigskip

Terminons par quelques r\'esultats sur les homog\'en\'eis\'es des id\'eaux premiers de $P$ ou $Q$~:

\bigskip

\begin{coro}\label{coro:homogeneise-premier}
Supposons que la graduation de $P$ s'\'etende \`a $Q$. Soient $\pG$ un id\'eal premier de 
$P$ et $\qG$ un id\'eal premier de $Q$ tel que $\qG \cap P = \pG$. Notons $\pGt$ (respectivement 
$\qGt$) l'id\'eal homog\`ene maximal de $P$ (respectivement $Q$) contenu dans $\pG$ (respectivement $\qG$). 
Alors $\pGt=\qGt \cap P$.
\end{coro}

\begin{proof}
Cela d\'ecoule de la preuve pr\'ec\'edente et du fait que le diagramme
$$\diagram
P \rrto \ddto|<\ahook&& (P/\pG)[\tb,\tb^{-1}]\ddto|<\ahook \\
&&\\
Q \rrto && Q/\qG[\tb,\tb^{-1}]
\enddiagram$$
est commutatif.
\end{proof}

\begin{coro}\label{coro:homogeneise-inertie}
Supposons que la graduation de $P$ s'\'etende \`a $Q$ et qu'il existe un groupe fini $G$ agissant sur $Q$ 
et tel que $P=Q^G$. Soit $\qG$ un id\'eal premier de $Q$ et notons 
$\qGt$ l'id\'eal homog\`ene maximal de $Q$ contenu dans $\qG$. 
Notons $D_\qG$ (respectivement $D_\qGt$) le groupe de d\'ecomposition de $\qG$ (respectivement $\qGt$) 
dans $G$ et $I_\qG$ (respectivement $I_\qGt$) le groupe d'inertie de $\qG$ (respectivement $\qGt$) 
dans $G$. Alors 
$$D_\qG \subset D_\qGt\qquad\text{et}\qquad I_\qG = I_\qGt.$$
\end{coro}

\begin{proof}
La premi\`ere inclusion est imm\'ediate, car $G$ respecte la graduation (voir le corollaire~\ref{graduation et automorphisme}). 
D'autre part, $Q/\qG$ est un quotient de $Q/\qGt$, donc $I_\qGt \subset I_\qG$. 
R\'eciproquement, si $g \in I_\qG \subset D_\qG \subset D_\qGt$ et si $q \in \qG \cap Q_i$, 
alors $g(q)-q \in \qG \cap Q_i \subset \qGt$. Donc $g \in I_\qGt$.
\end{proof}

\bigskip

\section{Graduation et groupes de r\'eflexions}\label{section:GR}

\medskip

\boitegrise{\noindent{\bf Notation.} 
{\it Dans cette section, nous fixons un corps commutatif $k$ de 
caract\'eristique nulle et une $k$-alg\`ebre commutative 
$\NM$-gradu\'ee {\bfit int\`egre} $R=\bigoplus_{i \in \NM} R_i$. Nous supposons de plus 
que $R_0=k$ et que $R$ est de type fini. Nous fixons aussi un groupe fini 
$G$ agissant {\bfit fid\`element} sur $R$ par automorphismes de $k$-alg\`ebre gradu\'ee et 
nous notons $P=R^G$. Notons $R_+=\bigoplus_{i > 0} R_i$~: c'est l'unique 
id\'eal maximal gradu\'e de $R$. Nous fixons un sous-espace 
vectoriel gradu\'e et $G$-stable $E^*$ de $R$ tel que $R_+=R_+^2 \oplus E^*$ 
(un tel sous-espace existe car $kG$ est semi-simple) et nous notons $E$ 
le $k$-dual de $E^*$.}}{0.75\textwidth}

\bigskip

Le groupe $G$ agit sur l'espace vectoriel $E$ et le but de cette section est 
de donner quelques crit\`eres permettant de d\'eterminer si $G$ est 
un sous-groupe de $\GL_k(E)$ engendr\'e par des r\'eflexions. 
Nos r\'esultats s'inspirent de~\cite{BBR}. 

Tout d'abord, le graduation sur $E^*$ induit une graduation sur $E$ et une 
graduation sur $k[E]$, l'alg\`ebre des fonctions polynomiales sur 
$E$ (ou encore l'alg\`ebre sym\'etrique de $E^*$). De m\^eme, 
$k[E]$ h\'erite d'une action de $G$, qui pr\'eserve la graduation. Nous noterons 
$k[E]_+$ l'unique id\'eal maximal gradu\'e de $k[E]$. L'inclusion $E^* \injto R$ 
induit un morphisme $G$-\'equivariant de $k$-alg\`ebres gradu\'es 
$$\pi : k[E] \longto R$$
dont il est facile de v\'erifier qu'il est surjectif et que 
\equat\label{nombre generateurs R}
\text{\it le nombre minimal de g\'en\'erateurs de la $k$-alg\`ebre $R$ est $\dim_k E$}
\endequat
(voir par exemple~\cite[lemme 2.1]{BBR}). Nous noterons 
$$I=\Ker \pi,$$
de sorte que
\equat\label{E/I}
R \simeq k[E]/I.
\endequat
En particulier, $G$ agit fid\`element sur $E$. 
Puisque $I$ est homog\`ene, il d\'ecoule imm\'ediatement du lemme de Nakayama 
gradu\'e que 
\equat\label{nombre generateurs}
\text{\it le nombre minimal de g\'en\'erateurs de l'id\'eal $I$ est $\dim_k I/k[E]_+ I$}
\endequat
D'autre part, il est tout aussi facile de v\'erifier que 
\equat\label{G trivial}
\text{\it $I=k[E] I^G$ si et seulement si $G$ agit trivialement sur $I/k[E]_+ I$.}
\endequat
(voir par exemple~\cite[lemme 3.1]{BBR}). 
Pour finir, puisque $kG$ est semi-simple, on a
\equat\label{PRG}
P \simeq k[E]^G/I^G.
\endequat
Nous aurons aussi besoin du lemme suivant~:

\bigskip

\begin{lem}\label{liberte G}
Si $R$ est un $P$-module libre, alors le rang du $P$-module $R$ est $|G|$.
\end{lem}

\begin{proof}
Notons $d$ le $P$-rang de $R$. Puisque $R$ est int\`egre, 
$P$ l'est aussi et, si on pose $K=\Frac(P)$ et $M=\Frac(R)$, 
alors $K=L^G$ (et donc $[L:K]=|G|$) et $L = K \otimes_P R$ (et donc $[L:K]=d$). 
Par cons\'equent, $d=|G|$. 
\end{proof}

\bigskip

Le r\'esultat principal de cette section est le suivant (comparer 
avec~\cite[th\'eor\`eme 3.2]{BBR}, dont nous reprenons presque mot pour 
mot la preuve)~:

\bigskip

\begin{prop}\label{intersection complete R}
On suppose que $P$ est r\'eguli\`ere et que $R$ est un $P$-module libre. Alors 
les assertions suivantes sont \'equivalentes~:
\begin{itemize}
\itemth{1} $R$ est d'intersection compl\`ete et $G$ agit trivialement sur $I/k[E]_+I$. 

\itemth{2} $G$ est un sous-groupe de $\GL_k(E)$ engendr\'e par des r\'eflexions.
\end{itemize}
\end{prop}

\begin{rema}\label{liberte platitude}
Si $P$ est r\'eguli\`ere et puisque l'on travaille avec des objets 
gradu\'es, les assertions suivantes sont \'equivalentes~: 
\begin{itemize}
\item[$\bullet$] $R$ est un $P$-module libre.

\item[$\bullet$] $R$ est un $P$-module plat.

\item[$\bullet$] $R$ est de Cohen-Macaulay.
\end{itemize}
D'autre part, si $R$ est d'intersection compl\`ete, alors $R$ est de Cohen-Macaulay.\finl
\end{rema}

\begin{proof}
Posons $e=\dim_k E$, $i=\dim_k I/k[E]_+ I$ et notons $d$ la dimension de Krull 
de $R$ (qui est aussi celle de $P$). De plus, $e$ est la dimension de Krull 
de $k[E]$ et de $k[E]^G$. 

\medskip

Montrons tout d'abord que (1) $\Rightarrow$ (2). Supposons donc que 
$R$ est d'intersection compl\`ete et que $G$ agit trivialement sur $I/k[E]_+I$. 
Puisque $R$ est d'intersection compl\`ete et d'apr\`es~(\ref{E/I}) et~(\ref{nombre generateurs}), on a
$$d=e-i.$$
De plus, puique $G$ agit trivialement sur $I/k[E]_+I$, l'id\'eal $I$ de $k[E]$ 
peut \^etre engendr\'e par $i$ \'el\'ements $G$-invariants homog\`enes 
$f_1$,\dots, $f_i$ et donc l'id\'eal $I^G$ de $k[E]^G$ est engendr\'e par 
$f_1$,\dots, $f_i$. Puisque $P$ est r\'eguli\`ere de dimension de Krull $d$, 
$P=k[E]^G/I^G$ peut-\^etre engendr\'ee par $d$ \'el\'ements $\pi(g_1)$,\dots, $\pi(g_d)$ 
o\`u $g_j \in k[E]^G$ est homog\`ene. Par cons\'equent, la $k$-alg\`ebre 
$k[E]^G$ est engendr\'ee par $f_1$,\dots, $f_i$, $g_1$,\dots, $g_d$, 
c'est-\`a dire qu'elle est engendr\'ee par $i+d=e$ \'el\'ements. 
Puisque la dimension de Krull de $k[E]^G$ est aussi \'egale \`a $e$, 
cela montre que $k[E]^G$ est une alg\`ebre de polyn\^omes, donc 
que $G$ est un sous-groupe de $\GL_k(E)$ engendr\'e par des r\'eflexions 
d'apr\`es le th\'eor\`eme~\ref{chevalley}.

\medskip

R\'eciproquement, montrons maintenant que (2) $\Rightarrow$ (1). 
Supposons donc que $G$ est un sous-groupe de $\GL_k(E)$ engendr\'e 
par des r\'eflexions. Alors $k[E]$ est un $k[E]^G$-module libre de 
rang $|G|$ (d'apr\`es le th\'eor\`eme~\ref{chevalley}) et donc 
$(k[E]^G/I^G) \otimes_{k[E]^G} k[E]$ est un $P$-module libre de 
rang $|G|$ (voir~(\ref{PRG})). D'autre part, $k[E]/I=R$ est un $P$-module libre de rang $|G|$ d'apr\`es 
le lemme~\ref{liberte G}. Donc la surjection canonique 
$(k[E]^G/I^G) \otimes_{k[E]^G} k[E] \surto k[E]/I$ (entre deux $P$-modules 
de m\^eme rang) est un isomorphisme, ce qui signifie que $I$ est engendr\'e 
par $I^G$ et donc que $G$ agit trivialement sur $I/k[E]_+ I$ (d'apr\`es~(\ref{G trivial})). 

D'autre part, puisque $k[E]^G$ et $k[E]^G/I^G=P$ sont toutes deux des alg\`ebres de 
polyn\^omes (d'apr\`es le th\'eor\`eme~\ref{chevalley} pour 
$k[E]^G$), $P$ est d'intersection compl\`ete et donc $I^G$ peut \^etre 
engendr\'e par $e-d$ \'el\'ements. On d\'ecuit de~(\ref{nombre generateurs}) et~(\ref{G trivial}) 
que $i \le e-d$ et donc forc\'ement $i=e-d$ et $R$ est d'intersection compl\`ete.
\end{proof}

\bigskip

\chapter{Blocs, matrices de d\'ecomposition}
\label{appendice: blocs}\setcounter{section}{0}

\bigskip

\boitegrise{{\bf Hypoth\`eses et notations.} 
{\it Nous fixons dans cet appendice un anneau commutatif $R$ que nous supposerons 
{\bfit noeth\'erien}, {\bfit int\`egre} et {\bfit int\'egralement clos}. 
Nous fixons un id\'eal premier $\rG$. Nous fixons aussi une $R$-alg\`ebre $\HC$ que nous 
supposons \^etre, comme $R$-module, libre et de type fini. 
Nous noterons $\Zrm(\HC)$ le centre de $\HC$,  $k=\Frac(R/\rG)=k_R(\rG)$ et $K=\Frac(R)=k_R(0)$. Pour finir, 
l'image d'un \'el\'ement $h \in \HC$ dans $k\HC$ sera not\'ee $\hba$.}}{0.75\textwidth}

\bigskip

%

\section{Blocs de $k\HC$}\label{section:definition blocs}

\medskip

Si $A$ est un anneau (non n\'ecessairement commutatif), nous noterons 
$\blocs(A)$  \indexnot{I}{\blocs(A)}  l'ensemble de ses idempotents primitifs. Par exemple, 
$\blocs(\Zrm(\HC))$ est l'ensemble des idempotents primitifs centraux de $\HC$. 
Puisque $\Zrm(\HC)$ est noeth\'erien, 
\equat\label{1}
1=\sum_{e \in \blocs(\Zrm(\HC))} e.
\endequat
D'autre part, le morphisme $\HC \to k\HC$ induit un morphisme $\pi_\Zrm : k\Zrm(\HC) \to \Zrm(k\HC)$ 
(qui peut n'\^etre ni surjectif ni injectif). Cependant, le r\'esultat suivant \`a 
\'et\'e d\'emontr\'e par M\"uller~\cite[th\'eor\`eme 3.7]{muller}~:

\bigskip

\begin{prop}[M\"uller]\label{muller}
\begin{itemize}
\itemth{a} Si $e \in \blocs(k\Zrm(\HC))$, alors $\pi_\Zrm(e) \in \blocs(\Zrm(k\HC))$. 

\itemth{b} L'application $\blocs(k\Zrm(\HC)) \to \blocs(\Zrm(k\HC))$, 
$e \mapsto \pi_\Zrm(e)$ est bijective.
\end{itemize}
\end{prop}
%

\bigskip

Dans une $k$-alg\`ebre commutative de dimension finie $\AC$ (par exemple $\Zrm(k\HC)$ ou $k\Zrm(\HC)$), 
les id\'eaux premiers sont maximaux et sont en bijection avec l'ensemble des idempotents 
primitifs de $\AC$~: si $\mG \in \Spec \AC$ et $e \in \blocs(\AC)$, alors $e$ et $\mG$ 
se correspondent par cette bijection si et seulement si $e \not\in \mG$ 
(ou encore si et seulement si $\mG = \Rad(\AC e) + (1-e)\AC$). 
La proposition~\ref{muller} montre donc que $\Spec k\Zrm(\HC)$ est en bijection avec 
$\blocs(\Zrm(k\HC))$, c'est-\`a-dire avec l'ensemble des idempotents primitifs centraux de $k\HC$. 

Par ailleurs, le morphisme naturel (et injectif) $R \injto \Zrm(\HC)$ induit un morphisme 
$\Upsilon : \Spec \Zrm(\HC) \to \Spec R$. L'application $\Zrm(\HC) \to k\Zrm(\HC)$ induit 
une bijection entre les ensembles $\Spec k\Zrm(\HC)$ et $\Upsilon^{-1}(\rG)$. Rappelons que
$$\Upsilon^{-1}(\rG) = \{\zG \in \Spec \Zrm(\HC)~|~\zG \cap R = \rG\}.$$
Au final, on obtient une bijection
\equat\label{xi}
\Xi_\rG : \blocs(\Zrm(k\HC)) \stackrel{\sim}{\longto} \Upsilon^{-1}(\rG)
\endequat
caract\'eris\'ee par la propri\'et\'e suivante~: 

\bigskip

\begin{lem}\label{lem:blocs}
Si $e \in \blocs(\Zrm(k\HC))$ et si $\zG \in \Upsilon^{-1}(\rG)$, alors les propri\'et\'es suivantes 
sont \'equivalentes~:
\begin{itemize}
\itemth{1} $\zG=\Xi_\rG(e)$.

\itemth{2} $e \not\in \pi_\Zrm(k \zG)$.

\itemth{3} $\zG$ est l'image inverse, dans $\Zrm(\HC)$, de $\pi_\Zrm^{-1}(\Rad(\Zrm(k\HC) e) + (1-e) \Zrm(k\HC))$.
\end{itemize}
\end{lem}

\bigskip

Maintenant, par localisation en $\rG$, $\Upsilon^{-1}(\rG)$ est 
en bijection avec $\Upsilon_\rG^{-1}(\rG R_\rG)$, o\`u 
$\Upsilon_\rG : \Spec R_\rG \Zrm(\HC) \to \Spec R_\rG$ est 
l'application induite par l'inclusion $R_\rG \injto R_\rG \Zrm(\HC)$. 
Les bijections, r\'eciproques l'une de l'autre, entre 
$\Upsilon^{-1}(\rG)$ et $\Upsilon_\rG^{-1}(\rG R_\rG)$ 
sont donn\'ees par
$$\fonctio{\Upsilon^{-1}(\rG)}{\Upsilon_\rG^{-1}(\rG R_\rG)}{\zG}{R_\rG \zG}$$
$$\fonctio{\Upsilon_\rG^{-1}(\rG R_\rG)}{\Upsilon^{-1}(\rG)}{\zG}{\zG \cap \Zrm(\HC).}\leqno{\text{et}}$$
Le centre de l'alg\`ebre $R_\rG \HC$ est \'egal \`a $R_\rG \Zrm(\HC)$ et le morphisme 
canonique $\HC \to k\HC$ s'\'etend en un morphisme $R_\rG \HC \to k\HC$, 
que nous noterons encore $h \mapsto \hba$. 
Pour finir, nous noterons $R_\rG \Zrm(\HC) \to k\Zrm(\HC)$, $z \mapsto \zhat$,  
le morphisme canonique (de sorte que $\zba = \pi_\Zrm(\zhat)$ si $z \in R_\rG \Zrm(\HC)$).

Pour r\'esumer, on obtient un diagramme naturel de bijections
\equat\label{diagramme bijections}
\diagram
\Upsilon^{-1}(\rG) \ar@{<->}[rr]^{\DS{\sim}} 
&& 
\Upsilon_\rG^{-1}(\rG R_\rG) \ar@{<->}[rr]^{\DS{\sim}} 
&& 
\Spec k\Zrm(\HC) \ar@{<->}[rr]^{\DS{\sim}}
 \ar@{<->}[dd]^{\DS{\reflectbox{\rotatebox[origin=c]{-90}{$\backsim$}}}} 
&& 
\Spec \Zrm(k\HC) \ar@{<->}[dd]^{\DS{\reflectbox{\rotatebox[origin=c]{-90}{$\backsim$}}}} \\
&&&&&&\\
&&
&& \blocs(k\Zrm(\HC)) \ar@{<->}[rr]^{\DS{\sim}}
&&
\blocs(\Zrm(k\HC)). \\
\enddiagram
\endequat
%

\section{Blocs de $R_\rG\HC$}

\medskip

%

\boitegrise{{\bf Hypoth\`ese.} 
{\it Dor\'enavant, et ce jusqu'\`a la fin de ce chapitre, nous supposerons 
que la $K$-alg\`ebre $K\HC$ est {\bfit d\'eploy\'ee}.}}{0.75\textwidth}

\bigskip

La question du rel\`evement des idempotents lorsque l'anneau local $R_\rG$ est complet pour la topologie 
$\rG$-adique est classique. Nous proposerons ici une autre version, valable lorsque la $K$-alg\`ebre 
$K\HC$ est d\'eploy\'ee (seule la normalit\'e de $R$ intervient~: aucune hypoth\`ese sur 
la dimension de Krull de $R$ ou sa compl\'etude n'est n\'ecessaire).

\bigskip

\subsection{Caract\`eres centraux}\label{section centrale}
Si $V$ est un $K\HC$-module simple, et si $z \in K\Zrm(\HC)$, alors $z$ agit sur 
$V$ par multiplication par un \'el\'ement $\o_V(z) \in K$ (car $K\HC$ \'etant 
d\'eploy\'ee, on a $\End_{K\HC}(V)=K$). 
Cela d\'efinit un morphisme de $K$-alg\`ebres
$$\o_V : K\Zrm(\HC) \longto K$$  \indexnot{oz}{\o_V, \o_V^\rG}  
dont la restriction \`a $\Zrm(\HC)$ est \`a valeurs dans $R$ (car $\Zrm(\HC)$ est entier sur $R$ et 
$R$ est int\'egralement clos). 
Ainsi, cela d\'efinit un morphisme de $R$-alg\`ebres
$$\o_V : \Zrm(\HC) \longto R.$$
Par composition avec la projection canonique $R \to R/\rG$, on obtient 
un morphisme de $R$-alg\`ebres
$$\o_V^\rG : \Zrm(\HC) \longto R/\rG.$$
Puisque $\o_V(1)=1$ et que $R/\rG$ est int\`egre, $\Ker \o_V$ est un id\'eal premier 
de $\Zrm(\HC)$ tel que $\Ker \o_V \cap R = \rG$. Donc
\equat\label{ker}
\Ker \o_V^\rG \in \Upsilon^{-1}(\rG).
\endequat
Cela nous d\'efinit une application 
$$\fonction{\KER_\rG}{\Irr(K\HC)}{\Upsilon^{-1}(\rG)}{V}{\Ker \o_V^\rG}.$$

\bigskip

\begin{defi}\label{defi:r-bloc}
Les fibres de l'application $\KER_\rG$ sont appel\'ees les {\bfit $\rG$-blocs} de $\HC$.
\end{defi}

\bigskip

Les $\rG$-blocs de $\HC$ sont donc des sous-ensembles de l'ensemble $\Irr(K\HC)$, dont ils 
forment une partition. Notons que, puisque $\Zrm(\HC) = R + \Ker(\o_V^\rG)$, 
le caract\`ere central $\o_V^\rG$ est d\'etermin\'e par son noyau. Ainsi, deux 
$K\HC$-modules simples $V$ appartiennent au m\^eme $\rG$-bloc si et seulement si 
$\o_V^\rG=\o_{V'}^\rG$.

\bigskip

\subsection{Rel\`evement des idempotents} 
Le r\'esultat principal de cette section est le suivant~:

\bigskip

\begin{prop}\label{relevement idempotent}
On a~:
\begin{itemize}
\itemth{a} Si $e \in \blocs(R_\rG \Zrm(\HC))$, alors $\ehat \in \blocs(k\Zrm(\HC))$.

\itemth{b} L'application $\blocs(R_\rG \Zrm(\HC)) \to \blocs(k\Zrm(\HC))$, $e \mapsto \ehat$ 
est bijective.
\end{itemize}
\end{prop}

\begin{proof}
Soit 
$$\fonction{\O}{R_\rG\Zrm(\HC)}{\prod_{V \in \Irr(K\HC)} R_\rG}{z}{(\o_V(z))_{V \in \Irr(K\HC)}.}$$
Alors $\O$ est un morphisme de $R_\rG$-alg\`ebres, de noyau $I$ \'egal \`a 
$R_\rG\Zrm(\HC) \cap \Rad(K\HC)$ et dont on notera $A$ l'image. 

Par cons\'equent, $I$ est nilpotent et donc $\O$ induit une bijection 
$\blocs(R_\rG \Zrm(\HC)) \longbij \blocs(A)$. 
De plus, d'apr\`es le corollaire~\ref{coro:relevement-idempotent}, 
la r\'eduction modulo $\rG$ induit une bijection $\blocs(A) \longbij \blocs(kA)$. 
Il nous reste donc \`a montrer que le noyau de l'application naturelle 
$k\Zrm(\HC) \surto kA$ est nilpotent, ce qui est \'evident car c'est l'image de 
$I$ dans $k\Zrm(\HC)$. 
\end{proof}

\bigskip

\begin{coro}\label{coro:r-blocs}
L'application $\KER_\rG : \Irr(K\HC) \to \Upsilon^{-1}(\rG)$ est surjective. 
Ses fibres sont de la forme $\Irr(K\HC e)$, o\`u $e \in \blocs(R_\rG \Zrm(\HC))$. 
\end{coro}

\begin{proof}
La premi\`ere assertion d\'ecoule de~\ref{L} et la deuxi\`eme de la preuve de 
la proposition~\ref{relevement idempotent}.
\end{proof}

\bigskip

En combinant les propositions~\ref{muller} et~\ref{relevement idempotent}, 
on obtient le corollaire suivant~:

\bigskip

\begin{coro}\label{muller plus}
On a~: 
\begin{itemize}
\itemth{a} Si $e \in \blocs(R_\rG \Zrm(\HC))$, alors $\eba \in \blocs(\Zrm(k\HC))$.

\itemth{b} L'application $\blocs(R_\rG \Zrm(\HC)) \to \blocs(\Zrm(k\HC))$, $e \mapsto \eba$ 
est bijective.
\end{itemize}
\end{coro}

\bigskip

Ainsi, on obtient une bijection 
\equat\label{bijection finale}
\Upsilon^{-1}(\rG) \stackrel{\sim}{\longleftrightarrow} \blocs(R_\rG \Zrm(\HC)) 
\endequat
Si $\zG \in \Upsilon_\rG^{-1}(\rG R_\rG)$ et si $e \in \blocs(R_\rG \Zrm(\HC))$, alors
\equat\label{bij e r}
\text{\it $e$ et $\zG$ sont associ\'es par cette bijection si et seulement si $e \not\in R_\rG \zG$.}
\endequat

Pour r\'esumer, on obtient un diagramme naturel de bijections
\equat\label{diagramme bijections deploye}
\diagram
\Upsilon^{-1}(\rG) \ar@{<->}[r]^{\DS{\sim}} 
\ar@{<-->}[ddr]^{\DS{\reflectbox{\rotatebox[origin=c]{40}{$\backsim$}}}}&
\Upsilon_\rG^{-1}(\rG R_\rG) \ar@{<->}[rr]^{\DS{\sim}} 
\ar@{<-->}[dd]^{\DS{\reflectbox{\rotatebox[origin=c]{-90}{$\backsim$}}}}&& 
\Spec k\Zrm(\HC) \ar@{<->}[rr]^{\DS{\sim}}
 \ar@{<->}[dd]^{\DS{\reflectbox{\rotatebox[origin=c]{-90}{$\backsim$}}}} && 
\Spec \Zrm(k\HC) \ar@{<->}[dd]^{\DS{\reflectbox{\rotatebox[origin=c]{-90}{$\backsim$}}}} \\
&&&&&\\
&\blocs(R_\rG \Zrm(\HC)) \ar@{<-->}[rr]^{\DS{\sim}}&& \blocs(k\Zrm(\HC)) \ar@{<->}[rr]^{\DS{\sim}}&&
\blocs(\Zrm(k\HC)), \\
\enddiagram
\endequat
les bijections en pointill\'e n'existant \`a coup s\^ur que parce que la $K$-alg\`ebre 
$K\HC$ est d\'eploy\'ee.

Nous noterons
$$\fonctio{\Upsilon^{-1}(\rG)}{\blocs(R_\rG\Zrm(\HC))}{\zG}{e_\zG}$$
la bijection du diagramme~\ref{diagramme bijections deploye}. On obtient une partition de 
$\Irr(K\HC)$ gr\^ace \`a l'action des idempotents centraux $e_\zG$~:
\equat\label{partition bloc}
\Irr(K\HC) = \coprod_{\zG \in \Upsilon^{-1}(\rG)} \Irr(K\HC e_\zG).
\endequat
Les sous-ensembles $\Irr(K\HC e_\zG)$ sont les $\rG$-blocs  
de $\HC$.

\bigskip

\begin{exemple}
Lorsque $\rG$ est l'id\'eal nul, alors $R_\rG=k=K$, 
$\Upsilon^{-1}(\rG) \simeq \Spec K\Zrm(\HC)$, 
$\blocs(R_\rG \HC)=\blocs(K\HC)$ et $\o_V^\rG = \o_V$.\finl
\end{exemple}

\bigskip
%
%

\subsection{Lieu de ramification} 
La proposition suivante est certainement classique (et n\'ecessite le fait que $R$ soit 
int\'egralement clos)~:

\bigskip

\begin{prop}\label{codimension un}
Supposons l'alg\`ebre $K\HC$ {\bfit d\'eploy\'ee}. Alors il existe un id\'eal $\aG$ de $R$ v\'erifiant les 
deux propri\'et\'es suivantes~:
\begin{itemize}
\itemth{1} $\Spec(R/\aG)$ est vide ou purement de codimension $1$ dans $\Spec(R)$~;

\itemth{2} Si $\rG$ est un id\'eal premier de $R$, alors $\blocs(R_\rG \Zrm(\HC))=\blocs(K\Zrm(\HC))$ si et seulement 
si $\aG \not\subset \rG$. 
\end{itemize}
\end{prop}

\begin{proof}
Soit $(b_1,\dots,b_n)$ une $R$-base de $\HC$ et soit $\blocs(K\Zrm(\HC))=\{e_1,\dots,e_l\}$ avec $l=|\blocs(\Zrm(\HC))|$. 
Fixons un id\'eal premier $\rG$ de $R$. On \'ecrit 
$$e_i=\sum_{j=1}^n k_{ij} \, b_j$$
avec $k_{ij} \in K$. Alors $\blocs(R_\rG \Zrm(\HC))=\blocs(K\Zrm(\HC))$ si et seulement si 
$$\forall~1 \le i \le l,~\forall~1 \le j \le n,~k_{ij} \in R_\rG.\leqno{(\clubsuit)}$$
Si $k \in K$, on pose $\aG_k=\{r \in R~|~rk \in R\}$. 
Alors $\aG_k$ est un id\'eal de $R$ et, si $\rG$ est un id\'eal premier de 
$R$, alors $k \in R_\rG$ si et seulement si $\aG_k \not\subset \rG$. Posons 
$$\aG=\prod_{\substack{1 \le i \le l \\ 1 \le j \le n}} \aG_{k_{ij}}.$$
Ainsi, $(\clubsuit)$ devient \'equivalent \`a $\aG \not\subset \rG$. 
Cela d\'emontre l'assertion (2). 

\medskip

Montrons maintenant que $\Spec(R/\aG)$ est vide ou purement de codimension $1$ dans $\Spec(R)$. 
Pour cela, il suffit de montrer que $\Spec(R/\aG_k)$ est vide ou purement de codimension $1$ dans $\Spec(R)$. 
Si $k \in R$, alors $\aG_k=R$ et il n'y rien \`a montrer. Supposons donc que $k \not\in R$, 
et montrons qu'alors $\Spec(R/\aG_k)$ est purement de codimension $1$ dans $\Spec(R)$. 
Soit $\rG$ un id\'eal premier minimal de $R$ contenant $\aG_k$. Alors $k \not\in R_\rG$. 
Il nous faut montrer que $\rG$ est de hauteur $1$. Mais, puisque $R$ est int\'egralement clos, 
il en est de m\^eme de $R_\rG$~: donc $R_\rG$ est l'intersection des localis\'es $R_{\rG'}$, 
o\`u $\rG'$ parcours l'ensemble des id\'eaux premiers de hauteur $1$ de $R$ contenus dans $\rG$ 
(voir~\cite[th\'eor\`eme 11.5]{matsumura}). Donc il existe un id\'eal premier $\rG'$ de $R$ de 
hauteur $1$ contenu dans $\rG$ et tel que $k \not\in R_{\rG'}$. Ainsi $\aG_k \subset \rG' \subset \rG$ et 
la minimalit\'e de $\rG$ implique que $\rG=\rG'$, c'est-\`a-dire que $\rG$ est de hauteur $1$.
\end{proof}

\bigskip

\section{Matrices de d\'ecomposition}\label{section:decomposition}

\medskip

\def\charpol{{\mathrm{Char}}}
\def\reduction{{\mathrm{r\acute{e}d}}}

Soit $R_1$ une $R$-alg\`ebre commutative et soit $\rG_1$ un id\'eal premier de $R_1$. 
On pose $R_2=R_1/\rG_1$, $K_1=\Frac(R_1)$ et $K_2=\Frac(R_2)=k_{R_1}(\rG_1)$. 
Notons $\FC(\HC,K_1[\tb])$ l'ensemble des applications $\HC \to K_1[\tb]$. 
Si $V$ est un $K_1\HC$-module de type fini et si $h \in \HC$, nous noterons $\charpol_{K_1}^V(h)$ le 
polyn\^ome caract\'eristique de $h$ pour son action sur le $K_1$-espace vectoriel de dimension finie $V$. 
Ainsi, $\charpol_{K_1}^V \in \FC(\HC,K_1[\tb])$. D'autre part, $\charpol_{K_1}^V$ ne d\'epend que de la classe de $V$ 
dans le groupe de Grothendieck $\groth(K_1\HC)$. Cela d\'efinit donc une application 
$$\charpol_{K_1} : \groth^+(K_1\HC) \longto \FC(\HC,K_1[\tb]),$$
o\`u $\groth^+(K_1\HC)$ d\'esigne le sous-mono\"{\i}de de $\groth(K_1\HC)$ form\'e des classes 
d'isomorphie de $K_1\HC$-modules de type fini. Il est bien connu que $\charpol_{K_1}$ est 
injective~\cite[proposition~2.5]{geck rouquier}. 

Nous dirons que le couple $(R_1,\rG_1)$ v\'erifie la propri\'et\'e 
$\propdec$ si les trois assertions suivantes sont satisfaites~:
\begin{quotation}
\begin{itemize}
\item[(D1)] $R_1$ est noeth\'erien, int\`egre.

\item[(D2)] Si $h \in R_1\Hb$ et si $V$ est un $K_1\Hb$-module simple, alors 
$\charpol_{K_1}^V(h) \in R_1[\tb]$ (notons que cette propri\'et\'e est automatiquement satisfaite 
si $R_1$ est int\'egralement clos). 

\item[(D3)] Les alg\`ebres $K_1\Hb$ et $K_2\Hb$ sont d\'eploy\'ees.
\end{itemize}
\end{quotation}
On note $\reduction_{\rG_1} : \FC(\HC,R_1[\tb]) \longto \FC(\HC,R_2[\tb])$ la r\'eduction modulo $\rG_1$. 
D'autre part, sous l'hypoth\`ese (D3), si $K_2'$ est une extension de $K_2$, l'extension des scalaires 
induit un isomorphisme $\groth(K_2\HC) \longiso \groth(K_2'\HC)$, et nous identifierons 
donc ces deux groupes de Grothendieck.

\bigskip

\begin{prop}[Geck-Rouquier]\label{prop:geck-rouquier}
Si $(R_1,\rG_1)$ v\'erifie $\propdec$, alors il existe une unique application 
$\dec_{R_2\HC}^{R_1\HC} : \groth(K_1\HC) \longto \groth(K_2\HC)$  \indexnot{d}{\dec_{R_2\HC}^{R_1\HC}}  rendant le diagramme 
suivant
$$\diagram
\groth(K_1\HC) \rrto^{\DS{\charpol_{K_1}}} \ddto_{\DS{\dec_{R_2\HC}^{R_1\HC}}} 
&& \FC(\HC,R_1[\tb]) \ddto^{\DS{\reduction_{\rG_1}}} \\
&&\\
\groth(K_2\HC) \rrto^{\DS{\charpol_{K_2}}} && \FC(\HC,K_2[\tb])
\enddiagram$$
commutatif. Si $\OC_1$ est un sous-anneau de $K_1$ contenant $R_1$, si $\mG_1$ est un id\'eal 
de $\OC_1$ tel que $\mG_1 \cap R_1 = \rG_1$, et si $\LC$ est un $\OC_1\HC$-module qui est $\OC_1$-libre 
et de type fini, alors $k_{\OC_1}(\mG_1)$ est une extension de $K_2$ et 
$$\dec_{R_1\HC}^{R_2\HC} \isomorphisme{K_1\LC}_{K_1\HC} = \isomorphisme{k_{\OC_1}(\mG_1) \LC}_{k_{\OC_1}(\mG_1)\HC}.$$
\end{prop}

\begin{proof}
Cette proposition est d\'emontr\'ee dans~\cite[proposition~2.11]{geck rouquier} lorsque $R_1$ est int\'egralement 
clos. Nous allons d\'eduire le cas g\'en\'eral du cas particulier. 
Si on suppose seulement que (D2) est vraie, notons $R_1'$ la cl\^oture int\'egrale de $R_1$ dans 
$K_1$. Puisque $R_1'$ est entier sur $R_1$, il existe un id\'eal premier $\rG_1'$ de $R_1'$ tel que 
$\rG_1' \cap R_1 = \rG_1$. Posons $R_2'=R_1'/\rG_1'$. Alors $k_{R_1'}(\rG_1')$ est une extension de $k_{R_1}(\rG_1)$ 
donc $k_{R_1'}(\rG_1)\HC$ est d\'eploy\'ee, ce qui signifie que, d'apr\`es~\cite[proposition~2.11]{geck rouquier}, 
$\dec_{R_2'\HC}^{R_1'\HC} : \groth(K_1\HC) \longto \groth(k_{R_1'}(\rG_1')\HC)$ est bien d\'efinie 
et v\'erifie les propri\'et\'es souhait\'ees. On d\'efinit alors 
$\dec_{R_2\HC}^{R_1\HC}$ en utilisant l'isomorphisme $\groth(k_{R_1'}(\rG_1')\HC) \simeq \groth(K_2\HC)$ 
et il est facile de v\'erifier que cette application v\'erifie les propri\'et\'es attendues.
\end{proof}

Il s'en suit un r\'esultat de transitivit\'e imm\'ediat~\cite[proposition~2.12]{geck rouquier}~:

\bigskip

\begin{coro}[Geck-Rouquier]\label{geck rouquier}
Soit $R_1$ une $R$-alg\`ebre, soit $\rG_1$ un id\'eal premier de $R_1$ et soit $\rG_2$ un id\'eal 
premier de $R_2=R_1/\rG_1$. On suppose que $(R_1,\rG_1)$ et $(R_2,\rG_2)$ v\'erifient $\propdec$ 
et on pose $R_3=R_2/\rG_2$. 
Alors 
$$\dec_{R_3\HC}^{R_1\HC} = \dec_{R_3\HC}^{R_2\HC} \circ \dec_{R_2\HC}^{R_1\HC}.$$
\end{coro}

\bigskip
%
%
%
%
%
%
%
%
%
%
%
%
%

%
%
%

\section{Idempotents et caract\`eres centraux}\label{section:relevements-idempotents}

\bigskip

Le but de cette section est de compl\'eter la preuve de la proposition~\ref{muller plus}. 
Soit $\OC$ un anneau local noeth\'erien et soit 
$A$ une sous-$\OC$-alg\`ebre de $\OC^d=\OC \times \OC \times \cdots \times \OC$ ($d$ fois). 
On note $\mG=J(\OC)$, $k=\OC/\mG$ et, si $r \in \OC$, on note $\rba$ son image dans $k$.

Si $1 \le i \le d$, on note $\pi_i : \OC^d \to \OC$ la $i$-\`eme projection et 
$$\o_i : A \longto \OC$$
d\'esigne la restriction de $\pi_i$ \`a $A$. On pose
$$\fonction{\bar{\o}_i}{A}{k}{a}{\overline{\o_i(a)}.}$$
Sur l'ensemble $\{1,2,\dots,d\}$, on note $\smile$ la relation d'\'equivalence 
d\'efinie par 
$$\text{$i \smile j$ \quad si et seulement si \quad $\bar{\o}_i=\bar{\o}_j$.}$$
Pour finir, on pose
$$e_i=(0,\dots,0,\underbrace{1}_{\text{$i$-\`eme position}},0,\dots,0) \in \OC^d.$$
Alors~:

\bigskip

\begin{lem}\label{lem:relevement-idempotent}
Soit $I \in \{1,2,\dots,d\}/\smile$. Alors $\sum_{i \in I} e_i \in A$.
\end{lem}

\bigskip

\begin{proof}
Quitte \`a r\'eordonner, on peut supposer que $I=\{1,2,\dots,d'\}$ avec $d' \le d$. 
Proc\'edons par \'etapes~:

$$\text{\it Si $i \in I$ et $j \not\in I$, alors il existe $a_{ij} \in A$ 
tel que $\o_i(a_{ij})=1$ et $\o_j(a_{ij})=0$.}\leqno{(\clubsuit)}$$

\medskip

\begin{quotation}
\noindent{\it Preuve de $(\clubsuit)$.} 
Puisque $i \not\smile j$, il existe $a \in A$ tel que $\bar{\o}_i(a) \neq \bar{\o}_j(a)$. 
Posons $r =\o_j(a)$ et $u = \o_i(a)-\o_j(a)$. Alors $u \in \OC^\times$ car $\OC$ est local 
et $a_{ij}=u^{-1}(a - r \cdot 1_A) \in A$ v\'erifie les conditions.\finl
\end{quotation}

$$\text{\it Il existe $a_1 \in A$ tel que $\o_1(a_1)=1$ et $\o_j(a_1)=0$ si $j \not\in I$.}
\leqno{(\diamondsuit)}$$

\medskip

\begin{quotation}
\noindent{\it Preuve de $(\diamondsuit)$.} 
D'apr\`es $(\clubsuit)$, il existe, pour tous $i \in I$ et $j \not\in J$, 
$a_{ij} \in A$ tel que $\o_i(a_{ij})=1$ et $\o_j(a_{ij})=0$. Notons que, 
si $i' \in I$, alors $\o_{i'}(a_{ij}) \equiv 1 \mod \mG$ car $\bar{\o}_i=\bar{\o}_{i'}$. 
Posons 
$a=\prod_{i \in I, j \not\in J} a_{ij}$. Alors il est clair 
que $\o_j(a)=0$ si $j \not\in I$ et $\o_i(a) \equiv 1 \mod \mG$ si $i \in I$. 
Il suffit alors de prendre $a_1 = \o_1(a)^{-1} a$.\finl
\end{quotation}

\bigskip

On d\'efinit alors par r\'ecurrence la suite $(a_i)_{1 \le i \le d'}$~: 
$$a_{i+1}=a_i^2 (1 + \o_{i+1}(a_i)^{-2}(1-a_i^2)).$$
on va montrer par r\'ecurrence sur $i \in \{1,2,\dots,d'\}$ les deux faits suivants~:
$$\text{\it L'\'el\'ement $a_i$ est bien d\'efini et appartient \`a $A$.}\leqno{(\heartsuit_i)}$$
$$\text{\it Si $1 \le i' \le i$ et $j \not\in I$, alors $\o_{i'}(a_i)=1$ et $\o_j(a_i)=0$.}
\leqno{(\spadesuit_i)}$$

\bigskip

\begin{quotation}
\noindent{\it Preuve de $(\heartsuit_i)$ et $(\spadesuit_i)$.} 
C'est clair pour $i=1$. On raisonne donc par r\'ecurrence en supposant que 
$(\heartsuit_i)$ et $(\spadesuit_i)$ sont v\'erifi\'ees (avec $i \le d'-1$). Montrons 
$(\heartsuit_{i+1})$ et $(\spadesuit_{i+1})$. 

Alors $i \smile i+1$ et donc $\o_{i+1}(a_i) \equiv \o_i(a_i)=1 \mod \mG$. 
Donc $\o_{i+1}(a_i)$ est inversible et donc $a_{i+1}$ est bien d\'efini et appartient \`a $A$ 
(ce qui est exactement $(\heartsuit_{i+1})$). 

D'autre part, posons pour simplifier $r=\o_{i+1}(a_i)$. Alors~:

$\bullet$ Si $1 \le i' \le i$, on a $\o_{i'}(a_{i+1})=1 \cdot (1 + r^{-2}(1-1^2))=1$.

$\bullet$ $\o_{i+1}(a_{i+1}) = r^2 (1 + r^{-2}(1-r^2))=1$.

$\bullet$ Si $j \not\in I$, alors $\o_j(a_{i+1})= 0 \cdot (1 + r^{-2}(1-0^2))=0$.

\noindent D'o\`u le r\'esultat.\finl
\end{quotation}

\bigskip

Ainsi, $a_{d'} = \sum_{i \in I} e_i \in A$.
\end{proof}

\bigskip

\begin{coro}\label{coro:blocs-A}
L'application 
$$\fonctio{\{1,2,\dots,d\}/\smile}{\blocs(A)}{I}{\DS{\sum_{i \in I} e_i}}$$
est bien d\'efinie et bijective.
\end{coro}

\begin{proof}
Le lemme~\ref{lem:relevement-idempotent} montre que, si $I \in \{1,2,\dots,d\}/\smile$, 
alors $e_I=\sum_{i \in I} e_i \in A$. Si $e_I$ n'est pas primitif, cela signifie, puisque $\OC$ est local, qu'il 
existe deux parties non vides $I_1$ et $I_2$ de $I$ telles que $e_{I_1}$, $e_{I_2} \in A$, 
et $I=I_1 \coprod I_2$. Mais, si $i_1 \in I_1$ et $i_2 \in I_2$, alors 
$\omeba_{i_1}(e_{I_1})=1 \neq 0 = \omeba_{i_2}(e_{I_1})$, ce qui est impossible car 
$i_1 \smile i_2$. Donc l'application $I \mapsto e_I$ est bien d\'efinie. 
Sa bijectivit\'e est claire.
\end{proof}

\bigskip

Si $a \in A$, on note $\ahat$ son image dans $kA=k \otimes_\OC A$.

\bigskip

\begin{coro}\label{coro:relevement-idempotent}
Avec ces notations, on a~:
\begin{itemize}
\itemth{a} Si $e \in \blocs(A)$, alors $\ehat \in \blocs(kA)$.

\itemth{b} L'application $\blocs(A) \to \blocs(kA)$, $e \mapsto \ehat$ 
est bijective.
\end{itemize}
\end{coro}

\begin{proof}
(a) Soit $e \in \blocs(A)$ et supposons que $\ehat=e_1+e_2$, 
o`u $e_1$ et $e_2$ sont des idempotents centraux orthogonaux de $kA$. 
L'anneau $\OC$ \'etant noeth\'erien, $kA$ est une $k$-alg\`ebre commutative de dimension finie. 
Donc il existe deux morphismes de $k$-alg\`ebres $\r_1$, $\r_2 : kA \to k'$ (o\`u $k'$ 
est une extension finie de $k$) tels que $\r_i(e_j)=\d_{i,j}$. Notons 
$\rhot_i$ la composition $A \to kA \stackrel{\r_i}{\longrightarrow} k'$. 

Posons $\aG_i=\Ker(\rhot_i)$. L'image de $\r_i$ \'etant 
un sous-corps de $k'$, $\aG_i$ est un id\'eal maximal de $A$. Puisque $\OC^d$ est entier 
sur $A$, il existe un id\'eal maximal $\mG_i$ de $\OC^d$ tel que $\aG_i = \mG_i \cap A$. 
Puisque $\OC$ est local, $\mG_i$ est de la forme 
$\OC \times \cdots \times \OC \times \mG \times \OC \times \cdots \times \OC$, o\`u $\mG$ est en $t_i$-i\`eme 
position (pour un $t_i \in \{1,2,\dots,d\}$), ce qui implique 
que $\rhot_i = \omeba_{t_i}$. 

Puisque $\r_1 \neq \r_2$ et $\r_i(e_j)=\d_{i,j}$, on obtient $\omeba_{t_1} \neq \omeba_{t_2}$ 
et $\omeba_{t_1}(e)=\r_1(e_1+e_2)=1=\r_2(e_1+e_2)=\omeba_{t_2}(e)$. Cela contredit 
le corollaire~\ref{coro:blocs-A}.
\end{proof}

\bigskip

Dans le cours de cette preuve, le r\'esultat suivant a \'et\'e d\'emontr\'e~: si $k'$ est une extension 
finie de $k$ et si $\r : kA \to k'$ est un morphisme de $k$-alg\`ebres, alors 
\equat\label{L}
\text{\it il existe $i \in \{1,2,\dots,d\}$ tel que $\r(\ahat)=\omeba_i(a)$ pour tout $a \in A$.}
\endequat

\newpage

\printindex

%
%
%
%
%
%
%

\end{document}